%% file: thesislenhardt.tex
\documentclass[ twoside,openright,titlepage,numbers=noenddot,headinclude,liststotoc, bibtotoc
                footinclude=true,cleardoublepage=empty,abstractoff, 
                BCOR=5mm,paper=a4,fontsize=11pt,
                ngerman,american%
                ]{scrreprt}
           
             \input{classicthesis-config}

\title{The Baum-Connes conjecture, Swan group actions and controlled topology}
\author{Fabian Lenhardt}

\begin{document}
\setlength\parindent{0pt}
\setcounter{chapter}{-1}
\frenchspacing
\raggedbottom
\selectlanguage{american} 
\pagestyle{plain}

\include{misc/titlepage}

\include{misc/dedication}


\include{misc/Acknowledgments}

\include{chapters/Preface}

\tableofcontents

\include{chapters/algebras}

\include{chapters/control}

\include{chapters/ktheoryofcstarcat}

\include{chapters/homology}

\include{chapters/swan}

\include{chapters/transfer}

\include{chapters/stability}

\include{chapters/rationally}

\include{chapters/inductionresults}

\include{chapters/farrell-hsiang}

\include{chapters/outlook} 
\bibliographystyle{alpha}
\bibliography{bib}

\include{misc/abstract}

\end{document}

%% file: classicthesis-config.tex
\PassOptionsToPackage{eulerchapternumbers,
				listings,
				 pdfspacing, floatperchapter,
				 subfig,beramono,
				 eulermath,parts}{classicthesis}										

\usepackage{ifthen}
\usepackage[]{currvita}
\newboolean{enable-backrefs} 
\setboolean{enable-backrefs}{false} 

\newcommand{\myTitle}{The Baum-Connes conjecture, Swan group actions and controlled topology}

\newcommand{\myName}{Fabian Lenhardt\xspace}

\newcommand{\myFaculty}{Fachbereich Mathematik und Informatik \xspace}

\newcommand{\myUni}{Freie Universit\"at Berlin\xspace}

\newcommand{\myVersion}{\xspace}

\newcounter{dummy} 
\providecommand{\mLyX}{L\kern-.1667em\lower.25em\hbox{Y}\kern-.125emX\@}

\PassOptionsToPackage{latin9}{inputenc}	
 \usepackage{inputenc}				

 \usepackage{babel}					

\PassOptionsToPackage{square,numbers}{natbib}
 \usepackage{natbib}				

\PassOptionsToPackage{fleqn}{amsmath}		
 \usepackage{amsmath}

\PassOptionsToPackage{T1}{fontenc} 
	\usepackage{fontenc}                 
\usepackage{xspace} 
\usepackage{mparhack} 
\usepackage{fixltx2e} 
\PassOptionsToPackage{printonlyused,smaller}{acronym}
	\usepackage{acronym} 

\usepackage{tabularx} 
\usepackage{caption}
\usepackage{subfig}  

\usepackage{listings} 
\PassOptionsToPackage{pdftex,hyperfootnotes=false,pdfpagelabels}{hyperref}
	\usepackage{hyperref}  
\PassOptionsToPackage{pdftex}{graphicx}
	\usepackage{graphicx} 

\newcommand{\backrefnotcitedstring}{\relax}
\newcommand{\backrefcitedsinglestring}[1]{(Cited on page~#1.)}
\newcommand{\backrefcitedmultistring}[1]{(Cited on pages~#1.)}
\ifthenelse{\boolean{enable-backrefs}}%
{%
		\PassOptionsToPackage{hyperpageref}{backref}
		\usepackage{backref} 
		   \renewcommand*{\backref}[1]{}  
		   \renewcommand*{\backrefalt}[4]{
		      \ifcase #1 %
		         \backrefnotcitedstring%
		      \or%
		         \backrefcitedsinglestring{#2}%
		      \else%
		         \backrefcitedmultistring{#2}%
		      \fi}%
}{\relax}    

\hypersetup{%
    colorlinks=true, linktocpage=true, pdfstartpage=3, pdfstartview=FitV,%
    breaklinks=true, pdfpagemode=UseNone, pageanchor=true, pdfpagemode=UseOutlines,%
    plainpages=false, bookmarksnumbered, bookmarksopen=true, bookmarksopenlevel=1,%
    hypertexnames=true, pdfhighlight=/O,
    urlcolor=webbrown, linkcolor=RoyalBlue, citecolor=webgreen, 
    pdftitle={\myTitle},%
    pdfauthor={\textcopyright\ \myName, \myUni, \myFaculty},%
    pdfsubject={},%
    pdfkeywords={},%
    pdfcreator={pdfLaTeX},%
    pdfproducer={LaTeX with hyperref and classicthesis}%
}   

\makeatletter
\@ifpackageloaded{babel}%
    {%
       \addto\extrasamerican{%
				}%
       \addto\extrasngerman{%
				}%
    }{\relax}
\makeatother

\usepackage{classicthesis} 
\usepackage{amsthm}
\usepackage{amssymb}
\usepackage{amscd}
\usepackage{amsfonts}  
\usepackage{makeidx}
\usepackage{hyperref}
\usepackage[all]{xy}
\usepackage{makeidx}     
\usepackage{graphicx}    
\usepackage{multicol}    
\usepackage{xspace}
\usepackage{ifthen}
\usepackage{verbatim}
\usepackage{xkeyval}
\usepackage{fixme}
\usepackage{url}
\usepackage{color}
\usepackage{titlesec}
\DeclareMathAlphabet\EuR{U}{eur}{m}{n}
\SetMathAlphabet\EuR{bold}{U}{eur}{b}{n}

\newcommand{\cala}{{\cal A}}
\newcommand{\calb}{{\cal B}}

\newcommand{\cald}{{\cal D}}
\newcommand{\cale}{{\cal E}}
\newcommand{\calf}{{\cal F}}
\newcommand{\calg}{{\cal G}}
\newcommand{\calh}{{\cal H}}

\newcommand{\calo}{{\cal O}}
\newcommand{\calp}{{\cal P}}

\newcommand{\calv}{{\cal V}}

\newcommand{\IC}{{\mathbb C}}

\newcommand{\IK}{{\mathbb K}}

\newcommand{\IN}{{\mathbb N}}

\newcommand{\IQ}{{\mathbb Q}} 
\newcommand{\IR}{{\mathbb R}}

\newcommand{\IZ}{{\mathbb Z}}

\newcommand{\alg}{\operatorname{alg}}

\newcommand{\coker}{\operatorname{coker}}
\newcommand{\colim}{\operatorname{colim}}

\newcommand{\hocolim}{\operatorname{hocolim}}

\newcommand{\id}{\operatorname{id}}
\newcommand{\im}{\operatorname{im}}
\newcommand{\ind}{\operatorname{ind}}
\newcommand{\inj}{\operatorname{inj}}

\renewcommand{\Im}{\operatorname{im}}

\newcommand{\map}{\operatorname{map}}
\newcommand{\Map}{\operatorname{Map}}

\newcommand{\res}{\operatorname{res}}

\newcommand{\Sw}{\operatorname{Sw}}

\newcommand{\tr}{\operatorname{tr}}

\newcommand{\End}{\operatorname{End}}

\newcommand{\Id}{\operatorname{Id}}

\newcommand{\C}{\mathcal{C}}
\newcommand{\D}{\mathcal{D}}
\newcommand{\Idem}{\operatorname{Idem}}

\newcommand{\Hom}{\operatorname{Hom}}
\newcommand{\supp}{\operatorname{supp}}

\newcommand{\Mod}{\operatorname{mod}}
\newcommand{\Aut}{\operatorname{Aut}}

\newcommand{\Or}{\operatorname{Or}}
\newcommand{\Grp}{\operatorname{GROUPOIDS}}
\newcommand{\Sp}{\operatorname{SPECTRA}}
\newcommand{\Top}{\operatorname{TOP}}
\newcommand{\minus}{\operatorname{-}}
\newcommand{\ev}{\operatorname{ev}}
\newcommand{\euc}{\operatorname{euc}}
\newcommand{\fin}{\mathcal{FIN}}

\newtheorem{lem}{Lemma}[section]

\newtheorem{thm}[lem]{Theorem}
\newtheorem*{thm*}{Theorem}
\newtheorem{prop}[lem]{Proposition}

\theoremstyle{remark}
\newtheorem{rem}[lem]{Remark}
\newtheorem{example}[lem]{Example}

\theoremstyle{definition}
\newtheorem{definition}[lem]{Definition}
\newtheorem*{definition*}{Definition}



%% file: misc/titlepage.tex

\begin{titlepage}
	\begin{addmargin}[-1cm]{-3cm}
    \begin{center}
        \large  

        \hfill

        \vfill

        \begingroup
          \spacedallcaps{\myTitle} \\ \bigskip
        \endgroup
        
        \vspace{3cm}
        \begingroup
        \textbf{Inauguraldissertation}
        \endgroup
	
	\begingroup
	zur Erlangung des Doktorgrades (Dr. rer. nat.)
	\endgroup
        
        \vspace{1cm}
        
        \begingroup
        am Fachbereich Mathematik und Informatik \\
        der Freien Universit\"at Berlin
        \endgroup
        
        \vspace{1cm}
        \begingroup
        vorgelegt von \\
        \endgroup
        
        \spacedlowsmallcaps{\myName}
	
	\vspace{3 cm}
	
	\begingroup
	Berlin 2013
	\endgroup

        \vfill                      

    \end{center}  
  \end{addmargin}       
\end{titlepage}   

%% file: misc/dedication.tex
\thispagestyle{empty}
\refstepcounter{dummy}

\vspace*{16cm}

\vspace{10ex}

Erstgutachter: Professor Holger Reich \\
Zweitgutachter: Professor Wolfgang L\"uck \\

Tag der Disputation: 15.01.2013

%% file: misc/Acknowledgments.tex

\begin{flushright}{\slshape    
  Seek simplicity, and distrust it.
} \\
---\emph{A.N. Whitehead}
\end{flushright}

\bigskip

\begingroup
\let\clearpage\relax
\let\cleardoublepage\relax
\let\cleardoublepage\relax
\chapter*{Acknowledgements}
First and foremost, I would like to thank my advisor Holger Reich for suggesting the topic of this thesis and for many fruitful discussions, and for answering many questions. \\
Furthermore, I would like to thank the Graduiertenkolleg 1150 "Homotopy and Cohomology" for their financial support and the lectures and seminars they organized. After moving to Berlin, the Berlin Mathematical School provided both financial and organizational support, for which I am grateful. \\
I would also like to thank Mark Ullmann for many mathematical discussions, and Ferit Deniz, Henrik R\"uping and Wolfgang Steimle for their participation in a joint seminar on the Farrell-Jones conjecture.
\endgroup

%% file: chapters/Preface.tex
\chapter{Preface}

In geometric topology, the algebraic $K$- and $L$-groups of the group ring $\IZ \pi_{1}X$ of the fundamental group of some space $X$ play a fundamental role as the domain of many obstructions. For example, the $s$-cobordism theorem tells us that all $h$-cobordisms over a given manifold $M$ are trivial if and only if the Whitehead group of $\pi_{1}M$, a quotient of $K_{1}(\IZ \pi_{1}(M))$, vanishes. \\
Unfortunately, the $K$-and $L$-groups of group rings are extremely hard to compute. The Farrell-Jones conjecture is a structured approach to obtain information about these groups. In its simplest form for algebraic $K$-theory, it predicts that for torsion-free groups $G$, the so-called assembly map is an isomorphism
\[
H_{*}(BG, \IK^{-\infty}\IZ) \rightarrow K_{*}(\IZ G)
\]
where $ \IK^{-\infty} \IZ$ is the $K$-theory spectrum of $\IZ$ and $H_{*}(-; \IK^{-\infty} \IZ)$ is the associated homology theory. The left-hand side is quite accessible: Once one knows the group homology of $G$ and some $K$-groups of $\IZ$, one can attack it with the Atiyah-Hirzebruch spectral sequence. \\
This simple form of the Farrell-Jones conjecture already fails when $G$ is finite, or infinite cyclic when one also considers group rings $RG$ over an arbitrary ring $R$ due to nil group phenomena. We will now explain the general form of the Farrell-Jones conjecture with coefficients in a ring, for which no counterexample is known. Associated to $R$, there is a $G$-homology theory $H_{*}^{G}(-, \IK R)$ such that
\[
H_{*}^{G}(*, \IK R) \cong K_{*}(RG)
\] 
The simple version of the Farrell-Jones conjecture is now obtained by plugging the map $EG \rightarrow pt$ into the homology theory to obtain the assembly map
\[
H_{*}(BG, \IK R)  \cong H_{*}^{G}(EG, \IK R)  \rightarrow  H_{*}^{G}(*, \IK^{-\infty} R) = K_{*}(RG)
\]
where $\IK^{-\infty} R$ is the non-connective $K$-theory spectrum of the ring $R$. We refer the reader to \cite{KThandbook} for further information on $G$-homology theories. Since the simple version fails for some groups, we have to account for these failures in general. This is encoded into the conjecture by considering not the space $EG$, which classifies free actions, but the space $E_{\mathcal{VCYC}}G$ which classifies $G$-actions with virtually cyclic isotropy. So the Farrell-Jones assembly map is the map
\[
 H_{*}^{G}(E_{\mathcal{VCYC}}G, \IK^{-\infty} R)  \rightarrow  H_{*}^{G}(*, \IK R) \cong K_{*}(RG)
\]
which is predicted to be an isomorphism.  In some sense, one pushes the problems with finite and infinite cyclic groups into the left-hand side of the conjecture.  The left-hand side is now more complicated, but still accessible to Atiyah-Hirzebruch spectral sequence type methods, and it has the advantage that now the conjecture stands a chance to be true. We refer the reader to \cite{KThandbook} for an overview over the Farrell-Jones conjecture. \\
In the world of $C^{*}$-algebras, the Baum-Connes conjecture is formally very similar to the Farrell-Jones conjecture. It can be formulated as follows. Associated to a $C^{*}$-algebra $A$, there is an equivariant homology theory $H_{*}^{G}(-, KA)$ such that
\[
H_{*}^{G}(*, KA) = K^{\mathcal{TOP}}_{*}(C^{*}_{r}A)
\]
where $C^{*}_{r}A$ is the reduced group $C^{*}$-algebra of $G$, a certain completion of the ring $AG$ to a $C^{*}$-algebra, and $K^{\mathcal{TOP}}$ denotes the topological $K$-theory of a $C^{*}$-algebra. Then the Baum-Connes conjecture predicts that the assembly map 
\[
H_{*}^{G}(E_{\mathcal{VCYC}}G, KA)  \rightarrow  H_{*}^{G}(*, KA) = K_{*}(C^{*}_{r}A)
\]
is an isomorphism. Typically, the Baum-Connes conjecture is formulated with the classifying space $E_{\mathcal{FIN}}G$ of proper $G$-actions, but since topological $K$-theory is better behaved than algebraic $K$-theory and in particular has no nil group phenomena, this does not make a difference. For general information on the Baum-Connes conjecture, we refer the reader to \cite{KThandbook} or \cite{valette}.\\
Both the Farrell-Jones and Baum-Connes conjecture are known for large classes of groups. The proofs of the Baum-Connes conjecture rely on equivariant $KK$-theory, see the original \cite{Kasparov} and, for example, \cite{Blackadar} and \cite{kkprimer}, and the Dirac-dual Dirac method, see for example \cite{trout}. Actual proofs of the conjecture are contained in \cite{haagerup} for groups with the Haagerup property, in \cite{mineyev} for hyperbolic groups and in \cite{lafforgue} for certain groups which have property $T$. \\
In recent years, there has been a lot of progress on the Farrell-Jones conjecture based on the geodesic flow techniques developed in \cite{kd1} and \cite{kd2}, starting with \cite{BFJR} and \cite{BR} for fundamental groups of closed Riemannian manifolds with negative sectional curvature. In \cite{BLR} and \cite{hypercover}, this is extended to hyperbolic groups, and in \cite{CAT0} and \cite{catflow} to groups acting cocompactly by isometries on a $CAT(0)$-spaces, at least for lower $K$-theory. The result is extended to all $K$-groups in \cite{wegner}. A somewhat different type of argument is given in \cite{FH} for certain groups, including for example crystallographic groups.\\
The strategy of the proofs of the Farrell-Jones and Baum-Connes conjectures are rather unrelated. The point of this work is to employ the strategy of one of the proofs of the Farrell-Jones conjecture, namely the one in \cite{FH}, to obtain a proof of the Baum-Connes conjecture for certain groups. This strategy has two main ingredients: Controlled algebra and induction theorems. \\
The basic idea of controlled algebra is to capture ``large-scale'' invariants of spaces endowed with some additional structure. For example, given a locally compact metric space $X$ and a ring $R$, we form the following category. Objects are sequences $(M_{x})_{x \in X}$ of free $R$-modules such that for each compact $K \subset X$, only finitely many $R_{k}$ with $k \in K$ are nonzero. A morphism $\phi \colon M \rightarrow N$ consists of morphisms of $R$-modules $\phi_{(x,y)} \colon R_{y} \rightarrow R_{x}$ such that for each $z \in X$, only finitely many $\phi_{(x,z)}$ and $\phi_{(z,y)}$ are nonzero; and furthermore we demand that $\phi$ is controlled: there is a constant $C > 0$ such that $\phi_{(x,y)} = 0$ whenever $d(x,y) > C$. This yields an additive category and we can take its algebraic $K$-theory. It turns out that this is well-behaved and yields something like a ``large-scale homology theory''. For example, this construction can distinguish the large-scale geometry of  $\IR^{n}$ for different values of $n$. \\
The properties of the metric $d$ which allow these kind of constructions can be axiomatized, yielding the notion of a coarse space. Furthermore, we can take a $G$-action on the space into account to obtain an equivariant variant of the additive category. The first step in the proofs of the Farrell-Jones conjecture essentially is to write down an equivariant  coarse space such that the algebraic $K$-groups whose construction we just indicated vanish if and only if the Farrell-Jones conjecture is true. \\
Similarly, one can associate topological $K$-groups to a coarse space, and the topological $K$-groups of the very same coarse space as in the Farrell-Jones conjecture are the only obstruction for the Baum-Connes conjecture to hold. Since the coarse space is quite explicit, this allows us to perform explicit constructions to obtain information about $K$-theory. \\
The other main ingredient, induction theory, is used to create more room for certain constructions. Given a finite group $G$, consider its complex representation ring $K_{0}(\IC G)$. For each subgroup $H \subset G$, we obtain an induction map
\[
\ind_{H}^{G} \colon K_{0}(\IC H) \rightarrow K_{0}(\IC G)
\] 
Given a collection $H_{i}$ of subgroups of $G$, one can ask under which conditions we obtain a surjection
\[
\oplus \ind_{H_{i}}^{G} \colon \oplus K_{0}(\IC H_{i}) \rightarrow K_{0}(\IC G)
\]
Some answers are provided by the Brauer and Artin induction theorems. These will play a crucial role, allowing us to create more room to make some constructions work. \\

Instead of looking at the ordinary Baum-Connes map
\[
H_{*}^{G}(E_{\mathcal{VCYC}}G, KA)  \rightarrow  H_{*}^{G}(*, KA) 
\]
we can take any family $\calf$ of subgroups of $G$ and consider the Baum-Connes map for the family $\calf$
\[
\operatorname{A}_{\calf}: H_{*}^{G}(E_{\calf}G, KA)  \rightarrow  H_{*}^{G}(*, KA)
\]
which yields a version of the Baum-Connes conjecture with respect to the family $\calf$, namely that $\operatorname{A}_{\calf}$ is an isomorphism.

Under the following condition on a group, we will verify the Baum-Connes conjecture with respect to the family $\calf$. The idea of the condition goes back to \cite{farrellhsiang} and was later on considered in \cite{quinn} and \cite{FH} in the context of the Farrell-Jones conjecture. Recall that an elementary group is a finite  group which is a product of a $p$-group for some prime $p$ and a cyclic group.

\begin{definition*}
Let $G$ be a group with a word metric $d_{G}$ and $\calf$ a family of subgroups of $G$.
We say that $G$ is a Farrell-Hsiang group with respect to $\calf$ if the following condition is satisfied: \\
There exists a natural number $N$ such that for each $n \in \IN$ there is a surjective group homomorphism $\alpha_{n}: G \rightarrow F_{n}$ with $F_{n}$ finite such that for each elementary subgroup $C \subset F_{n}$ and $H = \alpha_{n}^{-1}(C) \subset G$, there is a simplicial complex $E_{H}$ of dimension at most $N$ and a simplicial $H$-action with isotropy in $\calf$, and an $H$-equivariant map $f_{H}: G \rightarrow E_{H}$ such that for all $g,h \in G$ with $d_{G}(g,h) \leq n$, we have
\[
d^{1}_{E_{H}}(f_{H}(g), f_{H}(h)) \leq \frac{1}{n}
\] 
where $d^{1}_{E_{H}}$ is the $l^{1}$-metric on $E_{H}$.
\end{definition*}

In the context of the Farrell-Jones conjecture, elementary groups have to be replaced with hyperelementary groups, i.e. extensions of $p$-groups and finite cyclic groups since the integral analogue of the Brauer induction theorem requires hyperelementary, so this definition does not quite agree with the one in \cite{FH}. However, the difference is rather small and we use the induction results as a blackbox anyway, so we do not invent a new name for such groups.

Our main result is the following.

\begin{thm*}
If $G$ is a Farrell-Hsiang group with respect to $\calf$, then the Baum-Connes conjecture with respect to $\calf$ for $G$ is true.
\end{thm*}

To pass from the family version to the usual version of the Baum-Connes conjecture, there is the transitivity principle, which roughly states that if $G$ satisfies the Baum-Connes conjecture with respect to a family $\calf$ and all elements of $\calf$ satisfy the ordinary Baum-Connes conjecture, then also $G$ satisfies the ordinary Baum-Connes conjecture. The strategy employed in \cite{BFL} to prove the Farrell-Jones conjecture is essentially to prove that $G$ is a Farrell-Hsiang group with respect to a family for which we already know the conjecture by induction. To make up for the quite high level of abstraction, we have included an account of some of the results of \cite{BFL} in this thesis. Somewhat unfortunately, the groups considered in \cite{BFL} are amenable or the proof strategy also requires the Farrell-Jones conjecture for groups for which we do not know the Baum-Connes conjecture, and so we do not obtain new results on the Baum-Connes conjecture, but different proofs for already known results.

\section*{Notation and conventions}

Both algebraic and topological $K$-theory will play a role in this thesis. When we write an undecorated $K$, we will always mean topological $K$-theory, and use the notation $K^{\alg}$ when we have to refer to algebraic $K$-theory.  \\
We will also adopt a point of view which involves more category theory than the usual more analytic approaches to the Baum-Connes conjecture. We will usually not deal with $C^{*}$-algebras, but with $C^{*}$-categories and their $K$-theory. Actually, most of our arguments are not even about $C^{*}$-categories, but only about pre-$C^{*}$-categories: Categories with a norm which can be completed to a $C^{*}$-category. We will almost never explicitly work with an arbitrary morphism of a $C^{*}$-category; instead, all constructions take place in the uncompleted setting, where one has a better grasp of the morphisms, and are then completed to extend them to the whole $C^{*}$-category. \\
Throughout the thesis, $A$ will denote a $C^*$-algebra with a $G$-action, and we will consider the Baum-Connes conjecture with coefficients in $A$. This is in some sense a necessary evil: without coefficients, we do not have the inheritance properties on which parts of \cite{BFL} crucially rely on, and we want to apply the theory in \cite{BFL} to the Baum-Connes conjecture. However, coefficients add an additional layer of complication which distracts from the actual arguments. The reader who is interested in the actual proof and not in technical annoyances related to coefficients should concentrate on the case $A = \IC$ with the trivial $G$-action. \\
Aside from basic $C^*$-algebra theory, the only functional analysis we will use are some basics about Hilbert $C^*$-modules over $C^*$-algebras. Hilbert modules are generalizations of Hilbert spaces from $\IC$ to an arbitrary $C^*$-algebra $A$. The reader not familiar with Hilbert modules can either consult \cite{Lance} for some background, or can ignore coefficients and concentrate on $A = \IC$, replacing Hilbert modules with Hilbert spaces and ignoring the word adjointable whenever it appears.

\section*{Organization}

In the irst section, we give a short review of topological $K$-theory of $C^{*}$-algebras, with a special emphasis on representing $K$-groups as homotopy groups of a space. The sections 2, 3 and 4 set up the necessary topological $K$-theory machinery for $C^*$-categories, which is not developed in the literature in a way which fulfills all of our needs. Section 5 is a review of the necessary representation theory. The actual proof of the Baum-Connes conjecture, along the lines of the proof of the Farrell-Jones conjecture in \cite{FH}, is contained in the sections 6, 7 and 8. Section 6 introduces the basic strategy, section 7 is devoted to the transfer and section 8 to a stability result. Section 9 shows how to obtain rational results on the Baum-Connes conjecture using the Artin induction theorem. Section 10 is devoted to geometrically trivial examples which nevertheless yield interesting results. Finally,  section 11 is a short account of some of the results of \cite{BFL} to get our hands at some concrete, nontrivial examples of Farrell-Hsiang groups.

%% file: chapters/algebras.tex
\chapter{$C^{*}$-algebras and their $K$-theory}

This section is a short review of topological $K$-theory of $C^{*}$-algebras. The material is well-known and mostly well-documented in the literature, see for example \cite{HigsonRoe}, \cite{WO} or \cite{RLL}.  We will follow \cite{HigsonRoe}  since the exposition given there is more in the spirit of topo\-logy and algebraic $K$-theory and does not focus on explicit matrix manipulations so much. We will also consider a space-level version of topological $K$-theory, i.e. representing $K$-groups as homotopy groups of a certain space, which is well-known but not well-covered in the literature. \\
We assume a working knowledge of basic $C^*$-algebra theory as developed, for example, in \cite{rudin} or \cite{Murphy}. In particular, we will make use of the continuous function calculus for $C^*$-algebras. \\
Throughout this section, let $A$ be a $C^{*}$-algebra. Representing $A$ as an algebra of bounded operators on a Hilbert space $\calh$, we obtain a $*$-representation of the $n \times n$-matrices $M_{n}(A)$ over $A$ on $\calh^{n}$. This turns $M_{n}(A)$ into a $C^{*}$-algebra as well. We do not demand that $A$ has a unit; if it has, we call $A$ unital.

\fxnote{some reference for C*-algebras}
\section{Topological $K$-theory of $C^{*}$-algebras}

Let $A$ be a unital $C^*$-algebra. The cycles in the $K$-groups of $A$ will be given by projections and unitaries.

\begin{definition}
A projection in a $C^{*}$-algebra $A$ is a self-adjoint element $p$ with $p^{2} = p.$ A projection over $A$ is a projection in some matrix algebra $M_{n}(A)$.  
A unitary in $A$ is an element $u \in A$ such that $uu^{*} = u^{*}u = 1$. 
\end{definition}
\noindent Now we can immediately give the definition of the $K$-theory of $A$, starting with $K_{0}$. 

\begin{definition}
Let $A$ be a unital $C^{*}$-algebra. The group $K_{0}(A)$ is the free abelian group with one generator $[p]$ for each projection $p$ over $A$ modulo the following relations:
\begin{enumerate}
\item $[p]+[q] = [p \oplus q]$ where $\oplus$ denotes the direct block sum of matrices.
\item If $p,q \in M_{n}(A)$ are connected by a continuous path of projections in $M_{n}(A)$, then $[p] = [q]$.
\item $[0] = 0$ for any $0$-matrix.
\end{enumerate}
\end{definition}

\begin{rem}
Alternatively, one can define $K_{0}(A)$ as the Grothendieck group of the monoid of conjugacy classes of projections under the direct sum. It follows from \ref{kcat} and \ref{murneu} that this yields the same group $K_{0}(A)$.
\end{rem}

In the definition of $K_{0}(A)$,there are two other relations one could use instead of $[p] = [q]$ whenever $p$ and $q$ are connected by a path.

\begin{definition}
\label{murrayneumann}
Let $p, q \in M_{n}(A)$ be two projections.
\begin{enumerate}
\item We say that $p$ and $q$ are unitarily equivalent if there is a unitary $u \in M_{n}(A)$ with $upu^{*} = q$.
\item We say that $p$ and $q$ are Murray-von Neumann equivalent if there is $v \in M_{n}(A)$ such that $vv^{*} = p$ and $v^{*}v = q$.
\end{enumerate}
\end{definition}

\begin{prop}
\label{murneu}
Let $p,q$ be as above.
\begin{enumerate}
\item If $p$ and $q$ are connected by a path of projections inside $M_{n}(A)$, they are unitarily equivalent.
\item If $p$ and $q$ are unitarily equivalent, they are Murray-von Neumann equivalent.
\item If $p$ and $q$ are Murray-von Neumann equivalent, $p \oplus 0_{n}$ and $q \oplus 0_{n}$ are connected by a path of projections in $M_{2n}(A)$.
\end{enumerate}
It follows that we could have used any of the three relations in our definition of $K_{0}$.
\end{prop}

\begin{proof}
This is a standard fact; see for example \cite[4.1.11]{HigsonRoe}. We may assume that $p,q  \in A$ by renaming $M_{n}(A)$ if necessary. To prove i), it suffices to see that any two projections whose distance from one another is less than $1$ are unitarily equivalent; then one can cut the path connecting $p$ and $q$ into finitely many pieces, each of which has length $< 1$. \\
So assume $\left\Vert p-q \right\Vert < 1$. Let $x = qp+(1-q)(1-p)$. An easy calculation shows $xp = qx$ and $x-1 = (2q-1)(p-q)$. Since $\left\Vert p-q \right\Vert < 1$, $\left\Vert x-1 \right\Vert < 1$ as well and so $x$ is invertible. Now define a unitary $u$ as $u = x(x^{*}x)^{-\frac{1}{2}}$, compare \ref{unitaryvsinvertible}.
Since $xp = qx$, $p$ commutes with $x^{*}x$ and hence also with all elements obtained from $x^{*}x$ by functional calculus, in particular with $(x^{*}x)^{-\frac{1}{2}}$. It follows that $up = xp(x^{*}x)^{-\frac{1}{2}} = qx(x^{*}x)^{-\frac{1}{2}} = qu$ as desired. \\
If $up = qu$, setting $v = pu^{*}$ proves ii).  \\
Now let $v$ be given with $vv^{*} = p$ and $v^{*}v = q$. Let $ x = (1-vv^{*})v$. We compute
\begin{align*}
x^{*}x &= v^{*}(1-p)(1-p)v = v^{*}(1-p)v = q-v^{*}pv \\ 
&= q-v^{*}vv^{*}v = q-q^{2} = 0
\end{align*}
and hence $x = 0$ and $v = vv^{*}v$. Now consider the matrix 
\[ u = 
\begin{pmatrix}
v & 1-vv^{*} \\
v^{*}v-1 & v^{*}
\end{pmatrix}
\]
A straightforward calculation, using $v = vv^{*}v$, yields that this is a unitary matrix conjugating $p \oplus 0$ to $q \oplus 0$. Considering the path of  unitary matrices
\[ u('t) = 
\begin{pmatrix}
\cos(t)v & 1-(1-\sin(t))vv^{*} \\
(1-\sin(t))v^{*}v-1 & \cos(t)v^{*}
\end{pmatrix}
\]
we see that $v$ is connected by unitaries to 
\[
\begin{pmatrix}
0 & 1\\
-1 & 0
\end{pmatrix}
\]
which, in turn, is connected by unitaries to the identity matrix - just pick $v = 1$ in the above path. Concatenating these two paths, we obtain a path $u(t)$ connecting $u$ and the identity. It follows that $p \oplus 0$ and $q \oplus 0$ are connected by the path 
\[
t \mapsto u(t)(q \oplus 0)u(t)^{*} 
\]
This implies iii).
\end{proof}

\begin{prop}
The group $K_{0}(A)$ is canonically isomorphic to the algebraic $K$-group $K_{0}^{alg}(A)$.
\end{prop}

\begin{proof}
We defer the proof to \ref{kcat}, where it is a special case of iii).
\end{proof}

We now turn to the definition of $K_{1}$.

\begin{definition}
The group $K_{1}(A)$ for a unital $C^*$-algebra $A$ is the free abelian group with one generator $[u]$ for each unitary $u$ over $A$, modulo the following relations:
\begin{enumerate}
\item $[u]+[v] = [u \oplus v]$ for all unitaries $u, v$.
\item Whenever $u,v \in M_{n}(A)$ are connected by a path of unitaries in $M_{n}(A)$, we have $[u] = [v]$.
\item $[1_{n}]=0$ for all $n$.
\end{enumerate}
\end{definition}

\begin{prop}
\label{standardk1}
If $u,v \in U_{n}(A)$, the following relations hold in $K_{1}(A)$:
\[
[u \oplus v] = [v \oplus u] = [uv] = [vu]
\]
In particular, $[u^{*}]$ is inverse to $[u]$ in $K_{1}(A)$.
\end{prop}

\begin{proof}
This is a standard $K$-theory computation. We have
\[
\begin{pmatrix} u & 0 \\ 0 & v \end{pmatrix}  = \begin{pmatrix} uv & 0 \\0 & 1  \end{pmatrix}
 \begin{pmatrix} v^{*} & 0 \\0 & v\end{pmatrix}
\]
Now the path
\[
\begin{pmatrix} \cos(t)v^{*} & \sin(t) \\ -\sin(t) & \cos(t)v \end{pmatrix}  
\]
is a path of unitaries connecting
\[
 \begin{pmatrix} v^{*} & 0 \\0 & v\end{pmatrix}
\]
to
\[
 \begin{pmatrix} 0 & 1 \\-1 & 0\end{pmatrix}
\]
which in turn is connected by a path of unitaries to the identity - for example, by using the above path with $v = 1$. Now the relation $[u \oplus v]  = [uv \oplus 1]$ is obvious, and since $[1] = 0$, we obtain $[u \oplus v] = [uv]$. The other statements are direct consequences of this and $[u \oplus v]  =  [v \oplus u] $
\end{proof}

Both $K_{0}$ and $K_{1}$ are functorial for $*$-homomorphisms $f \colon A \rightarrow B$ since for a projection $p$ over $A$, $f(p)$ is a projection over $B$ and similarly for unitaries. We leave it to the reader to check the details. \\
To compare algebraic and topological $K_1$, we need another proposition.

\begin{definition}
If $A$ is unital, let $U(A)$ be the unitary group of $A$, i.e.
\[
U(A) = \{ u \in A \mid uu^{*} = u^{*}u = 1 \}
\]
and let
\[
GL(A) = \{a \in A \mid a \text{ is invertible } \}
\]
If $A$ is possibly nonunital, let $\tilde A$ be its unitalization and $p: \tilde A \rightarrow \IC$ the canonical map, compare \ref{nonunital}. We set
\[
U(A) = \{ u \in \tilde A \mid uu^{*} = u^{*}u = 1, p(u) = 1 \}
\]
and
\[
GL(A) = \{a \in \tilde A \mid a \text{ is invertible }, p(a) = 1 \}
\]
Since for unital $A$, the map $\tilde A \rightarrow A \oplus \IC, (a, \lambda) \mapsto (a+\lambda \cdot 1_A, \lambda)$ is an algebra isomorphism, both definitions give the same result for unital algebras.
\end{definition}

\begin{prop}
\label{unitaryvsinvertible}
Let $A$ be a $C^{*} \minus$ algebra. Then $U(A)$ is a deformation retract of $GL(A)$.
\end{prop}

\begin{proof}
See \cite[4.2.3]{WO}.
\end{proof}

\begin{prop}
\label{invk1}
If we replace unitary by invertible in the definition of $K_{1}(A)$, the group $K_{1}(A)$ does not change.
\end{prop}

\begin{proof}
This is a consequence of the fact  \ref{unitaryvsinvertible} that $Gl_{n}(A)$ is homotopy equivalent to $U_{n}(A)$.
\end{proof}

On the level of $K_{1}$, algebraic and topological $K$-theory differ, but we still have he following:

\begin{prop}
\label{k1quotient}
There is a natural, surjective map $K_{1}^{\alg}(A) \rightarrow K_{1}(A)$.
\end{prop}

\begin{proof}
We define $K_{1}^{\alg}(A)$ as in \cite[3.1.6]{Ros} as the free abelian group generated by the invertibles over $A$, modulo the following relations:
\begin{enumerate}
\item $[a]+[b] = [ab]$
\item If there is a commutative diagram with exact rows
\[
\xymatrix{
0 \ar[rr]  & & A^{n} \ar[rr]^{i} \ar[dd]^{b} & & A^{m} \ar[rr]^{p} \ar[dd]^{a} & & A^{k} \ar[rr] \ar[dd]^{c} & & 0 \\
\\
0 \ar[rr]  & & A^{n} \ar[rr]^{i}  & & A^{m} \ar[rr]^{p} & & A^{k} \ar[rr] & & 0 \\
}
\]
we have the relation $[b] + [c] = [a]$.
\end{enumerate}
By the above proposition \ref{invk1}, we can also view $K_{1}(A)$ as a quotient of the free abelian group generated by the invertibles over $A$. So to obtain a surjective map $K_{1}^{\alg}(A) \rightarrow K_{1}(A)$, we only have to check that the two types of relations in $K_{1}^{\alg}$ are also satisfied in $K_{1}(A)$. For the first, this is clear. For the second, note that we can write
\[
a = \begin{pmatrix} b & x \\0 & c  \end{pmatrix}
\]
after identifying $A^{m}$ with $A^{n+k}$ as prescribed by the short exact sequence in the rows in the diagram. Letting $x$ go to $0$ along the linear path, we obtain 
\[
[a] = \left [ \begin{pmatrix} b & 0 \\0 & c  \end{pmatrix} \right ]
\]
in $K_{1}(A)$ which implies $[a] = [b]+[c]$.
\end{proof}

\section{Non-unital algebras}

We will also consider $C^{*}$-algebras without identity and their $K$-theory. The definition is standard and uses a standard trick to extend a functor a priori only defined for unital algebras to all algebras.

\begin{definition}
\label{nonunital}
Let $A$ be a $C^{*}$-algebra. We define its unitalization $\tilde A$ to be the algebra $A \oplus \IC$ with involution and addition defined as prescribed by the direct sum decomposition and with multiplication
\[
(a, \lambda) \cdot (b, \mu) = (ab+\lambda \cdot b+\mu \cdot a, \lambda \cdot \mu)
\]
The norm is given by
\[
\left\Vert (a, \mu) \right\Vert = \left\Vert a \right\Vert + \left\vert \mu \right\vert
\]
It is easy to check that this is a unital $C^{*}$-algebra with unit $(0,1)$. The algebra $A$ embeds into $\tilde A$ via $a \mapsto (a,0)$. Also $\tilde A$ comes with a canonical ${}*$-homomorphism $p: \tilde A \rightarrow \IC$ sending $(a, \mu)$ to $\mu$. We define
\[
K_{i}(A) = \ker(p_{*} \colon K_{i}(\tilde A) \rightarrow K_{i}(\IC))
\]
\end{definition}

The following proposition shows that we have defined nothing new in case $A$ already was unital.

\begin{prop}
If $A$ is unital, $\tilde A$ is isomorphic to the direct sum $C^{*}$-algebra $A \oplus \IC$. Consequently, the two possible definitions of $K_{i}(A)$ agree.
\end{prop}

\begin{proof}
The embedding $A \rightarrow \tilde{A}$ is split by the map $(a, \mu) \mapsto a+\mu \cdot 1$. Note that the corresponding splitting $\tilde A \cong A \oplus \IC$ is not the splitting used in the definition of $\tilde{A}$.
\end{proof}

One can also find a direct description of $K_{1}(A)$ for any algebra $A$: It is the free abelian group with one generator for each unitary over the unitalisation of $A$ which maps to the identity matrix over $\IC$, i.e. each element of $U_{n}(A)$, modulo the same relations we used in defininig $K_{1}$ for unital algebras. We leave it to the reader to check that this definition is the same as the one above.

\section{A space-level description of $K$-theory}

Now we have defined topological $K$-theory on the group level; we will however also need a space-level version of $K$-theory, i.e. a natural functor from $C^{*}$-algebras to spaces or, later on, spectra representing $K$-theory. Nothing in this section is new, but there seems to be no place in the literature where the space-level version is investigated in detail.  \\
For definiteness, we will always work in the category of compactly generated weak Hausdorff spaces; see \cite{cgwh}. This does not really make a difference since most spaces we consider are metric spaces anyway.  In particular, we will use  \cite[3.6]{cgwh} to commute certain colimits and maps out of compact spaces with one another. \\
 Let $A$ be a $C^{*}$-algebra, possibly non-unital. Let $\tilde A$ be its unitalization and $p: \tilde A \rightarrow \IC$ the canonical map.

\begin{definition}
Recall that the unitary group of $A$ is
\[
U(A) = \{u \in \tilde A \mid u \text{ unitary }, p(u) = 1  \}
\]
If $A$ is already unital, $U(A)$ is isomorphic to the usual unitary group of $A$ by sending an actual unitary $u$ of $A$ to $u-1_{A}+1_{\IC} \in U(A)$. If $A$ is non-unital, we can (and often will) identify $U(A)$ with
\[
\{x \in A \mid xx^{*}+x+x^{*} = 0 = x^{*}x+x+x^{*} \}
\]
the identification being $x \mapsto x+1_{\IC}$. \\
For a natural number $n$, let $p \colon M_{n}(\tilde A) \rightarrow M_{n}(\IC)$ be the canonical map. Define
\[
U_{n}(A) = \{u \in M_{n}(\tilde A) \mid u \text{ unitary}, p(u) = 1_{n} \}
\]
We have $U_{n}(A) = U(\widetilde{M_{n}(A)})$. The reader should convince himself that this is true, and that this is not a total triviality: After all, $M_{n}(\tilde A)$ and $\widetilde{M_{n}(A)}$ are quite different. \\
Let $s_{n}: U_{n}(A) \rightarrow U_{n+1}(A)$ be the stabilization map sending a matrix $u$ to 
\[
\begin{pmatrix}
u &0 \\
0 & 1
\end{pmatrix}
\]
Let $U_{\infty}(A)$ be the colimit of the $U_{i}(A)$ along these stabilization maps. Note that $U_{\infty}(A)$ is a topological group, being the colimit of topological groups along group homomorphisms.
\end{definition} 

\begin{prop}
There is a canonical isomorphism of groups
\[
\phi \colon \pi_{0}(U_{\infty}) \rightarrow K_{1}(A)
\]
\end{prop}

\begin{proof}
Any point of $U_{\infty}$ is given by a unitary matrix $u$ of some size $n$. We define the map $\phi$ by sending $u$ to $[u]$. This is indeed well-defined: If $u, v \in U_{\infty}$ lie in the same path-component, each path connecting them actually lies in some $U_{k}(A)$ for some large $k$ since the unit interval is compact, and this implies $[u] = [v]$ in $K_{1}(A)$. The map $\phi$ is a group homomorphism by \ref{standardk1}, and it is clearly surjective. If $u \in U_{\infty}$ goes to $0$ in $K_{1}(A)$, there must be an $n$ such that $u \oplus 1_{n}$ is connected to the identity matrix, which in turn means that $u = 0$ in $\pi_{0}(U_{\infty})$.
\end{proof}

We also record the following:

\begin{prop}
\label{Morita}
Let $A$ be a $C^{*}$-algebra. Then for each $n$, the inclusion
\begin{align*}
A &\rightarrow M_{n}(A) \\
a &\mapsto \begin{pmatrix} a & 0 \\ 0 & 0 \end{pmatrix}
\end{align*}
induces an isomorphism $K_{i}(A) \cong K_{i}(M_{n}(A))$
\end{prop}

\begin{proof}
This is straightforward; see for example \cite[4.2.4]{HigsonRoe}.
\end{proof}

\begin{prop}
\label{conjugation}
Let $J$ be an ideal inside the unital $C^{*}$-algebra $A$. Then for each unitary $u \in A$, the $*$-homomorphism $Ad_{u} \colon J \rightarrow J$ given by conjugation with $u$ induces the identity on $K_{i}(J)$. Furthermore, the map induced on $U_{\infty}(J)$ is homotopic to the identity of $U_{\infty}(J)$.
\end{prop}

\begin{proof}
For the first claim, see \cite[4.6.1]{HigsonRoe}. For the second claim, note that this is clear as long as $u$ is connected to $1$ inside the unitaries of $A$, since then $Ad_{u}$ and $Ad_{1} = \Id$ are connected by a path of $*$-homomorphisms. This situation can always be arranged by replacing $J$ by $M_{2}(J)$ and conjugation by $u$ with conjugation by 
\[
\begin{pmatrix}
u & 0 \\ 0 & u^{*}
\end{pmatrix}
\]
which is connected to $1$ inside the unitaries of $M_{2}(A)$. Since the inclusion $J \rightarrow M_{2}(J)$ is a $K$-equivalence, the claim then easily follows.
\end{proof}

\begin{prop}
Let $J \subset A$ be an ideal in $A$. Then the map 
\[
U_{\infty}(A) \rightarrow U_{\infty}(A/J)
\]
is a Serre fibration with fiber $U_{\infty}(J)$.
\end{prop}

\begin{proof}
Given a diagram
\[
\xymatrix{
D^{n} \ar[d] \ar[rr] & & U_{\infty}(A) \ar[d] \\
D^{n} \times I \ar[rr] & & U_{\infty}(A/J) }
\]
we have to find a lift $D^{n} \times I \rightarrow U_{\infty}(A)$. Since $D^{n}$ is compact, we may view the top horizontal map as an element of $U_{\infty}(\Map(D^{n},A))$, and the lower horizontal map as a map $I \rightarrow U_{\infty}(\Map(D^{n}, A/J))$. Since also 
\[
\Map(D^{n}, J) \rightarrow \Map(D^{n}, A) \rightarrow \Map(D^{n}, A/J)
\]
is an ideal sequence of $C^{*}$-algebras, we see that it suffices to treat the case $n = 0$. To see that $\Map(D^{n}, A) \rightarrow \Map(D^{n}, A/J)$ is indeed onto, use \ref{tensors} and that $^*$-homomorphism have closed image. By maybe passing from $A$ to $M_{n}(A)$ for some $n$, we may assume that the map $I \rightarrow U_{\infty}(A/J)$ actually takes values in $U(A/J)$. Consider the diagram
\[
\xymatrix{
pt \ar[d] \ar[rr] & & U(A) \ar[d] \\
I \ar[rr]^{f} & & U(A/J) }
\]
We have a path $f$ of unitaries in $U(A/J)$ and a lift $u$ of $f(0)$ to a unitary in $A$. By multiplying everything with $u^{*}$, we can assume that $u = 1$ and $f$ is a path starting at $1$. Now view $f$ as an element of the identity component $U_{0}(\Map(I, A/J))$ of $U(\Map(I,A/J)$. Since 
\[
\Map(I,A) \rightarrow \Map(I, A/J)
\] is onto, we can apply \cite[2.1.7]{RLL} to see that also 
\[
U_{0}(\Map(I,A)) \rightarrow U_{0}(\Map(I, A/J))
\]
is onto. The reader should note that this completely fails when we leave out the zero. Now any preimage of $f$ yields the desired lift. The fiber over $1 \in U_{\infty}(A/J)$ is clearly $U_{\infty}(J)$.
\end{proof}

\section{Bott periodicity}

We define for $n \geq 1$ the higher $K$-groups of $A$ to be
\[
K_n(A) = \pi_{n-1} U_{\infty}(A)
\]
The fundamental theorem of topological $K$-theory is the Bott periodicity theorem.

\begin{thm}
For each $n \geq 0$, there is a natural isomorphism
\[
K_{n}(A) \cong K_{n+2}(A)
\]
On the space level, for each $C^{*}$-algebra $A$, there is a natural weak homotopy equivalence $b_{A} \colon U_{\infty}(A) \rightarrow \Omega^{2} U_{\infty}(A)$. 
\end{thm}

\begin{proof}
This is a space-level version of the Bott periodicity theorem; see for example \cite[4.7. 4.8]{HigsonRoe} and \cite[11.4]{RLL}.
\end{proof}

Out of the Bott periodicity map, we can build an $\Omega$-spectrum $K(A)$ as follows: In even dimensions, we have $K(A)_{n} = \Omega U_{\infty}(A)$ and in odd dimensions, we have $K(A)_{n} = U_{\infty}(A)$. The structure maps $K(A)_{n} \rightarrow \Omega K(A)_{n+1}$ is the identity if $n$ is even and the Bott map if $n$ is odd. It is straightforward to verify that this spectrum represents $K$-theory of $A$ in the sense that
\[
\pi_{n}(K(A)) = K_{n}(A)
\]
for each $n$. This yields a functor
\[
K \colon C^{*} \minus \alg \rightarrow SPECTRA
\]
representing $K$-theory. For an ideal $J \subset A$, the fibration sequence of spaces $U_{\infty}(J) \rightarrow U_{\infty}(A) \rightarrow U_{\infty}(A/J)$ is carried to a levelwise fibration of spectra $K(J) \rightarrow K(A) \rightarrow K(A/J)$ since $\Omega$ preserves fibrations. 

\section{Continuity of $K$-theory}

We will consider colimits of $C^{*}$-algebras; while it is well-documented that $K$-theory on the group level behaves well with respect to colimits, we need that $K$-theory also behaves well on the space level. This section supplies the necessary tools. We begin with recalling the group-level version of the continuity of $K$-theory.

Recall that for a directed system $(A_{i})_{i \in I}$ of $C^{*}$-algebras, its colimit in the category of $C^{*}$-algebras is the following algebra: First form the algebraic colimit $A'$. This algebra inherits an involution and a pseudonorm, see \cite[Appendix L]{WO}.
We then divide out the ideal of elements of norm $0$ and finally complete the resulting pre-$C^{*}$-algebra to a $C^{*}$-algebra $A$ which is the colimit of $(A_{i})_{i \in I}$. The only case of interest for us is the case of a directed union, i.e. a directed system $(A_{i})_{i \in I}$ where all structure maps $A_{i} \rightarrow A_{j}$ are injective. In this case, $A'$ is just the union of all $A_{i}$, with the unique norm such that each inclusion $A_{i} \rightarrow A'$ is isometric, which is well-defined since all structure maps are automatically isometric. Then $A$ is the completion of $A'$ under this norm. Note that due to the completion process, a directed union in this case is not a set-theoretic directed union. \\

\begin{prop}
Let $(A_{i})_{i \in I}$ be a directed system of $C^{*}$-algebras. Then the canonical map
\[
\colim K_{n}(A_{i}) \rightarrow K_{n} \colim A_{i}
\]
is an isomorphism for all $n$.
\end{prop}

\begin{proof}
See \cite[6.2.9, 7.1.7]{WO}.
\end{proof}

To obtain a space-level version, we will restrict for the sake of simplicity to the case of directed unions of $C^{*}$-algebras, i.e. directed systems in which all structure maps are isometric inclusions. \\
For a directed system $(X_{i})_{i \in I}$ of metric spaces and isometric inclusions, we can put two potentially different topologies on the set-theoretic union of the $X_{i}$: On the one hand, we can use the usual colimit topology to obtain $\colim X_{i}$, and on the other hand, $\cup X_{i}$ inherits a metric by defining $d(x,y) = d_{i}(x,y)$ where $i$ is such that $x,y \in X_{i}$. This is well-defined since all structure maps are isometric. We denote the space we obtain by $\colim^{met} X_{i}$.

\begin{prop}
\label{cont}
Let $(A_{i})_{i \in I}$ be a directed union of $C^{*}$-algebras. Then the canonical maps
\[
\colim^{met} U(A_{i}) \rightarrow U(\colim A_{i})
\]
and
\[
\colim^{met} U_{\infty}(A_{i}) \rightarrow U_{\infty}(\colim A_{i})
\]
are weak homotopy equivalences. Furthermore, the canonical maps induce isomorphisms
\[
\pi_n \colim^{met}(U(A_i)) \cong \colim \pi_n(U(A_i))
\]
and
\[
\pi_n \colim^{met}(U_{\infty}(A_i)) \cong \colim \pi_n(U_{\infty}(A_i))
\]
Even stronger, for each compact Hausdorff space $X$, we have
\[
[X, \colim^{met}(U(A_i))] \cong \colim [X, U(A_{i})]
\]
and
\[
[X, \colim^{met}(U_{\infty}(A_i))] \cong \colim [X,U_{\infty}(A_i)]
\]
\end{prop}

The reader should note that this is not true on the level of maps, i.e. for example there is no bijection between
$\Map(S^{n}, U(\colim A_{i})$ and $\colim \Map(S^{n}, U(A_{i}))$ since $U(\colim A_{i})$ is strictly bigger than $\colim U(A_{i})$. Furthermore, we will consider colimits over indexing posets where we cannot commute maps out of compact spaces and colimits anyway, compare \cite{cgwh}. \\
We will prove the above result in a series of lemmas.

\begin{definition}
For a compact Hausdorff space $X$ and a $C^{*}$-algebra $A$, let $C(X,A)$ be the $C^{*}$-algebra of continuous functions from $X$ to $A$. The norm is the $\sup$-norm
\[
\left\Vert f \right\Vert_{C(X,A)} = \sup_{x \in X} \left\Vert f(x) \right\Vert_{A}
\]
\end{definition}

\begin{lem}
\label{tensors}
For $f \in C(X) = C(X, \IC)$ and $a \in A$, let $(f,a)\colon X \rightarrow A$ be the function sending $x$ to $f(x)a$. Then the linear span of all $(f,a)$ with $f \in C(X)$ and $a \in A$ is dense in $C(X,A)$.
\end{lem}

\begin{proof}
Let $g: X \rightarrow A$ be an arbitrary continuous map. Let $\epsilon > 0$ be given.  Since $X$ is compact, there is a finite open covering $U_{1}, \dots U_{n}$ of $X$ such that $\left\Vert g(x)-g(y) \right\Vert < \epsilon$ whenever $x,y \in U_{i}$. Let $\phi_{i}:X \rightarrow \IC$ be a subordinate partition of unity and pick points $x_{i} \in U_{i}$. Then the function $\phi = \sum\limits_{i=1}^{n} g(x_{i})\phi_{i}$ satisfies
\begin{align*}
\left\Vert g-\phi \right\Vert =& \sup\limits_{y \in X} \left\Vert g(y)- \sum\limits_{i=1}^{n} g(x_{i})\phi_{i}(y) \right\Vert = \\
=& \sup\limits_{y \in X}  \left\Vert \sum\limits_{i=1}^{n} g(y)\phi_{i}(y) - \sum\limits_{i=1}^{n}  g(x_{i})\phi_{i}(y) \right\Vert \\
&\leq \epsilon \sum\limits_{i=1}^{n} \phi_{i}(y) = \epsilon
\end{align*}
Since $\phi$ is in our span and $\epsilon$ was arbitrary, we are done.
\end{proof}

\begin{rem}
For the reader familiar with tensor products, this lemma says that $C(X,A)$ is a (and actually, the) completion of $C(X) \otimes A$ to a $C^*$-algebra, where an elementary tensor $f \otimes a$ corresponds to $(f,a) \in C(X,A)$.
\end{rem}

\begin{prop}
\label{exactness}
For each compact Hausdorff space $X$, the functor
\[
C(X,-): C^{*}\minus \alg \rightarrow C^{*} \minus \alg
\]
preserves directed unions.
\end{prop}

\begin{proof}
Let $A = \cup A_{i}$ (recall the completion process).  There are compatible maps $C(X,A_{i}) \rightarrow C(X,A)$ induced by the inclusions $A_{i} \rightarrow A$, and for $i \leq j$, the canonical map $C(X,A_{i}) \rightarrow C(X, A_{j})$ is also an inclusion. So we get an induced map
\[
\cup C(X,A_{i}) \rightarrow C(X,A)
\]
which is an isometric inclusion: It is clearly isometric when restricted to the set-theoretic union of the $C(X,A_{i})$, and this carries over to the completions. For surjectivity, it is by \ref{tensors} enough to see that all functions of the form $(f,a)$ are in the image. But $a$ can be approximated arbitrarily well by elements $a_{i}$ lying in some $A_{k_{i}}$, and clearly $(f,a_{i})$ is in the image. The claim now follows since the image of an isometric inclusion of $C^{*}$-algebras is closed.
\end{proof}

\begin{proof}[Proof of \ref{cont}:] Let us first discuss the case where all structure maps $A_{i} \rightarrow A_{j}$ are unital, and hence also all inclusions $A_{i} \rightarrow A = \colim A_{i}$. \\
We begin with the isomorphism
\[
\pi_n \colim^{met}(U(A_i)) \cong \colim \pi_n(U(A_i))
\]
Let $f: S^{n} \rightarrow \colim^{met}(U(A_{i}))$ be a continuous map. We may consider $\colim^{met}(U(A_{i}))$ as a subspace of $A = \colim A_{i}$ and hence $f$ as an element of $C(S^{n}, A)$. By \ref{exactness}, we find an $f': S^{n} \rightarrow A_{i}$ such that $f'$ and $f$ are arbitrarily close. If, say, $\left\Vert f'-f \right\Vert < \frac{1}{10}$, $f'$ has to take values inside the invertibles of $A$ and then also in the invertibles of $A_{i}$, since $C^{*}$-subalgebras are closed under taking inverses. Also $f$ and $f'$ are homotopic as maps to $Gl(A)$ by the linear homotopy $t\cdot f+ (1-t)\cdot f'$. By \ref{unitaryvsinvertible}, we then find $f'': S^{n} \rightarrow U(A_{i})$ which is homotopic to $f'$ as maps to $U(A_{i})$. Together, it follows that $f$ comes from $\colim \pi_{n}(U(A_{i}))$, i.e. the canonical map $\colim \pi_{n}(U(A_{i})) \rightarrow \pi_n \colim^{met}(U(A_i))$ is onto. The argument for injectivity is pretty much the same, using that also $S^{n} \times I$ is compact.  The argument for the map
\[
\pi_n \colim^{met}(U_{\infty}(A_i)) \cong \colim \pi_n(U_{\infty}(A_i))
\]
is similar, using that for any $C^{*}$-algebra $B$, we have an isomorphism $\pi_{n} U_{\infty}(B) = \colim_{k} \pi_{n} U_{k}(B)$ since sequential colimits of closed inclusions of metric spaces commute with maps out of compact sets, see \cite{cgwh}. The only property of the sphere we have used is Hausdorff and compactness, so the same proof applies to all compact Hausdorff spaces $X$. \\
Now let us consider
\[
\colim^{met} U(A_{i}) \rightarrow U(\colim A_{i})
\]
Let $f: S^{n} \rightarrow U(\colim A_{i})$ be given. Again viewing $f$ as a map to $A$, we find by \ref{exactness} $f': S^{n} \rightarrow A_{i}$ arbitrarily close to $f$. Again $f'$ takes values in $GL(A_{i})$ and is homotopic to $f$ inside $GL(A)$. By \ref{unitaryvsinvertible}, we can homotope $f'$ to $f'': S^{n} \rightarrow U(A_{i})$. This proves surjectivity of 
$\pi_{n} \colim U(A_{i}) \rightarrow \pi_{n} U(\colim A_{i})$. Injectivity is again the same argument with $S^{n} \times I$ instead of $S^{n}$. The map
\[
\colim^{met} U_{\infty}(A_{i}) \rightarrow U_{\infty}(\colim A_{i})
\]
is treated similarly.\\
If the inclusions $A_{i} \rightarrow A_{j}$ are possibly non-unital, one can unitalize all involved algebras and use nearly the same argument, using the non-unital part of \ref{unitaryvsinvertible}. The only difficulty is that $f'$ may not take values in $GL(A) = \{a \in \tilde{A} \mid a \text{ invertible}, p(a) = 1\}$  directly since $p(f'(x)) \in \IC$ may not be $1$. This, however, is easily repaired  by multiplying $f'$ with $(p \circ f')^{-1}$. 
\end{proof}

\section{$G$-$C^{*}$-algebras}

For definiteness, we include the following definition.

\begin{definition}
Let $G$ be a (discrete) group. A $G$-$C^{*}$-algebra $A$ s a $C^{*}$-algebra $A$ together with a left action of $G$ on $A$ through $*$-homomorphisms.
\end{definition}

\begin{example}
If $X$ is a $G$-space, the $C^{*}$-algebra $C(X)$ inherits a $G$-action.
\end{example}

\fxnote{fix this}

%% file: chapters/control.tex
\chapter{Controlled topology and $C^{*}$-categories}

This section introduces the basic definitions of controlled algebra, with an emphasis on $C^*$-categories. Throughout this section, let $G$ be a discrete  group and $X$ a $G$-space.

\section{$C^*$-categories}

We begin with the definition of a $C^*$-category. The main source is \cite{cstarcat}, but note that unlike there, we will usually consider unital and additive $C^*$-categories.

\begin{definition}
Let $\C$ be an additive category.
\begin{itemize}
\item[(i)]
The category $\C$ is a $\IC$-category if all morphism sets are not only abelian groups, but complex vector spaces, with composition being complex bilinear.
\item[(ii)]
Let $\C$ be a $\IC$-category. An involution on $\C$ consists of a map $(-)^*\colon \Hom(A,B) \rightarrow \Hom(B,A)$ for each two objects $A$ and $B$ of $\C$ such that  $f^{**} = f$, $(\alpha f+\beta g)^* = \bar{\alpha}f^*+\bar{\beta}g^*$ and $(fg)^* = g^*f^*$ whenever $f$ and $g$ are composable.
\item[(iii)]
The $\IC$-category $\C$ is a normed category if each morphism space carries a norm such that composition is submultiplicative, i.e. we have $\left\Vert f \circ g \right\Vert \leq \left\Vert f \right\Vert \left\Vert g \right\Vert$ for all composable morphisms $f,g$.  It is a normed category with involution if it also carries an involution which is isometric, i.e. satisfies $\left\Vert f \right\Vert = \left\Vert f^{*} \right\Vert$ for all morphisms $f$.
\item[(iv)] A normed $\IC$-category with involution is a pre-$C^*$-category if it satisfies the $C^*$-identity $\left\Vert f^* \circ f \right\Vert = \left\Vert f \right\Vert^2$ for each $f \colon A \rightarrow B$ and when, additionally, $f^*f$ is a positive element of the pre-$C^*$-algebra $\End(A)$. It is a $C^*$-category if in addition each morphism space is complete with respect to the norm.
\item[(v)] A $C^*$-functor $F \colon \C \rightarrow \D$ between two $C^{*}$-categories is an additive functor which is compatible with the involutions in the sense that $F(\phi)^{*} = F(\phi^{*})$. Such a functor is automatically continuous with norm at most one, i.e. it satisfies $\left\Vert F(\phi) \right\Vert \leq \left\Vert \phi \right\Vert$: The norm of $F(\phi): A \rightarrow B$ is by the $C^{*}$-identity determined by the norm of $F(\phi^{*}\phi): A \rightarrow A$.  Since $F$ induces a $*$-homomorphism from the $C^{*}$-algebra $\End(A)$ to $\End(F(A))$ and $*$-homomorphisms of $C^{*}$-algebras are automatically continuous with norm at most one, we have
\[
\left\Vert F(\phi) \right\Vert^{2} =   \left\Vert F(\phi)^{*} \circ F(\phi)\right\Vert \leq \left\Vert  \phi^{*}\phi \right\Vert = \left\Vert \phi \right\Vert^{2}
\]
as desired.
\end{itemize}
\end{definition}

\begin{rem}
The positivity condition in (iv) is really necessary, as the example in \cite{cstarcat}[2.10] shows.
\end{rem}

\begin{rem}
Morally, we have to demand some compatibility between the norms and the direct sum. Fortunately, $C^{*}$-algebras are nice enough that we do not need such a condition: The norm on a $C^{*}$-category is uniquely determined by the norms on the endomorphism algebras because of the $C^{*}$-identity, and on the endomorphism algebras, basic $C^{*}$-algebra theory tells us the norm is uniquely determined by the algebra structure and the involution. In particular,  an equivalence of categories with involution between $C^{*}$-categories is automatically isometric. 
\end{rem}

\begin{example}
A concrete $C^*$-category is a norm-closed and involution-closed subcategory of the $C^*$-category of all Hilbert spaces and boun\-ded operators between them. By \cite[6.12]{cstarcat}, every $C^*$-category is isomorphic to a concrete one.
\end{example}

\begin{example}
Let $A$ be a $C^{*}$-algebra. Then we can view $A$ as a $C^{*}$-category with one object which has $A$ as its endomorphism ring. This category is not additive, so we pass to its additive completion $A_{\oplus}$ which is equivalent to the category of finitely generated free $A$-modules. Explicitly, for a (possibly non-additive) $C^{*}$-category $\C$,  define its additive completion $A_{\oplus}$ to be the category with objects finite sequences $(A_{1}, A_{2},\dots A_{n})$, and a morphism from $(A_{1}, A_{2},\dots A_{n})$ to $(B_{1}, \dots, B_{k})$ is a matrix of morphisms $A_{i} \rightarrow B_{j}$, $1 \leq i \leq n, 1 \leq j \leq m$. Composition is matrix multiplication. The category $\C_{\oplus}$ is additive, the direct sum of two objects just being the concatenation of the underlying sequences. It also inherits an involution which is given by conjugate-transposing matrices. Finally, by the preceding example, we can represent $\C$ on the category of Hilbert spaces; this representation can be extended to $\C_{\oplus}$ in the evident way. In this fashion, $\C_{\oplus}$ inherits a $C^{*}$-norm. Since there is at most one norm making $\C_{\oplus}$ into a $C^{*}$-category, this norm is independent of all choices.
\end{example}

\section{$C^*$-categories and idempotent completion}

Recall the following definition:

\begin{definition}
For an additive category $\C$, the idempotent completion $\Idem(\C)$ is the category with objects $(A,p)$ with $A$ an object of $\C$ and $p \colon A \rightarrow A$ an idempotent. A morphism from $(A,p)$ to $(B,q)$ is a morphism $f \colon A \rightarrow B$ such that $qfp = f$. Composition is inherited from $\C$: the composite of $f \colon (A,p) \rightarrow (B,q)$ and $g \colon (B,q) \rightarrow (C,e)$ is $gf$ which is possible since $egfp  = gf$.
\end{definition}
We would like the idempotent completion of a $C^*$-category to be a $C^*$-category again; however, using the above definition of idempotent completion, there is a problem: For a morphism $f \colon (A,p) \rightarrow (B,q)$, $f^*$ is only a morphism $(B,q^*) \rightarrow (A,p^*)$ and not a morphism $(B,q) \rightarrow (A,p)$ as required; or in other words, though the idempotent completion inherits an involution, this involution is no longer the identity on objects. The following very useful fact about $C^*$-algebras allows us to circumvent this problem:

\begin{prop}
\label{projections}
Let $e$ be an idempotent over a unital $C^*$-algebra $A$. Then the projective module $A^n e$ defined by $e$ can also be defined by a projection (i.e. a selfadjoint idempotent) $p$. By considering an idempotent $e:X \rightarrow X$ in a $C^*$-category $\C$ as an element of the $C^*$-algebra $\End(X)$, we get an analogous statement for $\C$ and objects of $\Idem(\C)$. 
\end{prop}

\begin{proof}
Replacing $A$ by matrices over $A$ if necessary, we may assume that $e$ is an actual element of $A$. Consider the element $h = 1+(e-e^*)(e^*-e)$. Since $(e-e^*)(e^*-e)$ is positive, $h$ is invertible, see \cite[Section 2.2]{Murphy}. It is also easy to check that $eh = ee^*e = he$ and $e^*h =  e^*ee^* = he^*$. Now set $p = ee^*h^{-1}$. Since $e$ and  $e^*$ commute with $h^{-1}$ and $e^*ee^*h^{-1} = e^*$, $p$ is a projection. Clearly $ep = p$, and it is easily checked that $pe = e$. Hence $p$ and $e$ define isomorphic projective modules; $e: (A,e)  \rightarrow (A,p)$ inducing the isomorphism with inverse $p$.  
\end{proof}

This leads to the following definition.

\begin{definition}
For a $C^*$-category $\C$, the idempotent completion $\Idem_{*}(\C)$ is the category with objects $(A,p)$ with $A$ an object of $\C$ and $p: A \rightarrow A$ a projection. A morphism from $(A,p)$ to $(B,q)$ is a morphism $f: A \rightarrow B$ such that $qfp = f$. Composition is inherited from $\C$: the composite of $f: (A,p) \rightarrow (B,q)$ and $g: (B,q) \rightarrow (C,e)$ is $gf$ which is possible since $egfp  = gf$.
\end{definition}

By applying the above proposition, we see that this idempotent completion is equivalent to the usual idempotent completion:
\begin{prop}
For a $C^{*}$-category $\C$, the inclusion functor $\Idem_{*}(\C) \rightarrow \Idem(\C)$ is an equivalence of categories.
\end{prop}

From now on, we will drop the $*$ and just write $\Idem(\C)$ for the latter definition. This definition allows us to make $\Idem(\C)$ into a $C^*$-category. To define the involution, given a morphism $f \colon (A,p) \rightarrow (B,q)$, i.e. a morphism $f \colon A \rightarrow B$ such that $qfp = f$, the morphism $f^*$ of $\C$ satisfies $pf^*q = f^*$ since $p$ and $q$ are selfadjoint, and is hence a morphism $(B,q) \rightarrow (A,p)$ as desired.  The morphisms from $(A,p)$ to $(B,q)$ form a subspace of $\Hom(A,B)$ which is closed since it is the subspace determined by the relation $pfq = f$. Hence $\Hom((A,p),(B,q))$ inherits a norm in which it is complete. The $C^*$-identity is inherited from $\C$. It remains to check that for any morphism $f: (A,p)\rightarrow (B,q)$, $f^*f \colon (A,p) \rightarrow (A,p)$ is positive. But this can be checked by inspecting the spectrum of $f^*f$ which is not changed by passing from $\End(A)$ to the smaller $C^*$-algebra $\End((A,p),(A,p))$, see \cite[2.1.11]{Murphy}, and by definition of a $C^*$-category, $f^*f$ is positive in $\End(A)$. Hence $\Idem(\C)$ is again a $C^*$-category which is equivalent to the usual $\Idem(\C)$ as additive categories.

\section{Controlled topology and $C^*$-categories}

Given a metric space, we can study its large-scale geometry with the help of the sets
\[
E_r = \{ (x,y) \in X \times X \mid d(x,y) < r \}
\] 
The properties of these sets are the motivation for the following definition; compare \cite[2.3]{BFJR}.

\begin{definition}
A coarse structure $(\cale, \calf)$ on the $G$-space $X$ consists of a family $\cale$ of subsets of $X \times X$ and a family $\calf$ of subsets of $X$ such that the following holds:
\begin{enumerate}
\item For $E, E' \in \cale$, there is an $E''$ in $\cale$ such that 
\[
E \circ E' = \{(x,z) \in X \times X \mid \exists y \in X: (x,y) \in E, (y,z) \in E' \} \subset E''
\]
\item
For $E, E' \in \cale$, there is an $E''$ such that $E \cup E' \subset E''$.
\item The diagonal of $X \times X$ is in $\cale$.
\item For $F, F' \in \calf$ there is an $F'' \in \calf$ such that $F \cup F' \subset F''$. 
\item For $E \in \cale$, also the flip $E^*=\{(y,x) \mid (x,y) \in E \}$ is in $\cale$. 
\item For $E \in \cale$ and $g \in G$, we have $E = gE = \{(gx,gy) \mid (x,y) \in E\}$, and for $F \in \calf$, we have $F = gF = \{ gx \mid x \in F \}$.
\end{enumerate}
Often, $\calf$ will be the power set of $X$ and will be omitted from the notation. We say that $(X, \cale, \calf)$ is a coarse $G$-space. \\
A $G$-equivariant coarse map $f: (X, \cale, \calf) \rightarrow (Y, \cale', \calf')$ is a $G$-equivariant map $f: X \rightarrow Y$ such that:
\begin{enumerate}
\item For each compact $K \subset Y$ and each $F \in \calf$, the set $f^{-1}(K) \cap F$ is relatively compact in $X$.
\item For each $E \in \cale$, there is $E' \in \cale'$ such that $(f \times f)(E) \subset E'$.
\item For each $F \in \calf$, there is $F' \in \calf'$ such that $f(F) \subset F'$.
\end{enumerate}
The definition is precisely such that the claim in ix) of the below definition \ref{controlcat} below is true.
\end{definition}

\begin{definition}
\label{controlcat}
Let $A$ be a $G$-$C^{*}$-algebra. Let $\calv$ be the additive category with objects $A^n, n \geq 0$, equipped with the standard $A$-valued inner product
\[
\langle (a_{1}, a_{2}, \cdots a_{n}), (b_{1}, b_{2}, \cdots b_{n}) \rangle = \sum\limits_{i=1}^{n} a_{i}b_{i}^{*}
\]
and with morphisms the right $A$-linear maps. The category $\calv$ has a left $G$-action by letting $g \in G$ act as identity on objects and by applying $g$ pointwise to the entries of a matrix in $\Hom(A^{n}, A^{m})$.

\begin{enumerate}
\item The category $\C(X)$ is the category with objects given by sequences $(M_{x})_{x \in X}$ with $M_{x} \in \calv$ such that $\{x \mid M_{x} \neq 0 \}$ is discrete in $X$. A morphism $\phi: M \rightarrow N$ is given by maps $\phi_{x,y}: M_{y} \rightarrow N_{x}$ in $\calv$ such that for all $x \in X$, the sets $\{s \in X \mid \phi_{s,x} \neq 0 \}$ and $\{s \in X \mid \phi_{x,s} \neq 0 \}$ are finite. Note that it is not a typo that $\phi_{x,y}$ is a map from $M_{y}$ to $N_{x}$. Composition is given by matrix multiplication: For two morphisms $\phi, \psi$, we set 
\[(\phi \circ \psi)_{x,y} = \sum\limits_{s \in X} \phi_{x,s} \circ \psi_{s,y}
\]
which is a finite sum. This definition is modeled on the usual matrix multiplication; this is the reason we are working with right modules. $\C(X)$ is an additive category with direct sum $M \oplus N$ defined by $(M \oplus N)_{x} = M_{x} \oplus N_{x}$. The direct sum is strictly associative, and it is commutative in the sense that $M \oplus N = N \oplus M$ as objects of $\C(X)$ - the various maps implementing the universal property are of course different.
\item The category $\C(X)$ also comes with an involution $^{*}$ which is the identity on objects and acts as the usual adjoint on morphisms, i.e. $(\phi^{*})_{x,y}$ is given by $(\phi_{y,x})^{*}$. This is a strict involution.
\item For an object $M$ of $\C(X)$, the total object $T(M)$ of $M$ is the pre-Hilbert-$A$-module $\bigoplus_{x \in X} M_{x}$, where the sum is orthogonal. We consider $T(M)$ as a right $A$-module. For each morphism $\phi \colon M \rightarrow N$, we obtain a morphism of right $A$-modules 
\[
T(\phi) \colon T(M) \rightarrow T(N)
\]
In this way, we can view $T$ as a functor from $\C(X)$ to the category of pre-Hilbert-$A$-modules and (possibly unbounded) $A$-linear maps between them, at least after choosing an order for the summands in the sum for each involved space $X$. We will mainly ignore this point, since a different order of the basis elements does not affect whether an operator $T(M) \rightarrow T(N)$ is bounded or not. The functor $T$ preserves involutions when the category of pre-Hilbert $A$-modules of the form $\oplus A$ and possibly unbounded operators is equipped with the involution given by conjugate-transposing matrices. 
\item $\C(X)$ comes with a right $G$-action given by $(M\cdot g)_{x} = M_{gx}$ on objects and $(\phi \cdot g)_{x,y} = g^{-1}(\phi_{gx,gy})$ on morphisms, where the $g^{-1}$ makes up for the fact that $G$ acts from the left on $\calv$, not from the right. Let $\C^{G}(X) \subset \C(X)$ be the fixed category under this action. Note that the $G$-action on $A$ only enters into the morphisms.
\item Let $\C_{b}(X)$ respectively $\C_{b}^{G}(X)$ be the subcategories of $\C(X)$ respectively $\C^{G}(X)$ with the same objects, but with morphisms $\phi \colon M \rightarrow N$ only those for which $T(\phi) \colon T(M) \rightarrow T(N)$ is a bounded, adjointable $A$-linear operator in the sense of Hilbert $A$-modules. Then $\C_{b}(X)$ and $\C_{b}^{G}(X)$ become normed categories with involution, with the norm of a morphism $\phi$ given by the operator norm of $T(\phi)$. We can view $T$ as a functor from $\C_{b}(X)$ respectively $\C_{b}^{G}(X)$  to the category of pre-Hilbert-$A$-modules and bounded, adjointable operators and hence, by completion, to the category of Hilbert-$A$-modules and bounded, adjointable operators. The latter category also has an involution given by the adjoint, and $T$ is strictly compatible with the involutions.  
\item For an object $M$ of $\C(X)$ or $\C^{G}(X)$, we define 
\[ \supp M = \{ x \in X \mid M_{x} \neq 0 \} \subset X
\]  
For a morphism $\phi$ in $\C(X)$ or $\C^{G}(X)$, we define 
\[
\supp(\phi) = \{ (x,y) \in X \times X \mid \phi_{y,x} \neq 0 \} \subset X \times X
\]
Given a coarse structure $(\cale, \calf)$, we define categories 
\[
\C(X; \cale, \calf), \C^{G}(X; \cale, \calf), \C_{b}(X; \cale, \calf), \C_{b}^{G}(X; \cale, \calf)
\]
as those subcategories of $\C(X), \C^{G}(X), \C_{b}(X), \C_{b}^{G}(X)$ with objects those objects $M$ for which there is an $F \in \calf$ with $\supp(M) \subset F$ and  morphisms those $\phi$ for which there is an $E \in \cale$ with $\supp(\phi) \subset E$. The axioms of a coarse structure are made such that these categories are actually closed under composition of morphisms, closed under the involution and additive.
\item The functor $T$ represents $\C_{b}(X; \cale, \calf)$ and $\C_{b}^{G}(X; \cale, \calf)$ as categories of Hilbert modules and bounded adjointable operators; by completion, we hence obtain $C^{*}$-categories $\C_{*}(X; \cale, \calf)$ respectively $\C_{*}^{G}(X; \cale, \calf)$. Explicitly, $\C_{*}^{G}(X; \cale, \calf)$ is defined as follows: objects are the same as the objects of $\C_{b}^{G}(X; \cale, \calf)$. For each two objects $N,M$ of $\C_{b}^{G}(X; \cale, \calf)$, $\Hom_{\C_{b}^{G}(X; \cale, \calf)}(N,M)$ is a normed space, and the involution is an isometric map
\[
\Hom_{\C_{b}^{G}(X; \cale, \calf)}(N,M) \rightarrow \Hom_{\C_{b}^{G}(X; \cale, \calf)}(M,N)
\]  
We define 
\[
\Hom_{\C_{*}^{G}(X; \cale, \calf)}(N,M)
\]
to be the completion of  $\Hom_{\C_{b}^{G}(X; \cale, \calf)}(N,M)$ to a Banach space. Since the involution is continuous, we obtain an involution
\[
\Hom_{\C_{*}(^{G}X; \cale, \calf)}(N,M) \rightarrow \Hom_{\C_{*}(X; \cale, \calf)}(M,N)
\]  
by extending the involution of $\C_{b}^{G}(X; \cale, \calf)$ to the completions, Since for a third object $L$, the composition map
\[
\Hom_{\C_{b}^{G}(X; \cale, \calf)}(N,M) \times \Hom_{\C_{b}^{G}(X; \cale, \calf)}(M,L) \rightarrow \Hom_{\C_{b}^{G}(X; \cale, \calf)}(N,L)
\]
is continuous, we can complete the composition map to obtain a map
\[
\Hom_{\C_{*}^{G}(X; \cale, \calf)}(N,M) \times \Hom_{\C_{*}^{G}(X; \cale, \calf)}(M,L) \rightarrow \Hom_{\C_{*}^{G}(X; \cale, \calf)}(N,L)
\]
which defines the composition in $\C_{*}^{G}(X; \cale, \calf)$. The $C^{*}$-identity is inherited from $\C_{b}^{G}(X; \cale, \calf)$, so $\C_{*}^{G}(X; \cale, \calf)$ is a $C^{*}$-category.
\item To simplify notation, we adopt the following conventions. If $M$ is an object of any of our categories, we let $m_{x}$ be the rank of $M_{x}$ as an $A$-module. Furthermore, $M_{x} = A^{m_{x}}$ has a canonical basis, which we will denote by $e_{(x,1)}, e_{(x,2)}, \cdots e_{(x,m_{x})}$ as long as $M$ is clear from the context. We will usually be sloppy and write something like $\{ e_{(x,n)} \mid x \in X, n \in \IN \}$ for the collection of all such vectors occuring in $M$; of course, $e_{(x,n)}$ does not really make sense when $n > m_{x}$. 
\item A coarse $G$-invariant map $f \colon X \rightarrow Y$ induces functors 
\begin{align*}
f_{*} \colon \C^{G}(X; \cale, \calf) &\rightarrow \C^{G}(Y; \cale, \calf) \\ 
f_{*} \colon \C^{G}_{b}(X; \cale, \calf) &\rightarrow \C^{G}_{b}(Y; \cale, \calf) \\ 
f_{*} \colon \C_{*}^{G}(X; \cale, \calf) &\rightarrow \C_{*}^{G}(Y; \cale, \calf)
\end{align*}
via $f_{*}(M)_{y} = \oplus_{x \in f^{-1}y} M_{y}$. The conditions on a coarse map guarantee that this sum is finite and that the resulting object over $Y$ is locally finite.  Since this does not change the total object $T(M)$ and the total maps of morphisms, we indeed get a functor on $\C^{G}_{b}$. For simplicity, we choose to ignore the problem of how to order the summands, which becomes a problem once one wants $(f \circ g)_* = f_* \circ g_*$ for two coarse maps $f,g$. This can be fixed as indicated in \cite[2.4]{BFJR}.
\end{enumerate}
\end{definition}

\begin{example} 
\begin{enumerate}
\item Let $(X,d)$ be a metric space with $G$ acting on $X$ by isometries. Then we define the metric coarse structure $\cale_{d}$ on $X$ by declaring $E \subset X \times X$ to be controlled if and only if $E$ is symmetric and $G$-invariant and there is $R > 0$ such that for all $(x,y) \in E$, $d(x,y) < R$.
\item If $X$ is any $G$-space, we define the $G$-compact object support condition $\calf_{Gc}$ by declaring that $F \subset X$ is in $\calf$ if and only if $F$ is relative $G$-compact, i.e. the closure of $F$ is $G$-compact.
\item Given two coarse structures $(\cale, \calf)$ and $(\cale', \calf')$, we define their intersection $(\cale \cap \cale', \calf \cap \calf')$ to be the coarse structure obtained by taking any $E \in \cale$, $E' \in \cale'$ and forming the intersection $E \cap E'$ which is an element of $\cale \cap \cale'$, and similarly for $\calf \cap \calf'$. Sloppily speaking, this means that we enforce both control conditions; for example, in the category $C^{G}(X; \cale \cap \cale', \calf \cap \calf')$, a morphism is allowed if and only if it is both $\cale$- and $\cale'$-controlled.
\item If $f: X \rightarrow Y$ is any $G$-equivariant map and $(Y;  \cale, \calf)$ is a coarse structure, we obtain a pullback coarse structure $(X; f^{-1}(\cale), f^{-1}(\calf))$, where for each $E \in \cale$, we form $(f \times f)^{-1}(E) \subset X \times X$ and let $f^{-1}(\cale)$ be the collection of subsets of $X \times X$ obtained in this way, and $f^{-1}(\calf)$ consists of the sets $f^{-1}(F)$ for $F \in \calf$.
\end{enumerate}
\end{example}

\noindent For a $G$-algebra $A$, we define the twisted group ring $A \rtimes G$ to have elements finite formal sums
\[
\sum\limits_{g \in G} a_{g} g
\]
with $a_{g} \in A$ and the evident addition. Multiplication is defined such that
\[
g \cdot a = \sigma_{g}(a)g
\]
for $g \in G, a \in A$ and extended distributively, where $\sigma_{g}: A \rightarrow A$ denotes the action of $g$ on $A$. It is easy to check that a right $A \rtimes G$-module $M$ is the same as a right $A$-module with a right $G$-action such that for $m \in M, g \in G, a \in A$ we have
\[
(v \cdot g) \cdot a = (v \cdot \sigma_{g}(a)) \cdot g
\]
A map of right modules is the same as a $G$-equivariant $A$-module map. There is a dual statement for left modules. \\
Let $l^{2}(G,A)$ be the Hilbert module with basis $G$. Then $A$ acts from the left on $l^{2}(G,A)$ via $a \cdot h = \sigma_{h^{-1}}(a)h$ for $h \in G$. Also, $G$ acts on $l^{2}(G,A)$ from the left by left multiplication. The reduced group $C^{*}$-algebra $A \rtimes_{r} G$ is the completion of the resulting representation of $A \rtimes G$ on the bounded, adjointable operators on $l^{2}(G,A)$.

\begin{prop}
\label{twistedgroup}
For each object $M$ of $\C^{G}(X; \cale, \calf)$, $\C^{G}_{b}(X; \cale, \calf)$ respectively $\C^{G}_{*}(X; \cale, \calf)$, $T(M)$ is a right $A \rtimes G$-module and for each morphism $f \colon M \rightarrow N$ in $\C^{G}(X; \cale, \calf)$ $\C^{G}_{b}(X; \cale, \calf)$ respectively $\C^{G}_{*}(X; \cale, \calf)$, \\ $T(f) \colon T(M) \rightarrow T(N)$ is an $A \rtimes G$-module map. \\
Furthermore, if the $G$-action on $X$ is free, $T(M)$ is a free $A \rtimes G$-module. If $X = G$ with no control conditions, the category $\C^{G}_{b}(X; \cale, \calf)$ is equivalent to the category of free right $A \rtimes G$ - modules and the category $\C^{G}_{*}(X; \cale, \calf)$ is equivalent to the category of free right $A \rtimes_{r} G$ - modules.
\end{prop}

\begin{proof}
To decrease the confusion potential, let us first discuss the case $A = \IC$ with the trivial $G$-action. Let $M$ be an object of  $\C^{G}(X; \cale, \calf)$. Then $T(M)$ is a complex vector space. To define a $G$-action on $T(M)$, we use the $G$-action on the underlying space $X$. Let for $x \in X$ \\ $e_{(x,1)}, \dots , e_{(x,m_{x})}$ be the standard basis of $M_{x} = \IC^{m_x}$. Varying over all $x$, the vectors $e_{(x,i)}$ by definition form a basis for $T(M)$. The $G$-action is now defined by setting
\[
e_{(x,n)} \cdot g = e_{(g^{-1}x,n)}
\]
which is possible since $M_{x}$ and $M_{g^{-1}x}$ have the same dimension. This gives $T(M)$ the structure of a right $\IC G$-module. Note that, at least in the case of non-commutative coefficients, our definitions force us to use right modules: Composition of morphisms in $\C_{*}^{G}(X)$ is modeled on the usual matrix multiplication, and multiplying with a matrix from the left is a right module homomorphism. Consequently, $T(\phi)$ for $\phi$ a  morphism in $\C_{*}^{G}(X)$ is naturally a right module homomorphism, so to make it into a $\IC G$-module map, $G$ has to act from the right.\\
Now let $f \colon M \rightarrow N$ be a morphism in  $\C^{G}(X; \cale, \calf)$. By definition, $T(f)$ is complex linear. We have to check that $T(f)$ is $G$-equivariant. It suffices to check this on the basis vectors $e_{(x,n)}$ where we compute
\begin{align*}
T(f)(e_{x,n} \cdot g) &= T(f)(e_{g^{-1}x,,n}) \\
&= \sum\limits_{y \in X} f_{y,g^{-1}x}(e_{g^{-1}x,n}) \\
&= \sum\limits_{y \in X} f_{g^{-1}y,g^{-1}x}(e_{g^{-1}x,n}) \\
&= (\sum\limits_{y \in X} f_{y,x}(e_{x,n})) \cdot g \\
&= T(f)(e_{x,n}) \cdot g
\end{align*}
where we changed indexing from $y$ to $g^{-1}y$ after line $2$ and used that $f_{g^{-1}y,g^{-1}x}$ and $f_{y,x}$ are the same matrices by $G$-invariance of $f$. \\
If $X$ is a free $G$-space and $M$ is only nonzero over a single $G$-orbit, then clearly $T(M)$ is a free $\IC G$-module. If $M$ is an arbitrary object, decompose it as a direct sum of objects, each of which is only nonzero over a single $G$-orbit. Since the functor $T$ is compatible with direct sums, $T(M)$ is a free $\IC G$-module. \\
Finally, assume $X = G$ without any control conditions. Then the category $\C_{*}^{G}(G)$ has one object for each natural number $n$, namely the one whose dimension is $n$ at each point - this dimension is constant by $G$-invariance. Hence $\C_{*}^{G}(G)$ is equivalent to the category of free modules over the endomorphism ring of the object $M$ whose dimension is one at each point, so we are done once we see that
\[
\End_{\C_{*}^{G}(G)}(M) = C_{*}^{r}(G)
\]
By definition, $\End_{\C_{*}^{G}(G)}(M)$ is the completion of $\End_{\C_{b}^{G}(G)}(M)$. We will identify $\End_{\C_{b}^{G}(G)}(M)$ with $\IC G$ and the norm on $\End_{\C_{b}^{G}(G)}(M)$ with the usual norm used in defining the reduced $C^{*}$-algebra. \\
Let $\phi: M \rightarrow M$ be a morphism in $C_{b}^{G}(G)$, where $M$ is still the object with the dimension of $M_{g}$ equal to one for all $g$. Let $e_{g}$ be the standard generator (i.e. the $1$) of $M_{g} = \IC$. The morphism $\phi$ is uniquely determined by 
\[
\phi(e_{1}) = \sum\limits_{g \in G} a_{g} e_{g}
\]
since we have by definition
\[
\phi(e_{h}) = \sum\limits_{g \in G} \phi_{hg,h} e_{hg}
\]
and by $G$-invariance of $\phi$, we have $a_{g} = \phi_{g,1} = \phi(hg,h)$. \\
When identifying $\phi$ with 
 \[
 \sum\limits_{g \in G} a_{g}g^{-1} \in \IC G
 \]
composition in the endomorphisms and multiplication in $\IC G$ agree. To see this, note we only have to consider $\phi_{g}$, $\psi_{h}$ with $\phi(e_{1}) = e_{g^{-1}}$ and $\psi(e_{1}) = e_{h^{-1}}$ for some $g,h \in G$ since any morphism is a complex linear combination of these and both the composition in the category and multiplication in the group ring are bilinear. Note that $\phi_{g}(e_{h}) = e_{hg^{-1}}$ and not $e_{g^{-1}h}$; this forces the occurence of $g^{-1}$ instead of $g$. We hence have 
\[
(\psi_{h} \circ \phi_{g})(e_{1}) = \psi_{h}(e_{g^{-1}}) = e_{g^{-1}h^{-1}}
\] 
Since $\psi_{h}$ corresponds to $h$ and $\phi_{g}$ to $g$ in the group ring, it follows that the compositions agree.\\
Any coefficients $a_{g}$ are allowed here, i.e. we actually have $\C_{b}^{G}(G) = \C^{G}(G)$: Again, each  $\phi$ is a linear combination of the maps $\phi_{g}$ with $\phi_{g}(e_{1}) = e_{g^{-1}}$, and $\phi_{g}: T(M) \rightarrow T(M)$ is bounded. In fact, $T(M)$ is the pre-Hilbert space with basis $G$ and $T(\phi_{g})$ is right multiplication by $g^{-1}$.

Hence  $\End_{\C_{b}^{G}(G)}(M) =  \End_{\C^{G}(G)}$ is canonically identified with the group ring, and the norm it inherits is precisely the norm we obtain from letting $G$ act on $l^{2}(G)$ by left multiplication once one identifies the completion of $T(M)$ with $l^{2}(G)$ via $e_{g} \mapsto g^{-1}$. Completing, we hence obtain
\[
\End_{\C_{*}^{G}(G)}(M) = C_{*}^{r}(G)
\]
as desired.  \\
Now we turn to the case of arbitrary coefficients. The principle is the same, but the presence of the $G$-action on $A$ complicates the formulas. Clearly, $T(M)$ is a right $A$-module. To avoid confusion, we will write $\sigma_g: A \rightarrow A$ for the action of $g \in G$ on $A$. Since the action will be on the right, we will view $A$ as a right $G-C^{*}$-algebra by letting $g$ act as $\sigma_{g^{-1}}$. We define a right $G$-action on $T(M) = \bigoplus_{x \in X} A^{m_{x}}$ by 
\[
 (\sum\limits_{x \in X, n \in \IN} a_{x,n} e_{x,n}) \cdot g= \sum\limits_{x \in X, n \in \IN} (\sigma_{g^{-1}}(a_{x,n})) e_{g^{-1}x,n}
\]
which is well-defined since $Mg = M$ implies $m_{x} = m_{gx}$ for all $x$.
For $v = \sum\limits_{x \in X, n \in \IN} a_{x,n} e_{x,n}$ and $a \in A$, we have to check $(va) \cdot g = (v \cdot g) \sigma_{g^{-1}}(a) $ to obtain a right $A \rtimes G$-module. We compute
\begin{align*}
((\sum\limits_{x \in X, n \in \IN} a_{x,n} e_{x,n})a) \cdot g &= (\sum\limits_{x \in X, n \in \IN} a_{x,n} a e_{x,n}) \cdot g \\
&= \sum\limits_{x \in X, n \in \IN} \sigma_{g^{-1}}(a_{x,n}) \sigma_{g^{-1}}(a) e_{g^{-1}x,n} \\
&=  (v \cdot g) \sigma_{g^{-1}}(a)
\end{align*}
If $f \colon M \rightarrow N$ is a morphism in $\C^{G}(X; \cale, \calf)$, $T(f)$ is a right $A$-module homomorphism. We have to check $T(f)(v \cdot g) = T(f)(v) \cdot g$. We compute for one of the canonical basis vectors $e_{x,n}$
\begin{align*}
T(f)(e_{x,n} \cdot g) &= T(f)(e_{g^{-1}x,n}) \\
&= \sum\limits_{y \in X}  f_{y,g^{-1}x} (e_{g^{-1}x,n}) \\
&=  \sum\limits_{y \in X} f_{g^{-1}y, g^{-1}x} (e_{g^{-1}x,n})  \\
&= \sum\limits_{y \in X} \sigma_{g^{-1}}(f_{y,x}) (e_{g^{-1}x,n}) \\
&=  T(f)(e_{x,n}) \cdot g
\end{align*}
From the second to the third line, we just changed the indexing from $y$ to $g^{-1}y$. \\
Furthermore, if  $T(f)(v \cdot g) = T(f)(v) \cdot g$ and  $T(f)(w \cdot g) = T(f)(w) \cdot g$, then $T(f)((v+w) \cdot g) = T(f)(v+w) \cdot g$. It remains to check that for $a \in A$, if  $T(f)(v \cdot g) = T(f)(v) \cdot g$, then also  $T(f)(va \cdot g) = T(f)(va) \cdot g$; then by linearity and since we have checked  $T(f)(v \cdot g) = T(f)(v) \cdot g$ on a basis, we can conclude it is true for all $v$. We compute, using $va \cdot g = v \cdot g \cdot \sigma_{g^{-1}}(a)$, $T(f)(v \cdot g) = T(f)(v) \cdot g$ and the $A$-linearity of $T(f)$
\begin{align*}
T(f)(va \cdot g) &= T(f)(v \cdot g \cdot \sigma_{g^{-1}}(a)) \\
&= T(f)(v \cdot g) \cdot \sigma_{g^{-1}}(a) \\
&= T(f)(v) \cdot g  \cdot \sigma_{g^{-1}}(a) \\
&= T(f)(v) \cdot a \cdot g  \\
&= T(f)(va) \cdot g
\end{align*}
\noindent The same is true for $\C^{G}_{b}(X; \cale, \calf)$. By continuity, also all morphisms of $\C^{G}_{*}(X; \cale, \calf)$ induce right $A \rtimes G$-module homomorphisms. \\
If the $G$-action on $X$ is free, we can write $M = \bigoplus_{s \in G \backslash X} M_{s}$, where $M_{s}$ is a module supported only over the orbit $s$. Note that $M_{s}$ is determined by a single number $n$. By decomposing $M_{s}$ further as in $n = 1+1+\dots+1$, we can assume we are in the situation $X = G$ with $M$ the object with $M_{g} = A$ for all $g$. Then it is easy to check that $T(M)$ is a free $A \rtimes G$-module. \\
Now let $X = G$ and consider $\C_{b}^{G}(X; \cale, \calf)$. Let $M$ be the object which has $M_{g} = A$ for all $g \in G$. Since $G$ has only one $G$-orbit, each object of  $\C_{b}^{G}(X; \cale, \calf)$ is a direct sum of copies of $M$, so it suffices to see that $\End(M) = A \rtimes G$ as rings. \\
Let $\phi: M \rightarrow M$. Since $\phi$ is $G$-invariant, it is determined by the components $\phi_{e,g}$. We send $\phi$ to $\sum \phi_{e,g} g \in A \rtimes G$. This is a bijection; we have to check that it is compatible with composition. For this, let $\phi, \psi: M \rightarrow M$ be given. By decomposing $\phi$ and $\psi$, we may assume that only $\phi_{e, g}$ and $\psi_{e, h}$ are nonzero. Then the only potentially nonzero component of $\phi \circ \psi$ into $e$ is the one from $gh$, and we have 
\[
(\phi \circ \psi)_{e, gh} = \phi_{e, g} \circ \psi_{g,gh} = \phi_{e, g}  (g \cdot \psi_{e, h})
\]
Hence $\phi \circ \psi$ is sent to $\phi_{e, g}  (g \cdot \psi_{e, h}) gh \in A \rtimes G$, which happens to be the same  as the product of $\phi_{e,g} g$ and $\psi_{e,h} h$ in $A \rtimes G$. \\
For the final claim about $\C^{G}_{*}(G; \cale, \calf)$, we regard $M$ as an object of the category $\C^{G}_{*}(G; \cale, \calf)$. It again suffices to see that $\End(M) = A \rtimes_{r} G$. Both are completions of $A \rtimes G$, so it suffices to see that the norms are the same. Consider $T(M) = \bigoplus_{g \in G} A$, which we regard as a free $A$-module with basis $G$. The ring $A \rtimes G$ now acts on $T(M)$ from the left since $A \rtimes G$ lies in the endomorphisms of $M$. For $g \in G$, consider the endomorphism $\phi: M \rightarrow M$ with $\phi_{e,g} = 1$ and $\phi_{e,h} = 0$ for $h \neq g$. \\
Now $T(\phi): T(M) \rightarrow T(M)$ is the map sending a basis element $h$ to $hg^{-1}$, since $\phi_{hg^{-1},h}$ is the only component of $\phi$ leaving $h$. We hence obtain a representation $\rho$ of $G$ on $T(M)$.
For $a \in A$, consider the endomorphism $\psi$ of $M$ with $\psi_{e,e} = a$ and $\psi_{e,h} = 0$ for $h \neq e$. Then $T(\phi): T(M) \rightarrow T(M)$ is the map sending a basis element $g$ to  
$\sigma_{g}(a)g$ since the only nonzero component of $\psi$ leaving $g$ is $\phi_{g,g}$ which is $\sigma_{g}(a)$. \\
Now identify the completion of $T(M)$ with $l^{2}(G,A)$ by the map sending $g$ to $g^{-1}$. The representation of $G$ on $l^{2}(G,A)$ obtained from the representation $\rho$ constructed above is given by left multiplication. The representation of $A$ on $T(M)$ constructed above identifies with the representation
\[
g \mapsto \sigma_{g^{-1}}(a) g
\]
Hence the obtained representation of $A \rtimes G$ is precisely the one used in the definition of the reduced group $C^{*}$-algebra.
\end{proof}

\begin{rem}
As long as the $G$-action on $A$ is nontrivial, the category $\C_{*}^{G}(X; \cale, \calf, A)$ is \emph{not} an (right) $A$-linear category: The only way to make it so in general would be to define $f \cdot a$ as having components $(f \cdot a)_{(x,y)} = f_{(x,y)} \cdot a$, which fails to be $G$-invariant.  However, in all of our examples, $X$ will be of the form $G \times Y$, i.e. all orbits of the $G$-action have a canonical starting point, namely $(e,y)$. Then we can define
\[
(f \cdot a)_{(g,y),(h,z)}  = f_{(g,y),(h,z)} \cdot g(a)
\]
which is $G$-invariant. Then we have that $T(f \cdot a) = T(f) \cdot a$, where the latter multiplication is right multiplication by $a$ as given by the $A \rtimes G$-module structure on the target of $T(f)$. This is different from obvious right action of $A$. After all, regarding $A$ as a subring of $A \rtimes G$ via $a \mapsto a \cdot e$, $A \rtimes G$ is an $A$-bimodule. The identification (as, say, abelian groups) $A \rtimes G \cong \oplus_{G} A$ is an isomorphism of left $A$-modules, but not of right $A$-modules since we have $g \cdot a = \sigma_{g} a \cdot g$ in $A \rtimes G$ and not $ga = ag$.  
\end{rem}

\begin{example}
\label{noinverse}
This example illustrates an easy-to-miss subtlety in the definition of $\C^{G}(X)$ respectively $\C^{G}_{*}(X)$. We use the ring $\IC$ as coefficients for this example and to keep matters simple assume that $G$ is trivial. Pick a coarse space $X$. Let $\phi: M \rightarrow M$ be a morphism in $\C(X)$ such that $T(\phi): T(M) \rightarrow T(M)$ is an isomorphism of $\IC$-modules.  \\
Then it does in general \emph{not} follow that $\phi$ is an isomorphism in $\C(X)$ or $\C_{*}(X)$, and it may happen that $\phi$ is an isomorphism in $\C_{*}(X)$, but not in $\C(X)$. In other words, control does not carry over to inverses of morphisms. For a concrete example, let $X = \IN$ with its metric control structure. Let $M$ be the object of $\C(X)$ with $M_{x} = \IC$ for all $x$. We write $M = M_{1} \oplus M_{2} \oplus M_{3} \oplus \cdots$ where $M_{1}$ is the part of $M$ supported on $1$, $M_{2}$ is the part supported on $2$ and $3$, $M_{3}$ is the part supported on $4,5$ and $6$ etc., such that the support of $M_{i}$ consists of $i$ subsequent natural numbers. Pick a sequence $a_{n} \in \IC$.  Let $e_{1}, \cdots ,e_{n}, \cdots$ be the canonical orthonormal basis of $M$ with $e_{j}$ supported over the $j$-th number in the support of $M$. Then the map $\phi_{i}: M_{i} \rightarrow M_{i}$ given by 
\[
e_{j} \mapsto \begin{cases} e_{j}+ a_{i}e_{j+1} &\text{ if } j < i  \\ e_{j} &\text{ else } \end{cases}
\]
is $1$-controlled. Hence the direct sum of the various $\phi_{i}$ fits together to a map
\[
\phi: M \rightarrow M
\]
which is $1$-controlled. The matrix form of $\phi_{i}$ is
\[
\begin{pmatrix}
1 & a_{i} & 0 & 0 & \cdots \\
0 & 1 & a_{i} & 0 & \cdots \\
0 & 0 & 1 & a_{i} & \cdots \\
\cdots & \cdots & \cdots  & \cdots  & \cdots 
\end{pmatrix} =  \Id + a_i N
\]
where 
\[
N = \begin{pmatrix}
0 & 1 & 0 & 0 & \cdots \\
0 & 0 & 1 & 0 & \cdots \\
0 & 0 & 0 & 1 & \cdots \\
\cdots & \cdots & \cdots  & \cdots  & \cdots 
\end{pmatrix}
\]
is nilpotent. 
However, the inverse of $\phi_{i}: M_{i} \rightarrow M_{i}$ is $i$-controlled but not $(i-1)$-controlled: The inverse of the above matrix is $\sum\limits_{k = 0}^{i-1} (-a)^k N^{k}$, which has an entry in the top-right corner. Since $i$ can be arbitrarily large, the inverse of $\phi$ is not controlled and  $\phi$ is not an isomorphism in $\C(X)$. \\
If $\left\Vert a_{i} \right\Vert < 1$ for all $i$, it is straightforward to see that each $T(\phi_{i})$ is bounded by $2$, and hence so is $T(\phi)$ as an orthonormal sum of the $T(\phi_{i})$. If $a_{i}$ converges to zero sufficiently fast, the potential inverse for $\phi$ is still not controlled, but can be approximated by controlled morphisms since the components of the inverse of the above matrix become very small towards the top right corner if $a_{i}$ is small, and hence $\phi$ is invertible in $\C_{*}(X)$. \
\end{example}

\section{Quotient categories}

This section collects some language concerning coarse spaces and their associated categories. The terminology is mainly from \cite{BFJR}. Let $(X, \cale, \calf)$ be a coarse $G$-space and $Y \subset X$ a $G$-invariant subspace. For any $E \in \cale$, we define the $E$-thickening of $Y$ to be
\[
Y^{E} = \{x \in X \mid \text{ there is } y \in Y \text{ such that } (y,x) \in E \}
\]
This is again $G$-invariant: If $(y,x)$ is in $E$, then so is $(gy,gx)$ by our assumptions on $\cale$. In terms of the controlled category $C^{G}(X; \cale, \calf)$, one should think of $Y^{E}$ as the subset of $X$ you can reach from an object supported in $Y$ by an $E$-controlled morphism. 
In particular, if an object $M$ over $X$ is isomorphic to an object $N$ supported over $Y$ via an $E$-controlled morphism, $M$ must be supported over $Y^{E}$. Even more precisely, if the support of $N$ is contained in $F \in \calf$, then the support of $M$ is contained in $(F \cap Y)^{E}$. Since also the support of $M$ is contained in some $F' \in \calf$, we find that $\supp(M) \subset (F \cap Y)^{E} \cap F'$. Since $F'' = F \cup F'$ is also in $\calf$, we conclude $\supp(M) \subset (F'' \cap Y)^{E} \cap F''$. So for an object $M$ of $\C^{G}(X)$ to be isomorphic to one supported over $Y$, it is necessary that $\supp(M) \subset (F \cap Y)^{E} \cap F$ for some $F \in \calf$, $E \in \cale$. This leads us to define a new class of object support conditions as
\[
\calf(Y) = \{ (F \cap Y)^{E} \cap F \mid F \in \calf, E \in \cale   \}
\]
Now we can form the categories $\C^{G}_{b}(X; \cale, \calf(A))$ and $\C^{G}_{*}(X; \cale, \calf(Y))$. One has the inclusion functors
\begin{align*}
\C^{G}_{b}(Y; i^{-1}\cale, i^{-1}\calf ) &\rightarrow \C^{G}_{b}(X; \cale, \calf(Y)) \\
\C^{G}_{*}(Y; i^{-1}\cale, i^{-1}\calf ) &\rightarrow \C^{G}_{*}(X; \cale, \calf(Y))
\end{align*}

and the definitions are nearly made such that these two functors are equivalences. However, this is not true in general, as the following example shows:
\begin{example} Let $X$ be an infinite-dimensional Banach space and $Y = \{ 0 \}$. Let $\cale$ be the metric control condition and $\calf$ be the power set of $X$, i.e. there is no control at all for the support of objects. Pick a discrete, infinite subset of the unit disk in $X$ and let $M$ be some object of $\C^{G}_{b}(X; \cale, \calf(Y))$ supported on this set. This object cannot be equivalent to an object supported only on $Y$. 
\end{example}
The following definition is tailor-made to circumvent this problem.

\begin{definition}\cite[4.1]{BFJR}
Let $(X, \cale, \calf)$ be a coarse $G$-space. We say that $X$ is $G$-proper if for each $G$-compact  $K \subset X$, each $E \in \cale$ and each $F \in \calf$, we find a $G$-compact $K'$ such that $(F \cap K)^{E} \cap F \subset K'$.
\end{definition}

\begin{prop}
\label{G-properness}
If the coarse structure on $X$ is $G$-proper, the $G$-action on $X$ is free and $Y \subset X$ is closed and $G$-invariant, the inclusion functors
\begin{align*}
\C^{G}_{b}(Y; i^{-1}\cale, i^{-1}\calf ) &\rightarrow \C^{G}_{b}(X; \cale, \calf(Y)) \\
\C^{G}_{*}(Y; i^{-1}\cale, i^{-1}\calf ) &\rightarrow \C^{G}_{*}(X; \cale, \calf(Y))
\end{align*}
are equivalences of categories.
\end{prop}

\begin{proof}
We only consider the first functor; the claim for the second just follows by completion since the first functor is isometric.\\
The only thing to check is essential surjectivity; full and faithful follow immediately from the definitions: The allowed morphisms are the same on both sides. So let $M$ be an object of $\C^{G}_{b}(X; \cale, \calf(Y))$. By definition of $\calf(Y)$, there is $F \in \calf$ and $E \in \cale$ such that $\supp(M) \subset (F \cap Y)^{E} \cap F$. Hence we find for $x \in \supp(M)$ an $f(x) \in F \cap Y$ such that $(x, f(x)) \in \cale$. This yields a (probably uncontinuous) map $f: \supp(M) \rightarrow F \cap Y$ which we can choose to be $G$-equivariant, and $f_{*}(M)$ is controlled by $\calf$. We only have to check that $f_{*}(M)$ is locally finite; then $M \cong f_{*}(M)$ along an $\cale$-controlled isomorphism since $(x,f(x)) \in \cale$ for all $x$ and $f_{*}(M)$ is a legit object of $\C^{G}_{b}(Y; i^{-1}\cale, i^{-1}\calf )$. \\
If $K \subset Y$ is $G$-compact, then we find by $G$-properness a $G$-compact $K'$ such that $(F \subset K)^{E} \cap F \subset K'$. Hence, if $f(x) \in K$, we must have $x \in K'$. Since $M$ is locally finite, it follows that $M'$ is locally finite.
\end{proof}

\begin{definition}
\label{quotients}
Let $X$ be a $G$-proper coarse space and $Y \subset X$ a $G$-invariant subspace. Consider the inclusion of categories
\[
\C_{*}^{G}(X; \cale, \calf(Y)) \rightarrow \C_{*}^{G}(X; \cale, \calf)
\]
We define the quotient category 
\[
\C_{*}^{G}(X; \cale, \calf)^{>Y}
\]
of this inclusion as follows. Its objects are the objects of $\C_{*}^{G}(X; \cale, \calf)$. We say that a morphism of $\C_{*}^{G}(X; \cale, \calf)$ is completely continuous if it factors through an object of $\C_{*}^{G}(X; \cale, \calf(Y))$. The completely continuous morphisms form an ideal, in the sense that if $f$ is completely continous, then so are $f \circ g$ and $g \circ f$ for any $g$. We say that $f: M \rightarrow N$ is approximately completely continuous if it  is in the norm closure of the completely continuous maps inside $\Hom(M, N)$.  Note that also the approximately completely continuous maps form an ideal in the above sense. Finally, we define the morphisms of $\C_{*}^{G}(X; \cale, \calf)^{>Y}$ to be the morphisms of $\C_{*}^{G}(X; \cale, \calf)$, modulo the approximately continuous maps. The composition of morphisms in $\C_{*}^{G}(X; \cale, \calf)$ descends to the quotients because of the ideal property. The resulting category $\C_{*}^{G}(X; \cale, \calf)^{>Y}$ inherits a $C^{*}$-norm under which it is again a $\C^{*}$-category, namely
\[\left\Vert f \right\Vert_{\C_{*}^{G}(X; \cale, \calf)^{>Y}} = \inf \left\Vert f-g \right\Vert_{\C_{*}^{G}(X; \cale, \calf)}
\]
where the infimum is taken over all approximately completely continuous maps $g$ with the same source and target as $f$.
\end{definition}

\begin{rem}
\label{quots}
The reader familiar with algebraic $K$-theory should note that this is not a Karoubi quotient as defined in \cite{karoubifil}. Indeed, the relationship of $\C_{*}^{G}(X; \cale, \calf)^{>Y}$ and the Karoubi quotient $\C_{b}^{G}(X; \cale, \calf)^{>Y}$ is as follows. The category  $\C_{b}^{G}(X; \cale, \calf)^{>Y}$ inherits a pseudonorm from  $\C_{b}^{G}(X; \cale, \calf)$, defined for a morphism $f: M \rightarrow N$ as
\[
\left\Vert f \right\Vert_{Y} = \inf_{g} \left\Vert f-g \right\Vert
\]
where the infimum is taken over all $g: M \rightarrow N$ which go to zero in  $\C_{b}^{G}(X; \cale, \calf)^{>Y}$. It is easy to see that $\C_{b}^{G}(X; \cale, \calf)^{>Y}$ with this norm satisfies all conditions of a normed category except that $\left\Vert f \right\Vert$ can be $0$ even if $f$ is nonzero, i.e. it is a pseudonormed category. Indeed, the approximately completely continuous maps  which are not completely continuous do not go to $0$ in  $\C_{b}^{G}(X; \cale, \calf)^{>Y}$, but by their very definition have norm $0$. We can make a normed category out of a pseudo-normed category by dividing out all morphisms of norm $0$. The resulting category is a normed category which we can complete to obtain a complete normed category. The result of this procedure is precisely  $\C_{*}^{G}(X; \cale, \calf)^{>Y}$. In particular, $\C_{b}^{G}(X; \cale, \calf)^{>Y}$ is not a subcategory of  $\C_{*}^{G}(X; \cale, \calf)^{>Y}$. \\
The reader should compare this situation to the corresponding situation in normed spaces: Given two normed spaces $B \subset A$, the quotient $A/B$ only inherits a pseudonorm, unless B is closed in A. Dividing out the elements of norm $0$ and then completing the resulting normed space to a Banach space $\overline{A/B}$ gives the same result as completing $A$ and $B$ and then considering $\bar{A}/\bar{B}$. The above is the categorical analogue of this. \\
Put another way, $A$ and $B$ give rise to the following diagram:
\[
\xymatrix{
B  \ar@{^{(}->} [rr] \ar@{^{(}->}[d] & & A \ar@{->>}[rr]  \ar@{=}[d] & & A/B \ar@{->>}[d] \\
\overline{B} \cap A \ar@{^{(}->} [rr] \ar@{^{(}->}[d] & & A \ar@{->>}[rr]  \ar@{^{(}->}[d] & & A/(\overline{B} \cap A) \ar@{^{(}->}[d] \\
\overline{B} \cap A \ar@{^{(}->} [rr]  & & \overline{A} \ar@{->>}[rr]  & & \overline{A/B} \\
}
\]
and similarly, we have a diagram of categories and functors, with $\D = \C_{*}^{G}(X; \cale, \calf(Y)) \cap \C_{b}^{G}(X; \cale, \calf)$:
\[
\xymatrix{
\C_{b}^{G}(X; \cale, \calf(Y)) \ar@{^{(}->} [r] \ar@{^{(}->}[d]  & \C_{b}^{G}(X; \cale, \calf) \ar@{->>}[r]  \ar@{=}[d] &  \C_{b}^{G}(X; \cale, \calf)^{> Y} \ar@{->>}[d] \\
\C_{*}^{G}(X; \cale, \calf(Y)) \cap \C_{b}^{G}(X; \cale, \calf) \ar@{^{(}->} [r] \ar@{^{(}->}[d] &  \C_{b}^{G}(X; \cale, \calf) \ar@{->>}[r]  \ar@{^{(}->}[d] &  \C_{b}^{G}(X; \cale, \calf)/\D \ar@{^{(}->}[d] \\
\C_{*}^{G}(X; \cale, \calf(Y)) \ar@{^{(}->}[r]  & \C_{*}^{G}(X;  \cale, \calf) \ar@{->>}[r]  &  \C_{*}^{G}(X;  \cale, \calf)^{> Y} \\
}
\]
where in the last column, we first divide out elements of norm $0$ and then complete the resulting category.
\end{rem}

%% file: chapters/ktheoryofcstarcat.tex
\chapter{$K$-theory of controlled $C^*$-categories}

In this section, we will give two definitions of the topological $K$-theory of controlled $C^{*}$-categories, one by directly defining groups $K_{0}$ and $K_{1}$ and one by associating a $C^{*}$-algebra and hence a space respectively spectrum to the category. We will check that these two definitions yield the same groups and play out the definitions against each other to prove various properties of $K$-theory. For example, it is not at all obvious from the $C^{*}$-algebra definition that equivalent categories have the same $K$-theory, but it is quite obvious that the directly defined $K$-groups are isomorphic, which allows us to conclude that the two $C^{*}$-algebras are $K$-equivalent through the backdoor. Since we need a spectrum-level definition of $K$-theory, using only the group approach would not suffice.

\section{$K$-theory of $C^{*}$-categories}

In this section, we define the abelian groups $K_0$ and $K_{1}$ associated to a $C^{*}$-category $\C$. Here it will not be necessary to assume that $\C$ is of the form $\C_{*}^{G}(X; \cale, \calf)$. The definitions are similar to those in \cite[II.3]{kar} for Banach categories.

\begin{definition}
Let $\C$ be a small $C^*$-category. A projection in $\C$ is a self-adjoint idempotent $p \colon A \rightarrow A$ for some object $A$ of $\C$. The group $K_0(\C)$ is the free abelian group generated by the projections $p \colon A \rightarrow A$ in $\C$, modulo the following relations:
\begin{itemize}
\item[(i)] For two projections $p \colon A \rightarrow A$ and $q \colon B \rightarrow B$, we form the projection $p \oplus q \colon A \oplus B \rightarrow A \oplus B$ and force the relation $p \oplus q = p+q$.
\item[(ii)] For any object $A$ and the zero projection $0_A: A \rightarrow A$, we force $0_A = 0$.
\item[(iii)] If $p,q: A \rightarrow A$ are two projections joined by a path of projections inside $\End(A)$, we have the relation $p = q$. 
\end{itemize}
A unitary in $\C$ is an endomorphism $u \colon A \rightarrow A$ such that $uu^{*} = 1 = u^{*}u$. The group $K_1(\C)$ is the free abelian group generated by the unitary endomorphisms $u \colon A \rightarrow A$ in $\C$, modulo the following relations:
\begin{itemize}
\item[(i)] For unitaries $u,v \colon A \rightarrow A$ which can be joined by a path of unitaries inside $\End(A)$, we have $u = v$.
\item[(ii)] $1_A = 0$ for any object $A$
\item[(iii)] For $u \colon A \rightarrow A$ and $v \colon B \rightarrow B$, $u \oplus v \colon A \oplus B \rightarrow A \oplus B$ is unitary and we force the relation $u \oplus v = u + v$.
\end{itemize}
\end{definition}

\begin{rem}
By definition, only endomorphisms can be unitaries. A map $u: A \rightarrow B$ satisfying $uu^{*} = \Id_{B}$ and $u^{*}u = \Id_{A}$ will be called unitary isomorphism. 
\end{rem}

Instead of the relation of  being connected by a path, we could have used unitary equivalence or Murray-von Neumann equivalence in our definition of $K_0$:

\begin{definition}
Let $p \colon A \rightarrow A$ and $q \colon B \rightarrow B$ be two projections in $\C$.
\begin{itemize}
\item[(i)] $p$ and $q$ are unitarily equivalent if there is a unitary morphism $u \colon A \rightarrow B$ such that $upu^* = q$.
\item[(ii)] $p$ and $q$ are Murray-von Neumann equivalent if there is a morphism $v: A \rightarrow B$ such that $v^*v = p$ and $vv^* = q$
\end{itemize}
\end{definition}

\begin{prop}
\label{murray}
Let $p$ and $q$ be as above.
\begin{itemize}
\item[(i)] If $p \oplus 0_B$ and $0_A \oplus q$ are connected by a path of projections, they are unitarily equivalent.
\item[(ii)] If $p$ and $q$ are unitarily equivalent, they are Murray-von Neumann equivalent. 
\item[(iii)] If $p$ and $q$ are Murray-von Neumann equivalent, $p \oplus 0_B \oplus 0_A \oplus 0_B$ and $0_A \oplus q \oplus 0_A \oplus 0_B$ are unitarily equivalent via a unitary connected to $1 \colon A \oplus B \rightarrow A \oplus B$ with a path of unitaries, and hence $p \oplus 0_B \oplus 0_A \oplus 0_B$ and $0_A \oplus q \oplus 0_A \oplus 0_B$ are connected by a path of projections.
\end{itemize}
It follows that we could have used any of the three relations as (iii) of our definition of $K_0$.
\end{prop}

\begin{proof}
The proofs are the obvious adaptions of the corresponding facts for $C^*$-algebras, see for example \cite[4.1.10]{HigsonRoe}or \ref{murneu}.  \\
For the first part, to simplify notation, we write $p$ for $p \oplus 0_B$ and $q$ for $0_A \oplus q$ and work in the $C^*$-algebra $D = \End(A \oplus B)$. By subdividing the path into smaller paths, we may assume $\left\Vert p-q \right\Vert < 1$. Now set $x = qp+(1-q)(1-p)$. We have $xp = qx$ and 
\[
x-1 = (2q-1)(p-q)
\]

Since $\left\Vert p-q \right\Vert < 1$ and $\left\Vert 2q-1 \right\Vert \leq 1$, we have $\left\Vert x-1\right\Vert < 1$ and hence $x$ is invertible.  Set $u = x(x^*x)^{-\frac{1}{2}}$. It is then easy to check that $up = qu$ since $x^*x$ and hence $(x^*x)^{-\frac{1}{2}}$ commute with $p$. \\
For the second claim, just set $v = qu$ for $u: A \rightarrow B$ with $p = u^*qu$. For the third claim, let $v: A \rightarrow B$ be a Murray-von Neumann equivalence between $p$ and $q$. We may as well consider $v$ as a morphism 
\[
v: A \oplus B \rightarrow A \oplus B
\]
where it implements a Murray- von Neumann equivalence between $p \oplus 0_B$  and $0_A \oplus q$. \\
Consider the matrix
\[
\begin{pmatrix} v & 1-vv^* \\ v^*v-1 & v^*  \end{pmatrix}
\] 
as an endomorphism of $A \oplus B \oplus A \oplus B$. A straightforward calculation using the fact that $v^*v$ and $vv^*$ are projections shows that this matrix is a unitary automorphism $u$ of $A \oplus B \oplus A \oplus B$ with 
\[
u(p \oplus 0_B \oplus 0_A \oplus 0_B)u^* = 0_A \oplus q \oplus 0_A \oplus 0_B
\]
The path
\[
t \rightarrow \begin{pmatrix} \cos(\frac{\pi}{2}t)v & 1-(1-\sin(\frac{\pi}{2}t))vv^* \\ (1-\sin(\frac{\pi}{2}t))v^*v-1 & \cos(\frac{\pi}{2}t)v^*  \end{pmatrix}
\] 
connects the matrix to 
\[
\begin{pmatrix} 0 & 1 \\ -1 & 0\end{pmatrix}
\]
through unitaries. Specializing the path to $v = 1$, we see that this matrix is connected to the identity as claimed. Hence $p \oplus 0_B \oplus 0_A \oplus 0_B$ and $0_A \oplus q \oplus 0_A \oplus 0_B$ are connected by a path of projections.
\end{proof}

The reader may be worried that we are not passing to isomorphism classes of unitaries or projections; the following proposition shows that this is not necessary.

\begin{prop}
Let $\C$ be a $C^{*}$-category and $A,B$ objects of $\C$. Let 
\[
\phi \colon A \rightarrow B
\]
be a unitary isomorphism, i.e. $\phi \circ \phi^{*} = \Id_B$ and $\phi^* \circ \phi = \Id_A$, and  \\$u \colon A \rightarrow A$ a unitary. Let $v = \phi u \phi^{*} \colon B \rightarrow B$. Then $u = v$ in $K_{1}(\C)$. The same holds for  projections and $K_{0}(\C)$.
\end{prop}

\begin{proof}
Since $-v = v^{*}$ in $K_{1}(\C)$, it suffices to see that 
\[
u \oplus v^{*} \colon A \oplus B \rightarrow A \oplus B
\] is zero in $K_{1}(\C)$. But under conjugation with the isomorphism 
\[
1 \oplus \phi \colon A \oplus A \rightarrow A \oplus B
\]
this map corresponds to $u \oplus u^{*} \colon A \oplus A \rightarrow A \oplus A$, which is connected to the identity of $A \oplus A$ via a path of unitaries. This path yields under conjugation with $1 \oplus \phi$ a path of unitaries from $u \oplus v^{*}$ to $\Id_{A \oplus B}$ as desired. The same argument applies to the $K_{0}$-case.
\end{proof}

\begin{rem}
\label{isos}
Let $\phi \colon A \rightarrow B$ be an isomorphism in a $C^{*}$-category. Then also $\phi^{*}$ is an isomorphism and $\phi^{*} \phi \colon A \rightarrow A$ is an automorphism of $A$. Via the continuous calculus and since $\phi^{*} \phi$ is by definition of a $C^{*}$-category an positive element of $\End(A)$, we can form 
\[
(\phi^{*} \phi)^{-\frac{1}{2}} \colon A \rightarrow A
\]
Then $\phi (\phi^{*} \phi)^{-\frac{1}{2}} \colon A \rightarrow B$ is a unitary isomorphism. So isomorphic objects in $C^{*}$-categories are in fact unitarily isomorphic.
\end{rem}

\begin{prop}
\label{kcat}
Let $\C$ be a $C^{*}$-category.
\begin{enumerate}
\item If $A$ is a $C^{*}$-algebra, then the $K$-theory of the $C^{*}$-category $A_{\oplus}$ is the same as the $K$-theory of $A$.
\item The groups $K_{i}(\C)$ depend functorially on $\C$, i.e. $C^{*}$-functors induce maps on $K_{i}(\C)$.
\item The group $K_{0}(\C)$ is isomorphic to $K_{0}^{\alg}(\Idem(\C))$.
\item In the definition of $K_{1}(\C)$, we may replace unitary by invertible. Then we obtain a quotient map $K_{1}^{\alg}(\C) \rightarrow K_{1}(\C)$.
\item Naturally isomorphic $C^{*}$-functors induce the same maps on $K$-groups.
\end{enumerate}
\end{prop}

\begin{proof}
Item (i) follows by definition. (ii) is also easy: If $F \colon \C \rightarrow \D$ is a $C^{*}$-functor, we define $K_{0}(F) \colon K_{0}(\C) \rightarrow K_{0}(\D)$ by sending a projection $p$ in $\C$ to the projection $F(p)$ in $\D$. Since $F$ is additive, continuous and sends zero morphisms to zero morphisms, the relations defining $K_{0}$ are respected. The same argument applies to $K_{1}$. \\
For (iii), we define $K_0^{\alg}(\Idem(\C))$ as in  \cite[3.1.6]{Ros} as the free abelian group on the isomorphism classes of objects of $\Idem(\C)$ modulo the relation $(A,p) + (B,q) = (A \oplus B, p \oplus q)$. The definition in \cite{Ros} is more general, covering exact categories, but this is what it boils down to in the split exact context. We have already seen in \ref{projections} that any object in $\Idem(\C)$ is isomorphic to an object $(A,p)$ with $p$ a projection, i.e. $p = p^{*}$ in addition to $p^{2} = p$. \\
Now let $(A,p)$ and $(B,q)$ be two isomorphic objects of $\Idem(\C)$ with $p$ and $q$ self-adjoint. The difficult step is to see that $p \oplus 0$ and $0 \oplus q$ are unitarily equivalent projections in $\End(A \oplus B)$. Once we have this, the relations defining $K_{0}$ and $K_{0}^{\alg}$ are the same, so the groups are the same. \\
So let $f \colon (A,p) \rightarrow (B,q)$ be an isomorphism with inverse $g \colon (B,q) \rightarrow (A,p)$. This means that
\begin{align*}
fp &= qf \quad  & gf &= p \\
gq &= pg \quad & fg &= q  \\
\end{align*}
The trick is now to consider
\[
\begin{pmatrix}
1-p & pg \\
qf & 1-q
\end{pmatrix} \colon A \oplus B \rightarrow A \oplus B
\]
compare \cite[1.2.1]{Ros}. Using the above relations, one checks that
\[
\begin{pmatrix}
1-p & pg \\
qf & 1-q
\end{pmatrix} \cdot \begin{pmatrix}
1-p & pg \\
qf & 1-q
\end{pmatrix} = \begin{pmatrix}
1 & 0 \\
0 & 1
\end{pmatrix} 
\]
and
\[
\begin{pmatrix}
1-p & pg \\
qf & 1-q
\end{pmatrix} \cdot \begin{pmatrix}
p & 0 \\
0 & 0
\end{pmatrix}  \cdot \begin{pmatrix}
1-p & pg \\
qf & 1-q
\end{pmatrix} = \begin{pmatrix}
0 & 0\\
0 & q
\end{pmatrix} 
\]
To simplify notation, we write $p$ for $p \oplus 0$ and $q$ for $0 \oplus q$. We have now found 
\[
a = \begin{pmatrix}
1-p & pg \\
qf & 1-q
\end{pmatrix} 
\]
such that $apa^{-1} = q$, but $a$ is not yet a unitary. Since $a$ and hence $a^{*}a$ is invertible, we can form $(a^{*}a)^{-\frac{1}{2}}$ via the continuous function calculus and we define $u = a(a^{*}a)^{-\frac{1}{2}}$. We compute
\[
uu^{*} = a (a^{*}a)^{-\frac{1}{2}}(a^{*}a)^{-\frac{1}{2}}a^{*} = a (a^{*}a)^{-1} a^{*} = 1
\]
and similarly $u^{*}u = 1$. So $u$ is a unitary. Furthermore, $p$ commutes with $a^{*}a$ since
\[
pa^{*}a = a^{*}qa = a^{*}ap
\]
and hence $p$ also commutes with $(a^{*}a)^{-\frac{1}{2}}$. It follows that
\begin{align*}
upu^{*} &= a(a^{*}a)^{-\frac{1}{2}} p(a^{*}a)^{-\frac{1}{2}}a^{*} = a(a^{*}a)^{-\frac{1}{2}} a (a^{*}a)^{-\frac{1}{2}} pa^{*}= \\
&=a(a^{*}a)^{-\frac{1}{2}} (a^{*}a)^{-\frac{1}{2}} a^{*}q = q
\end{align*}

\noindent Part (iv) follows directly from \ref{unitaryvsinvertible} and the same proof as in \ref{k1quotient}. For the last statement, note that also induced maps on $K_{1}$ are the quotient maps of the induced maps on $K_{1}^{\alg}$. Naturally isomorphic functors induce the same map on $K_{1}^{\alg}$ and hence also on the quotient groups $K_{1}$. Note that the claim really requires a proof since we do not assume that the natural isomorphism is unitary. The claim for $K_{0}$ is immediate from (iii). \\
Alternatively, let $\eta: F \rightarrow G$ be a natural isomorphism of $C^{*}$-functors. We argue that then $F$ and $G$ are unitarily isomorphic, i.e. there is a natural transformation $\tau: F \rightarrow G$ such that $\tau_{X}^{*}\tau_{X}= \Id$ and $\tau_{X}\tau_{X}^{*} = \Id$ for all objects $X$. As in \ref{isos}, we set
\[
\tau_{X}= \eta_{X} (\eta_{X}^{*}\eta_{X})^{-\frac{1}{2}}
\]
We have to check that $\tau$ is natural. First note that $\eta^{*}$ defines a natural transformation from $G$ to $F$: We have for $f: X \rightarrow Y$
\begin{align*}
F(f) \eta_{X}^{*} &= (\eta_{X} F(f^{*}))^{*} \\
&= (G(f)^{*} \eta_{Y})^{*} \\
&= \eta^{*}_{Y} G(f) 
\end{align*}
It follows that $\eta^{*} \eta \colon F \rightarrow F$ is a self-adjoint natural self-transformation of $F$. Then clearly every polynomial in $\eta^{*}\eta$ is also a natural self-transformation of $F$, and since for each object $X$ $(\eta_{X}^{*}\eta_{X})^{-\frac{1}{2}}$ is a norm-limit of polynomials in $\eta_{X}^{*}\eta_{X}$, it follows that $(\eta_{X}^{*}\eta_{X})^{-\frac{1}{2}}$ is natural as well. Then also $\tau$ is natural as a composition of natural transformations. 
\end{proof}

\begin{rem}
Let us issue one warning. Given a coarse $G$-space $X$, consider the following (wrong) argument for the wrong statement that the map $K^{\alg}_{1}(\C_{b}^{G}(X)) \rightarrow K_{1}(\C_{*}^{G}(X))$ is always surjective: \\
Start with an invertible morphism $\phi: M \rightarrow M$ in $\C_{*}^{G}(X)$ representing an element of $K_{1}(\C_{*}^{G}(X))$. Since the morphisms of $\C_{b}^{G}(X)$ are dense inside the morphisms of $\C_{*}^{G}(X)$, we find $\phi': M \rightarrow M$ in $\C_{b}^{G}(X)$ arbitrarily close to $\phi$, say such that $\left\Vert \phi-\phi' \right\Vert < \frac{1}{10}$. Then $\phi'$ is invertible and represents the same $K_{1}$-class as $\phi$. Now $\phi'$ represents an element of $K^{\alg}_{1}(\C_{b}^{G}(X))$ which is a preimage for $\phi$. \\
The argument fails because $\phi'$, while being a morphism in $\C_{b}^{G}(X)$ and invertible in $\C_{*}^{G}(X)$, may not be invertible in $\C_{b}^{G}(X)$, see \ref{noinverse} for a counterexample. 
\end{rem}

\begin{prop}
Any element of $K_{0}(\C)$ can be represented in the form $p-q$ for two projections $p,q$. Any element of $K_{1}(\C)$ can be represented by a single unitary $u$.
\end{prop}

\begin{proof}
Let $x = \sum p_{i}- \sum q_{i}$ be an arbitary element of the free abelian group generated by the projections, with $p_{i}$ and $q_{i}$ projections. Then set $p = \oplus p_{i}$ and $q = \oplus q_{i}$. The relations in $K_{0}(\C)$ then give $x = p-q$. The same argument applies to $K_{1}$, so we can write any element of $K_{1}(\C)$ as $u-v$ for unitaries $u$ and $v$. If $u: A \rightarrow A$ and $v: B \rightarrow B$, we may replace $u$ by $u \oplus 1_{B}: A \oplus B \rightarrow A \oplus B$ and $v$ by $1_{A} \oplus v: A \oplus B \rightarrow A \oplus B$, hence we can assume that $u$ and $v$ are unitary endomorphisms of the same object. Now a standard argument yields $-v = v^{*}$ and $u-v = u+v^{*} = u \oplus v^{*}$, proving the claim; compare \ref{standardk1}.

\end{proof}

\noindent We also record the following for reference:

\begin{prop}
If $p_{i}: A \rightarrow A$ and $q_{i}: B \rightarrow B$ are projections for $i=0,1$, then $p_{0}-q_{0}=p_{1}-q_{1}$ in $K_{0}(\C)$ if and only if there is an object $C$ such that $p_{0} \oplus q_{1} \oplus 0_{C}$ and $p_{1} \oplus q_{0} \oplus 0_{C}$ are homotopic in the projections of $A \oplus B \oplus C$.
 Similarly, if $u: A \rightarrow A$ and $v: B \rightarrow B$ are two unitaries, then $u=v$ in $K_{0}(\C)$ if and only if there is a third object $C$ such that $u \oplus 1_{B} \oplus 1_{C}$ and $1_{A} \oplus v \oplus 1_{C}$ are homotopic in the unitaries of $A \oplus B \oplus C$.
\end{prop}

\begin{proof}
We discuss the $K_{1}$-case. We can also define $K_{1}(\C)$ as the free abelian group generated by the homotopy classes of unitaries, modulo the relations $\Id_{D} = 0$ for each object $D$ and $u \oplus v = u +v$ for all unitaries $u,v$, which also makes sense on the homotopy level. If $u,v$ are as in the proposition, we may pass to $1 \oplus v, u \oplus 1: A \oplus B \rightarrow A \oplus B$ and hence can assume that $u,v$ are endomorphisms of the same object $A$. Now $u = v$ in $K_{1}$ if and only if we find finitely many unitaries $u_{i}, v_{i}$ and finitely many objects $A_{j}$, $B_{k}$ such that
\[
u+\sum\limits_{i \in I} v_{i} \oplus u_{i} + \sum\limits_{j \in J} \Id_{A_{j}} = v+ \sum\limits_{i \in I} u_{i}+ \sum\limits_{i \in I} v_{i} + \sum\limits_{k \in K} \Id_{B_{k}}
\] 
as formal sums of homotopy classes of unitaries. In particular, this means that the $B_{k}$ are only a reordering of the $A_{j}$.  But now, also the homotopy classes of
\[
u \oplus \bigoplus\limits_{i \in I} (v_{i} \oplus u_{i}) \oplus \bigoplus\limits_{j \in J} \Id_{A_{j}} 
\]
and
\[
v \oplus \bigoplus\limits_{i \in I} u_{i}\oplus \bigoplus\limits_{i \in I} v_{i} \oplus \bigoplus\limits_{k \in K} \Id_{B_{k}}
\]
have to agree once we view them as unitaries of the same object after reordering the sums. So $u = v$ if and only if there is a unitary $w: C' \rightarrow C'$ such that $u \oplus w$ and $v \oplus w$ are homotopic, where $w = \bigoplus\limits_{i \in I} u_{i}\oplus \bigoplus\limits_{i \in I} v_{i} \oplus \bigoplus\limits_{k \in K} \Id_{B_{k}}$. Finally, this implies that $u \oplus w \oplus w^{*}$ and $v \oplus w \oplus w^{*}$ are homotopic, and since $w \oplus w^{'}$ is in turn homotopic to the identity of $C = C' \oplus C'$, we see that 
$u \oplus \Id_{C}$ and  $v \oplus \Id_{C}$ are homotopic. The converse direction is obvious.
\end{proof}

\begin{prop}
\label{unions}
Let $\C$ be a $C^*$-category, and let $\C_{i} \subset \C$ be a directed system of full subcategories such that $\C = \cup \C_{i}$. Then the canonical map
\[
\colim_{i} K_{n} \C_{i} \rightarrow K_{n}(\C)
\] 
is an isomorphism for $n = 0,1$.
\end{prop}

\begin{proof}
This is straightforward.
\end{proof}

\section{Biexact functors and topological $K$-theory}

Let $\C$ be a $C^{*}$-category and $\cale$ an exact category in the sense of Quillen, see for example \cite[3.1.1]{Ros}.

\begin{definition}
A bilinear functor $F: \cale \times \C \rightarrow \C$ is called $C^*$-biexact if for each $A \in \cale$, $F(A,-): \C \rightarrow \C$ is a $C^{*}$-functor and for each exact sequence $A \rightarrow B \rightarrow C$ in $\cale$ and each object $M$ of $\C$, the sequence $0 \rightarrow F(A,M) \rightarrow F(B,M) \rightarrow F(C,M) \rightarrow 0$ is split exact.
\end{definition}

\begin{prop}
A $C^*$-biexact functor $F: \cale \times \C \rightarrow \C$ induces a $K^{\alg}_{0}(\cale)$-module structure on $K_{i}(\C)$, $i=0,1$, where an element $[A]-[B] \in K^{\alg}_{0}(\cale)$ acts on $K_{i}(\C)$ as $F(A,-)_{*}-F(B,-)_{*}: K_{i}(\C) \rightarrow K_{i}(\C)$.
\end{prop}

\begin{proof}
Note that we need the $C^{*}$-functor condition to even make sense of $F(A,-)_{*}-F(B,-)_{*}$.  What we have to check is that this definition is well-defined, i.e. that if $[A]-[B] = [A']-[B']$ in $K^{\alg}_{0}(\cale)$, we have $F(A,-)_{*}-F(B,-)_{*}  = F(A',-)_{*}-F(B',-)_{*}$.  However, this equation is true for the functors induced in the \emph{algebraic} $K$-theory of $\C$, and even of $\Idem(\C)$, and since the topological $K$-theory is a quotient of the algebraic $K$-theory (and induced maps are the induced quotient maps), the claim follows.
\end{proof}

We will usually abuse notation and call a $C^*$-biexact functor just biexact.

\section{A $C^{*}$-algebra representing $K$-theory of controlled $C^*$-categories}

Fix a $G$-$C^{*}$-algebra $A$. Let $\C = \C^{G}_{*}(X; \cale, \calf)$ be the $C^{*}$-category associated to the coarse $G$-space $(X, \cale, \calf)$. We associate a $C^{*}$-algebra to $\C$ whose $K$-groups are the $K$-groups of $\C$. Everything we say will also apply to quotient categories $ \C^{G}_{*}(X; \cale, \calf, A)^{> Y}$; we will not explicitly carry them through the notation. The construction is a more category-friendly and less functional-analytic variant of the Roe algebra; see \cite[Chapter 6]{HigsonRoe} and \cite{hpr}, in particular the remarks after 6.3. We will indicate at the end of this chapter that our constructions and the ones in \cite{HigsonRoe} are more closely related than one may think on the first glance. \\
In the following, when we say that two objects are equal, we really mean equal, not isomorphic.

\begin{definition}
Define a partial order on the set of objects of $\C$ as follows: If $M$, $N$ are two objects, then for each $x \in X$ look at $M_{x}$ and $N_{x}$. These are determined by their dimensions $m_{x}$ and $n_{x}$, and we say $M \leq N$ if for all $x$, $m_{x} \leq n_{x}$. Let $P$ be the poset obtained in this way, which we regard as a category in the usual way. Define a functor 
\[
T: P \rightarrow C^{*} \minus \alg
\]
as follows: Send $M \in \C$ to $\End(M)$. If $M \leq N$, we embed $\End(M)$ into $\End(N)$ as follows: There is a (in fact unique) way to write $M = N \oplus L$ with $L_{x} = A^{n_{x}-m_{x}}$. Then embed $\End(M)$ to $\End(M \oplus L)$ via $f \mapsto f \oplus 0$. This is clearly compatible with compositions, i.e. if $M \leq N \leq L$, then the two maps $\End(M) \rightarrow \End(L)$ agree. Now define the $C^*$-algebra 
\[
T(X; \cale, \calf) = T(\C)  = \colim_{P} T
\]
By abuse of notation, we will often leave out $(\cale, \calf)$ and just write $T(X)$ if no confusion is possible.
Similarly, define a functor
\[
U: P \rightarrow \mathcal{TOP}
\]
by sending $M \in \C$ to the topological space $U(M)$ of unitary endomorphisms of $M$. If $M \leq N$, we embed $U(M)$ into $U(N)$ as above, where we view $\End(M)$ and $\End(N)$ as being non-unital since the induced map between them is not unital. Now define
\[
U(\C)  = \colim^{met}_{P} U
\]
\end{definition}

From a philosophical point of view, this definition has serious flaws. For example, we obtain maps $U(M) \rightarrow U(N \oplus M)$ which will almost never be the obvious map $u \rightarrow 1 \oplus u$, which is annoying. However, the actual map $U(M) \rightarrow U(N \oplus M)$  differs from the obvious one only by conjugation by an automorphism of $N \oplus M$, so everything works out fine on the $K$-theory level, as we will now see.  Let us first record the following:

\begin{prop}
The induced map $U(\C) \rightarrow U(T(\C))$ is a weak homotopy equivalence.
\end{prop}

\begin{proof}
This follows from \ref{cont}.
\end{proof}

\begin{prop}
There is a canonical isomorphism 
\[
f \colon K_{i}(T(\C)) \rightarrow K_{i}(\C)
 \]
\end{prop}

\begin{proof}
The algebra $T(\C)$ is a filtered colimit of $C^{*}$-algebras. Since $K$-theory of $C^{*}$-algebras commutes with filtered colimits, it suffices to see that $K_{i}(\C)$ is isomorphic to $\colim_{P} K_{i}(\End(M))$. \\
For each $M$, we have a canonical map $K_{i}(\End(M)) \rightarrow K_{i}(\C)$ coming from the obvious $C^{*}$-functor $\End(M)_{\oplus} \rightarrow \C$. This map sends a projection $p: M^{\oplus k} \rightarrow M^{\oplus k}$ to the projection $p$ itself, where we now consider $M^{\oplus k}$ as an object of $\C$, and similarly for unitaries. For $M \leq N = M \oplus L$, the map $K_{i}(\End(M)) \rightarrow K_{i}(\End(N))$ is compatible with the two maps to $K_{i}(\C)$: If $p: M^{\oplus k} \rightarrow M^{\oplus k}$ is a projection, $p$ is sent (up to conjugation with the canonical isomorphism $M^{\oplus k} \oplus L^{\oplus k} \cong N^{\oplus k}$, which does not affect matters in $K$-theory) to $p \oplus 0_{L^{\oplus k}}$ in $\End(M^{\oplus k} \oplus L^{\oplus k})$ and then to $p \oplus 0_{L^{\oplus k}} = p$ in $K_{0}(\C)$, and similarly for unitaries. So we obtain a map $\colim_{P} K_{i}(\End(M)) \rightarrow K_{i}(\C)$. This map is clearly surjective: a projection $p \colon M \rightarrow M$ certainly comes from $K_{0}(\End(M))$, and similarly for unitaries. It remains to see injectivity.  \\
To see this, note that $\colim_{P} K_{i}(\End(M))$ is also a quotient of the free abelian group generated by the projections respectively unitaries in $\C$. So it is enough to see that the relations in $K_{i}(\C)$ are also true in  $\colim_{P} K_{i}(\End(M))$. But this is now nearly obvious, except for the relation $p \oplus q = p+q$ respectively $u \oplus v = u+v$. However, if for example $u: A \rightarrow A$, $v: B \rightarrow B$, then $u \oplus v \in \End(A \oplus B)$ and the relations in $K_{1}(\End(A \oplus B))$ give $u \oplus v = u+v$ in the colimit. The same argument applies to $K_{0}$. This finishes the proof.
\end{proof}

\section{Representing $T(X)$ as a concrete $C^{*}$-algebra}

\fxnote{Relative case for quotients?}

Let $(X; \cale, \calf)$ be a coarse space. We consider the Hilbert $A$-module $\calh = l^{2}(X \times \IN,A)$ with basis $X \times \IN$.  The total object of each object $M$ of $\C_{*}^{G}(X; \cale, \calf)$ can be viewed as a Hilbert submodule of $\calh$: For each $x$, if $M_{x}$ is $n$-dimensional, we take the vectors $(x,1), \dots (x,n)$ and $T(M)$ is the Hilbert submodule of $\calh$ spanned by all these vectors. Note that this submodule is complemented. 
A morphism $f: M \rightarrow N$ gives rise to an operator on $\calh$ by first projecting onto $T(M)$, then using $T(f)$ to map to $T(N)$ and then include $T(N)$ into $\calh$. This gives rise to a bounded, adjointable operator on $\calh$ since the projection onto $T(M)$ and the inclusion of $T(N)$ into $\calh$ are adjointable.

\begin{prop}
\label{concrete}
The algebra $T(X; \cale,\calf)$ is the completion of the bounded, adjointable operators on $\calh$ which come from a morphism $f \in \C_{*}^{G}(X; \cale,\calf)$. Hence we can view $T(X; \cale,\calf)$ as a subalgebra of the bounded, adjointable operators on $\calh$.
\end{prop}

\begin{proof}
For an object $M$, let $E(M) \subset \calb(\calh)$ be the image of $\End(M)$ in $\calb(\calh)$. Letting $M$ vary over the objects of $\C_{*}^{G}(X; \cale,\calf)$, we obtain a directed system of $C^{*}$-subalgebras of $\calb(\calh)$ which is easily seen to be the same directed system we used to define $T(X; \cale,\calf)$. \\
Potentially, the bounded, adjointable operators coming from morphisms in $C_{*}^{G}(X)$ might be bigger than this algebra since we have so far only handled endomorphisms, not arbitrary morphisms. However, we can view any morphism as an endomorphism in the following way: Begin with a morphism $f \colon M \rightarrow N$.
Define an object $M'$ as 
\[
M'_{x} = \begin{cases} 
0 \text{ if } m_{x} \geq n_{x} \\
A^{n_{x}-m_{x}} \text{ else }
\end{cases}
\]
and similarly $N'$ as 
\[
N'_{x} = \begin{cases} 
0 \text{ if } n_{x} \geq m_{x} \\
A^{m_{x}-n_{x}} \text{ else }
\end{cases}
\]
These are vaild objects of $\C_{*}^{G}(X)$ since their support is contained in the support of $N \oplus M$. We have
\[
M \oplus M' = N \oplus N'
\]
and can view $f$ as an endomorphism of this object via
\[
\begin{pmatrix}
f & 0 \\
0 & 0
\end{pmatrix}: M \oplus M' \rightarrow N \oplus N'
\]
The images of $f$ and of 
\[
\begin{pmatrix}
f & 0 \\
0 & 0
\end{pmatrix}
\]
in $\calb(\calh)$ are the same, so we see that the morphism $f$ is already covered in the endomorphisms.
\end{proof}

\section{Functoriality and lack thereof}

Unfortunately, $T(\C) = T(\C^{G}_{*}(X; \cale, \calf))$ does not depend functorially on the coarse space $X$. For example, let $X = S^{0}$ and $Y = pt$ without any control conditions and with $G$ the trivial group. The map $X \rightarrow Y$ induces a $C^{*}$-functor $F: \C(X) \rightarrow \C(Y)$. But let $M$ be the object over $X$ with $M_{x} = A$ for both points of $X$ and $N$ the object with $N_{x} = A^{2}$. Then the square
\[
\xymatrix{
\End(M) \ar[rr] \ar[d] & & \End(N) \ar[d] \\
\End(F(M)) \ar[rr] & & \End(F(N)) \\
}
\]
does not commute: The upper composition corresponds to embedding $T(M) = A^{2}$ into $T(F(N)) = A^{4}$ as the first and third summand, whereas the lower composition corresponds to embedding it as the first and second summand. Consequently, we cannot hope for an induced map on the colimits. There seems to be no way around this problem; but using some tricks, we can at least turn $\C^{G}_{*}(X; \cale, \calf)$ into a functor from coarse spaces to the homotopy category of spaces or spectra.  Let us first notice that everything is fine as long as we only consider injective maps. 
\fxnote{Are there induced maps somewhere?}

\begin{prop}
Let $(X; \cale, \calf)$ and $(Y; \cale', \calf')$ be two coarse $G$-spaces and $f: X \rightarrow Y$ a coarse injective $G$-map. Then we get a functorial $C^{*}$-homomorphism $T(f): T(X) \rightarrow T(Y)$, i.e. if $g: Y \rightarrow Z$ is another injection, we have $T(g \circ f) = T(g) \circ T(f)$.
\end{prop}

\begin{proof}
For each object $M$ of $\C^{G}_{*}(X)$, we obtain a $C^{*}$-homomorphism $\End(M) \rightarrow \End(F(M))$. For $M \leq N$, we have to check commutativity of the square 
\[
\xymatrix{
\End(M) \ar[rr] \ar[d] & & \End(N) \ar[d] \\
\End(F(M)) \ar[rr] & & \End(F(N)) \\
}
\]
to obtain an induced map on the colimits. But this is immediate to check since $f$ is injective.

\end{proof}

Now let $f \colon X \rightarrow Y$ be an arbitrary coarse map with $X$ and $Y$ free $G$-spaces. We will employ a trick to make $f$ injective up to equivalence of categories.

\begin{definition}
\label{dummys}
Let $f \colon X \rightarrow Y$ be a coarse $G$-map. We define a new coarse structure on $X \times Y$ with its diagonal $G$-action as follows:
\begin{enumerate}
\item[-] For morphisms, we pull back the morphism support $\cale$ of $Y$ back to $X \times Y$. \fxnote{explain}
\item[-] For the object support, we also pull back the object support $\calf$ of $Y$ and also insist on finiteness in $X$-direction: $F \subset X \times Y$ is controlled if and only if it is the pullback of a controlled subset of $Y$ and for each $y \in Y$, there are only finitely many $x \in X$ with $(x,y) \in F$. Note that for any object $M$ over $X$ and any $y \in Y$, the conditions on a coarse map force that $\supp(M) \cap f^{-1}(y)$ is finite, so that pushing $M$ forward along $x \mapsto (x, f(x))$ yields a vaild object over $X \times Y$.
\end{enumerate}
We denote the category $\C^{G}_{*}(X \times Y)$ with the above control conditions as $\C^{G}_{*}(Y)_{X}$ since the $X$ is only a dummy variable we need to create enough space. Let $f_{X} \colon X \rightarrow X \times Y$ be the map sending $x$ to $(x, f(x))$. This is again a coarse $G$-map and, since it is injective,  we obtain an induced map $(f_{X})_{*} \colon T(X) \rightarrow T(X \times Y)$, which by abuse of notation we will often write as $f_{*}$ again.
\end{definition}

\begin{prop}
\label{welldef}
Let $g \colon Y \rightarrow X$ be any $G$-equivariant map of sets, which exists since $X$ and $Y$ are free $G$-spaces. Then the map $g \times \Id \colon Y \rightarrow X \times Y$ is coarse and induces an equivalence of categories $F \colon \C_{*}^{G}(Y) \rightarrow \C_{*}^{G}(X \times Y)$. Also the induced $C^{*}$-homomorphism 
\[
i \colon T(Y) \rightarrow T(X \times Y)
\] 
is a $K$-equivalence and in particular induces a weak homotopy equivalence of $K$-theory spectra. The homotopy class of this map does not depend on the choice of $g$. By inverting $K(i)$, we hence obtain a well-defined homotopy class 
\[
K(f) \colon K(T(X)) \stackrel{K(f_X)}{\rightarrow} K(T(X \times Y)) \stackrel{K(i)^{-1}}{\rightarrow} K(T(Y))
\]
\end{prop}

\begin{proof}
It is easy to check that $g \times \id$ is compatible with the control conditions. Let $M$ be an object in $\C_{*}^{G}(X \times Y)$. Let $M'$ be the object in $\C_{*}^{G}(X \times Y)$ with
\[
M'_{(x,y)} = \begin{cases} 0 \text{ if } x \neq g(y) \\
		\bigoplus_{x' \in X: M_{(x',y) \neq 0}} M_{(x',y)} \text{ else }
		\end{cases}
\]
The object $M'$ is clearly in the image of the functor 
\[
F \colon \C_{*}^{G}(Y) \rightarrow \C_{*}^{G}(X \times Y)
\] 
Since there are no morphism control conditions in the $X$-direction of $X \times Y$, it is easy to write down an isomorphism from $M$ to $M'$, so $F$ is essentially surjective.  That $F$ is full and faithful is immediate. Hence $F$ induces an isomorphism on the $K$-groups of the categories $\C_{*}^{G}(Y)$ and $\C_{*}^{G}(X \times Y)$. Since these $K$-groups are also the $K$-groups of $T(Y)$ respectively $T(X \times Y)$, the map $i \colon T(Y) \rightarrow T(X \times Y)$ is a $K$-equivalence. \\
It remains to see that  the homotopy class of $K(i)$ does not depend on the choice of $g$. Let $g'$ be another choice. For each $y \in Y$, pick a self-bijection $\phi_y \colon X \rightarrow X$ sending $g(y)$ to $g'(y)$ and consider the bijection
\begin{align*}
F \colon X \times Y &\rightarrow X \times Y \\
(x,y) &\mapsto (\phi_y(x),y)
\end{align*}
$F$ induces a unitary operator $V$ from the Hilbert module with basis $X \times Y \times \IN$ to itself by acting on the basis as $(x,y,n) \mapsto (F(x,y),n)$. Considering $T(X \times Y)$ as a subalgebra of $\calb(\calh(X \times Y \times \IN))$, it is easy to check that $V$ is a multiplier of $T(X \times Y)$, i.e. for each operator $S \in T(X \times Y)$, also $VS$ and $SV$ are in $T(X \times Y)$: By \ref{concrete}, it suffices to check that for any morphism $f: M \rightarrow N$ in $\C_{*}^{G}(X \times Y)$, which we can view as a bounded, adjointable operator on $l^{2}(X \times Y \times \IN, A)$ as in \ref{concrete},  $Vf$ and $fV$ - again viewed as bounded operators on $l^{2}(X \times Y \times \IN, A)$ - also come from morphisms in the category. However $Vf$ comes from 
\[
f' \colon M \rightarrow F_{*}N
\]
which is obtained by composing $f$ with the isomorphism $N \cong F_{*}(N)$ identifying $N_{(x,y)}$ with $F_{*}(N)_{F(x,y)}$, and similarly for $fV$. Since there are no control conditions in the $X$-direction and $F$ preserves the fibers in $X$-direction, the isomorphism $N \cong F_{*}(N)$ is indeed controlled. \\
Furthermore, $i$ and $i'$ precisely differ by conjugation with $V$. Since conjugation with $V$ induces the identity in $K$-theory up to homotopy, $K(i)$ and $K(i')$ are homotopic. 
\fxnote{Explain}
\end{proof}

So we now have defined a potential functor $K$ from coarse $G$-spaces to the homotopy category spectra on objects and morphisms and have shown this depends on no choices. It remains to see that $K$ is compatible with compositions.

\begin{prop}
Let $f: X \rightarrow Y$ and $g: Y \rightarrow Z$ be two coarse $G$-maps, with $X$ and $Y$ free $G$-spaces. Then $K(g) \circ K(f) = K(g \circ f)$ in the homotopy category of spectra. 
\end{prop}

\begin{proof}
For a map $s: A \rightarrow B$, we write $s_{A} \colon A \rightarrow A \times B$ or \\ $s_{A} \colon A \rightarrow B \times A$, depending on context, for the map sending $a$ to $(a, s(a))$ respectively $(s(a),a)$. Fix $G$-equivariant maps (probably non-continuous) $h \colon Y \rightarrow X$ and $k \colon Z \rightarrow Y$.
Consider the following diagram of $C^*$-algebras and $*$-homomorphisms:

\[
\xymatrix{
T(X) \ar[rr]^{(g \circ f)_{X}} \ar[ddr]_{f_{X}} \ar@{.>}@/^5pc/[rrrr]^{K(g \circ f)} \ar@{.>}@/_5pc/[ddddrr]_{K(f)}  & & T(X \times Z) \ar@<-2ex>@{-->}[d]_{f_{X} \times \id}^{\simeq}  \ar@<2ex>@{-->}[d]^{\id \times k_{Z}}_{\simeq}  & & T(Z) \ar[ll]^{(h \circ k)_{Z}}_{\simeq} \ar[ddl]^{k_{Z}}_{\simeq} \\
& & T(X \times Y \times Z) \\
& T(X \times Y) \ar[ur]^{\id \times g_{Y}} & & T(Y \times Z) \ar[ul]_{h_{Y} \times \id}^{\simeq} \\
\\
& & T(Y) \ar[uul]^{h_{Y}}_{\simeq} \ar[uur]_{g_{Y}} \ar@{.>}@/_5pc/[uuuurr]_{K(g)} \\
\\}
\]

Some explanations are in order.  We always use the last variable of the various products to build the actual coarse structures, and the remaining factors are dummies. For example, $T(X \times Y \times Z)$ is $T(\C_{*}^{G}(Z)_{X \times Y})$ with notation as in \ref{dummys}. All solid maps in the diagram are induced by the maps of spaces as indicated. The left-hand dashed arrow is such that the left-hand part of the diagram commutes and the right-hand dashed arrow is such that the right-hand part commutes. All arrows labeled with a $\simeq$ are $K$-equivalences. The dotted arrows only exist in the homotopy category of spaces or spectra after taking $K$-theory of the involved $C^{*}$-algebras. They are defined such that the triangles
\[
\xymatrix{
& & & T(X \times Y) \\
T(X) \ar[urrr] \ar@{.>}[rrr]_{K(f)} & & & T(Y) \ar[u]_{\cong}
}
\]
\[
\xymatrix{
& & & T(Y \times Z) \\
T(Y) \ar[urrr] \ar@{.>}[rrr]_{K(g)} & & & T(Z) \ar[u]_{\cong}
}
\]
\[
\xymatrix{
& & & T(X \times Z) \\
T(X) \ar[urrr] \ar@{.>}[rrr]_{K(g \circ f)} & & & T(Z) \ar[u]_{\cong}
}
\]
commute in the homotopy category after taking $K$-theory. The claim is that the triangle consisting of the three dotted arrows is commutative in the homotopy category. \\
The two dashed arrows are different on-the-nose, but induce homotopic maps on $K$-theory. This argument is similar to the multiplier argument in the proof of \ref{welldef}. 
There is a bijection 
\[
F \colon X \times Y \times Z \rightarrow X \times Y \times Z
\]
 such that 
\[
F \circ (f_{X} \times \id) = \id \times k_{Z}
\]
and such that $F$ preserves the $Z$-component, i.e. we have $F(x,y,z) = (x',y',z)$ for all $x,y,z$.
The map $F$ induces a unitary operator on $l^{2}(X \times Y \times Z, A)$ which again is a multiplier of $T(X \times Y \times Z)$, and the two maps $T(X \times Z) \rightarrow T(X \times Y \times Z)$ differ by conjugation with this unitary. So when we consider things in the homotopy category after taking $K$-theory, we can replace the two dashed arrows with a single arrow which makes the whole diagram commutative. Now a diagram chase and using that many maps are actually isomorphisms immediately implies that the triangle consisting of the three dotted arrows is commutative.
\end{proof}

\section{Properties of $K$-theory}

This section collects some properties of topological $K$-theory of controlled categories. We will assume that all involved coarse $G$-spaces are free $G$-spaces.

\begin{prop}
\label{homotopyfiber}
Let $X$ be a proper coarse $G$-space and $Y \subset X$ a $G$-invariant subspace. Then:
\begin{enumerate}
\item The inclusion $\C^{G}_{b}(X; \cale, \calf(Y))  \rightarrow \C^{G}_{b}(X; \cale, \calf)$ is a Karoubi filtration. Hence there is a fibration sequence
\[
\IK^{-\infty}\C^{G}_{b}(Y; i^{-1}\cale, i^{-1}\calf) \rightarrow \IK^{-\infty}\C^{G}_{b}(X; \cale, \calf) \rightarrow \IK^{-\infty}\C^{G}_{b}(X; \cale, \calf)^{> Y}
\]
where $\IK^{-\infty}$ denotes the non-connective algebraic $K$-theory spectrum and we used $\ref{G-properness}$ to replace $\C^{G}_{b}(X; \cale, \calf(Y))$ by $\C^{G}_{b}(Y; i^{-1}\cale, i^{-1}\calf)$.
\item There is a fibration sequence
\[
K \C^{G}_{*}(Y; i^{-1}\cale, i^{-1}\calf) \rightarrow K \C^{G}_{*}(X; \cale, \calf) \rightarrow K \C^{G}_{*}(X; \cale, \calf)^{> Y}
\]
\end{enumerate}
\end{prop}

\begin{proof}
The first part is standard; see \cite{karoubifil}. 
For the second part, consider the projection $\C^{G}_{*}(X; \cale, \calf) \rightarrow \C^{G}_{*}(X; \cale, \calf)^{> Y}$ and let $J$ be the kernel of this projection, i.e. the category without identities with the same objects as $\C^{G}_{*}(X; \cale, \calf)$, but with only those morphisms which map to zero in $\C^{G}_{*}(X; \cale, \calf)^{> Y}$. In the language of \ref{quotients}, $J$ consists precisely of the approximately continuous morphisms. Even though $J$ is not really a category - not all objects have identites-, nothing keeps us from defining the $C^{*}$-algebra $T(J)$ as in the usual case as
\[
T(J) = \colim_{P} \End_{J}(M)
\]
where now $\End_{J}(M)$ is a potentially non-unital algebra. \\
We will prove the following facts:
\begin{enumerate}
\item The induced inclusion $T(\C^{G}_{*}(Y; i^{-1}\cale, i^{-1}\calf)) \rightarrow T(J)$ induces an isomorphism in $K$-theory.
\item The sequence $J \rightarrow T(\C^{G}_{*}(X; \cale, \calf))  \rightarrow T(\C^{G}_{*}(X; \cale, \calf)^{> Y})$ is an ideal sequence and hence induces a fiber sequence in $K$-theory.
\end{enumerate}
Together, the two facts imply the claim. \\
Let us start with (i). We will actually give two different proofs of this. For the first one, let $E \in \cale$ be arbitrary. We define a new object support condition $\calf(E)$ on $X$ as follows: A subset $F \subset X$ is in $\calf(E)$ if and only if there is $F' \in \calf$ such that $F= (F' \cap Y)^{E}$. We define furthermore the $C^{*}$-algebras
\[
T(E) = T(\C_{*}^{G}(X; \cale, \calf(E))) 
\]
for each $E \in \cale$. The categories $\C_{*}^{G}(X; \cale, \calf(E))$ form a directed system of categories in the $E$-variable, and for $E \subset E'$ the inclusion
\[
\C_{*}^{G}(X; \cale, \calf(E)) \rightarrow \C_{*}^{G}(X; \cale, \calf(E'))
\]
is an equivalence of categories: Indeed, any object in any of our categories is by definition isomorphic to an object supported only over $Y$ since $X$ was assumed to be $G$-proper. Now we claim that
\[
T(J) = \colim_{E} T(E)
\]
which is directed by the definition of a coarse structure. Since all the structure maps on the right-hand side are $K$-equivalences, the claim follows from this. \\
For each $E$, the category $\C_{*}^{G}(X; \cale, \calf(E))$ is a subcategory of $J$. This gives rise to compatible inclusion maps
\[
T(E) \rightarrow T(J)
\]
and hence to an injection
\[
\colim_{E} T(E) \rightarrow T(J)
\]
By definition of $J$, each morphism in $J$ can be approximated arbitrarliy closely by morphisms in some $\C_{*}^{G}(X; \cale, \calf(E))$: Each morphism factoring through $\C_{*}^{G}(Y; \cale, \calf)$ lies in some $\C_{*}^{G}(X; \cale, \calf(E))$. This implies surjectivity of the map $\colim_{E} T(E) \rightarrow T(J)$ and hence
\[
T(\C^{G}_{*}(Y; i^{-1}\cale, i^{-1}\calf)) \rightarrow T(J)
\]
is a $K$-equivalence. \\
Unfortunately, we would like to obtain the statement of the proposition in a slightly more general setting in the proof of \ref{stability}, where the proof above does not apply. So we give a second proof of (i) which will also apply in the more general setting. This proof will apply whenever we are in a situation where we have a Karoubi filtration before completing to $C^*$-categories. \\
We will make heavy use of the space-level continuity of $K$-theory \ref{cont}. Pick a map $f: S^{n} \rightarrow U(T(J))$. By \ref{cont}, we find a map homotopic to $f$ which factors through some $U(\End_{J}(M))$ where $M$ is an object of $J$. For simplicity, let us call this map $f$ again, i.e. $f: S^{n} \rightarrow U(\End_{J}(M))$. Recall that the object $M$  is actually an object of $ \C^{G}_{*}(X; \cale, \calf)$.  For each $F \in \calf(Y)$, we write $M = M_{F} \oplus M_{X-F}$ where $M_{F}$ is supported only over $F$ and $M_{X-F}$ is only supported over $X-F$. This yields a (non-unital) inclusion map $\End_{J}(M_{F}) \rightarrow \End_{J}(M)$. By definition of the quotient construction, $\End_{J}(M)$ has to be the completion of the ascending union of the $\End_{J}(M_{F})$: The morphisms coming from some $\End(M_{F})$ are precisely those factoring over $Y$ (thanks to $G$-properness, compare \ref{G-properness}), and these morphisms are dense in $\End_{J}(M) $ by definition. Now apply \ref{cont} again to conclude that $f$ is actually homotopic to a map $S^{n} \rightarrow U(M_{F})$ for some $F$. But $M_{F}$ is an object of $\C_{*}^{G}(X; \cale, \calf(Y))$ which is equivalent to $\C^{G}_{*}(Y; i^{-1}\cale, i^{-1}\calf)$. This implies surjectivity of
\[
T(\C^{G}_{*}(Y; i^{-1}\cale, i^{-1}\calf)) \rightarrow T(J)
\]
on the $K$-theory level. Injectivity is the same argument, using that also $S^{n} \times I$ is compact. \\
For ii), we need some more notation. Let $T^{\alg}(J)$ be the colimit of $\End_{J}(M)$ in the category of (non-unital) algebras, and similarly we define $T^{\alg}(\C^{G}_{*}(X; \cale, \calf) )$ and $T^{\alg}(\C^{G}_{*}(X; \cale, \calf)^{> Y})$. Since algebraic colimits are exact, we have a short exact sequence of algebras
\[
T^{\alg}(J) \rightarrow T^{\alg}(\C^{G}_{*}(X; \cale, \calf) ) \rightarrow T^{\alg}(\C^{G}_{*}(X; \cale, \calf)^{> Y})
\]
We have to argue this carries over to the completions to $C^{*}$-algebras. Certainly, after completing the composition of the two maps is still $0$. Also, the left-hand map is an isometric inclusion, which also carries over to the completions. What we have to see is that the quotient pre-$C^{*}$-norm $T^{\alg}(\C^{G}_{*}(X; \cale, \calf)^{> Y})$ inherits is the same as the one it inherits by the colimit structure. Let $f: M \rightarrow M$ be an endomorphism of an object of $\C^{G}_{*}(X; \cale, \calf)$. Its norm in the quotient is
\[
\left\Vert f \right\Vert_{Y} = \inf_{g:M \rightarrow M \in J} \left\Vert f-g \right\Vert
\]
Considering $f$ as an element of $T^{\alg}(\C^{G}_{*}(X; \cale, \calf) )$, its quotient norm is
\[
\left\Vert f \right\Vert_{Y} = \inf_{g \in T^{\alg} J} \left\Vert f-g \right\Vert
\]
The second norm is clearly at most as large as the first one. However, we can decompose $f-g$ in the second norm as the direct orthogonal sum of two operators: one whose support is precisely the support of $M$, and one whose support is the rest. The second operator makes the norm larger, and the first operator is already covered in the first norm. This proves the claim.

\fxnote{ii) is too sketchy.}
\end{proof}

The following proposition concerns potential Eilenberg Swindles on controlled categories.

\begin{prop}
\label{swindle}
Let $(X, \cale, \calf)$ be a coarse $G$-space. Let $s: X \rightarrow X$ be a $G$-equivariant map such that:
\begin{enumerate}
\item For every compact $K \subset X$, every $F \in \calf$ and every $n \geq 0$ the set $(s^{n})^{-1}(K) \cap F$ is compact and eventually empty.
\item For every $E \in \cale$ and $F \in \calf$, there is an $E' \in \cale$ with
\[
\bigcup_{n \geq 1} (s \times s)^{n} (E \cap F \times F) \subset E'.
\]
\item For every $F \in \calf$ there exists an $F' \in \calf$ with  $\bigcup_{n \geq 1}s^{n}(F) \subset F'$.
\item For every $F \in \calf$ the set $\{(x, s(x)) \mid x \in \bigcup_{n \geq 1} s^{n}(F)\}$ is contained in some $E \in \cale$.
\end{enumerate}
Then the topological $K$-theory of $\C_{*}^{G}(X; \cale, \calf)$ is trivial.
\end{prop}

\begin{proof}
The proof is the same as the one in \cite[4.4]{BFJR}, with some continuity arguments added in.  For each $n \geq 1$, consider the map $s^{n}: X \rightarrow X$ and the induced $C^{*}$-functor $s^{n}: \C^{G}_{*}(X; \cale, \calf) \rightarrow \C^{G}_{*}(X; \cale, \calf)$. Now consider $S = \bigoplus_{n \geq 1} s_{n}: \C^{G}_{*}(X; \cale, \calf) \rightarrow \C^{G}_{*}(X; \cale, \calf)$. As long as this makes sense, i.e. respects the control conditions and does not create non-locally finite objects, there are no worries about continuity since if $f: M \rightarrow N$ is bounded, $s^{n}(f)$ also has norm $\left\Vert f \right\Vert$ and an infinite direct sum of bounded operators with uniformly bounded norm is bounded. The conditions i)-iv) guarantee that $S$ respects the control conditions and that $\Id \oplus S \cong S$.  The condition i) guarantees that $S(M)$ is locally finite for each object $M$, ii) guarantees that $S(f)$ is controlled for each controlled $f$, iii) checks that $S(M)$ is a controlled object and finally iv) guarantees that $\tau  = s_{*}: \Id \oplus S \cong S$ is controlled.
\end{proof}

\fxnote{A counterexample to Mayer-Vietoris? And some text somewhere}

To obtain homological properties of topological $K$-theory of controlled categories, we need a kind of Mayer-Vietoris principle. For this end, we make the following definition.

\begin{definition}
Let $(X, \cale, \calf)$ be a $G$-proper coarse space and $A,B \subset X$ be two $G$-invariant subsets with $A \cup B = X$. We say that the triple $(X,A,B)$ is coarsely excisive if the following condition holds: \\
For each $E \in \cale$ and $F \in \calf$ there are $E' \in \cale$ and $F' \in \calf$ such that $(A \cap F)^{E} \cap (B \cap F)^{E} \cap F \subset (A \cap B \cap F')^{E'} \cap F'$.
\end{definition}

\begin{rem}
Let $M$ be an object of $C_{*}^{G}(X; \cale, \calf)$. Then $\supp(M) \subset (A \cap F)^{E} \cap (B \cap F)^{E} \cap F$ says that $M$ is close to both $A$ and $B$, and $\supp(M) \subset (A \cap B \cap F')^{E'} \cap F'$ says that $M$ is close to $A \cap B$. \fxnote{Bad remark}
\end{rem}

\begin{prop}
\label{MV}
Let $(X, \cale, \calf)$ be a $G$-proper coarse space and $A,B \subset X$ be two $G$-invariant subsets with $A \cup B = X$. If the triple $(X,A,B)$ is coarsely excisive,  the following diagram induces a homotopy pull-back square in $K$-theory
\[
\xymatrix{
\C_{*}^{G}(A \cap B, i^{-1}\cale, i^{-1}\calf) \ar[dd] \ar[rr] & & \C_{*}^{G}(A, i_{A}^{-1}\cale, i_{A}^{-1}\calf) \ar[dd] \\
\\
\C_{*}^{G}(B, i_{B}^{-1}\cale, i_{B}^{-1}\calf) \ar[rr] & & \C_{*}^{G}(X, \cale, \calf) 
}
\]
where $i: A \cap B \rightarrow X$, $i_{A}: A \rightarrow X$ and $i_{B}: B \rightarrow X$ are the inclusions.
\end{prop}

\begin{proof}
Note that by $G$-properness of the coarse structure, we may as well work with the square
\[
\xymatrix{
\C_{*}^{G}(X, \cale, \calf(A \cap B)) \ar[dd] \ar[rr] & & \C_{*}^{G}(X, \cale, \calf(A)) \ar[dd] \\
\\
\C_{*}^{G}(X, \cale, \calf(B)) \ar[rr] & & \C_{*}^{G}(X, \cale, \calf) 
}
\]
The idea of the proof is the same as the one used in \cite[4.3]{BFJR}: Check that the induced $C^{*}$-functor on the quotients
\[
F: \C_{*}^{G}(X, \cale,\calf(A))^{> A \cap B} \rightarrow \C_{*}^{G}(X, \cale, \calf)^{> B}
\]
is an equivalence of categories and then use that both rows in the resulting diagram induce fiber sequences in $K$-theory. \\
The execution is a bit more subtle since our quotient process is more complicated. We will follow the usual strategy: Prove a suitable statement for the $\C^{G}_{b}$-categories and complete it. First consider the uncompleted categories
\[
\C_{b}^{G}(X, \cale, \calf(A \cap B)) \rightarrow \C_{b}^{G}(X, \cale, \calf(A)) \rightarrow \C_{b}^{G}(X, \cale, \calf(A))^{> A \cap B}
\]
where the right-hand category is the Karoubi quotient of the left-hand inclusion functor. The problem is now that $\C_{*}^{G}(X, \cale, \calf(A))^{> A \cap B}$ is \emph{not} a completion of $\C_{b}^{G}(X, \cale, \calf(A))^{> A \cap B}$: Indeed, a morphism of 
\[
\C_{b}^{G}(X, \cale, \calf(A))^{> A \cap B}
\]
which is a norm-limit of morphisms factoring through $\C_{b}^{G}(X, \cale, \calf(A \cap B))$ is $0$ in 
\[
\C_{*}^{G}(X, \cale, \calf(A))^{> A \cap B}
\]
but not necessarily in $\C_{b}^{G}(X, \cale, \calf(A))^{> A \cap B}$, compare \ref{quots}. \\
As in \ref{quots}, we can define a pseudonorm on $\C_{b}^{G}(X, \cale, \calf(A))^{> A \cap B}$, and similarly we can put a pseudonorm on the other quotient category $\C_{b}^{G}(X; \cale, \calf)^{> B}$. It follows from the discussion in \ref{quots} that it suffices to see that the induced functor
\[
F_{b}:  \C_{b}^{G}(X, \cale, \calf(A))^{> A \cap B} \rightarrow \C_{b}^{G}(X; \cale, \calf)^{> B}
\]
is an isometric equivalence of pre-normed categories. Then the "completion" (which really consists of first dividing out morphisms of norm $0$ and then completing) of $F_{b}$ to a $C^{*}$-functor is also an equivalence of categories as desired. \\
So let us set out to prove this. First we prove that $F_{b}$ is essentially surjective. 
Let $M$ be an object of $\C_{b}^{G}(X, \cale, \calf)^{> B}$. Let $M'$ be the object of $\C_{b}^{G}(X, \cale, \calf)^{> B}$ with
\[
M'_{x} = \begin{cases} M_{x} &\text{ if } x \notin B \\ 0 & \text{ else } \end{cases}
\]
Since we disregard things happening over $B$, $M'$ and $M$ are isomorphic in $\C_{b}^{G}(X, \cale, \calf)^{> B}$. But $M'$ clearly is in the image of $F_{b}$.  \\
That $F_{b}$ is full is a direct consequence of the definitions. \\
Let $f \colon M \rightarrow N$ be a morphism in $\C_{b}^{G}(X, \cale, \calf(A))^{> A \cap B}$. We can assume that both $M$ and $N$ are not supported over $A \cap B$. Let $M'$ be the direct summand of $M$ consisting of those $M_{x}$ on which $f$ does not vanish, and let $N'$ be the direct summand of $N$ consisting of those $N_{x}$ which $f$ actually hits. 
If $F_{b}(f) = 0$, $F_{b}(f)$ has to factor over an object $K$ supported only over $B$, say via $g: M \rightarrow K$ and $h: K \rightarrow N$. Let $g$ be $E$-controlled, $h$ $E'$-controlled and $N'$ $F$-controlled, with $E,E' \in \cale$ and $F \in \calf$. It follows that
\begin{align*}
&\supp(N') \subset \supp(M')^{E} \cap \supp(K)^{E'} \cap F \\ &\subset (A \cap \supp(M'))^{E \cup E'} \cap (B \cap \supp(K))^{E \cup E'} \cap F 
\end{align*}
By the coarse excisiveness of the triple $(X,A,B)$, we conclude that there is an $E'' \in \cale$ and $F'' \in \calf$ with 
\[
\supp(N') \subset (A \cap B \cap F'')^{E''} \cap F''
\]
But this implies $N' = 0$ in $\C_{b}^{G}(X, \cale, \calf(A))^{> A \cap B}$ and hence $f = 0$ in \\ $\C_{b}^{G}(X, \cale, \calf(A))^{> A \cap B}$. \\
It remains to see that $F_{b}$ is indeed isometric. Let $f: M \rightarrow N$ be a morphism of $\C_{b}^{G}(X, \cale, \calf(A))$, representing a morphism of $\C_{b}^{G}(X, \cale, \calf(A))^{> A \cap B}$. We may assume that $M$ and $N$ are supported over $A \setminus B$. Its norm as a morphism in the quotient category is
\[
\left\Vert f \right\Vert_{A \cap B}  = \inf\{ \left\Vert f-g \right\Vert \mid g: M \rightarrow N \text{ completely continuous }\}
\] 
where $g$ is a morphism of $\C_{b}^{G}(X, \cale, \calf(A))$. Considering $f$ as a morphism of $\C_{b}^{G}(X; \cale, \calf)$, the norm of the morphism represented by $f$ in $\C_{b}^{G}(X; \cale, \calf)^{> B}$ is
\[
\left\Vert f \right\Vert_{B}  = \inf\{ \left\Vert f-g \right\Vert \mid g: M \rightarrow N \text{ completely continuous }\}
\]
where now $g$ is a morphism of $\C_{b}^{G}(X; \cale, \calf)$. It follows that we are done once we can prove that each completely continuous map $g: M \rightarrow N$ in $\C_{b}^{G}(X; \cale, \calf)$ is the image of a completely continuous map under the inclusion functor $\C_{b}^{G}(X; \cale, \calf(A)) \rightarrow \C_{b}^{G}(X; \cale, \calf)$. Since $M$ and $N$ are supported over $A \setminus B$, each morphism $g: M \rightarrow N$ is in the image of the inclusion functor: $g: M \rightarrow N$ just is a valid morphism in $\C_{*}^{G}(X; \cale, \calf)$. If $g$ where not completely continuous as a morphism of $\C_{b}^{G}(X; \cale, \calf(A))$, the class of $g$ in the quotient would be nonzero and its image in $\C_{b}^{G}(X; \cale, \calf)^{>B}$ would be zero since $g$ is completely continuous as a morphism of $\C_{b}^{G}(X; \cale, \calf)$. This contradicts the faithfulness of $F_{b}$.
\fxnote{There should really be an easier argument for the last part}
\end{proof}

\section{Comparison with the Roe algebra}

In this section, we outline a relationship between our constructions and the Roe algebra as defined in, for example, \cite{HigsonRoe}. For simplicity, we only give the argument for the case of a metric space with its metric control and without group actions.

\begin{definition}
Let $X$ be a proper metric space. Let 
\[
\rho: C_{0}(X) \rightarrow \calb(\calh)
\] 
be a representation of $C_{0}(X)$ on a separable Hilbert space $H$.
\begin{enumerate}
\item We say that $\rho$ is nondegenerate if $\rho(C_{0}(X))(H)$ is dense in $H$.
\item We say that $\rho$ is ample if in addition, no nonzero $f \in C_{0}(X)$ acts as a compact operator on $\calh$.
\end{enumerate}
\end{definition}

\begin{example}
Let $S \subset X$ be a countable dense subset and let $\calh = l^{2}(S \times \IN)$. Then $C_{0}(X)$ acts on $H$ via 
\[
\rho_{S}(f)(s,n) = f(s) \cdot (s,n)
\]
This is nondegenerate, since each $(s,n)$ is in the image of some $\rho(f)$: just pick $f$ such that $f(s) \neq 0$. It is also ample, because if $f(s) \neq 0$ for some $s$, the image of unit ball in $l^{2}(s \times \IN) \subset l^{2}(S \times \IN)$ is the ball of radius $f(s)$ in $l^{2}(\{s\} \times \IN)$ which is certainly not compact.
\end{example}

\begin{definition}
Let $\rho$ be an ample representation of $C_{0}(X)$ on $\calh$.
\begin{enumerate}
\item An operator $T: \calh \rightarrow \calh$ is locally compact if for all $f$, the operators $\rho(f)T$ and $T\rho(f)$ are compact.
\item The support of an operator $T: \calh \rightarrow \calh$ is the complement of the set of all $(x,y) \in X \times X$ such that there are open neighborhoods $U$ of $x$ and $V$ of $y$ such that $\rho(f)T\rho(g) = 0$ for all $g \in C_{0}(V)$, $f \in C_{0}(U)$. 
\item We say that $T$ is controlled if the support of $T$ is a controlled subset of $X \times X$ with respect to the metric control condition.
\item Let the Roe-algebra $C_{*}^{Roe}(X,\rho)$ of $X$ be the $C^{*}$-algebra generated by the controlled, locally compact operators.
\end{enumerate}
\end{definition}

\begin{prop}
Let $C_{*}(X)$ be the controlled $C^{*}$-category associated to $X$ with its metric coarse structure. Let $S \subset X$ be dense and countable and $\rho_{S}$ the associated representation of $C_{0}(X)$. Then there is a canonical algebra homomorphism
\[
C_{*}^{Roe}(X,\rho_{S}) \rightarrow T(C_{*}(X)) 
\]
inducing an isomorphism on $K$-theory.
\end{prop}

\begin{proof}
In the proof, we omit $\rho_{S}$ from the notation. Let $C_{*}(X,S)$ be the full subcategory of $C_{*}(X)$ spanned by the objects whose support is in $S$. The inclusion $C_{*}(X,S) \rightarrow C_{*}(X)$ is an equivalence of categories: If $M$ is any object of $C_{*}(X)$, we can pick  for each $x \in \supp(M)$ an $s \in S$ with, say, $d(s,x) < 1$ and push $M_{x}$ away from $x$ to $s$. Since there are only finitely many $x$ with distance $<1$ from $s$ such that $M_{x} \neq 0$, this indeed yields an object of $C_{*}(X,S)$ which is clearly isomorphic to $M$. In other words, $(S,d)$ and $(X,d)$ yield coarsely equivalent coarse structures. \\
Since the inclusion $C_{*}(X,S) \rightarrow C_{*}(X)$ is induced by the injective space map $S \rightarrow X$, we obtain an induced map
\[
T(C_{*}(X,S)) \rightarrow T(C_{*}(X))
\]
which induces an isomorphism on $K$-theory since the functor $C_{*}(X,S) \rightarrow C_{*}(X)$ does. We will be done once we can prove
\[
T(C_{*}(X,S)) = C_{*}^{Roe}(X)
\]
where we think of $T(C_{*}(X,S))$ as an algebra of operators on $l^{2}(S \times \IN)$.
Let $f: M \rightarrow M$ be an endomorphism of an object $M$ of $C_{*}(X,S)$. Then the bounded operator $f: l^{2}(S,\IN) \rightarrow l^{2}(S, \IN)$ is controlled and locally finite-dimensional, i.e. for all $\phi \in C_{0}(X)$ with compact support we have that $\phi f$ and $f \phi$ are finite-dimensional operators. Since the norm-limit of a locally finite-dimensional operator is locally compact, it follows that
\[
T(C_{*}(X,S)) \subset C_{*}^{Roe}(X)
\]
The other direction is more involved. Note that also bounded Borel functions act on $\calh$ in the obvious way. Decompose $X$ into a disjoint, countable union of Borel-measurable, relatively compact subsets $M_{i}$. Let $T$ be a locally compact operator and let $1 > \epsilon > 0$ be given. Denote by $\calh(M_{i})$ the sub-Hilbertspace of $\calh$ spanned by the vectors $(s,n)$ with $s \in M_{i}$. Let $\chi_{M_{i}}$ be the characteristic function of $M_{i}$. Then $\chi_{M_{i}}: \calh_{S} \rightarrow \calh_{S}$ is the orthogonal projection onto $\calh(M_{i})$. Consider the operator 
\[
\chi_{M_{i}}T: \calh = \oplus \calh(M_{i}) \rightarrow \calh(M_{i})
\]
This operator is compact since $M_{i}$ is relative compact and we have $T = \oplus \chi_{M_{i}} T$. \\
Note that $\calh(M_{i})$ has $M_{i} \cap S$ as an orthogonal basis. Enumerate the countably many elements of $(S \cap M_i) \times \IN$ in a way such that for $x \in X$ and $n < n'$, $(x,n)$ occurs before $(x,n')$ and let $S_k: \calh(M_{i}) \rightarrow \calh(M_{i})$ be the projection onto the first $k$ basis vectors with respect to this enumeration. \\
For each $i$, pick $k_{i}$ such that $S_{k_{i}}: \calh{M_{i}} \rightarrow \calh M_{i})$ satisfies 
\[
\left\Vert \chi_{M_{i}}T-S_{k_{i}}\chi_{M_{i}}T \right\Vert^{2} < \epsilon^{i}
\] 
This can be arranged since $\chi_{M_{i}}T: \calh(M_{i}) \rightarrow \calh(M_{i})$ is compact and the usual argument proving that compact operators on a Hilbert space can be approximated by finite-dimensional ones proceeds precisely by using projection operators $S_{k}$ as above. \\
Set $S = \oplus S_{k_{i}} \colon \calh_{S} \rightarrow \calh_{S}$. Clearly, $S$ is locally finite-dimensional, and we have 
\begin{align*}
\left\Vert T-ST \right\Vert &= \left\Vert \oplus (\chi_{M_{i}}T-S_{k_{i}} \chi_{M_{i}}T) \right\Vert \\ &\leq \sqrt{\sum \left\Vert \chi_{M_{i}}T - S_{k_{i}}\chi_{M_{i}}T \right\Vert^{2}} \leq \sqrt{\frac{\epsilon}{1-\epsilon}}
\end{align*}
which goes to zero as $\epsilon$ goes to zero. Now $ST$ may not be locally finite-dimensional, but it is at least left locally finite-dimensional in the sense that $\rho(f)ST$ is finite-dimensional for each $f \in C_{0}(X)$. If $T$ was right finite-dimensional to start with, then $ST$ is finite-dimensional. We now apply the same arguments to the other sides to find an operator $S' = \oplus S_{n_{i}}$ such that $STS'$ approximates $ST$ arbitrarily close. Finally, $STS'$ comes from a morphism in $\C_{*}(X)$: Let $M$ be the object of $\C_{*}(X)$ with the dimension of $M_{x}$ the dimension of the image of $S'$ intersected with $l^{2}(\{x\} \times \IN)$ and $N$ the object with the dimension of $N_{x}$ the dimension of the image of $S$ intersected with $l^{2}(\{x\} \times \IN)$. Then $STS'$ is more or less a morphism from $M$ to $N$: Pick a  basis vector $e_{(x,m)}$ of $M$. Then write $STS'((x,m))$ as 
\[
\sum a_{i} (x_{i}, m_{i})
\] 
and set 
\[
f(e_{(x,m)}) = \sum a_{i} e_{(x_{i}, m_{i})}
\]
By definition of $N$, $e_{(x_{i}, m_{i})}$ is a basis vector of $N$ as long as $a_{i} \neq 0$. This finishes the argument.

\end{proof}

\section{A functor to spectra}

For some of our purposes, a functor to the homotopy category of spectra is not enough - we need an honest functor to spectra. Using infinite loop space techniques, we can lift our functor from coarse $G$-spaces to the homotopy category of spectra to an actual functor to spectra:

\begin{prop}
There is a functor 
\[
K: C^{*} \minus cat \rightarrow \Sp
\]
lifting the $K$-theory functor from coarse spaces to $Ho(\Sp)$ we have constructed.
\end{prop}

\begin{proof}
The argument is the same as the one in \cite[6.4]{hpr} or \cite[Section 2.2]{PH}. The partition of unity-argument alluded to in \cite[2.3]{PH} is really an application of \ref{cont}.
\end{proof}

Luckily, it is enough for our purposes that such a functor exists; we will always use the up to homotopy-version, tacitly assuming that there is an actual functor to spectra in the background.

%% file: chapters/homology.tex
\chapter{An equivariant homology theory}

In this section, we define the equivariant homology theory which occurs in the definition of the Baum-Connes map. Furthermore, we use controlled topology to give a very explicit description for this homology theory. This yields an explicit coarse space $X$ whose $K$-theory is the homotopy fiber of the Baum-Connes map, i.e. such that the Baum-Connes conjecture holds if and only if
\[
K \C_{*}^{G}(X) = 0
\]
which we will later on use in the proof of the Baum-Connes conjecture for Farrell-Hsiang groups.

\section{$G$-homology theories}

\begin{definition}
A $G$-homology theory $h^{G}(-)$ is a functor from $G \minus CW$-complexes to $SPECTRA$ such that
\begin{enumerate}
\item $h^{G}$ is homotopy invariant: If $f: X \rightarrow Y$ is a $G$-equivariant homotopy equivalence, the induced map $h^{G}(f)$ is an equivalence of spectra.
\item $h^{G}$ turns homotopy pushout squares of $G \minus CW$-complexes into homotopy pullback squares, i.e. satisfies Mayer-Vietoris.
\item For a $ G \minus CW$-complex $X$, let $\{X_{i} \mid i \in I\}$ be the ordered set of $G$-finite subcomplexes of $X$. Then the natural map
\[
\operatorname{hocolim}_{i} h^{G}(X_{i}) \rightarrow h^{G}(X)
\]
is an equivalence of spectra.
\end{enumerate}
We denote $\pi_{*}(h^{G}(X))$ by $H_{*}^{G}(X)$.
\end{definition}

\section{A review of the Davis-L\"uck construction}

Let $G$ be a group. 

\begin{definition}
The orbit category $\Or(G)$ of $G$ is the category with objects the transitive $G$-spaces $G/H$ and morphisms the $G$-maps. 
\end{definition}

The orbit category is a convenient way to keep track of the group $G$. For example, any $G$-space $X$ gives rise to a contravariant functor from $\Or(G)$ to $\Top$ by sending $G/H$ to $X^{H} = \map_{G}(G/H, X)$. \\

Recall that a non-equivariant homology theory is described by a spectrum, at least as long as one restricts attention to $CW$-complexes.  For $G$-homology theories, we get a similar description using $\Or(G)$-spectra:

\begin{definition}
An $\Or(G)$-spectrum is a covariant functor from $\Or(G)$ to the category of spectra. 
\end{definition}

\begin{thm}
Associated to each $\Or(G)$-spectrum $E$, there is a $G$-homology theory $H^{G}(-; E)$ such that
\[
H_{*}^{G}(G/H; E) \cong \pi_{*}(E(G/H))
\]
\end{thm}

\begin{proof}
Given a $G$-space $X$, we form the associated contravariant $G$-space $\map_{G}(-,X): \Or(G) \rightarrow \Top$ as above. Now form the coend of this contravariant functor and the covariant functor $E$ over the smash product of a space and a spectrum. In \cite{DL}, this is called the balanced tensor product of $E$ and $\map_{G}(-,X)$. The result is a spectrum $h^G(X; E) = E \wedge_{G}  \map_{G}(-,X)$, and we set
\[
H_{*}^{G}(X; E) = \pi_{*}(E \wedge_{G}  \map_{G}(-,X))
\]
We refer the reader to \cite{DL} for further details. 
\end{proof}

\section{Groupoids and $\Or(G)$-spectra}
\label{orgktop}
Fix a $C^{*}$-algebra $A$ (with trivial $G$-action for now). As in \cite[2.2]{DL}, we will associate an $\Or(G)$-spectrum
\[
A_{r}^{G} \colon \Or(G) \rightarrow \Sp
\]
to $A$. The first step is a functor
\[
Gr \colon \Or(G) \rightarrow \Grp^{\inj}
\] 
where $\Grp^{\inj}$ is the category of groupoids and faithful functors. On objects, $Gr$ sends a transitive $G$-set $S$ to the groupoid $Gr(S)$ whose objects are the points of $S$. The morphisms in the groupoid $Gr(S)$ are given as
\[
\hom_{Gr(S)}(s,t) = \{ g \in G \mid gs = t \}
\]
and composition is group multiplication. This construction is often called the transport groupoid. Composing the functor $Gr$ with the functor
\[
A^{*}_{r} \colon \Grp^{\inj} \rightarrow \Sp
\]
we will now define yields an $\Or(G)$-spectrum. \\
Let $S$ be a groupoid. Let $AS$ be the category with the same objects as $S$, but with morphism sets $\hom_{AS}(s,t)$ the free left $A$-module with basis $\hom_{S}(s,t)$. Composition is extended bilinearly from the composition in $S$. The category $AS$ carries an involution which is given by $(\sum a_{i} f_{i})^{*} = \sum a_{i}^{*} f_{i}^{-1}$ for $f_{i}$ morphisms in $S$. We define a norm on $\hom_{AS}(s,t)$ as follows. Pick an object $z$ such that $t$ maps to $z$; it does not matter which one since all objects mapping to $z$  are isomorphic in the groupoid anyway. We get a map
\[
i_{s,t;z} \colon \hom_{AS}(s,t) \rightarrow \calb(l^{2}(\hom_{S}(t,z),A), l^{2}(\hom_{S}(s,z),A))
\]
by precomposition, where we consider the Hilbert modules in the target as right modules so that composition from the left is $A$-linear. Note that this definition is not precisely the same as the one in \cite{DL}, where composition from the right is used to define the above map. The map $i_{s,t;z}$ indeed lands in the bounded, adjointable operators: For each morphism $f: s \rightarrow t$ in $S$, composition with $f$ sends the orthogonal basis $\hom_{S}(t,z)$ of the left-hand Hilbert module bijectively to the orthogonal basis $\hom_{S}(s,z)$ of the right-hand one and hence $i_{s,t;z}(f)$ is indeed bounded. The claim follows since any element of $\hom_{AS}(s,t)$ is an $A$-linear combination of such morphisms. We then equip $\hom_{AS}(s,t)$ with the restriction of the operator norm of the right-hand $C^{*}$-algebra. \\
It is easy to check that this representation is compatible with the involution, i.e. it does not matter whether one represents $\hom_{AS}(s,t)$ as bounded operators first and then takes the adjoint, ending up in 
\[ 
\calb(l^{2}(\hom_{S}(s,z),A), l^{2}(\hom_{S}(t,z),A))
\]
or uses the involution on the category $AS$ first to go to $\hom(t,s)$ and then represent this as bounded operators on 
\[
\calb(l^{2}(\hom_{S}(s,z),A), l^{2}(\hom_{S}(t,z),A))
\]
To check this, it suffices by linearity to consider a morphism  $f \colon s \rightarrow t$ in $S$, where the claim is clear. So the involution is isometric and satisfies the $C^{*}$-identity. It is also straightforward to check that composition is submultiplicative, since again it does not matter whether one composes first and then represents or represents first and then composes the arising bounded operators. So $AS$ is a pre-$C^{*}$-category (without direct sums). We then form the additive completion and norm complete the resulting additive pre-$C^{*}$-category to obtain a $C^{*}$-category $A_{r}^{*} S$. Finally, we take its topological $K$-theory to land in $\Sp$. This defines the functor $A^{*}_{r}: \Grp \rightarrow \Sp$ on objects. \\
For morphisms, let $F \colon S \rightarrow T$ be a faithful functor. It is clear how $F$ induces a functor from $AS$ to $AT$, even when $F$ is not faithful. We have to check that $AF \colon AS \rightarrow AT$ is isometric; then it extends to the completions. So let $f = \sum a_{i} f_{i} \colon s \rightarrow s'$ be a morphism in $AS$, where the $f_{i}$ are morphisms of $S$. The involved norms do not change when we compose with some morphisms $s' \rightarrow s$ from $S$, so we may assume $s = s'$. Then we may regard $H = \End_{S}(s)$ as a subgroup of $G = \End_{T}(F(s))$ since $F$ is faithful. Then we have
\[
l^{2}(\hom_{T}(F(s), F(s))) = \bigoplus_{G/H} l^{2}(\hom_{S}(s,s))
\]
where the sum is orthogonal. The operator 
\[
F(f): l^{2}(\hom_{T}(F(s), F(s))) \rightarrow l^{2}(\hom_{T}(F(s), F(s)))
\]
respects this direct sum decomposition since all the $f_{i}$ actually lie in $H = \End_{S}(s,s)$. It follows that $f$ and $F(f)$ have the same norms.

\begin{example}
Let $G$ be a group which we consider as a groupoid with one object. Then it follows straightforwardly from the definitions that the $C^{*}$-category $A^{*}_{r}G$ is equivalent to the category of finitely generated free $A^*_r G$-modules. 
\end{example} 

Now we can define the Baum-Connes map.

\begin{definition}
Let $A_{r}^{G}$ be the $\Or(G)$-spectrum we have just defined. Let $\calf$ be a family of subgroups of $G$ and let $E_{\calf}G$ denote the corresponding classifying space. The Baum-Connes assembly map of $G$ with respect to $\calf$ with coefficients in $A$ is the map
\[
H_{*}^{G}(E_{\calf}G; A_{r}^{G}) \rightarrow H_{*}^{G}(*,A_{r}^{G})
\]
induced by the projection of $E_{\calf}G$ to the point $*$.
\end{definition}

Note that this definition does not yet cover the case of a nontrivial $G$-action on $A$, which we also want to account for. We will discuss below the necessary modifications needed to cover this case as well.  \\
The functor $A_{r}^{*}$ yields a compatible way to write down $\Or(G)$-spectra for any group $G$. This is made precise in \cite[6.1]{KThandbook}, where it is shown that the homology theories we obtain from $A_{r}^{*}$ for different groups $G$ fit together to a so-called equivariant homology theory as defined in \cite{KThandbook}. We are mainly interested in the transitivity properties of assembly maps this additional structure implies:

\begin{thm}[Transitivity principle]
\label{transitivity}
Fix a $C^{*}$-algebra $A$ with trivial $G$-action. Let $G$ be a group and $\calf \subset \calg$ be two families of subgroups of $G$. If the Baum-Connes conjecture for $G$ with coefficients in $A$ with respect to $\calg$ is true and for each $H \in \calg$, the Baum-Connes conjecture for $H$ with coefficients in $A$ with respect to the family $H \cap \calf$ is true, then the Baum-Connes conjecture for $G$ with coefficients in $A$ with respect to $\calf$ is true.
\end{thm}

\begin{proof}
See \cite[2.9]{KThandbook}.
\end{proof}

Unfortunately, this result does not cover coefficients, and the above construction cannot take into account a $G$-action on $A$ for a given, fixed group $G$. In \cite{BEL}, the problem of non-trivial $G$-action on the coefficients is treated as follows. \\
Fix a group $G$. Let $A$ be a $G$-$C^*$-algebra. Instead of looking at functors  $\Grp^{\inj}\rightarrow \Sp$ we look at functors from groupoids over $G$ to spectra, where the category of groupoids over $G$ is the category with objects pairs $(S,F)$ where $S$ is a groupoid and $F \colon S \rightarrow G$ is a functor, regarding $G$ as a groupoid with one object. A morphism from $(S,F)$ to $(S',F')$ is a faithful functor $T \colon S \rightarrow S'$ such that $F'T = F \colon S \rightarrow G$. We denote this category as $\Grp^{\inj} \downarrow G$. The most canonical objects are group homomorphism from a group $H$ to $G$, and group homomorphisms into $G$ are a good setup if we want to handle coefficients since we can now regard a $G \minus C^{*}$-algebra $A$ also as an $H \minus C^{*}$-algebra by restriction along the given homomorphism and compare the two arising assembly maps. \\
If a functor 
\[
E: \Grp^{\inj} \downarrow G \rightarrow \Sp
\]
is given, we obtain for each group $H$ equipped with a group homomorphism $f \colon H \rightarrow G$ an $\Or(H)$-spectrum as follows. First consider for a transitive $H$-set $S$ the groupoid $Gr(S)$ as above. This groupoid comes with a functor to $G$ which on morphisms is given by sending a morphism $h$, which is really an element of $H$, to $f(h)$. So $Gr$ actually takes values in groupoids over $G$ and hence can be composed with $E$ to yield an $\Or(H)$-spectrum
\[
A_{r}^{G}: \Or(H) \rightarrow \Sp
\] 

The advantage of this approach is that one can now prove the transitivity principle for all coefficients, see \cite[3.3]{BEL}. \\
It remains to construct a suitable functor 
\[
E: \Grp^{\inj} \downarrow G \rightarrow \Sp
\]
Its construction is deferred to an unpublished paper in the proof of \cite[7.1]{BEL}, so we give the details of the construction of a functor $\Grp^{\inj} \downarrow G \rightarrow \Sp$. Of course, when $G$ is trivial, we would like to recover the definitions we made above. \\
Let $(S, F: S \rightarrow G)$ be a groupoid over $G$ and $A$ a $G$-$C^{*}$-algebra. We define a category $A \rtimes_{G} S$ as follows. Its objects are those of $S$. As above, the morphism set $\hom_{A \rtimes_{G} S}(s,t)$ is the free $A$-module generated by $\hom_{S}(s,t)$. The functor $F$ only enters into the composition law and the involution. Composition will be $\IZ$-bilinear, but only $A$-linear in one of the two variables. To avoid confusion, we write $\sigma_{g} \colon A \rightarrow A$ for the action of $g \in G$ on $A$. If $e \colon s \rightarrow t$ and $f: t \rightarrow z$ are morphisms in $S$ and $a,b \in A$, we define the composition of $ae$ and $bf$ to be $(b \sigma_{F(f)}(a) )(f \circ e)$ and extend $\IZ$-bilinearly. It is straightforward to verify that this composition law is associative. As above, we define a norm on $\hom_{A \rtimes_{G} S}(s,t)$ by picking an object $z$ which is the target of some map from $t$, consider the map
\[
i_{s,t;z}: \hom_{A \rtimes_{G} S}(s,t) \rightarrow \calb(l^{2}(\hom_{S}(t,z),A), l^{2}(\hom_{S}(s,z),A))
\]
given by composition and equip $\hom_{A \rtimes_{G} S}(s,t)$ with the operator norm of the right-hand side. The involution is defined by
\[
(\sum a_{i} f_{i})^{*} = \sum \sigma_{F(f_{i})^{-1}}(a_{i})^* f_{i}^{-1}
\]
for morphisms $f_{i}$ in $S$. It is checked as above that this gives $A \rtimes_{G} S$ the structure of a (non-additive) pre-$C^{*}$-category. By completing with respect to the norm , we obtain a (non-additive) $C^{*}$-category $A \rtimes_{r} S$ and  consider its additive completion $(A \rtimes_{r} S)_{\oplus}$. Taking its $K$-theory finally yields a spectrum. If $G: (S,F) \rightarrow (S',F')$ is a morphism in $\Grp^{\inj} \downarrow G$, it is checked as above that $G$ induces a $C^{*}$-functor $A \rtimes_{r} S \rightarrow A \rtimes_{r} S'$. This finishes the construction. \\

\begin{example}
We consider $G$ as a groupoid over itself via the identity map. Then the category $A \rtimes_{r} G$ we just constructed is equivalent to the $C^{*}$-category of finitely generated free modules over the algebra $A \rtimes_{r} G$. We ignore the conflict of notation since we will never again use the notation $A \rtimes_{r} G$ for the first category.
\end{example}
\fxnote{this requires an argument}

Now we can define the Baum-Connes assembly map with coefficients in $A$ and with respect to the family of subgroups $\calf$ as above to be the map
\[
H_{*}^{G}(E_{\calf}G; A_{r}^{G}) \rightarrow H_{*}^{G}(*,A_{r}^{G})
\]
induced by projection onto a point.  The classical Baum-Connes conjecture is the case $\calf = \mathcal{FIN}$. Since the classical Baum-Connes conjecture is true for virtually cyclic groups, we see using the transitivity principle that the classical Baum-Connes conjecture is equivalent to the Baum-Connes conjecture for the family $\mathcal{VCYC}$. We will always use the family $\mathcal{VCYC}$ instead of $\mathcal{FIN}$.

\section{Equivariant homology theories from controlled topology}

Fix a group $G$ and a $G$-$C^{*}$-algebra $A$. Let $X$ be a $G$-space. We will define a coarse structure on the $G$-space $G \times X \times [1, \infty)$. We begin by defining a morphism support condition $\cale_{Gcc}$ on $X \times [1, \infty)$ which we will pull back to  $G \times X \times [1, \infty)$ eventually. The definition is the same as \cite[2.7]{BFJR}.

\begin{definition}
Let $E \subset (X \times [1, \infty))^{2}$. We say that $E$ is continuously controlled if the following holds:
\begin{enumerate}
\item For any $x \in X$ and any $G_{x}$-invariant neighborhood $U$ of $(x, \infty)$ in $X \times [1, \infty]$, there is a smaller $G_{x}$-invariant neighborhood $V \subset U$ of $(x ,\infty)$ such that
\[
(U^{c} \times V) \cap E = \varnothing
\]
where $U^{c}$ is the complement of $U$ in $X \times [1, \infty]$.
\item There is $R > 0$ such that whenever $((x,t), (y,t')) \in E$, we have $d(t,t') < R$.
\item $E$ is symmetric and $G$-invariant.
\end{enumerate}
\end{definition}

\begin{definition}
Let $X$ be a $G$-space. Pick a word metric $d_{G}$ on $G$. We define a coarse structure $(\cale_{X}, \calf_{X})$ on $G \times X \times [1,\infty)$. Let 
\begin{align*}
p \colon G \times X \times [1, \infty) &\rightarrow X \times [1, \infty) \\
q \colon G \times X \times [1, \infty) &\rightarrow G \times X \\
s \colon G \times X \times [1, \infty)  &\rightarrow G
\end{align*}
be the projections. We let
\[
(\cale, \calf) = (p^{-1} \cale_{Gcc} \cap s^{-1} \cale_{d_{G}}, q^{-1} \calf_{G \minus \text{compact}})
\]   
We define the categories $\cald_{b}^{G}(X)$ and $\cald_{*}^{G}(X)$ as
\[
\cald_{b}^{G}(X) = \C_{b}^{G}(G \times X \times [1,\infty), \cale_{X}, \calf_{X})
\]
and similarly for $\cald_{*}^{G}$. 
\end{definition}

\begin{definition}
Consider the inclusion $i \colon G \times X \subset G \times X \times [1, \infty)$. We define $\cald_{*}^{G}(X)^{\infty}$ to be the quotient of the inclusion
\[
\C_{*}^{G}(X \times G \times [1, \infty),\cale_{X}(X \times G \times \{1\}), \calf_X(X \times G \times \{1\}) \rightarrow \D_{*}^{G}(X)
\]
and similarly, we define $\cald_{b}^{G}(X)^{\infty}$. Even though $(\cale_X, \calf_{X})$ may not be proper, it is proper in the $[1, \infty)$-direction, i.e. there is an equivalence of categories
\[
\xymatrix{
\C_{*}^{G}(X \times G \times \{1\}, i^{-1}\cale_{X}, i^{-1} \calf_{X}) \ar[d] \\ \C_{*}^{G}(X \times G \times [1, \infty),\cale_{X}(X \times G \times \{1\}), \calf_X(X \times G \times \{1\})
}
\]
\end{definition}

\begin{prop}
\label{constant}
For each $X$, the category $\C_{*}^{G}(X \times G \times [1, \infty),\cale_{X}, \calf(X \times G))$ is equivalent to the category of finitely generated free $A \rtimes_{r} G$-modules, and each map $X \rightarrow Y$ induces an equivalence $\C_{*}^{G}(X \times G \times [1, \infty),\cale_{X}, \calf(X \times G)) \rightarrow\C_{*}^{G}(Y \times G \times [1, \infty),\cale_{X}, \calf(Y \times G))$.
\end{prop}

\begin{proof}
The allowed morphisms are the same in the two categories. Since any object is isomorphic to one only supported over $X \times G \times \{1\}$ respectively $Y \times G \times \{1\}$, we can concentrate on such objects. By $G$-equivariance, such an object is determined by its restriction to $X \times {e} \times \{1\}$, where it is finitely supported thanks to the $G$-compactness condition. Now pick a point $x \in X$ and move all summands of the object over $X \times {e} \times \{1\}$ to $(x,e,1)$, which are only finitely many thanks to the $G$-compactness condition. Extending this $G$-equivariantly yields a new object isomorphic to the old one which is supported over the orbit of $(x,e,1)$. The claims are now easy to check.
\end{proof}

\begin{prop}
A $G$-equivariant map $f: X \rightarrow Y$ induces a functor
\[
\cald_{*}^{G}(X) \rightarrow \cald_{*}^{G}(Y)
\]
and similarly for $\cald_{b}^{G}$.
\end{prop}

\begin{proof}
This is not a triviality and is proven in \cite[3.3]{BFJR}. The reader should note that this fails if we do not pass from $X$ to $G \times X$.
\end{proof}

\begin{thm}
The functor sending the $G$-$CW$-complex $X$ to the topological $K$-theory spectrum of $\D_{*}^{G}(X)^{\infty}$ is an equivariant homology theory. Similarly, also sending $X$ to the algebraic $K$-theory of $\D_{b}^{G}(X)^{\infty}$ is an equivariant homology theory.
\end{thm}

We prove this in several steps, concentrating on the statement for $\D_{*}^{G}(X)^{\infty}$; the statement about $\D_{b}^{G}(X)^{\infty}$ is proven in the same way. First note that by \ref{constant}, in the fiber sequence 
\[
K \C_{*}^{G}(X \times G \times [1, \infty),\cale_{X}, \calf(X \times G)) \rightarrow K \cald_{*}^{G}(X) \rightarrow K \cald_{*}^{G}(X)^{\infty}
\]
the left-hand term is constant. It follows that the right-hand term is homotopy invariant and satisfies Mayer-Vietoris if and only if the middle term does, and we will argue in the middle term. The quotient construction is only necessary to obtain the disjoint union axiom. 

\begin{prop}
For an arbitrary $G$-$CW$-complex $X$, the natural map
\[
\hocolim_{C \subset X} K \D_{*}^{G}(X_{C}) \rightarrow K \D_{*}^{G}(X)
\]
is an equivalence, where $C$ runs over all $G$-compact subsets of $X$. 
\end{prop}

\begin{proof}
Because of the $G$-compact object control condition in the definition of $\D_{*}^{G}$, it is easy to see that $\D_{*}^{G}(X) = \cup \D_{*}^{G}(X_{c})$. This time, no completion is required: The union is basically a union of objects. Since $K$-theory of $C^{*}$-categories commutes with filtered unions by \ref{unions}, the claim follows.  
\end{proof}

It hence suffices to prove the axioms of a homology theory for $G$-compact $G$-$CW$-complexes. We will from now on assume that all involved $G$-$CW$-complexes are $G$-compact.

\begin{prop}
The functor sending $X$ to $\cald_{*}^{G}(X)$ is homotopy invariant.
\end{prop}

\begin{proof}
The proof is nearly the same as the one in \cite[5.6]{BFJR} or the one in \cite[7.1]{ull}. We only have to be careful about the metric control condition in the $G$-direction which does not appear in \cite{BFJR}. \\
Let $X$ be a $G$-$CW$-complex. We have to see that $\cald_{*}^{G}(X \times I) \rightarrow \cald_{*}^{G}(X)$ is an equivalence. Let $X_{i} = X \times [i,1]$ and $Z_{i} = X \times [i, \infty)$ for $i = 0,1$. We obtain coarse spaces $X_{i}^{cc} = (G \times X_{i} \times [1, \infty), \cale_{X_{i}}, \calf_{X_{i}})$ and $Z_{i}^{cc} = (G \times Z_{i} \times [1, \infty), \cale_{Z_{i}}, \calf_{Z_{i}})$. Let $Z_{i}^{cc'}$ be the coarse space obtained from $Z_{i}^{cc}$ by relaxing the object support condition $\calf_{Z_{i}} = p^{-1}(\calf_{Gc}(G \times Z_{i}))$ to $G$-compact support only in $G \times X$-direction, i.e. not in the $[i, \infty)$-direction. Now consider the triple of coarse spaces $(Z_{0}^{cc'}, Z_{1}^{cc'}, X_{0}^{cc})$.  We want to check that this triple is coarsely excisive, i.e. satisfies the coarse Mayer-Vietoris principle \ref{MV}. Using \cite[5.3]{BFJR}, this is not hard to verify. \fxnote{really? where does the relaxed object support condition enter?} \\
Hence we get a homotopy pullback square
\[
\xymatrix{
K(\C_{*}^{G}(X_{1}^{cc})) \ar[dd] \ar[rr] & & K(\C_{*}^{G}(X_{0}^{cc})) \ar[dd] \\
\\
K(\C_{*}^{G}(Z_{1}^{cc'})) \ar[rr] & & K(Z_{0}^{cc'}) \\
}
\]
We will argue in the next lemma that both lower categories admit an Eilenberg Swindle. It follows that the upper map is an equivalence of spectra, which was our goal.
\fxnote{continue}
\end{proof}

\begin{lem}
For $i = 0, 1$, the map 
\begin{align*}
s: [i, \infty) \times G \times X \times [1, \infty) &\rightarrow [i, \infty) \times X \times G \times [1, \infty) \\
(a,x,g,r) &\mapsto (a+\frac{1}{r},x,g,r)
\end{align*}
induces an Eilenberg swindle on $\C_{*}^{G}(Z_{i}^{cc'})$
\end{lem}

\begin{proof}
See \cite[7.12]{ull}.
\fxnote{do it yourself}
\end{proof}

\begin{prop}
The functor sending $X$ to $\cald_{*}^{G}(X)$ satisfies Mayer-Vietoris.
\end{prop}

\begin{proof}
Let $X \leftarrow Y \rightarrow Z$ be a diagram of $G$-$CW$-complexes. Assume $Y$ is nonemtpy; the case $Y = \emptyset$ is a special case of the disjoint union axiom and will be proven below. By homotopy invariance of $\cald_{*}^{G}(-)$, we can replace $Y \rightarrow X$ and $Y \rightarrow Z$ by the inclusions of $Y$ into the respective mapping cylinders $A = M_{X}$ and $B = M_{Z}$, such that $A \cap B = Y$. We now verify the conditions of the coarse Mayer-Vietoris principle \ref{MV}:  For each $E \in \cale_{X}, F \in \calf_{X}$, there is $E' \in \cale_{X}, F' \in \calf_{X}$ such that
\begin{align*}
(G \times &A \times [1, \infty) \cap F)^{E} \cap (G \times B \times [1, \infty) \cap F)^{E}  \cap F \subset \\ & \subset (G \times A \cap B \times [1, \infty) \cap F')^{E'} \cap F'
\end{align*}
Since nothing really happens in the $G$-direction, the metric control of $E'$ in this direction can be taken to be the same as the control of $E$ with respect to $d_{G}$. The difficult part is the continuous control condition. If $p: G \times X \times [1, \infty) \rightarrow X \times [1, \infty)$ is the projection, $U \subset X \times [1, \infty)$ and $E'' \in \cale_{cc}(X)$ are given, it is easy to verify that
\[
p^{-1}(U)^{p^{-1}(E'')} = p^{-1}(U^{E''})
\] 
(this is a general statement and has really nothing to do with continuous control). It follows that it (nearly) suffices to see the Mayer-Vietoris condition for $A \times [1, \infty)$, $B \times [1, \infty)$ and $A \cap B \times [1, \infty)$ with respect to the continuous control condition. This is verified in \cite[5.3]{BFJR}. Finally, if $E$ is $k$-controlled in the $G$-direction, let $F' \subset G \times X$ be the set consisting of all $(h,x)$ such that there is $g \in G$ with $d_{G}(g,h) \leq k$ and $(g,x) \in F$. It is easy to check that $F'$ is $G$-compact and that this $F'$ does the job. 
\end{proof}

\begin{prop}
\label{connected}
Let $X$ be a $G$-compact $G$-$CW$-complex. Let $f$ be any morphism in $\cald_{b}^{G}(X)^{\infty}$. Then $f$ has a representative $F$ in $\cald_{b}^{G}(X)$ with the additional property that whenever $x,y$ lie in different components of $X$, $F_{x,y} = 0$.
\end{prop}

\begin{proof}
Pick any representative $F$ of $f$, with $F$ $E$-controlled. Fix a point $x \in X$. The component $C_{x}$ containing $x$ is certainly invariant under the stabiliser of $X$. We can apply the defining property of  the continuous control condition to $U = C_{x} \times [1, \infty]$ to find a $G_{x}$-invariant neighborhood $V$ of $(x, \infty)$ such that $f$ cannot connect points from inside $G \times V$ to points outside of $G \times U$. By definition of the product topology, we can assume $V$ after maybe throwing away some parts to be of the form $S \times [t, \infty]$ where $S$ is a neighborhoof of $x$ in $C_{x}$ and $t \in \IR$, and $t$ may depend on $x$. It follows that after time $t$, $F$ does not connect $(g,x,t')$ to any point outside of $G \times C_{x} \times [1, \infty)$ as long as $t' > t$. Since we throw away all information in $[1, t]$ when passing to $f$, we may as well assume that $F$ does not connect points from inside $G \times C_x \times [1, \infty)$ to any point outside of $G \times C_{x} \times [1, \infty)$ at all. By $G$-invariance, this is also true for all points in the orbit of $G$. Finally, since the source object of $F$ is locally finite and $X$ is $G$-compact, there are only finitely many $G$-orbits on which $F$ can have components, so a finite repitition of the above argument yields the claim.

\end{proof}

\begin{prop}
If $X = Y \coprod Z$ as $G$-$CW$-complexes, then
\[
K \cald_{*}^{G}(X)^{\infty} \cong K \cald_{*}^{G}(Y)^{\infty} \vee K \cald_{*}^{G}(Z)^{\infty}
\]
\end{prop}

\begin{proof}
Applying the above proposition, we see that
\[
\cald_{b}^{G}(X)^{\infty} = \cald_{b}^{G}(Y)^{\infty} \times \cald_{b}^{G}(Z)^{\infty}
\]
as pre-normed categories. By completion, it follows 
\[
\cald_{*}^{G}(X)^{\infty} = \cald_{*}^{G}(Y)^{\infty} \times \cald_{*}^{G}(Z)^{\infty}
\]
The claim then easily follows.
\end{proof}

\section{Identifying the coefficients}

To identify the $G$-homology theory we have constructed, we need to study the associated $\Or(G)$-spectrum, i.e. the functor
\[
G/H \mapsto K \cald_{*}^{G}(G/H)^{\infty}
\]
Then we can employ \cite[Section 6]{DL} which shows that the homology theory is uniquely determined by the $\Or(G)$-spectrum.
Recall that there is a fiber sequence
\[
K \C_{*}^{G}( G \times G/H, i^{-1}(\cale_{G/H}), i^{-1}(\calf_{G/H})) \rightarrow K \cald_{*}^{G}(G/H) \rightarrow K \cald_{*}^{G}(G/H)^{\infty}
\]
The only control condition playing a role in the left-hand spectrum is the $G$-compact object support condition, i.e.
\[
K \C_{*}^{G}( G \times G/H, i^{-1}(\cale_{G/H}), i^{-1}(\calf_{G/H}))  = K \C_{*}^{G}( G \times G/H, \calf_{Gc})
\]
By \ref{connected}, we may add the additional control condition to the left and middle categories that a morphism is not allowed to connect different components of $G/H$ without changing the quotient. So let $\cale_{\Delta}$ be the morphism support condition on $G/H$ which only contains the diagonal. We obtain a homotopy fiber sequence
\[
\xymatrix{
K \C_{*}^{G}( G \times G/H, p^{-1}(\cale_{\Delta}), \calf_{Gc}) \ar[d] \\ K \C_{*}^{G}(G \times G/H \times [1, \infty); \cale_{G/H} \cap p^{-1} \cale_{\Delta}, \calf_{G/H}) \ar[d] \\ K \cald_{*}^{G}(G/H)^{\infty} \\
}
\]
Since $G/H$ is discrete, the continuous control condition in the middle category is no additional restriction once one has accounted for metric control in the $[1, \infty)$-direction. Hence the map $G \times G/H \times [1, \infty) \rightarrow G \times G/H \times [1, \infty)$, $(g, g'H,t) \mapsto (g,g'H, t+1)$ induces an Eilenberg swindle on the middle term of the above homotopy fiber sequence. As in the proof of \ref{homotopyfiber}, up to equivalence we may replace $K \C_{*}^{G}( G \times G/H, p^{-1}(\cale_{\Delta}), \calf_{Gc})$ by the $K$-theory of the non-unital category $J$ which is the kernel of 
\[
\C_{*}^{G}(G \times G/H \times [1, \infty); \cale_{G/H} \cap p^{-1} \cale_{\Delta}, \calf_{G/H}) \rightarrow \cald_{*}^{G}(G/H)^{\infty}
\]
Then the homotopy fiber sequence becomes an actual fibration of spectra, coming from the short exact sequence of $C^{}*$-algebras
\[
T(J) \rightarrow T(G \times G/H) \rightarrow T(G \times G/H)/J
\]
Regarding all three spectra as depending functorially on $G/H$, we hence obtain a pointwise fiber sequence of $\Or(G)$-spectra
\[
\xymatrix{
K J_{*}^{G}( G \times -, p^{-1}(\cale_{\Delta}), \calf_{Gc}) \ar[d] \\ K \C_{*}^{G}(G \times - \times [1, \infty); \cale_{-} \cap p^{-1} \cale_{\Delta}, \calf_{-}) \ar[d] \\ K \cald_{*}^{G}(-)^{\infty}
}
\]
where ''$-$`` denotes functoriality in $\Or(G)$. The middle $\Or(G)$-spectrum is contractible, and since we have a pointwise fibration, the boundary map of the fiber sequence gives an identification of $\Or(G)$-spectra
\[
\Omega K J_{*}^{G}( G \times -, p^{-1}(\cale_{\Delta}), \calf_{Gc}) \cong K \cald_{*}^{G}(-)^{\infty}
\]
and we have an equivalence of $\Or(G)$-spectra
\[
K J_{*}^{G}( G \times -, p^{-1}(\cale_{\Delta}), \calf_{Gc})  \cong K \C_{*}^{G}( G \times -, p^{-1}(\cale_{\Delta}), \calf_{Gc})
\]
Altogether, it follows:

\begin{prop}
The homology theory $\cald_{*}^{G}(X)$ is, up to a dimension shift of $1$, equivalent to the homology theory given by the $\Or(G)$-spectrum
\[
G/H \mapsto K \C_{*}^{G}( G \times G/H, p^{-1}(\cale_{\Delta}), \calf_{Gc})
\]
\end{prop}

It remains to identify the $\Or(G)$-spectrum on the right-hand side with the $\Or(G)$-spectrum $A^{G}_{r}$ constructed before. 

\begin{prop}
The $\Or(G)$-spectrum given by 
\[
G/H \mapsto K \C_{*}^{G}(G \times G/H; p^{-1}\cale_{\Delta}, \calf_{Gc}; A)
\]
is naturally weakly equivalent to the functor $A_{r}^{G}$ constructed after \ref{transitivity}.
\end{prop}

\begin{proof}
Both functors are actually functors from $\Or(G)$ to $C^{*}$-categories composed with topological $K$-theory. So it suffices to see that the functors from $\Or(G)$ to $C^{*}$-categories are equivalent. \\
We first argue that the uncompleted versions $\C_{b}^{G}(G \times G/H; p^{-1}\cale_{\Delta}, \calf_{Gc})$ and $(A \rtimes_{G} G/H)_{\oplus}$ are equivalent (in fact, even isomorphic).  
For $s \in G/H$, let $A_{s}$ be the object of $\C_{b}^{G}(G \times G/H; p^{-1}\cale_{\Delta}, \calf_{Gc})$ which is only supported over the $G$-orbit of $(e,s)$, where it is $A$. Because of the $G$-compact object support condition,  each object of the category 
\[
\C_{b}^{G}(G \times G/H; p^{-1}\cale_{\Delta}, \calf_{Gc})
\]
is a direct sum of copies of $A_{s}$ for varying $s$ in a canonical way. So it suffices to see that the full subcategory of $\C_{b}^{G}(G \times G/H; p^{-1}\cale_{\Delta}, \calf_{Gc})$  spanned by the $A_{s}$ is equivalent to $A \rtimes_{G} G/H$. Both categories have precisely one object for each $s \in G/H$, and the desired equivalence sends $A_{s}$ to the object $s$ of $A \rtimes_{G} G/H$. Now pick a morphism $f \colon A_{s} \rightarrow A_{t}$. By $G$-invariance, $f$ is determined by its components entering $(e,t)$. Because of the morphism support condition $\cale_{\Delta}$, the only potential sources for these components are the $(g,gs)$ for $g \in G$ with $gs = t$. Now 
\[
\sum\limits_{g \in G; gs = t} f_{(e,t),(g,gs)}
\]
is a valid morphism of $A \rtimes_{G} G/H$. 
Since we may choose the $f_{(e,t),(g,gs)}$ arbitrary, this gives a bijection between the Hom-sets $\Hom(A_{s}, A_{t})$ in $\C_{b}^{G}(G \times G/H; p^{-1}\cale_{\Delta}, \calf_{Gc})$ and $\Hom(s,t)$ in $A \rtimes_{G} G/H$. We have to check compatibility with compositions. Since composition is bilinear, it suffices to consider $f: A_{s} \rightarrow A_{t}$ which has only a single nonzero component $f_{(e,t),(g,gs)}$ entering $(e,t)$ and $d: A_{t} \rightarrow A_{r}$ with only one nonzero component $d_{(e,r),(h,ht)}$. The composition of $f$ and $d$ in $\C_{b}^{G}(G \times G/H; p^{-1}\cale_{\Delta}, \calf_{Gc})$ also has only one nonzero component entering $(e,r)$, namely the one from $(hg,hgs)$. This component is 
\[
d_{(e,r),(h,ht)} \cdot (h \cdot f_{(e,t),(g,gs)})
\]
This morphism is sent to $d_{(e,r),(h,ht)} \cdot (h \cdot f_{(e,t),(g,gs)}) hg$, which is the same as the composition of $d_{(e,r),(h,ht)}h $ and $f_{(e,t),(g,gs)} g$ in $A \rtimes_{G} G/H$. \\
It is straightforward to check that this equivalence of categories is natural in the $\Or(G)$-variable $G/H$. It remains to see that the norms on $\C_{b}^{G}(G \times G/H; p^{-1}\cale_{\Delta}, \calf_{Gc})$ and $A \rtimes_{G} G/H$ agree under this isomorphism.  Fix $s,t \in G/H$ and consider the map
\[
i_{s,t} \colon \hom_{A \rtimes_{G} G/H}(s,t) \rightarrow \calb(l^{2}(\hom_{G/H}(t,z),A), l^{2}(\hom_{G/H}(s,z),A))
\]
which defines the norm on $\hom_{A \rtimes_{G} G/H}(s,t)$. To compute the norm on the category $\C_{b}^{G}(G \times G/H; p^{-1}\cale_{\Delta}, \calf_{Gc})$, we obtain an orthogonal splitting
\[
T(A_{s}) = \oplus_{(g,gs)} A  \cong \oplus_{p \in G/H} \oplus_{g \in \hom_{G/H}(s,p)} A_{(g,p)}
\]
and similarly for $T(A_{t})$. For each morphism $f \colon A_{s} \rightarrow A_{t}$ in $\C_{b}^{G}(G \times G/H; p^{-1}\cale_{\Delta}, \calf_{Gc})$, $T(f)$ respects this splitting, i.e. $T(f)$ is an orthogonal direct sum over $G/H$ of maps 
\[
\oplus_{g \in \hom_{G/H}(s,p)} A_{(g,p)} \rightarrow \oplus_{g \in \hom_{G/H}(t,p)} A_{(g,p)}
\]
where $p \in G/H$ is fixed. The norm of each of these maps is the norm of $f$ in $A \rtimes_{G} G/H$ by definition, and since the sum is orthogonal, also the norm of $T(f)$ is the norm of $f$ in $A \rtimes_{G} G/H$. This finishes the proof.
\end{proof}

Altogether, we arrive at the following result.

\begin{prop}
\label{criterion}
The Baum-Connes conjecture for the group $G$ with coefficients in $A$ with respect to the family $\calf$ is true if and only if
\[
K_{*} \D_{*}^{G}(E_{\calf}) = 0
\]
\end{prop}

\begin{proof}
Consider the following diagram of categories
\[
\xymatrix{
\C_{*}^{G}(G \times E_{\calf}) \ar[d] \ar[rr] & & \D_{*}^{G}(E_{\calf}) \ar[d] \ar[rr] & & \D_{*}^{G}(E_{\calf})^{\infty} \ar[d] \\
\C_{*}^{G}(G) \ar[rr] & & \D_{*}^{G}(pt) \ar[rr] & & \D_{*}^{G}(pt)^{\infty} \\
}
\] 
induced by the projection $E_{\calf} \rightarrow pt$. The Baum-Connes map is the map in $K$-theory induced by the right-hand vertical map. The left-hand vertical map induces an isomorphism in $K$-theory since both categories are equivalent to finitely generated free $A \rtimes_{r} G$-modules, and both rows give rise to long exact sequences of $K$-groups. Hence the Baum-Connes map is an isomorphism if and only if the middle vertical map is an isomorphism in $K$-theory.  The category 
\[
\D_{*}^{G}(pt) = \C_{*}^{G}(G \times pt \times [1, \infty))
\]
allows an Eilenberg swindle $(g,t) \mapsto (g,t+1)$ and hence has trivial $K$-theory. So the middle vertical map is an isomorphism in $K$-theory if and only if $K_{*} \D_{*}^{G}(E_{\calf}) = 0$.
\end{proof}

Later on, we will need more flexible notation, so we also make the following definition:

\begin{definition}
\label{o-notation}
We denote $\D_{*}^{G}(E_{\calf})$ as $\calo_{*}^{G}(E_{\calf}, G, d_{G})$.
\end{definition}

The reason for this is that we later on want to vary the metric space $(G,d_{G})$ and want a uniform notation for this.

\section{Inheritance properties}

We have the following inheritance properties for the Baum-Connes conjecture.

\begin{thm}
Consider a short exact sequence of groups 
\[
0 \rightarrow G \rightarrow \Gamma \rightarrow H \rightarrow 0
\]
 Suppose that $H$ satisfies the Baum-Connes conjecture with coefficients and that for each virtually cyclic subgroup $V$ of $H$, its preimage in $\Gamma$ satisfies the Baum-Connes conjecture. Then $\Gamma$ satisfies the Baum-Connes conjecture with coefficients.
\end{thm}

\begin{proof}
This is \cite[3.3]{oyono}. The assumptions there are even somewhat weaker, replacing virtually cyclic subgroups of $H$ with finite ones.
\end{proof}

\begin{thm}
The Baum-Connes conjecture with coefficients passes to subgroups, i.e. if $H \subset G$ and $G$ satisfies the Baum-Connes conjecture with coefficients, so does $H$.
\end{thm}

\begin{proof}
See \cite[2.5]{inheritanceBC}.
\end{proof}

%% file: chapters/swan.tex
\chapter{Swan group actions}

In this section, we collect the necessary tools from the representation theory of finite groups and give a classical application to algebraic $K$-theory, which is the original motivation for the constructions we will consider later on.

\section{Swan groups and induction}

Let $G$ be a group. We will consider the categories of finite-dimensional complex respectively integral representations of $G$, together with the involution which associates to a representation $\rho \colon G \rightarrow \Aut(V)$ the dual representation $\rho^{*} \colon G \rightarrow \Aut(V^*)$ given by $\rho^{*}(g) = \rho(g^{-1})^{*}$. However, to make our constructions work, we need a stricter version of these categories. Additionally, for technical reasons which will be apparent later on we only want to consider unitary complex representations of $G$. This leads to the following definition.

\begin{definition}
Fix a group $G$. Let $\Mod_{(\IC,G)}$ be the category with objects pairs $(\IC^{n}, \rho)$ where $\rho \colon G \rightarrow U(n)$ is a unitary representation of $G$. A morphism from $(\IC^{n}, \rho)$ to $(\IC^{k}, \tau)$ is a linear map $f \colon \IC^{n} \rightarrow \IC^{k}$ such that $f(\rho(g)v) = \tau(g)(f(v))$ for all $g \in G$ and $v \in \IC^{n}$. This is an exact category where we call a sequence exact if its underlying sequence of $\IC$-modules is exact. Since we imposed the unitary condition, this category is in fact split exact, i.e. each exact sequence
\[
0 \rightarrow (\IC^{n}, \rho) \rightarrow (\IC^{k}, \tau) \stackrel{f}{\rightarrow} (\IC^{m}, \eta) \rightarrow 0
\]
splits: The orthogonal complement of the kernel of $f$ is $G$-invariant and isomorphic to $ (\IC^{m}, \eta)$. \\
The category $\Mod_{(\IC,G)}$ also carries an involution, which is the identity on objects since $\rho$ is unitary and so $\rho(g^{-1})^{*} = \rho(g)$, and is given on a morphism $f: (\IC^{n}, \rho) \rightarrow (\IC^{k}, \tau)$ by sending the matrix $A$ representing $f$ to its conjugate transpose $A^{*}$ to obtain
\[
f^* \colon (\IC^{k}, \tau) \rightarrow (\IC^{n}, \rho)
\]
Finally, we define the unitary Swan group $\Sw^{u}(G)$ of $G$ as
\[
Sw^u(G) = K_{0}(\Mod_{(\IC,G)})
\]
\end{definition}

\begin{definition}
Let $\Mod_{(\IZ,G)}$ be the category with objects pairs $(\IZ^{n}, \rho)$ where $\rho: G \rightarrow Gl_{n}(\IZ)$ is a representation of $G$.  Morphisms are defined as above. This is an exact category where we call a sequence exact if its underlying sequence of $\IZ$-modules is exact. Note that this category is not split exact. The involution on $\Mod_{(\IZ, G)}$ is given by $(\IZ^{n}, \rho)^{*} = (\IZ^{n}, \rho^{*})$ where $\rho^{*}(g) = (\rho(g^{-1}))^{*}$, where the last $*$ means transposing the $n \times n$-matrix $\rho(g)$. Finally, we define the (integral) Swan group $\Sw(G)$ of $G$ to be $K_{0}(\Mod_{(\IZ,G)})$.
\end{definition}

Note that for a unitary representation, $ (\rho(g^{-1}))^{*} = \rho(g)$. This is the reason we can arrange the involution to be the identity on objects in the unitary case.  We will  concentrate on the complex case.\\
Let $H \subset G$ be a finite-index subgroup. Then there are functors 
\[
\res_{H}^{G} \colon \Mod_{(\IC,G)} \rightarrow \Mod_{(\IC,H)}
\]
and 
\[
\ind_{H}^{G} \colon \Mod_{(\IC,H)} \rightarrow \Mod_{(\IC,G)}
\]
defined as follows. Let $i \colon H \rightarrow G$ be the inclusion. The functor $\res_{H}^{G}$ just sends a $G$-representation $(\IC^{n}, \rho)$ to the $H$-representation 
\[
\res_H^G(\IC^{n}, \rho)  =  (\IC^{n}, \rho \circ i) 
\]
and acts as identity on the underlying morphisms of vector spaces. The induction functor is more complicated to define since we have to write it down explicitly in our fixed coordinates; compare \cite[2.C]{FH} for the following definition. \\ 
Pick representatives $g_{1}, \dots, g_{m}$ for $G/H$. Let $(\IC^{n}, \rho) \in \Mod_{(\IC,H)}$ be a complex representation of $H$. Its image under $\ind_{H}^{G}$ is the object $(\IC^{mn}, \eta)$, where $\eta(g) \in U(mn)$ is the matrix whose entry at $(kn+i, k'n+i')$ is $0$ as long as $gg_{k'}H \neq g_{k}H$, and is $\rho(g_{k}^{-1}gg_{k'})_{i,i'}$ if $gg_{k}'H =  g_{k}H$. On morphisms, $\ind_H^G f$ is just $\oplus_{G/H} f$.
 \fxnote{some explanation?}

When $G$ is finite, $\Sw^{u}(G)$ is just the complex representation ring of $G$. The following classical result is an important ingredient for our strategy. Recall that a finite group $G$ is called $p$-elementary if it is the product of a cyclic group $C$ and a $p$-group $P$, and elementary if it is $p$-elementary for some prime $p$.

\begin{thm}
\label{brauer}
\emph{\textbf{The Brauer Induction Theorem:}} Let $G$ be a finite group. Let $\calh$ be the family of elementary subgroups of $G$. Then the map
\[
\bigoplus_{H \in \calh} \ind^{G}_{H}: \bigoplus_{H \in \calh}  \Sw^{u}(H) \rightarrow \Sw^{u}(G)
\]
is onto. In particular, we find $T_{H} \in \Sw^{u}(H)$ such that 
\[
1_{\Sw^{u}(G)} = \sum\limits_{H \in \calh} \ind_{H}^{G} T_{H}
\]
\end{thm}

For a proof, see \cite[Theorem 10.19]{Serre}. We will also make use of the following, exchanging a simpler family of subgroups for a weaker statement.

\begin{thm}
\label{repcyclic}
Let $G$ be a finite group of order $n$. Let $\mathcal{CYC}$ be the family of cyclic subgroups. Then the map
\[
\bigoplus_{C \in \mathcal{CYC}} \ind^{G}_{C}: \bigoplus_{C \in \mathcal{CYC}}  \Sw^{u}(C) \rightarrow \Sw^{u}(G)
\]
has finite cokernel of exponent at most $n$, i.e. the cokernel is annihilated by multiplication with $n$. In particular, we find $T_{C} \in \Sw(C)$ such that
\[
n \cdot 1_{\Sw^{u}(G)} = \sum\limits_{C \in \mathcal{CYC}} \ind_{C}^{G} T_{C}
\]
\end{thm}

This is a stronger version of the Artin induction theorem. For a proof, see \cite[Proposition 9.27]{Serre}.

\section{Action of the Swan group on the algebraic $K$-theory of complex group rings}

For completeness' sake, we include this classical material which is the original motivation for our later constructions. \\
Given a $\IC$-algebra $A$, a right $AG$-module $M$ and an element \\ $T = (\IC^{n}, \rho)$ of $\Mod_{(\IC, G)}$,  we can form the tensor product over $\IC$
\[
T \otimes M = T \otimes_{\IC} M
\]
and equip it with the diagonal $G$-action, where we also view $T$ as a right $G$-module via $vg = \rho(g^{-1})(v)$. We will denote this module as $T \otimes_{\Delta} M$, to emphasize that the $G$-action is diagonal.

\begin{prop}
\label{action}
If $M$ is a finitely generated projective $AG$-module, then $T \otimes_{\Delta} M$ is also a finitely generated projective $AG$-module. Hence we get a biexact pairing
\[
\Mod_{(\IC,G)} \times \calp(AG) \rightarrow \calp(AG)
\]
and hence an induced $\Sw^{u}(G)$-module action on the algebraic $K$-theory of $AG$. The same holds for a complex algebra with $G$-action $A$ and the twisted group ring $A \rtimes G$.
\end{prop}

\begin{proof}
The reader should note that the unitary condition we imposed on our representations is not necessary here. It is easy to check that $T \otimes_{\Delta} -$ is an additive functor. Hence it suffices to check that $T \otimes_{\Delta} AG$ is projective as an $AG$-module. Let $T \otimes_{r} AG$ denote the tensor product with $G$ only acting on $AG$. There is an $AG$-module isomorphism $T \otimes_{\Delta} AG \rightarrow T \otimes_{r} AG$ sending $t \otimes g$ to $tg^{-1} \otimes g$ and extending $A$-linearly, with inverse sending $t \otimes g$ to $tg \otimes g$. Hence $T \otimes_{\Delta} AG$ is free since 
\[
T \otimes_{r} AG \cong \bigoplus\limits_{1}^{dim T} AG
\]
clearly is. \\
The only potential difficulty in checking biexactness is that for a short exact sequence $0 \rightarrow R \rightarrow S \rightarrow T \rightarrow 0$ in $\Mod_{(\IC,G)}$ and a projective $AG$-module $M$, also the sequence 
\[
0 \rightarrow R \otimes_{\Delta} M \rightarrow S \otimes_{\Delta} M \rightarrow T \otimes_{\Delta} M \rightarrow 0
\]
is an exact sequence of projective $AG$-modules. But the underlying sequence of $\IC$-modules is exact since tensoring over $\IC$ is exact, hence the sequence is also exact as a sequence of $AG$-modules. Alternatively, we can construct a section $S \otimes_{\Delta} M \rightarrow T \otimes_{\Delta} M$ as follows: Consider the case $M = AG$. First pick a section $q \colon T \rightarrow S$ as $\IC$-module maps. Then define 
\[
s \colon T \otimes_{\Delta} AG \rightarrow S \otimes_{\Delta} AG
\]
as $s(t \otimes g) = q(tg^{-1})g \otimes g$. To understand this formula, note that $s$ is the composition of the isomorphism $T \otimes_{\Delta} AG \cong T \otimes_{r} AG$ from above, the map $q \otimes \Id \colon T \otimes_{r} AG \rightarrow S \otimes_{r} AG$ which is $G$-equivariant since $G$ does not act on $T$ respectively $S$, and finally the inverse of the isomorphism $S \otimes_{\Delta} AG \cong S \otimes_{r} AG$. Since $p(g \cdot -) = g \cdot (p(-))$ and $p \circ q = Id_{T}$, it follows that $(p \otimes \Id) \circ s = Id_{T \otimes_{\Delta} AG}$. For general $M$, the claim now follows easily since a direct sum of two potential short exact sequences is exact if and only if the two original sort sequences where exact. \\ 
The claim about the module structure is a standard $K$-theory fact, see for example \cite[Cor. 1 on page 106]{quillenktheory}. The argument for twisted group rings is the same once one convinces oneself that $S \otimes_{\Delta} (A \rtimes G)$ is indeed an $A \rtimes G$-module.
\end{proof}

\begin{rem}
On first glance, the formulas we used here for the isomorphism $T \otimes_{\Delta} AG \rightarrow T \otimes_{r} AG$ are different from the ones we will use in the construction of the transfer in \ref{defoftransfer}. The difference is that in the controlled setup, the right action of $G$ on $G$ will not come from right multiplication, but from left multiplication with the inverse, i.e. we view the $A$-module $AG$ as a right $G$-module via $g \cdot h = h^{-1}g$. This is still a free $AG$-module of rank one - a canonical isomorphism to  $AG$ with the usual right action is given by $g \mapsto g^{-1}$. As a consequence, formulas where we used an inverse here will be without an inverse later on and conversely.
\end{rem}

\section{Some computations}

Fortunately, we will never need any computations of concrete Swan groups; but for a few easy groups, we will indicate some results. \\
The integral Swan group $\Sw(G)$ is extremely hard to compute, even if $G$ is an easy group. For some partial results for $G = \IZ$ see \cite[Cor. 5]{graysonendo}. For finite groups, \cite{swaninduced} gives at least some structural induction results, but concrete computations are not easy. The interested reader may consult \cite{curtis-reiner} for some concrete results.
 \\
In the complex case, one can do better. For finite groups, we are just considering the complex representation ring which is a well-studied object for a finite group.  For abelian groups, we have the following:

\begin{prop}
For $G = \IZ^{n}$ and $s = (s_{1}, \dots s_{n}) \in (S^{1})^{n}$, sending $s$ to the $\IZ^{n}$-representation on $\IC$ given by letting the $i$-th unit vector $e_{i}$ act as multiplication by $s_{i}$ induces an isomorphism
 \[
 \IZ[(S^{1})^{n}] \rightarrow \Sw^{u}(\IZ^n)
 \]
Similarly, if $G$ is abelian, $\Sw^{u}(G)$ is the free abelian group on the one-dimensional representations, or in other words, the free abelian group on the set of group homomorphisms
\[
G \rightarrow S^{1}
\]
\end{prop}

\begin{proof}
A $\IZ^{n}$-representation on $\IC^{k}$ is determined by the actions of $e_{i}$, i.e. by $n$ commuting unitary matrices $A_{i}$. Since unitary matrices are diagonalizable and commuting diagonalizable matrices are simultaneously diagonalizable, we can assume that all $A_{i}$ are diagonal. But then clearly the representation splits as a sum of $k$ one-dimensional representations. Now the claim is easy to check. The same argument works for any abelian group $A$ to reduce to one-dimensional representations.
\end{proof}

Decomposing a representation into its irreducible components and employing Schur's lemma to deduce that this decomposition is basically unique, we obtain the following:

\begin{prop}
For any group $G$, $\Sw^{u}(G)$ is free abelian on the set of isomorphism classes of finite-dimensional irreducible unitary representations of $G$.
\end{prop}

\section{The relationship with equivariant $KK$-theory}

Let $G$ be a discrete group. We want to shed some light on the relationship between $\Sw^{u}(G)$ and the Kasparov representation ring $R(G) = KK_0^G(\IC,\IC)$. 

\begin{definition}
A Kasparov $G$-module for $(\IC,\IC)$ is a $\IZ/2\IZ$-graded separable Hilbert space $\calh = \calh_{0} \oplus \calh_{1}$, a unitary grading-preserving representation $\rho$ of $G$ on $\calb(\calh)$ and a self-adjoint degree-$1$ operator $F: \calh \rightarrow \calh$ such that $F^{2}-1$ is a compact operator and for each $g \in G$, $\rho(g) \circ F \circ \rho(g^{-1})-F$ is a compact operator. Two Kasparov modules $(\calh, \rho, F)$ and $(\calh', F', \rho')$ are isomorphic if there is a grading-preserving isomorphism $T: \calh \rightarrow \calh'$ preserving all the structure, i.e. such that $F = T^{-1} \circ F' \circ T$ and $\rho(g) = T^{-1}  \circ \rho'(g) \circ T$ for all $g \in G$. \\
A Kasparov $G$-module is degenerate if $F^{2} = 1$ and $\rho(g) \circ F\circ \rho(g)^{-1} = F$ on the nose, not only up to compact operators. Two Kasparov $G$-modules $(\calh, \rho, F)$ and $(\calh',  \rho', F')$ are homotopic if $\calh = \calh', \rho = \rho'$ and $F,F'$ can be connected by a continuous path of operators $F_{t}$ such that $(\calh, \rho, F_{t})$ is a Kasparov $G$-module for all $t$. 
\end{definition}

\noindent
We can form the direct sum of two Kasparov $G$-modules to obtain a new Kasparov $G$-module, so the isomorphism classes of Kasparov $G$-modules form a monoid. When one divides out homotopy and degenerate modules, the following proposition tells us that we obtain a group.

\begin{prop}
Let $(\calh, \rho, F)$ be a Kasparov $G$-module. Let $\calh'$ be $\calh$ with grading reversed. Then $(\calh, \rho, F) \oplus (\calh', \rho, F)$ is homotopic to a degenerate Kasparov module. 
\end{prop}

\begin{proof}
 See \cite[4.2.7]{valette}.
\end{proof}

\begin{definition}
We define $KK_{0}^{G}(\IC,\IC)$ to be the monoid of isomorphism classes of Kasparov $G$-modules, modulo homotopy and degenerate Kasparov modules. The above lemma tells us that this is a group. 
\end{definition}

We say that a Kasparov module $(\calh, \rho, F)$ is strict if it is strictly $G$-equivariant, i.e. satisfies $\rho(g) \circ F = F \circ \rho(g)$ for all $g \in G$. A strict homotopy between two strict modules is a homotopy $(\calh, \rho, F_{t})$ such that $F_{t}$ is strict for all $t$. 

\begin{definition}
Let $KK_{0}^{G, strict}(\IC,\IC)$ be the abelian group of isomorphism classes of strict Kasparov $G$-modules, modulo strict homotopy and degenerate modules.
\end{definition}

The following theorem gives the precise relationship between $KK$ and the unitary Swan group:

\begin{thm}
Let $G$ be a group.
\begin{enumerate}
\item There is a canonical map $\gamma \colon KK_{0}^{G,strict}(\IC,\IC) \rightarrow KK_{0}^{G}(\IC,\IC)$. 
\item If $G$ is finite, $\alpha$ is an isomorphism.
\item For any group $G$, there is a canonical isomorphism
\[
\alpha: \Sw^{u}(G) \rightarrow KK_{0}^{G,strict}(\IC,\IC)
\]
\end{enumerate}
\end{thm}

\begin{proof}
The map $\gamma$ is obtained by considering a strict module as an ordinary module. This is clearly well-defined. \\
To prove  (ii), start with a Kasparov module $(\calh, \rho, F)$ and define
\[
F^{G} = \frac{1}{\left\vert G \right\vert} \sum\limits_{g \in G} \rho(g) \circ F \circ \rho(g^{-1})
\]
which makes sense since $G$ is finite. Since $\rho(g) \circ F \circ \rho(g^{-1})-F$ is compact for all $G$, $F^{G}-F$ is compact, and $F^{G}$ is clearly $G$-equivariant. The linear path from $F$ to $F^{G}$ then defines a homotopy from $F$ to $F^{G}$; this path indeed passes through Kasparov modules since $F^{G}-F$ is compact, so the path is just a compact perturbation. Hence each ordinary module is homotopic to a strict one, proving surjectivity of $\gamma$. Similarly, we can average out any homotopy between equivariant Kasparov modules to obtain an equivariant homotopy, which proves injectivity. \\
For (iii),  pick some element $\eta \in \Sw^{u}(G)$, represented by a formal difference $(\IC^{n}, \rho)-(\IC^{k},\tau)$ of two finite-dimensional representations of $G$. Let $\calh = \IC^{n} \oplus \IC^{k}$ with the representation $\rho \oplus \tau$, with $\IC^{n}$ the even part and $\IC^{k}$ the odd part, and let $F = 0 \colon \calh \rightarrow \calh$. Since $\calh$ is finite-dimensional, the compactness conditions are automatically satisfied, so we obtain a Kasparov $G$-module $(\calh, \rho \oplus \tau, 0)$. In the case $\rho = \tau$, this module is homotopic to a degenerate one: Then
\[
\begin{pmatrix}
0 & 1 \\ 1 & 0
\end{pmatrix} \colon \IC^{n} \oplus \IC^{n} \rightarrow \IC^{n} \oplus \IC^{n}
\]
defines a degenerate module which is homotopic to $(\calh, \rho \oplus \rho, 0)$ via the linear homotopy - again, the compactness conditions are empty since all operators are compact anyway. This construction is compatible with direct sums. If $(\calh, \rho \oplus \tau, 0)$ is homotopic to a degenerate module, then $\rho$ and $\tau$ are isomorphic: The potential degenerate operator on $\calh$ is of the form
\[
F = \begin{pmatrix}
0 & A \\ B & 0
\end{pmatrix} \colon \IC^n \oplus \IC^k \rightarrow \IC^n \oplus \IC^k
\]
and since $F^2 = 1$, we have $AB = 1, BA = 1$ and both $A$ and $B$ are $G$-equivariant since $F$ is equivariant by definition of a degenerate module. Altogether, we see that the class of $(\calh, \rho \oplus \tau, 0)$ in $KK_{0}^{G}(\IC, \IC)$ does not change when we pass from $\eta = \rho-\tau$ to $\eta = \rho \oplus \rho'- \tau \oplus \rho'$, so we indeed get a well-defined group homomorphism
\[
\alpha: \Sw^{u}(G) \rightarrow KK_{0}^{G}(\IC,\IC)
\]
Let $(\calh, \rho, F)$ be a strict Kasparov module. Write $\calh = \calh_{0} \oplus \calh_{1}$ and 
\[
F = \begin{pmatrix}
0 & F_{2} \\ F_{1} & 0
\end{pmatrix} 
\]
with respect to this decomposition. Since $F^{2} = 1$ up to compact operators, $F_{1}$ is a $G$-equivariant Fredholm operator $\calh_{0} \rightarrow \calh_{1}$, so its kernel and cokernel are finite-dimensional unitary representations of $G$. Set 
\[
\beta((\calh, \rho, F)) = [\ker(F_{1})]-[\coker(F_{1})] \in \Sw^{u}(G)
\]
If $(\calh, \rho, F)$ was degenerate to begin with, we have $\beta((\calh, \rho, F)) = 0$ since $F^{2} = 1$ means that $F_{1}$ is invertible. \\
Under on-the-nose equivariant operator homotopies $F^{t}$, the isomorphism classes of $\ker(F_{1}^{t})$ and $\coker(F_{1}^{t})$ can change. We have to argue that their formal difference does not change.  The proof is similar to the fact that the Fredholm Index is locally constant, which indeed is the special case $G = {1}$; compare \cite[2.1.5]{HigsonRoe} and \cite[III.9.4]{LM}. \\
Let $A,B,C$ be Hilbert spaces equipped with unitary $G$-representations. Let $S: A \rightarrow B$ be an equivariant Fredholm operator. We define the equivariant index of $S$ by
\[
index_{G}(S) = [\ker(S)]-[\coker(S)] \in \Sw^{u}(G)
\]
What we want to prove is that this equivariant index is locally constant. If $T \colon B \rightarrow C$ is another equivariant Fredholm operator, we have
\[
index_{G}(TS) = index_{G}(T)+index_{G}(S)
\]
since there is an exact sequence of $G$-modules
\begin{align*}
0 \rightarrow \ker(S) &\rightarrow \ker(TS) \rightarrow \ker(T) \\ &\rightarrow \coker(S) \rightarrow \coker(TS) \rightarrow \coker(T) \rightarrow 0
\end{align*}
and hence the alternating sum of these six terms is zero in $\Sw^{u}(G)$. \\
Let $J \colon A_{1}= \ker(S)^{\perp} \rightarrow A$ be the inclusion and $P \colon  B \rightarrow \Im(S) = B_{1}$ the orthogonal projection. Both $J$ and $P$ are Fredholm and $PSJ: A_{1} \rightarrow B_{1}$ is invertible and hence has index zero. But then for each equivariant Fredholm operator $S'$ sufficiently close to $S$, also $PS'J$ is invertible since invertible operators form an open subset of all bounded operators. Altogether, it follows
\begin{align*}
&index_{G}(P)+index_{G}(S')+index_{G}(J) = index_{G}(PS'J) = 0 \\& 
index_{G}(P)+index_{G}(S)+index_{G}(J) = index_{G}(PSJ) = 0
\end{align*}
and so $index_{G}(S) = index_{G}(S')$. Hence the equivariant index is locally constant and so homotopic Kasparov modules have the same index. Altogether, we obtain from $\beta$ a map which we will also call $\beta$
\[
\beta \colon KK_{0}^{G}(\IC,\IC) \rightarrow \Sw^{u}(G)
\]
Clearly, $\beta \circ \alpha = \id$. For the other direction, start with an equivariant Kasparov module $(\calh, \rho, F)$. Write $\calh_{0} = \ker(F_{0}) \oplus X$ where the sum is orthogonal, and similarly $\calh_{1} = \im(F_{0}) \oplus A$, where $A \cong \coker(F_{0})$. The restriction of $F_{0}$ to $X$ then yields an isomorphism $X \rightarrow \im(F_{0})$ with inverse $F_{2}$. One can then manufacture a new Kasparov module with $\calh' = X \oplus \im(F_{0})$ and the restrictions of $F_{0}$ and $F_{1}$ which is degenerate, and the original Kasparov module is the direct sum of this degenerate module and $(\ker(F_{0}) \oplus \coker(F_{0}), \rho, 0)$, which means that $\alpha \circ \beta$ is the identity.
\end{proof}

%% file: chapters/transfer.tex
\chapter{The core of the argument}

We have now set up the technical machinery and will dive into the actual proof of the Baum-Connes conjecture. Both the mathematics and the notation is heavily influenced by the proof of the Farrell-Jones conjecture in \cite{FH}.  We begin with some categorical preliminaries.

\section{Tensor products on $\IC$-categories}

Let $\C$ be an additive $\IC$-category. Up to equivalence, we can assume that $\C$ has a strictly associative direct sum: If necessary, replace $\C$ by the category $\C_{\oplus}$ whose objects are finite sequences $(A_{1}, A_{2},\dots, A_{n})$ of objects of $\C$ and where a morphism from a sequence $(A_{1}, \dots, A_{n})$ to $(B_{1}, \dots, B_{m})$ consists of a matrix of morphisms $f_{ij} \colon A_{j} \rightarrow B_{i}$. Composition is matrix multiplication. The category $\C_{\oplus}$ has direct sums by just concatenating two sequences to one sequence, which is strictly associative, and the functor $\C \rightarrow \C_{\oplus}$, $A \mapsto (A)$ is an equivalence of categories with inverse
given by
\[
(A_{1}, \dots A_{n}) \mapsto A_{1} \oplus A_{2} \oplus \dots \oplus A_{n}
\]
So we can assume that $\C$ is strictly associative. \\
Let $\calv$ be the category with objects $\IC^{n}, n \in \IN$ and morphisms the complex linear maps. The category $\calv$ is strict monoidal under the tensor product, which is a special case of the product we will define below. For any strictly associative additive $\IC$-category $\C$, there is a strictly associative paring
\[
- \otimes -: \calv \otimes \C \mapsto \C
\]
defined as follows. On objects,
\[
\IC^{n} \otimes A = A \oplus A \oplus \dots \oplus A (n \text{ summands})
\]
and on morphisms, given $f: \IC^{n} \rightarrow \IC^{k}$ and $\phi: A \rightarrow B$, we let
\[
f \otimes \phi: \IC^{n} \otimes A \rightarrow \IC^{k} \otimes B
\]
be the map whose $(i,j)$-component is $f_{ij}\phi: A \rightarrow B$, where $f_{ij} \in \IC$ and we are using the $\IC$-structure on $\C$ to make sense of $f_{ij}\phi$. It is straightforward to check that this yields a strictly associative pairing, i.e. such that the two possible functors
\[
\calv \times \calv \times \C \rightarrow \C
\]
are equal. 

\section{The strategy of the proof}
Throughout this section, fix a $G$-$C^{*}$-algebra $A$. We will usually suppress $A$ from the (already rather overloaded) notation. Our aim is to prove the Baum-Connes conjecture in some special cases. Recall the following result from \ref{criterion}, together with the notation from \ref{o-notation}:

\begin{prop}
The Baum-Connes conjecture for $G$ with respect to the family $\calf$ is true if and only if
\[
K_* \calo_{*}^{G}(E_{\calf}, G, d_{G}) = 0
\]
\end{prop}
For a metric space $(X,d)$ with a free, proper and isometric $G$-action and a $G$-space $E$, let 
\begin{align*}
p: X \times E \times [1, \infty) &\rightarrow E \times [1,\infty) \\
q: X \times E \times [1, \infty) &\rightarrow X \\
r: X \times E \times [1, \infty) &\rightarrow X \times E
\end{align*}  
be the projections. Generalising the notation $\calo_{*}^{G}(E_{\calf}, G, d_{G})$, we set
\[
\calo_{b}^{G}(E, X,d) = \C_{b}^{G}(X \times E \times [1, \infty); p^{-1}\cale_{cc} \cap q^{-1}\cale_{d}, r^{-1}\calf_{Gc})
\]

Let $(X_{n},d_{n})$ be a sequence of metric spaces equipped with free, proper and isometric $G$-actions. Then we define a subcategory 
\[
\calo^{G}_{b}(E, (X_{n}, d_{n})_{n \in \IN})
\]
of 
\[
\prod\limits_{n=1}^{\infty} \calo_{*}^{G}(E,X_{n},d_{n})
\]
as follows:  Objects are those of the product category $\prod\limits_{n=1}^{\infty} \calo_{*}^{G}(E,X_{n},d_{n})$. A morphism $f$ from $(M_{n})_{n \in \IN}$ to $(N_{n})_{n \in \IN}$ in the product category is in $\calo^{G}_{b}(E, (X_{n}, d_{n})_{n \in \IN})$ if the following conditions hold:
\begin{enumerate}
\item There is $R > 0$ such that each $f_{n}$ is $R$-controlled in the $X_{n}$-direction. Then we say that $f$ is $R$-controlled.
\item The sequence of norms $\left\Vert f_{n} \right\Vert$ is bounded.
\end{enumerate}
This category is a normed category under the norm 
\[
\left\Vert f \right\Vert = \sup\limits_{n \in \IN}  \left\Vert f_{n} \right\Vert
\] 
It is also closed under the involution, since if $f$ is $R$-controlled, so is $f^{*}$, and since $\left\Vert f_{n} \right\Vert = \left\Vert f_{n}^{*} \right\Vert$.  In this way, $\calo^{G}_{b}(E, (X_{n}, d_{n})_{n \in \IN})$ becomes a pre-$C^{*}$-category, and we let $\calo_{*}^{G}(E, (X_{n},d_{n})_{n \in \IN})$ be its completion.

\begin{rem}
Assume we have a $C^{*}$-algebra $A$ which is faithfully represented on a Hilbert space $\calh$. Then the algebraic infinite product  $\prod_{i=1}^{\infty} A$ acts on $\bigoplus_{i=1}^{\infty} \calh$ diagonally. However, this representation of $\prod_{i=1}^{\infty} A$ does not land in the bounded operators on $\oplus_{i=1}^{\infty} \calh$; a sequence $(a_{i})_{i \in \IN}$ is sent to a bounded operator if and only if $\left\Vert a_{i} \right\Vert$ is bounded. When imposing this additional condition, it is clear that the resulting subalgebra of $\prod_{i=1}^{\infty} A$ is a $C^{*}$-algebra again. The above is the categorical analogue of this. 
\end{rem}

\begin{rem}
\label{productcoarse}
There is another way to think about the categories 
\[
\calo^{G}_{b}(E, (X_{n}, d_{n})_{n \in \IN})
\]
Namely, consider the $G$-space $Y = \coprod_{n \in \IN} E \times [1, \infty) \times X_{n}$ and equip it with a coarse structure $(\cale, \calf)$ as follows: A set $F \subset Y$ is in $\calf$ if and only if for each $n$, the intersection $F \cap  E \times  [1, \infty) \times X_{n}$  is controlled in  $E \times  [1, \infty) \times X_{n}$, and a set $D \subset Y \times Y$ is in $\cale$ if and only if, whenever $(x,y) \in D$, we have $x, y \in  E \times  [1, \infty) \times X_{n}$ for the same $n$, i.e. morphisms are not allowed to connect different components, and such that the set
\[
\{(x,y)\in D \mid x \in  E \times [1, \infty) \times X_{n} \} \subset  (E \times  [1, \infty) \times X_{n}) \times ( E \times  [1, \infty) \times X_{n})
\]
is controlled in $E \times [1, \infty) \times X_{n}$ for each $n$ and additionally, such that there is an $r > 0$ (independent of $n$) such that, whenever $(x = (e,z,t),y = (e',z',t')) \in \cale$, we have $d(z,z')<r$. It is not hard to see that 
\[
\calo^{G}_{b}(E, (X_{n}, d_{n})_{n \in \IN}) =  \C_{b}^{G}(Y; \cale, \calf)
\]
and
\[
\calo^{G}_{*}(E, (X_{n}, d_{n})_{n \in \IN}) =  \C_{*}^{G}(Y; \cale, \calf)
\]
\end{rem}

\section{The diagram}

Recall the definition of a Farrell-Hsiang group.
 
\begin{definition}
Let $G$ be a group with a word metric $d_{G}$ and $\calf$ a family of subgroups of $G$.
We say that $G$ is a Farrell-Hsiang group with respect to $\calf$ if the following condition is satisfied: \\
There exists a natural number $N$ such that for each $n \in \IN$ there is a surjective group homomorphism $\alpha_{n}: G \rightarrow F_{n}$ with $F_{n}$ finite such that for each elementary subgroup $C \subset F_{n}$ and $H = \alpha_{n}^{-1}(C) \subset G$, there is a simplicial complex $E_{H}$ of dimension at most $N$ and a simplicial $H$-action with isotropy in $\calf$, and an $H$-equivariant map $f_{H}: G \rightarrow E_{H}$ such that
\[
d^{1}_{E_{H}}(f_{H}(g), f_{H}(h)) \leq \frac{1}{n}
\] 
for all $g,h \in G$ with $d_{G}(g,h) \leq n$, where $d^{1}_{E_{H}}$ is the $l^{1}$-metric on $E_{H}$ and $H$ acts by left multiplication on $G$.
\end{definition}
 
 \begin{rem}
 Strictly speaking, this definition is not good. We should define a Farrell-Hsiang group to be a group $G$, together with a collection of finite-index subgroups $H_{i}$ such that the map induced by induction
 \[
 \bigoplus_{i} \Sw^{u}(H_{i}) \rightarrow \Sw^{u}(G)
 \]
 is onto, and such that for each $H_{i}$ there is a simplicial complex $E_{H_{i}}$ as in the above definition. The groups $F_{n}$ are only used to get a hand on such a collection of subgroups. However, whenever we are in such a situation, i.e. have a collection of finite-index subgroups $H_{i}$ such that $\bigoplus_{i} \Sw^{u}(H_{i}) \rightarrow \Sw^{u}(G)$ is onto, there is a finite quotient $F$ of $G$ such that all $H_{i}$ are preimages of subgroups of $F$: If $1 \in \Sw^{u}(G)$ is in the image, it has to come from finitely many of the $H_{i}$, say $H_{1}, \dots, H_{n}$. Then $\bigcap_{i=1}^{n} H_{i}$ still has finite index in $G$ and hence has a finite-index subgroup $K$ which is normal in $G$. Now set $F = G/K$. So all potentially interesting induction results have to arise in finite groups, though not necessarily from the Brauer induction theorem.
 \end{rem}

Let $G$ be a Farrell-Hsiang group with respect to the family $\calf$. 
Let $\calh_{n}$ be the family of subgroups of $G$ obtained by pulling back the family of elementary subgroups of $F_{n}$.  We define pseudo-metric spaces $X_{n}, S_{n}$ as

\begin{align*}
X_{n} &= G \times \coprod_{H \in \calh_{n}} \ind_{H}^{G} E_{H} \\
S_{n} &=   \coprod_{H \in \calh_{n}} G \times G/H
\end{align*}

with metrics

\begin{align*}
d_{X_{n}}((g,x),(h,y)) &= d_{G}(g,h)+n \cdot d^{1} (x,y)\\
d_{S_{n}}((g,aH), (h,bK)) &= \begin{cases} 
d_{G}(g,h) \text{ if } K=H,aH = bK \\
\infty \text{ else }
\end{cases}
\end{align*}
where $d^{1}$ denotes the $l^{1}$-metric. Recall that the $l^{1}$-metric on the realization of an abstract simplicial complex $X$ is defined as follows: If $M = (m_{i})_{i \in I}$ is the set of $0$-simplices of $X$, let $x = \sum a_{i} m_{i}$ and $y = \sum b_{i} m_{i}$ be two points in 
the realization of $X$ and set
\[
d^{1}(x,y) = \sum\limits_{i \in I} \left\vert a_{i}-b_{i} \right\vert
\]
The reader should note that the metric on $X_{n}$ is basically scaled by a factor of $n$, i.e. it is quite hard to be controlled over $X_{n}$ since the metric is so large. The proof of the Baum-Connes conjecture, relative to the family $\calf$, for Farrell-Hsiang groups is now organized around the following diagram, which is our analogue of \cite[Diagram 4.1]{FH}:

\begin{equation}
\label{core}
\xymatrix{
& & \bigoplus_{n \in \IN} \calo^{G}_{*}(E,X_{n}, d_{X_{n}}) \ar[d]^{{\text{inc}}} \\
\calo^{G}_{*}(E, (S_{n}, d_{S_{n}})_{n \in \IN}) \ar[dd]_{P_{k}} \ar[rr]^{F} & & \calo_{*}^{G}(E,(X_{n},d_{X_{n}})_{n \in \IN}) \ar[dd]^{{Q_{k}}} \\
\\
\calo^{G}_{*}(E,G,d_{G}) \ar[rr]^{\id} & &  \calo^{G}_{*}(E,G,d_{G})
}
\end{equation}

Some explanations are in order. The functor $P_{k}$ is obtained by projecting $\calo^{G}_{*}(E, (S_{n}, d_{S_{n}})_{n \in \IN})$ onto its $k$-th factor $\calo^{G}_{*}(E,S_{k}, d_{S_{k}})$ and then using the projection $S_{k} \rightarrow G$, and similarly for $Q_{k}$. The functor $F$ will be constructed in \ref{horizontal} such that the diagram is commutative; here the contracting properties of the maps $G \rightarrow E_{H}$ occuring in the definition of a Farrell-Hsiang group are crucial because the metrics on the right-hand scale are extremely large compared to those on the left-hand side, so to get a well-defined map, we need to contract distances. We will see  
\begin{itemize}
\item[(i)] in \ref{stability} that $\text{inc}$ induces an isomorphism in $K$-theory, and
\item[(ii)] in \ref{surjectivity}. that for each $a \in K_{*}\calo^{G}_{*}(E,G,d_{G})$, there is an $a' \in K_{*} \calo^{G}_{*}(E, (S_{n}, d_{S_{n}})_{n \in \IN})$ such that for all $k$, we have $(P_{k})_{*}(a') = a$. Note that this is stronger than $(P_{k})_{*}$ being onto: $a'$ is independent of $k$.
\end{itemize}
Assuming these facts for the moment, we can prove the Baum-Connes conjecture for $G$:

\begin{thm}
Assuming the above two facts, the topological $K$-theory of  \\ $\calo^{G}_{*}(E,G,d_{G})$ is trivial.
\end{thm}

\begin{proof}
Pick $a \in K_{*}\calo^{G}_{*}(E,G,d_{G})$. We find $a' \in K_{*} \calo^{G}_{*}(E, (S_{n}, d_{S_{n}})_{n \in \IN})$ such that $(P_{k})_{*}(a') = a$ for all $k$. Since $K_{*}(\text{inc})$ is an isomorphism, we find $b \in K_{*} (\bigoplus_{n \in \IN} \calo^{G}_{*}(E,X_{n}, d_{X_{n}}))$ with $\text{inc}_{*}(b) = F_{*}(a')$. There is a $k \in \IN$ such that $b$ comes from $K_{*}( \bigoplus_{n=1}^{k} \calo^{G}_{*}(E,X_{n}, d_{X_{n}}))$ since any element in the topological $K$-theory is represented by an actual morphism in the category. But now $Q_{k+1} \circ \text{inc}_{*}(b) = 0$, hence we have $Q_{k+1}(F_{*}(a')) = 0$. By commutativity, this implies $P_{k+1}(a') = 0$. Since $a = P_{k+1}(a')$, it follows $a = 0$, which is what we wanted to prove.
\end{proof}

\noindent The next two chapters are devoted to the proofs of the facts (i) and (ii).

\pagebreak
 
\chapter{The transfer}

In this section, we construct a single section which is a right inverse for all the maps
\[
K_{*} P_{k}: K_{*} \calo^{G}_{*}(E, (S_{n}, d_{S_{n}})_{n \in \IN}) \rightarrow K_{* }\calo^{G}_{*}(E,G,d_{G})
\]
from section $6$. The construction is the topological $K$-theory analogue of the construction in \cite{FH}. 

\section{Swan group actions on controlled categories}
As a warm-up, we want to define a functor
\[
\tr \colon \Mod_{(\IC,G)} \times \calo_{*}^{G}(E,G,d_{G}) \rightarrow \calo_{*}^{G}(E,G,d_{G})
\]
which tries to mimic the Swan group action on $K_*^{\alg}(AG)$. This will be the functor obtained by completion of the functor
\[
\tr \colon \Mod_{(\IC,G)} \times \calo_{b}^{G}(E,G,d_{G}) \rightarrow  \calo_{b}^{G}(E,G,d_{G})
\]
defined as follows. Let $M$  be an object of $\calo_{b}^{G}(E,G,d_{G})$. Let $S = (\IC^{n}, \rho)$ and $T = (\IC^{k}, \tau)$ be objects of $\Mod_{(\IC,G)}$ and $f \colon S \rightarrow T$. On objects, $\tr$ is given by
\[
(\tr(S,M))_{z} = \IC^{n} \otimes M_{z}
\]
and on morphisms, we define
\begin{equation}
\label{defoftransfer}
(\tr(f, \phi))_{z,z'} = (f \circ \rho(g^{-1}g')) \otimes \phi_{z,z'}
\end{equation}
for $z = (g,e,t)$, $z' = (g',e',t')$. Recall that $\phi_{z,z'}$ is the component of $\phi$ mapping from $z'$ to $z$.  \\
Let us say a few words on the motivation for this formula. Since $T(M)$ is a right $A \rtimes G$-module, we have to view a representation $(S, \rho)$ as a right $G$-module via
\[
vg = \rho(g^{-1})v
\]  
Recall that in \ref{action}, we defined a biexact functor
\[
\Mod_{(\IC,G)} \times \calp(AG) \rightarrow \calp(AG)
\]
via the tensor product with the diagonal $G$-action from the right. It turned out that there is an isomorphism, depending on the choice of a basis,
\[
S \otimes_{\Delta} (AG)^{k} \rightarrow S \otimes_{r} (AG)^{k}
\]
So when we restrict to a suitable category of $AG$-modules, say the category with objects only $(AG)^{n}$, which come with a preferred basis, we could define a biexact pairing via
\[
(S,M) \rightarrow S \otimes_{r} M
\]
on objects and push the $G$-action on $S$ entirely into the morphisms: For $f \colon S \rightarrow T$ a morphism in $\Mod_{(\IC,G)}$ and $\phi \colon (AG)^{n} \rightarrow (AG)^{k}$, we define a morphism $S \otimes_{r} (AG)^{n} \rightarrow T \otimes_{r} (AG)^{k}$ as the unique morphism which makes the diagram
\[
\xymatrix{
S \otimes_{r} (AG)^{n} \ar[d]^{\cong} \ar[r] & T \otimes_{r} (AG)^{k} \ar[d]^{\cong} \\
S \otimes_{\Delta} (AG)^{n} \ar[r]^{f \otimes \phi}  & T \otimes_{\Delta} (AG)^{k}
}
\]
commutative. In our controlled setup, we can view an object $M$ of $ \calo_{*}^{G}(E,G,d_{G})$ -- or rather, its total object -- as a  free right $A \rtimes G$-module with a preferred basis, thanks to the $G$-factor in the coarse space.Then we form for $\phi: M \rightarrow N$, $f: S \rightarrow T$ the map $S \otimes_{r}T(M) \rightarrow T \otimes_{r} T(N)$ as just described, where we have to view $S$ and $T$ as right $G$-modules via $vg = \rho(g^{-1})v$. The objects $S \otimes_{r} T(M)$ and $T \otimes_{r} T(N)$ make sense as objects of $ \calo_{*}^{G}(E,G,d_{G})$, basically just as direct sums of copies of $M$ respectively $N$, and we can write out the components of the map $S \otimes_{r}T(M) \rightarrow T \otimes_{r} T(N)$. The result is the above formula \ref{defoftransfer} for $\tr$. \\
More precisely,  the map
\begin{align*}
S \otimes_{r} T(M) &\rightarrow S \otimes_{\Delta} T(M) \\
v \otimes e_{(g',e',t'),n} &\mapsto vg'^{-1} \otimes e_{(g',e',t'),n}
\end{align*}
is an isomorphism with inverse
\[
v \otimes e_{(g',e',t'),n} \mapsto vg' \otimes e_{(g',e',t'),n}
\]
When checking these maps are indeed $G$-equivariant, remember that $G$ acts by $e_{(g',e',t'),n} \cdot g = e_{(g^{-1}g',g^{-1}e',t'),n}$ and not by right multiplication on $G$. This is also the reason the formulas differ from the ones in the proof of \ref{action} by some inverses.\\
When we start with $v \otimes e_{(g',e',t'),n} \in S \otimes_{r} T(M)$, it is mapped to $vg'^{-1} \otimes e_{(g',e',t'),n} \in S \otimes_{\Delta} T(M)$, which in turn is mapped to
\[
\sum\limits_{(g,e,t)} f(v)g'^{-1} \otimes \phi_{(g,e,t),(g',e',t')} e_{(g',e',t'),n} \in T \otimes_{\Delta} T(N)
\]
under $f \otimes \phi$. Mapping back to $T \otimes_{r} T(N)$ and projecting down onto the part of $T \otimes_{r} T(N)$ supported over a single point $(g,e,t)$ - which just means forgetting all but one summands in the sum - we obtain
\[
f(v)g'^{-1}g \otimes \phi_{(g,e,t),(g',e',t')}(e_{(g',e',t'),n})
\]
Since $f(v)g'^{-1}g = f(\rho(g^{-1}g')(v))$, this is precisely 
\[
\tr(f,\phi)_{(g,e,t),(g',e',t')}(v \otimes e_{(g',e',t'),n})
\]
as claimed.

The following proposition shows that $\tr(f, \phi)$ is indeed a morphism of $\calo_{b}^{G}(E,G,d_{G})$.

\begin{prop}
The functor $\tr$ is well-defined, continuous, compatible with involutions in the sense that $\tr(S,-)$ is a $*$-functor, biexact and extends to a biexact functor 
\[
\tr \colon \Mod_{(\IC,G)} \times \calo_{*}^{G}(E,G,d_{G}) \rightarrow \calo_{*}^{G}(E,G,d_{G})
\]
\end{prop}

\begin{proof}
The control conditions are clearly satisfied since $\supp(\tr(f,\phi)) = \supp(\phi)$. Biexactness is clear since both categories carry the split exact structure. Alternatively, one can
argue as follows. We have to see that for a short exact sequence
\[
0 \rightarrow R \rightarrow S \rightarrow T \rightarrow 0
\]
in $\Mod_{(\IC,G)}$ and for $M \in \calo_{*}^{G}(E,G,d_{G})$, the sequence
\[
0 \rightarrow \tr(R,M) \rightarrow \tr(S,M) \rightarrow \tr(T,M) \rightarrow 0
\]
is split exact. For this, pick a section $s \colon T \rightarrow S$ of the projection $p:S \rightarrow T$ as $\IC$-modules, i.e. $s$ is not necessarily $G$-equivariant. Then the map $s': \tr(T,M) \rightarrow \tr(S,M)$ with components
\[
s'_{z,z'} = \begin{cases}  s \otimes \id_{M_{z}} \text{ if } z = z' \\ 0 \text{ else }  \end{cases}
\]
is a section for $\tr(p,M)$. Recall that $\tr$ is essentially rewriting $- \otimes_{\Delta} - $ in terms of $- \otimes_{r} - $, and it is clear that $s$, even if it is not equivariant, is a good candidate for a section in the $ - \otimes_{r} -$-picture since the $G$-actions on the $(\IC,G)$-modules do not play a role anyway on objects. \\
Since $T(s') = s \otimes \Id \colon T \otimes T(M) \rightarrow S \otimes T(M)$, $T(s')$ is clearly bounded and adjointable, and so $s'$ is indeed an allowed morphism in the category $\calo_{*}^{G}(E,G,d_{G})$.\\
It is straightforward to check that for a fixed object $S$  of $\Mod_{(\IC,G)}$, the functor $\tr(S,-)$ is a $^{*}$-functor. \\
To check that $\tr(f, \phi)$ is again bounded and adjointable, we have to understand what happens on the total objects.  
Let $M$ be an object of $\calo_{b}^{G}(E,G,d_{G})$. Then we can view $T(M) = \bigoplus_{x \in G \times E \times [1,\infty)} M_x$ as a right $A \rtimes G$-module as in \ref{twistedgroup}.
Given an object $S = (\IC^{n}, \rho)$ of $\Mod_{(\IC, G)}$, we view $S$ as a right $G$-module via $vg = \rho(g^{-1})(v)$ and we can consider $T(\tr(S,M)) \cong \IC^{n} \otimes_{\IC} T(M)$ as a right $A \rtimes G$-module in two different ways: we can let $G$ act diagonally or only on the right-hand side. As usual, we write $\IC^{n} \otimes_{\Delta} T(M)$ for the tensor product with the diagonal action and $\IC^{n} \otimes_{r} T(M)$ with the $G$-action only on $T(M)$.  These two modules are isomorphic as right $A \rtimes G$-Modules via the isomorphism
\begin{align*}
U: \IC^{n} \otimes_{r} T(M) &\rightarrow \IC^{n} \otimes_{\Delta} T(M) \\
v \otimes e_{(g,e,t),k} &\mapsto vg^{-1} \otimes e_{(g,e,t),k}
\end{align*}
and extending $A$-linearly, where we use the notation from \ref{controlcat} (viii). 
Since $\rho$ is a unitary representation, $U$ is a unitary operator: Pick an orthonormal basis $v_{i}$ of $\IC^{n}$. Then 
\[
(v_{i} \otimes e_{(g,e,t),k})_{(g,e,t) \in G \times E \times [1, \infty) ,k,i}
\]
 is an orthonormal basis for $\IC^{n} \otimes T(M)$ as a right Hilbert $A$-module, and $U$ sends this basis to another orthonormal basis since $\rho(g)$ sends $v_{i}$ to an orthonormal basis. This is the reason we had to impose the unitary condition on our representations.
 \\ Now for morphisms $f: S = (\IC^{n}, \rho) \rightarrow T = (\IC^{k}, \tau)$ in $\Mod_{(\IC,G)}$ and $\phi: M \rightarrow N$ in $\calo_{b}^{G}(E,G,d_{G})$, we consider the diagram
\[
\xymatrix{
\IC^{n} \otimes_{r} T(M) \ar[rr]^{T(\tr(f,\phi))} \ar[dd]_{U} & & \IC^{k} \otimes_{r} T(N) \ar[dd]^{U} \\
\\
\IC^{n} \otimes_{\Delta} T(M) \ar[rr]^{f \otimes \phi}  & & \IC^{k} \otimes_{\Delta} T(N)\\
}
\]
The definition of the transfer functor is precisely such that this diagram commutes. Since both vertical maps are unitary operators, the norms of the top and bottom horizontal operators are equal. For the norm of the lower horizontal map, we have $\left\Vert f \otimes \phi \right\Vert \leq  \left\Vert f \right\Vert \left\Vert \phi \right\Vert$. Hence the map
\[
\tr(f,-) \colon \Hom_{\calo_{b}^{G}(E,G,d_{G})}(M,N) \rightarrow \Hom_{\calo_{b}^{G}(E,G,d_{G})}(\tr(T,M), \tr(S,M))
\]
indeed takes values in $ \Hom_{\calo_{b}^{G}(E,G,d_{G})}(\tr(T,M), \tr(S,M))$ and is boun\-ded by $\left\Vert f \right\Vert$. Since $f \otimes \phi$ is adjointable, so is $T(\tr(f,\phi))$. So we can complete this map to obtain
\[
\tr(f,-) \colon  \Hom_{\calo_{*}^{G}(E,G,d_{G})}(M,N) \rightarrow \Hom_{\calo_{*}^{G}(E,G,d_{G})}(\tr(T,M), \tr(S,M))
\]
In this way, we obtain a functor 
\[
\tr \colon \Mod_{(\IC,G)} \times \calo_{*}^{G}(E,G,d_{G}) \rightarrow \calo_{*}^{G}(E,G,d_{G})
\]
as desired.  Exactness in the second variable is automatic since the exact structure is the split exact one and $\tr(S,-)$ is still an additive functor as the completion of an additive functor.  Exactness in the first variable is proven as above. \fxnote{Check}
 \end{proof}
 
 \begin{rem}
 \label{choice}
The definition here relies on the fact that our space has a factor $G$, so we have a canonical way to identify each orbit of the $G$-action with $G$. In general, if we tried to get the same definition working with some arbitrary free $G$-space $X$, we would have to choose isomorphisms of each orbit with $G$, which we cannot do in a coherent way in general. Basically we exploit that we have a canonically chosen basis of $T(M)$ as $A \rtimes G$-module for any object $M$. 
 \end{rem}

Now let $H \subset G$ be a finite-index subgroup of $G$ -- for example one of the subgroups pulled back from a subgroup of the finite group $F_{n}$ under the map $G \rightarrow F_{n}$ appearing in the definition of a Farrell-Hsiang group. We obtain a functor
 \[
 \tr \circ \ind_{H}^{G} \colon \Mod_{(\IC,H)} \times  \calo_{*}^{G}(E,G,d_{G}) \rightarrow  \calo_{*}^{G}(E,G,d_{G})
 \]
 by first inducing up from $\Mod_{(\IC,H)}$ to $\Mod_{(\IC,G)}$ and then applying $\tr$.  Now we want to lift this functor to a functor
 \[
 \tr_{H} \colon \Mod_{(\IC,H)} \times  \calo_{*}^{G}(E,G,d_{G}) \rightarrow  \calo_{*}^{G}(E,G \times G/H,d)
 \]
 where $G \times G/H$ is equipped with the diagonal $G$-action and the quasimetric $d$ is given by
 \[
 d((g,aH), (h,bH)) = \begin{cases} d_{G}(g,h) \text { if } aH = bH \\
 \infty \text{ else }
 \end{cases}
 \]
 More precisely, let $p \colon G \times G/H \rightarrow G$ be the projection. We want to define $\tr_{H}$ such that the following diagram commutes up to equivalence:
 \[
 \xymatrix{
 & & &  \calo_{*}^{G}(E,G \times G/H,d) \ar[d]^{p} \\
 \Mod_{(\IC,H)} \times  \calo_{*}^{G}(E,G,d_{G}) \ar[rrru]^{\tr_{H}} \ar[rrr]^{ \tr \circ \ind_{H}^{G}} & & &  \calo_{*}^{G}(E,G,d_{G}) \\
 }
 \]
\noindent To write down $\tr_{H}$, it is more convenient to consider $G \times_{l} G/H$ where $G$ only acts on the left factor. This is $G$-homeomorphic to $G \times_{\Delta} G/H$ via the map $G \times_{\Delta} G/H \rightarrow G \times_{l} G/H, (g,aH) \mapsto (g,g^{-1}aH)$. Under this isomorphism, the above metric becomes the metric
\[
 d_{l}((g,aH), (h,bH)) = \begin{cases} d_{G}(g,h) \text { if } gaH = hbH \\
 \infty \text{ else }
 \end{cases}
\]
The motivation for the definition of $\tr_{H}$ is as follows: If $T \in \Mod_{(\IC,H)}$ and $P$ is a $\IC G$-Module, we have the Frobenius reciprocity formula
\[
(\ind_{H}^{G} T) \otimes_{\Delta} P \cong \ind_{H}^{G}(T \otimes_{\Delta} \res_{H}^{G} P)
\]
The left-hand side is analogous to $\tr \circ \ind_{H}^{G}$, and we will model the definition of $\tr_{H}$ on the right-hand side.
The formula actually nearly makes sense in our controlled context: We can restrict $G \times E \times [1,\infty)$ to an $H$-space, which induces a functor 
\[
\calo_{*}^{G}(E,G,d_{G}) \rightarrow  \calo_{*}^{H}(E,G,d_{G})
\]
then tensor with $T$, then induce up from $H$ to $G$. However, in the last step of inducing up, we also have to induce up the underlying space of our controlled categories, whence the additional factor $G/H$. Also,  when we want to let $T$ act on $\calo_{*}^{H}(E,G,d_{G})$, we have to explicitly identify the $H$-orbits with $H$, compare \ref{choice}. The easiest way to do this is to identify $G$ with $H \times G/H$ as $H$-spaces. So let $g_{1}, \dots g_{l}$ be a complete set of representatives of $G/H$ and define an $H$-isomorphism
\begin{align*}
\phi:  G &\rightarrow H \times G/H \\
g &\mapsto (gg_{i},g_{i}H)
\end{align*}
 where $i$ is the unique $i$ such that $g^{-1}H = g_{i}H$.
Breaking down what the constructions in the Frobenius reciprocity formula mean on total objects, the design criterion for $\tr_{H}$ is now as follows: For objects $M,N$ of  $\calo_{*}^{G}(E,G,d_{G})$, $\phi: M \rightarrow N$ and $S,T \in \Mod_{(\IC,H)}, f: S \rightarrow T$, we want the following diagram to commute:
\begin{equation}
\label{boundeddiagram}
\xymatrix{
T(\tr_{H}(T,M)) = (\ind_{H}^{G} S) \otimes_{r} T(M) \ar[dd]^{\cong} \ar[rr]^{T(\tr_{H}(f,\phi))} & & T(\tr_{H}(S,N)) = (\ind_{H}^{G} T) \otimes_{r} T(N) \ar[dd]^{\cong} \\
\\
(\ind_{H}^{G} S) \otimes_{\Delta} T(M) \ar[dd]^{\cong} \ar[rr]^{(\ind_{H}^{G} f) \otimes \phi} & &  (\ind_{H}^{G} T) \otimes_{\Delta} T(N) \ar[dd]^{\cong}  \\
\\
\ind_{H}^{G}(S \otimes_{\Delta} \res_{H}^{G} T(M)) \ar[rr]^{\ind_{H}^{G}(f \otimes \phi)} & & \ind_{H}^{G}(T \otimes_{\Delta} \res_{H}^{G} T(N))
}
\\
\end{equation}
Recall that the choice of the $g_{i}$ enters into the coordinate description of $\ind_{H}^{G} S$ and hence into this diagram.
This leads to the following definition: On objects, we define $\tr_{H}(T,M)$ for $T = (\IC^{n}, \rho)$ as
\[
\tr_{H}(T,M)_{(g,aH,e,t)} = \IC^{n} \otimes M_{(g,e,t)}
\]
For a morphism $\phi: M \rightarrow N$ in  $\calo_{b}^{G}(E,G,d_{G})$ and $f: T \rightarrow S$, we define
\[
\tr_{H}(f, \phi)_{(g,g_{j}H,e,t), (g',g_{i}H,e',t')} = \begin{cases} f \circ \rho(g_{j}^{-1}g^{-1}g'g_{i}) \otimes \phi_{(g,e,t),(g',e',t')}  \\  \hspace{2.7 cm} \text{ if }  g_{j}^{-1}g^{-1}g'g_{i} \in H \\
0 \text{ \hspace{2.7 cm} else } \end{cases}
\]

\begin{rem}
The reader should note that the corresponding definition in the published version of \cite{FH} is not correct and correspondingly looks very different. This is fixed in the version available on the homepage of the authors, where the formula still looks slightly different due to the fact that we use $G \times_l G/H$ and not $G \times_{\Delta} G/H$.
\end{rem}

\begin{lem}
\label{surjectivity}
The functor $\tr_{H}$ is well-defined, continuous, compatible with the involutions and extends to the completions, i.e. we obtain a functor
 \[
 \tr_{H}: \Mod_{(\IC,H)} \otimes  \calo_{*}^{G}(E,G,d_{G}) \rightarrow  \calo_{*}^{G}(E,G \times_{l} G/H,d_l)
 \] 
Furthermore, if we let $P: G \times_{l} G/H \rightarrow G$ be the canonical projection, the functors $\tr \circ \ind_H^G$ and $P \circ \tr_{H}$ are equivalent.
\end{lem}

\begin{proof}
 We first have to check that the functor respects the control conditions in the right-hand category.
But by definition
\[
\tr_{H}(f, \phi)_{(g,g_{j}H,e,t), (g',g_{i}H,e',t')}
\] 
is nonzero only if $gg_{j}H = g'g_{i}H$, i.e. only if $d_{l}(g,g_{j}H), (g',g_{i}H)) < \infty$ in $G \times_{l} G/H$. If $\phi$ is $k$-controlled in $G$-direction, it is then clear that $\tr_{H}(f, \phi)$ is $k$-controlled in $G \times_{l} G/H$-direction. \\
The transfer is now defined precisely such that the above diagram \ref{boundeddiagram} is commutative. Since the vertical maps are unitary, we find that $\tr_{H}(f,-)$ is bounded by $\left\Vert \ind_{H}^{G} f \right\Vert = \left\Vert f \right\Vert$ and hence extends to the completions. \\
For the second claim, pick $S = (\IC^{k}, \rho) \in \Mod_{(\IC,H)}$ and an object $M$ of $\calo_{*}^{G}(E,G,d_{G})$. Then $(\tr \circ \ind_{H}^{G})(S,M)$ is the object of $\calo_{*}^{G}(E,G,d_{G})$ which is given by
\[
(\tr \circ \ind_{H}^{G})(S,M)_{(g,e,t)} = (\IC^{k})^{n} \otimes M_{(g,e,t)} \cong \oplus_{i=1}^{n} \IC^{k} \otimes M_{(g,e,t)}
\]
and $P(\tr_{H}(S,M))$ is given by
\[
P(\tr_{H}(S,M))_{(g,e,t)} = \oplus_{G/H} \IC^k \otimes M_{(g,e,t)}
\]
since $P$ projects away from $G/H$. These two objects are isomorphic by identifying the $i$-th summand of $\oplus_{i=1}^{n} \IC^{k} \otimes M_{(g,e,t)}$ with the summand given by $g_{i}H$ in $ \oplus_{G/H} \IC^{k} \otimes M_{(g,e,t)}$. We have to check that this isomorphism is natural. However, given $f: S \rightarrow T$ and $\phi: M \rightarrow N$, the $(i,j)$-components of $P(\tr_{H}(f,\phi))_{((g,e,t),(g',e',t'))}$ with respect to the above direct sum decompositions 
\[
P(\tr_{H}(S,M))_{(g,e,t)} = \oplus_{G/H} \IC^{k} \otimes M_{(g,e,t)}
\]
and 
\[
P(\tr_{H}(T,N))_{(g',e',t')} = \oplus_{G/H} \IC^{k'} \otimes N_{(g',e',t')}
\]
is given by
\[
f \circ \rho(g_{j}^{-1}g^{-1}g'g_{i}) \otimes \phi_{(g,e,t),(g',e',t')}
\]
as long as $g_{j}^{-1}g^{-1}g'g_{i} \in H$, and is zero else. The components of $(\tr \circ \ind_{H}^{G})(f,\phi)_{(g,e,t),(g',e',t')}$ with respect to the sum decomposition are the same. This finishes the proof.
\end{proof}

\section{The transfer}

From now on, the argument proceeds as in \cite{FH}. We will construct a map $K_{*}(\calo^{G}_{*}(E,G,d_{G})) \rightarrow  K_{*}(\calo^{G}_{*}(E, (S_{n}, d_{S_{n}})_{n \in \IN}))$ which, for each $k$, gives a $K$-theoretic section to 
\[
P_{k} \colon \calo^{G}_{*}(E, (S_{n}, d_{S_{n}})_{n \in \IN}) \rightarrow \calo^{G}_{*}(E,G,d_{G})
\]
To motivate the following constructions, note that splitting each functor $\calo^{G}_{*}(E,S_{n}, d_{S_{n}}) \rightarrow \calo^{G}_{*}(E,G,d_{G})$ individually nearly gives the desired result: We could try to define a functor 
\[
\calo^{G}_{*}(E,G,d_{G}) \rightarrow \calo^{G}_{*}(E, (S_{n}, d_{S_{n}})_{n \in \IN})
\]
 as the functor induced by the individual sections for each $n$, using that the right-hand category is nearly a product. The situation is not quite that easy, since the category is not a product on the nose and we cannot split the projections $\calo^{G}_{*}(E,S_{n}, d_{S_{n}}) \rightarrow \calo^{G}_{*}(E,G,d_{G})$ on the nose by a functor in the other direction; instead, we construct two functors 
\[
F_{n}^{\pm}: \calo^{G}_{*}(E,G,d_{G}) \rightarrow \calo^{G}_{*}(E,S_{n}, d_{S_{n}})
\]
in the other direction such that $K_{*}(F_{n}^{+})-K_{*}(F_{n}^{-})$ splits the map on $K$-theory induced by the projection; but $F_{n}^{+}-F_{n}^{-}$ does not make sense as a functor. Then we have to be careful with the control conditions to assemble the $F_{n}^{+}$ respectively $F_{n}^{-}$ to two functors
\[
F^{\pm}: \calo^{G}_{*}(E,G,d_{G}) \rightarrow \calo^{G}_{*}(E, (S_{n}, d_{S_{n}})_{n \in \IN})
\]
Finally, $K_*(F^+)-K_*(F^-)$ is the desired splitting.
\\
Let $p_{n} \colon S_{n} = \coprod\limits_{H \in \calh_{n}} G \times G/H \rightarrow G$ be the map which is the projection on each summand. It induces a functor 
\[
p_{n}: \calo_{*}^{G}(E,S_{n}, d_{S_{n}}) \rightarrow \calo_{*}^{G}(E,G,d_{G})
\]

\begin{prop}
\label{pm}
For each $n$, there are $C^{*}$-functors 
\[
F_{n}^{+}, F_{n}^{-} \colon \calo^{G}_{*}(E,G,d_{G}) \rightarrow \calo_{*}^{G}(E,S_{n}, d_{S_{n}})
\]
such that $K_{*}(p_{n} \circ F_{n}^{+})-K_{*}(p_{n} \circ F_{n}^{-})$ is the identity. The functors $F_{n}^{\pm}$ restrict to functors $F_{n}^{\pm} \colon \calo^{G}_{b}(E,G,d_{G}) \rightarrow \calo_{b}^{G}(E,S_{n}, d_{S_{n}})$, and these functors have the property that they send $R$-controlled morphisms to $R$-controlled morphisms. 
\end{prop}

\begin{proof}
We now consider $G \times G/H$ for $H \in \calh_{n}$   with the diagonal $G$-action again. By the Brauer induction theorem \ref{brauer} and since the family $\calh_{n}$ is the pullback of the family of elementary subgroups of $F_{n}$, we can find elements $T_{H} \in \Sw(\IC,H)$ such that 
\[
1_{\Sw^u(\IC,G)} = \sum\limits_{H \in \calh_{n}} \ind_{H}^{G} T_{H}
\]
 For each $H \in \calh_n$, we can find $T^{+}_{H}$ and $T^{-}_{H}$ in $\Mod_{(\IC,H)}$ such that $T_{H} = T^{+}_{H}-T^{-}_{H}$ and we have an inclusion map $G \times G/H \rightarrow S_{n}$ inducing a map
\[
\phi_{H} \colon \calo_{*}^{G}(E, G \times G/H, d) \rightarrow \calo_{*}^{G}(E, S_{n}, d_{S_{n}})
\]
Now let $F^{+}_{H}\colon \calo_{*}^{G}(E,G,d_{G}) \rightarrow  \calo_{*}^{G}(E, S_{n}, d_{S_{n}})$ be the composition $\phi_{H} \circ \tr_{H}(T^{+}_{H},-)$ and let $F^{-}_{H}$ be the composition 
$\phi_{H} \circ \tr_{H}(T^{-}_{H},-)$. Now set
\[
F^{\pm}_{n}= \bigoplus_{H \in \calh_{n}} F_{H}^{\pm}
\]
As a direct sum of $C^{*}$-functors, these are $C^{*}$-functors. Now we compute for $a \in K_{*}(\calo_{*}^{G}(E,G,d_{G}))$
\begin{align*}
&(K_{*}(p_{n} \circ F_{n}^{+}) - K_{*}(p_{n} \circ F_{n}^{-}))(a) \\
&= \sum\limits_{H \in \calh_{n}} K_{*}(p_{n} \circ F^{+}_{H})(a) - K_{*}(p_{n} \circ F^{-}_{H})(a) \\
&= \sum\limits_{H \in \calh_{n}} K_{*}(p_{n} \circ \phi_{H} \circ \tr_{H}(T_{H}^{+}))(a) - K_{*}(p_{n} \circ \phi_{H} \circ \tr_{H}(T_{H}^{-}))(a) \\
&= \sum\limits_{H \in \calh_{n}} K_{*}(\tr(\ind_{H}^{G} T_{H}^{+}, -))(a) - K_{*}(\tr(\ind_{H}^{G}(T_{H}^{-}, -)))(a) \\
&= \sum\limits_{H \in \calh_{n}} (\ind_{H}^{G} T_{H}) \otimes a \\
&= 1_{\Sw^{u}(\IC,G)} \otimes a = a \\
\end{align*}
where the $\otimes$ in the fourth and fifth line stands for the pairing 
\[
\Sw^{u}(\IC,G) \times K_{*}(\calo_{*}^{G}(E,G,d_{G})) \rightarrow K_{*}(\calo_{*}^{G}(E,G,d_{G}))
\]
induced by the biexact functor $\tr$ and we used that $p_{n} \circ \phi_{H}$ is the projection $G \times G/H \rightarrow G$. The last claim of the proposition involving the $R$-controlled condition is not hard to check: All involved functors are defined as completions of functors between the bounded categories, so also $F_{n}^{\pm}$ are, and the control condition follows since the transfer does not change control in the $G \times G/H$-direction.
\end{proof}

\begin{thm}
\label{surjective}
For each $a \in K_{*}(\calo_{*}^{G}(E,G,d_{G}))$, there is 
\[
b \in K_{*} \calo_{*}^{G}(E,(S_{n}, d_{S_{n}})_{n \in \IN})
\]
such that for all $k$, we have $(K_{*}(P_{k}))(b) = a$.
\end{thm}

\begin{proof}
Recall the functors $F_{n}^{\pm}: \calo^{G}_{b}(E,G,d_{G}) \rightarrow \calo_{b}^{G}(E,S_{n}, d_{S_{n}})$ from \ref{pm}. Since these are restrictions of $C^{*}$-functors, they all have norm at most $1$. Because of this and the fact that they all send $R$-controlled morphisms to $R$-controlled morphisms, the product functors
\[
\prod\limits_{n} F_{n}^{\pm}: \calo^{G}_{b}(E,G,d_{G}) \rightarrow \prod\limits_n \calo_{b}^{G}(E,S_{n}, d_{S_{n}})
\]
give rise to functors
\[
F^{\pm}: \calo_{b}^{G}(E,G,d_{G}) \rightarrow \calo_{b}^{G}(E,(S_{n}, d_{S_{n}})_{n \in \IN})
\]
These are continuous with norm at most $1$, hence extend to the completions, yielding functors
\[
F^{\pm}: \calo_{*}^{G}(E,G,d_{G}) \rightarrow \calo_{*}^{G}(E,(S_{n}, d_{S_{n}})_{n \in \IN})
\]
Of course, $F^{\pm}$ is the same functor we would have obtained by the restriction of
\[
\prod F_{n}^{\pm}: \calo_{*}^{G}(E,G,d_{G}) \rightarrow \prod \calo_{*}^{G}(E,(S_{n}, d_{S_{n}}))
\]
but we needed the roundabout argument to ensure that it actually takes values in $\calo_{*}^{G}(E,(S_{n}, d_{S_{n}})_{n \in \IN})$. \\
Inspection of the formulas now immediately gives $P_{k} \circ F^{\pm} = p_{k} \circ F_{k}^{\pm}$. Hence we have by \ref{pm}
\[
K_{*}(P_{k})(K_{*}(F^{+})-K_{*}(F^{-}))(a) = a
\]
for all $a \in K_{*} \calo_{*}^{G}(E,G,d_{G})$. Now set $b= K_{*}(F^{+})(a)-K_{*}(F^{-}(a))$. 
\end{proof}

\section{The horizontal map}

Now we will define the horizontal map in our core diagram; compare \cite[7.1]{FH}. We define a map $f_{n}: S_{n} \rightarrow X_{n}$ as follows. First consider the map $f_{H}: G \rightarrow E_{H}$. Inducing up, we obtain a map
\begin{align*}
\ind_{H}^{G} f_{H}: G \times G/H \cong \ind_{H}^{G} G &\rightarrow \ind_{H}^{G} E_{H} \\
(a, gH) \mapsto (g, f_{H}(g^{-1}a))
\end{align*}
and then set
\[
f_{n}(a,gH) = (a, \ind_{H}^{G} f_{H}(a, gH))
\]

\begin{thm}
\label{horizontal}
The sequence of maps $f_{n}$ fit together to induce a functor
\[
F: \calo^{G}_{*}(E, (S_{n}, d_{S_{n}})_{n \in \IN}) \rightarrow \calo^{G}_{*}(E,(X_{n}, d_{X_{n}})_{n \in \IN})
\]
such that $Q_{k} \circ F = P_{k}$ for all $k$, i.e. the core diagram \ref{core} commutes. 
\end{thm}

\begin{proof}
\fxnote{FH could be better explained}
The proof is virtually the same as \cite[7.1]{FH}. Again, it is enough to argue that we get a continuous functor on the uncompleted categories
\[
F: \calo^{G}_{b}(E, (S_{n}, d_{S_{n}})_{n \in \IN}) \rightarrow \calo^{G}_{b}(E,(X_{n}, d_{X_{n}})_{n \in \IN})
\]
For this, we only have to check the control conditions; continuity is automatic. We have to see that for each $r > 0$ there is $R > 0$, independent of $n$, such that for all $s,s' \in S_{n}$, we have
\[
d_{S_{n}}(s,s') < r \Longrightarrow d_{X_{n}}(f_{n}(s), f_{n}(s')) < R
\]
First we claim that, as long as $n > r$, $R = r+1$ does the job. Let $s = (a, gH)$; then if $d(s,s') < r$, we must have $s' = (a', gH)$. Also, $d_{G}(a,a') \leq d_{S_{n}}((a,gH), a',gH)) < r$ and by $G$-invariance of $d_{G}$ also $d_{G}(g^{-1}a, g^{-1}a') < r$. It follows
\begin{align*}
&d_{X_{n}}(f_{n}(s), f_{n}(s')) = d_{X_{n}}((a, \ind_{H}^{G} f_{H}(a,gH)), (a', \ind_{H}^{G} f_{H}(a',gH))) \\
&= d_{G}(a,a') +n \cdot d^{1}_{\ind_{H}^{G}E_{H}} ( \ind_{H}^{G} f_{H}(a,gH),  \ind_{H}^{G} f_{H}(a',gH)) \\
&= d_{G}(a,a')+n \cdot d^{1}_{E_{H}}(f_{H}(g^{-1}a), f_{H}(g^{-1}a')) \\ 
&< r + n \cdot \frac{1}{n} = r+1 
\end{align*}
where in the last line, we use the contracting property of $f_{H}$ and that $d_{G}(g^{-1}a, g^{-1}a') < r$. If $n < r$, we use that $S_{n}$ is $G$-cofinite and that for each $s$, there are only finitely many $s'$ with $d_{S_{n}}(s,s') < r$. Together, this implies that there are only finitely many cases to consider for $n < r$, which we can easily handle by letting $R$ be the maximum of these finitely many $d_{X_{n}}(f_{n}(s), f_{n}(s'))$.\\
The commutativity of the diagram is immediate, since we project away from all factors where our map $F$ does anything. 
\end{proof}

%% file: chapters/stability.tex
\chapter{Stability}
In this section, we want to see that the map

\[
\oplus_{n \in \IN} \calo^{G}_{*}(E,X_{n}, d_{X_{n}}) \rightarrow \calo^{G}_{*}(E, (X_{n}, d_{X_{n}})_{n \in \IN})
\]
induces an equivalence in topological $K$-theory.  

\section{Topological $K$-theory and stability}

The above claim will be a consequence of the following more general statement, which is the topological $K$-theory analogue of \cite[7.2]{BLR}. The proof given here is an adaption to topological $K$-theory of the proof given in \cite{BLR}.

\begin{thm}
\label{stability}
Let $X_{n}$ be a sequence of simplicial complexes equipped with simplicial $G$-actions such that the dimensions of the $X_{n}$ are uniformly bounded by $N \in \IN$. Let $d_{n}$ be a quasimetric on $X_{n}$ satisfying
\[
d_{n}(x,y) \geq nd^{1}(x,y)  \quad \forall x,y \in X_{n}
\]
with equality if $x$ and $y$ lie in the same simplex of $X_{n}$. Assume that all isotropy groups of the $G$-actions on $X_{n}$ are contained in the family $\calf$ and let $E = E_{\calf} G$. Let $\widetilde{d_{n}}$ be the metric on $G \times X_{n}$ given by
\[
\widetilde{d_{n}}((g,x), (g',x')) = d_{G}(g,g')+d_{n}(x,x')
\]
Then the functor
\[
\bigoplus_{n \in \IN} \calo_{*}^{G}(E, G \times X_{n}, \widetilde{d_{n}}) \rightarrow \calo^{G}_{*}(E, (G \times X_{n}, \widetilde{d_{n}}))
\]
induces an isomorphism in topological $K$-theory.
\end{thm}

\begin{rem}
In our case, we have
\[
X_{n} = \coprod_{H \in \calh_{n}} \ind_{H}^{G} E_{H} 
\]
and
\[
d_{n}(x,y) =  n \cdot d^{1}(x,y)
\]
The assumptions on the isotropy groups and on the dimension of the $Z_{n}$ are true by definition of a Farrell-Hsiang group. 
\end{rem}

We will split up the proof into a series of propositions. Let us first discuss what we have to prove. We can form the quotient $C^{*}$-category $ \calo^{G}_{*}(E, (X_{n}, d_{X_{n}})_{n \in \IN})^{> \oplus}$ of the inclusion 
\[
\oplus_{n \in \IN} \calo^{G}_{*}(E,X_{n}, d_{X_{n}}) \rightarrow \calo^{G}_{*}(E, (X_{n}, d_{X_{n}})_{n \in \IN})
\]
The second proof of the fiber sequence theorem \ref{homotopyfiber} also applies in this setup to give a homotopy fiber sequence
\[
K (\bigoplus_{n \in \IN} \calo^{G}_{*}(E,X_{n}, d_{X_{n}})) \rightarrow K \calo^{G}_{*}(E, (X_{n}, d_{X_{n}})_{n \in \IN}) \rightarrow K \calo^{G}_{*}(E, (X_{n}, d_{X_{n}})_{n \in \IN})^{> \oplus}
\]
We see that the original claim is equivalent to 
\[
K \calo^{G}_{*}(E, (X_{n}, d_{X_{n}})_{n \in \IN})^{> \oplus} = 0
\]
This is what we will prove. \\

Fix a natural number $N$. Let $Y_{n}$ be the disjoint union of the $N$-simplices of $X_{n}$ with the quasi-metric
\[
d_{n}^{\infty}(x,y) = \begin{cases}  nd^{1}(x,y) &\text{ if $x,y$ lie in the same simplex} \\  \infty &\text{ else } \end{cases}
\]
Let $\partial Y_{n}$ be the disjoint union of the boundaries of the simplices of $X_{n}$ with the restricted metric. Denote by $\widetilde{d_{n}^{\infty}}$ the metric on $G \times Y_{n}$ given by
\[
\widetilde{d^{\infty}_{n}}((g,x), (g',x')) = d_{G}(g,g')+d_{n}(x,x')
\]
Both $Y_{n}$ and $\partial Y_{n}$ inherit $G$-actions whose isotropy is in the familiy $\calf$. 

\fxnote{I think from this point forward the proof is really the same}

\begin{prop}
\label{Swindle-stability}
The $K$-theory of $\calo_{*}^{G}(E, (G \times Y_{n}, d_{n}^{\infty})_{n \in \IN})^{> \oplus}$ vanishes.
\end{prop}

\begin{prop}
\label{MV-stability}
The square induced from the attaching of the $N$-simplices to $X_{n}$
\[
\xymatrix{
\calo_{*}^{G}(E, (G \times \partial Y_{n}, \widetilde{d_{n}^{\infty}})_{n \in \IN})^{> \oplus} \ar[rr] \ar[dd] & & \calo_{*}^{G}(E, (G \times Y_{n}, \widetilde{d_{n}^{\infty}})_{n \in \IN})^{> \oplus} \ar[dd] \\
\\
\calo_{*}^{G}(E, (G \times X_{n}^{N-1}, \widetilde{d_{n}})_{n \in \IN})^{> \oplus} \ar[rr] & & \calo_{*}^{G}(E, (G \times X_{n}, \widetilde{d_{n}})_{n \in \IN})^{>\oplus}\\
}
\]
becomes homotopy cartesian after applying $K$-theory.
\end{prop}

Together, these two propositions imply \ref{stability}:

\begin{proof}[Proof of \ref{stability}] The proof is now by induction on $N$. The case $N = 0$ is covered by \ref{Swindle-stability}. For the induction step, we argue in the diagram in \ref{MV-stability}. By induction hypothesis, the $K$-theory of the two left-hand categories vanish since the simplicial complexes there have dimension at most $N-1$. By \ref{Swindle-stability}, the $K$-theory of the upper right-hand category vanishes. We conclude that the $K$-theory of the lower right-hand category also vanishes since the square is homotopy cartesian after $K$-theory.
\end{proof}

\begin{proof}[Proof of \ref{Swindle-stability}]
This is proven by writing down an Eilenberg swindle on the coarse space (compare \ref{productcoarse}) underlying the category $\calo_{*}^{G}(E, (G \times Y_{n}, d_{n}^{\infty})_{n \in \IN})$ which passes to $\calo_{*}^{G}(E, (G \times Y_{n}, d_{n}^{\infty})_{n \in \IN})^{> \oplus}$. The Eilenberg swindle is then the same as the one in \cite[7.4]{BLR}. This is the place where the universal property of $E$ enters into the proof. 
\end{proof}

\begin{proof}[Proof of \ref{MV-stability}]
The proof is nearly the same as \cite[7.5]{BLR}. We give the details here to point out one subtle point one has to be careful about. To simplify notation, we abbreviate
\begin{align*}
B &= \coprod\limits_{n \in \IN} G \times \partial Y_n \times E \times [1, \infty) \\
A &= \coprod\limits_{n \in \IN} G \times X_n^{N-1} \times E \times [1, \infty) \\
Y &= \coprod\limits_{n \in \IN} G \times Y_n \times E \times [1, \infty) \\
X &= \coprod\limits_{n \in \IN} G \times  X_n \times E \times [1, \infty) \\
\end{align*}
and introduce coarse structures on these four spaces as follows. A set $E \subset B \times B$ is in $\cale_{B}$ if and only if there is for each $n$ a continuously controlled subset $J_{n} \subset (E \times [1, \infty))^{2}$ such that $E \subset \coprod\limits_{n \in \IN} p_{n}^{-1}(J_{n})$ and additionally, there is $R > 0$ such that whenever $((g,y,e,t),(g',y',e',t')) \in E$, we have $d(g,g') < R$ and $d_{n}^{\infty}(y,y') < R$. Note that the first condition on $E$ forces that $y$ and $y'$ lie in the same $\partial Y_{n}$ and compare \ref{productcoarse}. Similarly, we define $\cale_{A}, \cale_{Y}$ and $\cale_{X}$ with $d_n$ in place of $d_n^{\infty}$. Let $\calf^{B}_{Gc}$ be the set of subsets of $B$ of the form $\coprod\limits_{n \in \IN} K_{n}$, where $K_{n} \subset G \times \partial Y_n \times E \times [1, \infty)$ is the preimage of a $G$-compact set in $G \times \partial Y_n \times E$ under the canonical projection. Similarly, we define  $\calf^{A}_{Gc}$, $\calf^{X}_{Gc}$ and $\calf^{Y}_{Gc}$. Note that
\[
\calo_{*}^{G}(E, (G \times \partial Y_{n}, \widetilde{d_{n}^{\infty}})_{n \in \IN}) = \C_{*}^{G}(B; \cale_{B}, \calf^{B}_{Gc})
\]
and similarly for $A,X,Y$; compare \ref{productcoarse}.  We furthermore define $\calf^{B}_{\oplus}$ to be the collection of subsets of $B$ of the form
\[
\coprod\limits_{n=1}^{k} G \times \partial Y_{n} \times E \times [1, \infty)
\]
for $k \in \IN$ and similarly we define $\calf^{A}_{\oplus}$, $\calf^{X}_{\oplus}$, $\calf^{Y}_{\oplus}$. We have inclusions
\[
\C_{*}^{G}(B; \cale_{B}, \calf_{Gc}^{B} \cap \calf_{\oplus}^{B}) \rightarrow \C_{*}^{G}(B; \cale_{B}; \calf_{Gc}^{B})
\]
and similarly for $A,X,Y$. The quotient of this inclusion is precisely $\calo_{*}^{G}(E, (G \times \partial Y_{n}, \widetilde{d_{n}^{\infty}})_{n \in \IN})^{> \oplus}$; we from now on denote this quotient category by $\C_{*}^{G}(B; \cale_{B}, \calf_{Gc}^{B})^{> \oplus}$ from now on. The diagram we consider in the proposition then becomes
\[
\xymatrix{
\C_{*}^{G}(B; \cale_{B}, \calf_{Gc}^{B})^{> \oplus} \ar[rr] \ar[dd] & &\C_{*}^{G}(Y; \cale_{Y}, \calf_{Gc}^{Y})^{> \oplus}\ar[dd] \\
\\
\C_{*}^{G}(A; \cale_{A}, \calf_{Gc}^{A})^{> \oplus}\ar[rr] & &\C_{*}^{G}(X; \cale_{X}, \calf_{Gc}^{X})^{> \oplus} \\
}
\]
As in the proof of \ref{MV}, we would like to replace $A$ and $B$ by $X$ and $Y$ with some changed control conditions. For $R > 0$ and $F \in \calf_{Gc}^{B}$, let $F^{R}$ be the subset of $Y$ consisting of all points $(g,y,e,t)$ such that there is $(g, y',e,t) \in F$ with $d_{n}^{\infty}(y,y') < R$ and let $\calf_{B}^{Y}$ be the collection of all these subsets; and similarly for $\calf_{A}^{X}$. We obtain inclusion functors
\begin{align*}
\C_{*}^{G}(B; \cale_{B}; \calf^{B}_{Gc})^{> \oplus} &\rightarrow \C_{*}^{G}(Y; \cale_{Y}, \calf_{Gc}^{Y} \cap \calf_{B}^{Y})^{> \oplus} \\
\C_{*}^{G}(A; \cale_{A}; \calf^{A}_{Gc})^{> \oplus} &\rightarrow \C_{*}^{G}(X; \cale_{X}, \calf_{Gc}^{X} \cap \calf_{A}^{X})^{> \oplus}
\end{align*}
and these functors are equivalences of categories. Indeed, that these functors are full and faithful is a direct consequence of the definitions. To see essential surjectivity, start with an object 
\[
M \in \C_{*}^{G}(Y; \cale_{Y}, \calf_{Gc}^{Y} \cap \calf_{B}^{Y})
\]
It is supported over some $F^{R}$ where $F \in \calf_{Gc}^{B}$. Because of the $> \oplus$-quotient, we can assume that $M$ is supported only over $\coprod\limits_{n=k}^{\infty} G \times Y_{n} \times E \times [1, \infty)$ where $k$ is chosen such that $k > 2R$. For $(g,y,e,t)$ in the support of $M$, we hence find $y' \in B$ such that $d(y,y') < R$, and we can make these choices in a $G$-invariant way. Pushing $M_{(g,y,et)}$ forward to $M_{(g,y',e,t)}$, we obtain a potential object of $\C_{*}^{G}(B; \cale_{B}; \calf^{B}_{Gc})^{> \oplus}$, and similarly for $A,X$. The only potential worry is that this object may not be locally finite. But this is impossible since for fixed $g \in G$, $e \in E$, $t \in [1, \infty)$, only finitely many points of the form $(g,y,e,t)$ in the support of $M$ can have distance $< R$ from a given $(g,y',e,t)$ since any such $y$ has to lie in a simplex intersecting a simplex containing $y'$ thanks to the assumption $k > 2R$, but the set of all such $y'$ has to be compact since else the support of $M$ would not be $G$-compact in the $G \times Y \times E$-direction. Clearly the object we constructed is isomorphic to $M$, which finishes the claim. \\
So we now consider the diagram
\[
\xymatrix{
\C_{*}^{G}(Y; \cale_{Y}, \calf_{Gc}^{Y} \cap \calf_{B}^{Y})^{> \oplus} \ar[rr] \ar[d] & &\C_{*}^{G}(Y; \cale_{Y}, \calf_{Gc}^{Y})^{> \oplus}\ar[d] \\
\C_{*}^{G}(X; \cale_{X}, \calf_{Gc}^{X} \cap \calf_{A}^{X})^{> \oplus}\ar[rr] & &\C_{*}^{G}(X; \cale_{X}, \calf_{Gc}^{X})^{> \oplus} \\
}
\]
and want to see it is homotopy cartesian. To see this, we again argue as in \ref{MV} on the quotient categories; note that we really need the second proof of the fibration sequence theorem \ref{homotopyfiber} to conclude that we obtain fiber sequences in $K$-theory: This is not a quotient which arises out of killing a subspace. \\
So we consider the induced functor $F$ on the quotients
\[
\C_{*}^{G}(Y; \cale_{Y}, \calf^{Y}_{Gc})^{\oplus, B} \rightarrow \C_{*}^{G}(X; \cale_{X}, \calf^{X}_{Gc})^{\oplus,A}
\]
Since in the right-hand category we ignore all objects supported over $A$, we may assume that an object is only supported over $X-A$. But since $X-A \cong Y-B$, this means that $F$ is essentially surjective. \\
The functor $F$ is also faithful since the original functor 
\[
\C_{*}^{G}(Y; \cale_{Y}, \calf^{Y}_{Gc})^{\oplus} \rightarrow \C_{*}^{G}(X; \cale_{X}, \calf^{X}_{Gc})^{\oplus}
\]
was and the preimage of each $F \in \calf_{A}^{X}$ is contained in some $F' \in \calf_{B}^{Y}$ because of the condition $d_{n}(x,y) \geq nd^{1}(x,y)$ in $X_{n}$. \\
Now let $\phi$ be any morphism in the category $\C_{b}^{G}(X; \cale_{X}, \calf^{X}_{Gc})$. We can write $\phi$ as a  sum $\psi + \tau$ where $\psi$ has no component connecting different $N$-simplices of $\coprod_{n \in \IN} X_{n}$ and $\tau$ only has such components. We can lift $\psi$ to $\C_{b}^{G}(Y; \cale_{Y}, \calf^{Y}_{Gc})$. Lemma \ref{distancel1} below implies that $\tau$ factors over some $F \in \calf_{A}^{X}$ and is hence $0$ in the quotient $\C_{b}^{G}(X; \cale_{X}, \calf^{X}_{Gc})^{\oplus,A}$. Since $C^{*}$-functors have closed image, this implies that 
\[
\C_{*}^{G}(Y; \cale_{Y}, \calf^{Y}_{Gc})^{\oplus, B} \rightarrow \C_{*}^{G}(X; \cale_{X}, \calf^{X}_{Gc})^{\oplus,A}
\]
is full. This finishes the proof.
\end{proof}

\begin{lem}
\label{distancel1}
For an $n$-dimensional simplicial complex $X$, a point $x$ in some $n$-simplex $\Delta$ and a point $y \in X-\Delta$ there is $z \in \partial \Delta$ such that $d^{1}(x,z) \leq 2d^{1}(x,y)$.
\end{lem}

\begin{proof}
See \cite[7.15]{BLR}
\end{proof}

%% file: chapters/rationally.tex
\chapter{Rational results}

In this section, we employ the Artin induction theorem \ref{repcyclic} to obtain rational results or results after localisation at a prime for the Baum-Connes or Farrell-Jones conjecture. For this end, we change the definition of a Farrell-Hsiang group from \cite{BFL} slightly so that it fits into this context.

\section{Rational Farrell-Hsiang groups}

\begin{definition}
Let $G$ be a group with a word metric $d_{G}$ and $\calf$ a family of subgroups of $G$.
\begin{enumerate}
\item We say that $G$ is a rational Farrell-Hsiang group with respect to $\calf$ if the following condition is satisfied: \\
There exists a natural number $N$ and for each $n \in \IN$ there is a surjective group homomorphism $\alpha_{n} \colon G \rightarrow F_{n}$ with $F_{n}$ finite such that for each \emph{cyclic} subgroup $C \subset F_{n}$ and $H = \alpha_{n}^{-1}(C) \subset G$, there is a simplicial complex $E_{H}$ of dimension at most $N$ and a simplicial $H$-action with isotropy in $\calf$, and an $H$-equivariant map $f_{H} \colon G \rightarrow E_{H}$ such that
\[
d^{1}_{E_{H}}(f_{H}(g), f_{H}(h)) \leq \frac{1}{n}
\] 
for all $g,h \in G$ with $d_{G}(g,h) \leq n$, where $d^{1}_{E_{H}}$ is the $l^{1}$-metric on $E_{H}$
\item If in addition each of the finite quotients $F_{n}$ has order not divisible by a prime $p$, we say that $G$ is a Farrell-Hsiang group with respect to $\calf$ at the prime $p$.
\end{enumerate}
\end{definition}

\begin{thm}
\label{rationalbcc}
If $G$ is a rational Farrell-Hsiang group, the Baum-Connes assembly map for $\calf$ is a rational isomorphism. If $G$ is a Farrell-Hsiang group at the prime $p$, the Baum-Connes assembly map for $\calf$ is an isomorphism after localising at $p$.
\end{thm}

To prove this, recall the main diagram we organized the proof around:

\[
\xymatrix{
& & \oplus_{n \in \IN} \calo^{G}_{*}(E,X_{n}, d_{X_{n}}) \ar[d]^{{\text{inc}}} \\
\calo^{G}_{*}(E, (S_{n}, d_{S_{n}})_{n \in \IN}) \ar[dd]_{P_{k}} \ar[rr]^{F} & & \calo_{*}^{G}(E,(X_{n},d_{X_{n}})_{n \in \IN}) \ar[dd]^{{Q_{k}}} \\
\\
\calo^{G}_{*}(E,G,d_{G}) \ar[rr]^{\id} & &  \calo^{G}_{*}(E,G,d_{G})
}
\]

We will again argue in this diagram, with slightly changed definitions. 
Let $\calh_{n}$ be the family of subgroups of $G$ obtained by pulling back the family of \emph{cyclic} subgroups of $F_{n}$.  As before, we define pseudo-metric spaces $X_{n}, S_{n}$ as

\begin{align*}
X_{n} &= G \times \coprod_{H \in \calh_{n}} \ind_{H}^{G} E_{H} \\
S_{n} &=   \coprod_{H \in \calh_{n}} G \times G/H
\end{align*}

with metrics

\begin{align*}
d_{X_{n}}((g,x),(h,y)) &= d_{G}(g,h)+n \cdot d^{1}_{\ind_{H}^{G} E_{H}} (x,y)\\
d_{S_{n}}((g,aH), (h,bK)) &= \begin{cases} 
d_{G}(g,h) \text{ if } K=H,aH = bK \\
\infty \text{ else }
\end{cases}
\end{align*}
where $d^{1}$ denotes the $l^{1}$-metric.  The only difference to the definitions we had before is that $H$ only runs through the preimages of cyclic subgroups of $F_{n}$. The existence of the upper horizontal map is proven precisely as in \ref{horizontal}. Theorem \ref{rationalbcc} will be a consequence of the following two propositions and some basic representation theory of finite groups.

\begin{prop}
\label{rational}
Assume that for each $a \in K_{*}(\calo_{*}^{G}(E,G,d_{G}))$ there is $b \in K_{*} \calo^{G}_{*}(E, (S_{n}, d_{S_{n}})_{n \in \IN})$ such that for all $k$, the element $K_{*}(P_{k})(b)$ is some integral multiple of $a$. Then we have
\[
K_{*}(\calo_{*}^{G}(E,G,d_{G})) \otimes \IQ = 0
\] 
and hence the assembly map for $\calf$ is a rational isomorphism.
\end{prop}

\begin{proof}
Consider $K_{*}(F)(b) \in K_{*}\calo_{*}^{G}(E,(X_{n},d_{X_{n}})_{n \in \IN})$. Pick a preimage $c$ of $K_{*}(F)(b)$ in $K_{*} \oplus_{n \in \IN} \calo^{G}_{*}(E,X_{n}, d_{X_{n}})$ and choose $k$ large enough such that $K_{*}(Q_{k} \circ \text{inc})(c) = 0$, which is possible by \ref{stability}.
It follows that $K_{*}(Q_{k} \circ F)(b) = 0$, but by assumption, $K_{*}(Q_{k} \circ F)(b) = K_{*}(P_{k})(b)$ is a multiple of $a$. So there is an $n$ with $na = 0$ which proves the claim. 
\end{proof}

\begin{prop}
\label{p-local}
Let $p$ be a prime. If  for each $a \in K_{*}(\calo_{*}^{G}(E,G,d_{G}))$ there is $b \in K_{*} \calo^{G}_{*}(E, (S_{n}, d_{S_{n}})_{n \in \IN})$ such that for all $k$, $K_{*}(P_{k})(b) = n_{k}a$ with $n_{k}$ not divisible by $P$, then
\[
K_{*}(\calo_{*}^{G}(E,G,d_{G})) \otimes \IZ_{(p)}= 0
\] 
and hence the assembly map for $\calf$ is an isomorphism after localizing at $p$.
\end{prop}

\begin{proof}
As above, consider $K_{*}(F)(b) \in K_{*}\calo_{*}^{G}(E,(X_{n},d_{X_{n}})_{n \in \IN})$. Pick a preimage $c$ of $K_{*}(F)(b)$ in $K_{*} \oplus_{n \in \IN} \calo^{G}_{*}(E,X_{n}, d_{X_{n}})$ and choose $k$ large enough such that $K_{*}(Q_{k} \circ \text{inc})(c) = 0$. So $K_{*}(Q_{k} \circ F)(b) = 0$, but by assumption, $K_{*}(Q_{k} \circ F)(b) = n_{k}a$. So $n_{k}a = 0$, and since $n_{k}$ is inverted when we localize at $p$, $a = 0$ in  $K_{*}(\calo_{*}^{G}(E,G,d_{G})) \otimes \IZ_{p}$. 
\end{proof}

\begin{proof}[Proof of \ref{rationalbcc}.]
We check the conditions of \ref{rational} respectively \ref{p-local}. Let $n \in \IN$ be given and let $\calh$ be the collection  of subgroups of $G$ consisting of $H \subset G$ which are preimages of cyclic subgroups of $F_{n}$ under $\alpha_{n}$. Let $p_{n} \colon O_{*}^{G}(E,S_{n}, d_{S_{n}}) \rightarrow \calo_{*}^{G}(E,G,d_{G})$ be the projection. As in \ref{pm}, we construct functors
\[
F_{n}^{\pm} \colon \calo_{*}^{G}(E,G,d_{G}) \rightarrow \calo_{*}^{G}(E,S_{n},d_{S_{n}})
\]
such that $K_{*}(p_{n} \circ F_{n}^{+})-K_{*}(p_{n} \circ F_{n}^{-})$ is multiplication with the order of $F_{n}$. Then the argument proceeds as in the proof of \ref{surjective}. Let $m$ be the order of $F_{n}$. 
By the Artin induction theorem \ref{repcyclic}, we find elements $T_{H} \in \Sw(\IC,H)$ for $H \in \calh$ with
\[
m_{\Sw(\IC,G)} = \sum\limits_{H \in \calh} \ind_{H}^{G} T_{H}
\]
 Pick $T_{H}^{\pm} \in \Mod_{(\IC,H)}$ with $T_{H} = T_{H}^{+}-T_{H}^{-}$. Let 
 \[
 \phi_{H} \colon \calo_{*}^{G}(E, G \times G/H, d) \rightarrow \calo_{*}^{G}(E, S_{n}, d_{S_{n}})
 \]
 be the map induced by the inclusion $G \times G/H \rightarrow S_{n}$. Now define $F_{H}^{\pm} \colon \calo_{*}^{G}(E,G,d_{G}) \rightarrow \calo_{*}^{G}(E,S_{n},d_{S_{n}})$ by
 \[
 F_{H}^{\pm} = \phi_{H} \circ \tr_{H}(T_{H}^{\pm},-)
 \]
 and set
 \[
 F_{n}^{\pm} = \oplus_{H \in \calh} F_{H}^{\pm}
 \]
 The proof of and the definitions in \ref{surjective}  can be copied word-by-word, with the only difference in the end being that
 \[
 K_{*}(p_{n} \circ F_{n}^{+}) - K_{*}(p_{n} \circ F_{n}^{-})(a) = ma
 \]
 instead of $a$. 
 \end{proof}

The same proof yields a similar statement for the Farrell-Jones conjecture:

\begin{thm}
Let $G$ be a rational Farrell-Hsiang group with respect to the family $\calf$. Let $A$ be an additive category with $G$-action which is enriched over a field $F$ of characteristic $0$. Then the Farrell-Jones conjecture with coefficients in $A$ with respect to the family $\calf$ is true rationally. Similarly, if $G$ is a Farrell-Hsiang group with respect to $\calf$ at the prime $p$, then the Farrell-Jones conjecture with coefficients in $A$ with respect to $\calf$ is true after inverting $p$.
\end{thm}
\noindent
Note that we need the assumption about the field $F$ to ensure that we can employ the Artin induction theorem.

%% file: chapters/inductionresults.tex
\chapter{Varying the family of subgroups}

In this section, we will show how even in seemingly trivial cases, we can extract quite some information from the fact that Farrell-Hsiang groups satisfy the Baum-Connes or Farrell-Jones conjectures, which yields induction theorems for the Farrell-Jones and Baum-Connes conjecture allowing one to restrict to smaller families than $\mathcal{VCYC}$ and $\mathcal{FIN}$. The results are not new and are mainly contained in \cite{BLind}, but with different proofs, though in the end they also rely on the same induction theorems.

\section{Some examples of Farrell-Hsiang groups}

\begin{definition}
Let $\fin$ be the family of finite groups, $\calh$ the family of hyperelementary groups and $\cale$ the family of elementary groups. For a subfamily $\calf \subset \fin$, let $\calf'$ be the family consisting of all groups $G$ which lie in $\calf$ or for which there is an extension
\[
1 \rightarrow \IZ \rightarrow G \rightarrow F \rightarrow 1
\]
with $F \in \calf$. Note that $\fin' = \mathcal{VCYC}$.
\end{definition}

\begin{thm}
\label{kinduction}
Let $G$ be a group. Let $A$ be a $G$-$C^{*}$-algebra and $\calb$ an additive category with $G$-action Let $K A$ respectively $K^{\alg} \calb$ be the $\Or(G)$-spectra associated to $A$ respectively $\cala$ as in \ref{orgktop} respectively \cite{coeff}. Then the following relative assembly maps are bijective:
\begin{align*}
H^{G}_{*}(E_{\calh}, K^{\alg} \calb) &\rightarrow H^{G}_{*}(E_{\fin}, K^{\alg} \calb) \\
H^{G}_{*}(E_{\calh}, K A) &\rightarrow H^{G}_{*}(E_{\fin}, K A) \\
H^{G}_{*}(E_{\cale}, K A) &\rightarrow H^{G}_{*}(E_{\fin}, K A) \\
H^{G}_{*}(E_{\calh'}, K^{\alg} \calb) &\rightarrow H^{G}_{*}(E_{\mathcal{VCYC}}, K^{\alg} \calb) \\
H^{G}_{*}(E_{\calh'}, K A) &\rightarrow H^{G}_{*}(E_{\mathcal{VCYC}}, K A) \\
H^{G}_{*}(E_{\cale'}, K A) &\rightarrow H^{G}_{*}(E_{\mathcal{VCYC}}, K A) \\
\end{align*}
\end{thm}

\begin{proof}
We concentrate on the hyperelementary result; the result about $\cale$ is proven in the same way. By the transitivity principle, it suffices to prove that each group in $\fin$ respectively $\mathcal{VCYC}$ satisfies the Baum-Connes respectively Farrell-Jones conjecture with respect to the family $\calh$ respectively $\calh'$. Let $G$ be a finite group and $\calh(G)$ the family of hyperelementary subgroups of $G$. By the results of \cite{FH} and our own results, it suffices to see that $G$ is a Farrell-Hsiang group with respect to the family $\calh(G)$. For each $n \in \IN$, we set $F_{n} = G$ and $\alpha_{n} = \Id \colon G \rightarrow F_{n}$. For each hyperelementary subgroup $H$ of $G$, we let $E_{H}$ be the one-point space and let $f_{H} \colon G \rightarrow E_H$ be the obvious map. Clearly, $f_{H}$ is $H$-invariant and sufficiently contracting; and since $H$ is in the family with respect to which we want to prove the Baum-Connes respectively Farrell-Jones conjecture, the isotropy of the one-point space is contained in the family. This handles the first three maps. \\
For the second set of maps, it again suffices to prove that each virtually cyclic group $G$ satisfies the Farrell-Jones respectively Baum-Connes conjecture with respect to the family $\calh'$. So let $G$ be virtually cyclic and consider an extension
\[
1 \rightarrow \IZ \rightarrow G \rightarrow F \rightarrow 1
\]
with $F$ finite. We prove that $G$ is a Farrell-Hsiang group with respect to the family $\calh'$. For each $n$, let $F_{n} = F$ and let $\alpha_{n} \colon G \rightarrow F$ be the projection from the extension above. For each hyperelementary subgroup $H \subset F$, the preimage $\alpha^{-1}(H) \subset G$ is contained in the family $\calh'$. This again allows us to set $E_{H} = *$, and hence $G$ is a Farrell-Hsiang group with respect to $\calh'$.
\end{proof}

We can formulate a slightly more general version of the same argument:

\begin{prop}
\label{trivialFH}
Let $G$ be a group, together with a group homomorphism $f: G \rightarrow F$ to a finite group. Assume that for each hyperelementary subgroup $H \subset F$, $f^{-1}(H)$ satisfies the Farrell-Jones or Baum-Connes conjecture. Then also $G$ satisfies the Farrell-Jones or Baum-Connes conjecture.
\end{prop}

\begin{proof}
Again we can set $F_{n} = F$ and $E_{H} = *$ to conclude that $G$ is a Farrell-Hsiang group with respect to the family of preimages of hyperelementary groups. 
\end{proof}

\section{Rational reductions}

With the same kind of arguments and the rational version of Farrell-Hsiang groups, we can prove the following:

\begin{thm}
\label{kinduction}
Let $G$ be a group. Let $A$ be a $G$-$C^{*}$-algebra and $\calb$ an additive category with $G$-action Let $K A$ respectively $K^{\alg} \calb$ be the $\Or(G)$-spectra associated to $A$ respectively $\calb$ as in \ref{orgktop} respectively \cite{coeff}. Let $\mathcal{FINCYC}$ be the family of finite cyclic groups. Then the following relative assembly maps are rational isomorphisms:
\begin{align*}
H^{G}_{*}(E_{\mathcal{FINCYC}}, K^{\alg} \calb) &\rightarrow H^{G}_{*}(E_{\fin}, K^{\alg} \calb) \\
H^{G}_{*}(E_{\mathcal{FINCYC}}, K A) &\rightarrow H^{G}_{*}(E_{\fin}, K A) \\
H^{G}_{*}(E_{\mathcal{FINCYC}'}, K^{\alg} \calb) &\rightarrow H^{G}_{*}(E_{\mathcal{VCYC}}, K^{\alg} \calb) \\
H^{G}_{*}(E_{\mathcal{FINCYC}'}, K A) &\rightarrow H^{G}_{*}(E_{\mathcal{VCYC}}, K A) \\
\end{align*}
\end{thm}

\begin{rem}
In the case of the Baum-Connes map, this statement is true even before tensoring with $\IQ$; see \cite[Section 6]{BLind}. The technique used there significantly differs from the rest of \cite{BLind} in that it uses no induction results.
\end{rem}

%% file: chapters/farrell-hsiang.tex
\chapter{Farrell-Hsiang groups}

The following is mostly an account of the results of \cite{BFL}, which we include to have a concrete example of non-trivial Farrell-Hsiang groups. We will prove that all virtually finitely generated abelian groups satisfy the Baum-Connes conjecture. The strategy of the proof is a mixture of using the Farrell-Hsiang method and inheritance properties of the Baum-Connes conjecture. Note that since virtually finitely generated abelian groups are amenable, the Baum-Connes conjecture is known for these groups and we get no new results, only a different proof.

\section{The case $\IZ^{2}$}

Before going on to more complicated examples, it is worthwile studying the case of $\IZ^{2}$ in detail.

\begin{prop}
\label{C2}
The group $\IZ^{2}$ is a Farrell-Hsiang group with respect to the family of virtually cyclic subgroups and hence satisfies the Baum-Connes Conjecture.
\end{prop}

For this, we need the following lemma.

\begin{lem}
\label{lattice}
Let $p$ be a prime and $C \subset (\IZ/p)^{2}$ a cyclic non-trivial subgroup. Then there is a map $r \colon \IZ^{2} \rightarrow \IZ$ whose kernel reduces modulo $p$ to $C$ and such that
\[
r_{\IR} = r \otimes_{\IZ} \id_{\IR} \colon \IR^{2} \rightarrow \IR
\]
satisfies
\[
d(r_{\IR}(x), r_{\IR}(y)) \leq \sqrt{2p} \cdot d(x,y)
\]
for all $x,y \in \IR^{2}$, where $d$ is the euclidean metric on $\IR^{2}$ respectively $\IR$.
\end{lem}

\begin{proof}
If $C$ is one of the two factors of $(\IZ/p)^{2}$, the projection onto the other factor does the job. Otherwise, $C$ has a generator which lifts under projection  mod $p$ to an element of the form $(1,k) \in \IZ^2$, with $k \in \IZ$ and $0 < k < p$. 
Consider all pairs $(s,t) \in \IZ^{2}$ with $\sqrt{p} \geq s,t \geq 0$ and $(s,t) \neq (0,0)$. Since there are at least $p+1$ such pairs, there have to be two such pairs $(s,t), (s',t')$ such that
\[
s+tk = s'+t'k \mod p
\]
or, put otherwise, $(s-s')+(t-t')k = 0 \mod p$. Since $s,s',t,t'$ are nonnegative, the absolute values of $S = s-s'$ and $T = t-t'$ are still at most $\sqrt{p}$. Now set $r(a,b) = Sa+Tb$. Then $r(1,k) = S+Tk$ is divisible by $k$ and hence $(1,k)$ lies in the mod $p$-reduction of the kernel of $r$ as desired. Since $r$ is nonzero and also nonzero mod $p$, $C$ is indeed all of the kernel. The inequality
\[
d(r_{\IR}(x), r_{\IR}(y)) \leq \sqrt{2p} \cdot d(x,y)
\]
 easily follows from $\left\vert S \right\vert \leq \sqrt{p}$ and $\left\vert T \right\vert \leq \sqrt{p}$
\end{proof}

\begin{proof}[ Proof of \ref{C2}] This is adapted from \cite[2.8]{BFL}.
We let $\IZ^{2}$ act on $\IR^{2}$ by isometries in the obvious way. Pick a word metric $d$ on $\IZ^{2}$ and a point $x \in \IR^{2}$; for distinction, we denote the euclidean metric on $\IR^{2}$ by $d_{euc}$. Then evaluation at $x$ defines a map
\[
e: \IZ^{2} \rightarrow \IR^{2}
\]
We could pick $x = 0$, and then $e$ is just the canonical inclusion, but since the situation will be more complicated later on, we directly write things up this way.   
By the \v{S}varc-Milnor Lemma \cite[8.19]{BH}, $e$ is a quasi-isometry, hence there are constants $C, D > 0$ such that
\[
d_{euc}(e(g), e(h)) \leq Cd(g,h)+D
\]
for all $g,h \in \IZ^{2}$. Pick a natural number $n$. To prove that $\IZ^{2}$ is a Farrell-Hsiang group, we need to provide a quotient map 
\[
\alpha_{n} \colon \IZ^{2} \rightarrow F_{n}
\]
into a finite group $F_{n}$ with certain properties. Since $F_{n}$ should not be elementary itself,
the simplest possible quotient of $\IZ^{2}$ to use is $F_{n} = (\IZ/pq)^{2}$ for two different primes $p$ and $q$. As long as one picks $p$ and $q$ big enough, this will actually be enough; one could work out how big precisely in the end, dependent on $n$. So let us set out to verify the conditions on a Farrell-Hsiang group.\\
Let $E \subset (\IZ/pq)^{2}$ be an $r$-elementary subgroup. Without loss of generality, we can assume $r = q$. Then the quotient map $(\IZ/pq)^{2} \rightarrow (\IZ/p)^{2}$ sends $E$ to a cyclic subgroup $C$ of $(\IZ/p)^{2}$. If $C$ is nontrivial, let $r: \IZ^{2} \rightarrow \IZ$ be a map as in the above lemma; if $C$ is trivial, let $r$ be the projection onto the first factor. Let $H$ be the preimage of $E$ under the projection $\IZ^{2} \rightarrow (\IZ/pq)^{2}$. Now consider the following commutative diagram:
\[
\xymatrix{
H \ar[d] \ar@{^{(}->}[rr] & & \IZ^{2} \ar[d] \ar[rr]^{r} & & \IZ \ar[dd]\\
E \ar[d] \ar@{^{(}->}[rr] & & (\IZ/pq)^{2} \ar[d] \\
C \ar@{^{(}->}[rr] & & (\IZ/p)^{2} \ar[rr] & & \IZ/p \\
}
\]
Chasing around this diagram, one sees that $r(H) \subset p\IZ$ since $H$ has to go to zero in the lower right corner. Now let $E_{H}$ be the simplicial complex with underlying space $\IR$, with the integers as $0$-simplices. We define a map $\IZ^2 \rightarrow E_H$ as the composite
\[
\xymatrix{
f_{H}: \IZ^{2} \ar[rr]^{e} & & \IR^{2} \ar[rr]^{r_{\IR}} & & \IR \ar[rr]^{\cdot \frac{1}{p}} & &  \IR = E_{H} \\
}
\]
Now we need to define an $H$-action on $E_{H}$ such that this map is $H$-equivariant. Of course, there is a very easy $H$-action on $\IR$: $r_{\IR}$ sends $H$ to $\IZ$, and $\IZ$ acts on $\IR$ by translations. However, because of the $\frac{1}{p}$-factor, this action will not make $f_{H}$ $H$-equivariant, and we cannot omit this factor because then $f_{H}$ is not contracting at all and hence useless for our purpose. So we would like to divide the $H$-action on $\IR$ by $p$ - and we can because $r(H) \subset p\IZ$. So we define an $H$-action on $\IR$ by letting $h$ act as translation with the integer $\frac{r_{\IR}(h)}{p}$. With this $H$-action, $f_{H}$ is $H$-equivariant and all isotropy groups of $E_{H}$ for the $H$-action are cyclic since the isotropy subgroups have to lie in the kernel of $r$. It remains to see that $f_H$ is sufficiently contracting. Let $g,h \in \IZ^{2}$ be such that $d(g,h) \leq n$. Then we estimate
\begin{align*}
d^{l^{1}}(f_{H}(g), f_{H}(h)) &\leq d_{euc}(f_{H}(g), f_{H}(h)) \\
&= \frac{1}{p} d_{euc}(r_{\IR}(e(g)), r_{\IR}(e(h))) \\
&\leq \frac{1}{p} \sqrt{2p} \cdot d_{euc}(e(g), e(h)) \\
&\leq \frac{\sqrt{2}}{\sqrt{p}}(Cd(g,h)+D) \\
&\leq \frac{\sqrt{2}}{\sqrt{p}}(Cn+D)
\end{align*}
If $p$ (and also $q$, since we picked $p$ without loss of generality above) is sufficiently large, this is smaller than $\frac{1}{n}$. So if we pick $p,q$ large enough, $f_{H}$ has the desired contracting properties. Hence $\IZ^{2}$ is a Farrell-Hsiang group with respect to the family of virtually cyclic subgroups.
\end{proof}

The next-most complicated example is the group $\IZ^{2} \rtimes_{-\id} \IZ/2$. Trying to apply a similar argument, we will run into some problems. \fxnote{what to write here?}

\begin{prop}
The group $\IZ^{2} \rtimes_{-\id} \IZ/2$ is a Farrell-Hsiang group with respect to the family $\mathcal{VCYC}$.
\end{prop}

\begin{proof}
This is \cite[2.8]{BFL}. We will try to mimick the above proof for $\IZ^{2}$. Pick a word metric $d$ on $G = \IZ^{2} \rtimes_{-\id} \IZ/2$. The group $G$ acts by isometries on $\IR^{2}$, and evaluation at some point of $\IR^{2}$ yields a quasi-isometry
\[
e \colon G \rightarrow \IR^{2}
\]
So there are constants $A,B > 0$ such that
\[
d^{euc}(e(g), e(h)) \leq Ad(g,h)+B
\]
for all $g,h \in G$. Now fix a natural number $n > 0$. We need to provide a map from $G$ onto a finite group $F_{n}$. Since $G$ is non-abelian, it does not make sense to form $G/(pq)$ as we did for $\IZ^{2}$; however, we can consider the subgroup $pq\IZ^{2} \subset \IZ^{2} \subset G$, which is normal since $pq\IZ^{2}$ is characteristic in the normal subgroup $\IZ^{2}$. So we define for primes $p,q > 2$ the group
\[
G_{pq} = G/(pq\IZ^{2})
\]
and similarly $G_{p}$ and $G_{q}$. We let $a \colon G \rightarrow G_{pq}$ be the projection. Let $E \subset G_{pq}$ be an elementary subgroup. We may assume without loss of generality that $E$ is either $q$-elementary or $2$-elementary.  Then the image of $E \cap \IZ^{2}/pq\IZ^{2}$ under the projection $ \IZ^{2}/pq\IZ^{2} \rightarrow \IZ^{2}/p\IZ^{2}$ is a cyclic group $C$. If $C$ is trivial, let $r: \IZ^{2} \rightarrow \IZ$ be the projection onto the first factor; if not, let $r: \IZ^{2} \rightarrow \IZ$ be a map as in \ref{lattice} with respect to $C$. Let $r_{\IR} = r \otimes \Id \colon \IR^{2} \rightarrow \IR$. We have
\[
d(r_{\IR}(x_{1}),r_{\IR}(x_{2})) \leq \sqrt{2p} d^{euc}(x_{1}, x_{2})
\]
Let $H \subset G$ be the preimage of $E$ under the projection $G \rightarrow G_{pq}$. Then 
\[
r(H \cap \IZ^2) \subset p\IZ
\]
This would allow us to define an action of $H \cap \IZ^2$ on $\IR$ by letting $e \in H \cap \IZ^2$ act as $\frac{r(e)}{p}$. However, this is not sufficient: We need an action of $H$, so we have to take the $\IZ/2$-factor into account. Out of $r$, we can manufacture a map
\[
\overline{r} = r \rtimes \Id: G = \IZ^{2} \rtimes_{-\Id} \IZ/2 \rightarrow D_{\infty} = \IZ \rtimes_{-\Id} \IZ/2
\]
We know that $\overline{r}(H) \cap \IZ \subset p\IZ$. We now have to extend multiplication by $p$, defined on $\IZ$, to a map $\phi$ defined on all of $D_{\infty}$ such that
\[
\overline{r}(H) \subset \phi(D_{\infty})
\] 
Then we could try to define an $H$-action on $\IR$ by letting $h \in H$ act as the $d \in D_{\infty}$ which satisfies $\overline{r}(h) = \phi(d)$. So let
\[
\phi = p \rtimes \Id \colon \IZ \rtimes_{-\Id} \IZ/2 \rightarrow \IZ \rtimes_{-\Id} \IZ/2
\]
It follows that $\overline{r}(H) \subset \phi(D_{\infty})$. Let $E_{H}$ be the simplicial complex with underlying space $\IR$ and $0$-simplices at $\frac{z}{2}$, $z \in \IZ$. The usual $D_{\infty}$-action on $\IR$ is a simplicial action. Now define a map $f \colon G \rightarrow E_{H}$ as the composition
\[
G \stackrel{e}{\rightarrow} \IR^{2} \stackrel{r_{\IR}}{\longrightarrow} \IR \stackrel{\frac{1}{p}}{\rightarrow} \IR
\] 
Define an $H$-action on $\IR$ by letting $h \in H$ act on $\IR$ as the unique $d \in D_{\infty}$ with $\phi(d) = \overline{r}(h)$. Then $f$ is $H$-equivariant thanks to the fact that
\[
p \cdot gx = \phi(g) \cdot px
\] 
for $x \in \IR, g \in D_{\infty}$. All isotropy groups of the $H$-action on $E_{H}$ are virtually cyclic:  For $x \in \IR$, certainly the kernel $Z$ of $\overline{r} \colon H \rightarrow D_{\infty}$ stabilizes $x$, so we obtain a map $Z \rightarrow H_{x}$. The group $Z$ is cyclic and the cokernel consists at most of the finite stabilizer of $x$ under the $D_{\infty}$-action. 
Finally, we estimate for $g,h \in G$ with $d(g,h) \leq n$
\begin{align*}
d^{l^{1}}(f_{H}(g), f_{H}(h)) &\leq 2d_{euc}(f_{H}(g), f_{H}(h)) \\
&= 2\frac{1}{p} d_{euc}(r_{\IR}(e(g)), r_{\IR}(e(h))) \\
&\leq 2\frac{1}{p} \sqrt{2p} \cdot d_{euc}(e(g), e(h)) \\
&\leq 2\frac{\sqrt{2}}{p}(Ad(g,h)+B) \\
&\leq 2\frac{\sqrt{2}}{p}(An+B)
\end{align*}
Picking $p$ large enough makes this arbitarily small. Since $p$ was picked without loss of generality, note that this also forces us to pick $q$ large.
\end{proof}

When trying to generalize this line of argument to arbitrary groups acting by isometries on some euclidean space, one runs into two problems:
\begin{enumerate}
\item In general, such groups are not semidirect products, so we cannot define a map $\phi$ as in the above proof.
\item Even when we have some suitable map $\phi$, it may not satisfy
\[
p \cdot gx = \phi(g) \cdot px
\] 
\end{enumerate}

Both problems can be solved using group cohomology and will be addressed in the next section.

\section{Crystallographic groups}

In this section, we assemble the necessary facts on virtually finitely generated abelian groups and crystallographic groups. We will use \cite{brown} as a standard reference on group cohomology. 

\begin{definition}
Let $G$ be a group.
\begin{enumerate}
\item We say that $G$ is an abstract crystallographic group of rank $n$ if it has a normal, finite-index subgroup $A$ which is free abelian of rank $n$ and such that $cen(A) = A$, i.e. no element of $G$ which is not already in $A$ commutes with all elements of $A$. 
\item We say that $G$ is a concrete crystallographic group of rank $n$ if $G$ is isomorphic to a cocompact discrete subgroup of the group of isometries of $\IR^{n}$.
\end{enumerate}
\end{definition}

\begin{prop}
The subgroup $A$ in the definition of an abstract crystallographic $G$ is unique and so the holonomy group $F = G/A$ of $G$ is well-defined. Furthermore, a group is an abstract crystallographic group of rank $n$ if and only if it is a concrete crystallographic group. Furthermore, $n$ is unique.
\end{prop}

\begin{proof}
Let $G$ be an abstract crystallographic group of rank $n$ and let $B$ be another subgroup of $G$ with the same properties as $A$. Then $A \cap B$ is of finite index in both $A$ and $B$. Since any group automorphism of $\IZ^{n}$ inducing the identity on a finite-index subgroup is itself the identity, the centralizer of $A \cap B$ is equal to the centralizers of both $A$ and $B$, which implies $A = B$ because $A = cen(A)$ and $B = cen(B)$. It also follows that the $n$ in the definition of an abstract crystallographic group is unique. \\
Let $F$ be the quotient $G/A$. The conjugation action of $F$ on $A$ is faithful since $A$ is its own centralizer and hence conjugation by $g \in G$ only induces the identity on $A$ if $g \in A$. Pick an isomorphism $A \cong \IZ^{n}$ and consider the inclusion $\IZ^{n} \subset \IR^{n}$. We obtain a diagram of extensions as in \cite[Exercise IV.3.1]{brown}
\[
\xymatrix{
0 \ar[rr] & & A \ar[dd] \ar[rr] & & G \ar[rr] \ar[dd] & & F \ar[dd] \ar[rr] & & 0 \\
\\
0 \ar[rr] & & \IR^{n}  \ar[rr] & & G' \ar[rr] & & F \ar[rr] & & 0 \\
}
\]

All vertical maps are injective, and the lower horizontal row is semi\-split since $H^{2}(F, \IR^{n}) = 0$ by finiteness of $F$, see \cite[IV.10.2]{brown}. Picking a scalar product on $\IR^{n}$ such that $F$ acts by isometries on $\IR^{n}$, it is easy to see that $G'$ acts cocompactly and discretely on $\IR^{n}$ by isometries. \\
The other direction is Bieberbach's theorem, see \cite[Theorem 14]{Farkas} or \cite{Auslander}. 
\end{proof}

\begin{example}
Not every crystallographic group $G$ is a split extension of $A$ and $F = G/A$. For example, let $\phi, \psi$ be the isometries of $\IR^2$ given by
\begin{align*}
\phi(x,y) &= (-x,y+1) \\
\psi(x,y) &= (x+1,-y)
\end{align*}
Then $\phi^2$ is translation by $(0,2)$ and $\psi^2$ is translation by $(2,0)$, generating the translation subgroup of the group of isometries generated by $\phi$ and $\psi$. The canonical extension of this crystallographic group is
\[
0 \rightarrow \IZ\phi^2 \oplus \IZ \psi^2 \rightarrow \IZ\phi \oplus \IZ\psi \rightarrow \IZ/2\IZ \oplus \IZ/2\IZ \rightarrow 0
\]
which clearly is not split.
\end{example}

\begin{definition}
Let $G$ be a crystallographic group with $A$ the subgroup of translations and $F$ the finite holonomy group and $s$ an integer. We say that a group homomorphism $\phi: G \rightarrow G$ is $s$-expansive if it fits into a commutative diagram
\[
\xymatrix{
0 \ar[rr] & & A \ar[dd]^{s \cdot -} \ar[rr] & & G \ar[rr] \ar[dd] & & F \ar[dd]^{\id} \ar[rr] & & 0 \\
\\
0 \ar[rr] & & A  \ar[rr] & & G \ar[rr] & & F \ar[rr] & & 0 \\
}
\]
\end{definition}

\begin{prop}
\label{exp}
Let $G$ be a crystallographic group and $s \neq 0$ an integer.
\begin{enumerate}
\item If $s =1 \text{ mod } \left\vert F \right\vert$, there exists an $s$-expansive homomorphism $\phi: G \rightarrow G$.
\item For every such $\phi$, there is $u \in \IR^{n}$ such that the map
\[
f = f_{(s,u)}: \IR^{n} \rightarrow \IR^{n}, x \mapsto sx+u
\]
is $\phi$-equivariant, i.e. satisfies $f(g(x))  = \phi(g)f(x)$ for all $x \in \IR^{n}$.
\end{enumerate}
\end{prop}

\begin{proof}
For i), fix the $F$-module structure on $A$ given by the extension $G$. We certainly find some extension $0 \rightarrow A \rightarrow G' \rightarrow F \rightarrow 0$ inducing the given $F$-action on $A$ and fitting into a commutative diagram
\[
\xymatrix{
0 \ar[rr] & & A \ar[dd]^{s \cdot -} \ar[rr] & & G \ar[rr] \ar[dd]^{\phi} & & F \ar[dd]^{\id} \ar[rr] & & 0 \\
\\
0 \ar[rr] & & A  \ar[rr] & & G' \ar[rr] & & F \ar[rr] & & 0 \\
}
\]
The class of the lower extension in $H^{2}(F,A)$ is the $s$-fold multiple of the class of the upper extension. See \cite[Exercise IV.3.1]{brown}. Since $H^{2}(F,A)$ is $\left\vert F \right\vert$-torsion by \cite[III.10.2]{brown}, it follows that multiplication by $s$ is an isomorphism on $H^{2}(F,A)$ and hence the two extensions are isomorphic. Then the $\phi$ in the above diagram is our desired $s$-expansive map. \\
For ii), let us write out what the condition on $u$ means. For $g \in G$, we find $M_{g} \in F$ and $v_{g} \in \IR^{n}$ such that $gx = M_{g}x+v_{g}$. Then the equation $f(g(x))  = \phi(g)f(x)$ reads
\[
s(M_{g}x+v_{g})+u = M_{\phi(g)}sx+M_{\phi(g)}u+v_{\phi(g)}
\]
However, since $M_{g}$ is the image of $g$ under the projection $G \rightarrow F$ and $\phi$ induces the identity on $F$, we have $M_{g} = M_{\phi(g)}$. Then we can simplify and reach
\[
u-g\cdot u = v_{\phi(g)}-s \cdot v_{g}
\]
where $g \cdot -$ is the conjugation action of $G$ on $A$ and hence on $\IR^{n}$. Define a map $d \colon G \rightarrow \IR^{n}$ by $d(g) = v_{\phi(g)}-s \cdot v_{g}$. It is easy to check that $d$ is a derivation, i.e. satisfies $d(gh) = dg+gdh$.  If $g \in A$, we have $v_{\phi(g)} = sv_{g}$ since $\phi$ is multiplication by $s$ on $A$. Hence $d = 0$ on $A$ and we get an induced derivation
\[
d: F \rightarrow \IR^{n}
\]
Since $H^{1}(F, \IR^{n}) = 0$, this derivation has to be principal by \cite[Exercise III.1.2]{brown}. So there is $u \in \IR^{n}$ such that $d$ is equal to the principal derivation sending $g$ to $u-gu$. This was what we set out to prove.
\end{proof}

\section{The Baum-Connes conjecture for virtually finitely generated abelian groups}

In this section, we put the Farrell-Hsiang method to some serious work to conclude the Baum-Connes conjecture for virtually finitely generated abelian groups. We start with a reduction to crystallographic groups.

\begin{prop}
\label{vfga}
Each virtually finitely generated abelian group $G$ of virtually cohomological dimension $n$ admits an epimorphism with finite kernel onto a crystallographic group of rank $n$.
\end{prop}

\begin{proof}
Let $G$ be a virtually finitely generated abelian group. Let $A$ be a finitely generated abelian subgroup of finite index. We may assume without loss of generality that $A$ is normal in $G$ with quotient $F$. We have $A = \IZ^{n} \oplus H$ for some finite abelian group $H$; let $f: A \rightarrow \IZ^{n}$ be the projection. Since $F$ is finite, there is an $F$-action on $\IZ^{n}$ such that $f$ is an $F$-map. Now consider the extension induced by $f$ and $G$
\[
0 \rightarrow \IZ^{n} \rightarrow G' \rightarrow F \rightarrow 0
 \]
Note that actually $H$ is normal in $G$ and $G' = G/H$, so there is a map $G \rightarrow G'$ with kernel $H$. Now consider the inclusion $\IZ^{n} \rightarrow \IR^{n}$ and the induced map on extensions
\[
\xymatrix{
0 \ar[rr] & & \IZ^{n} \ar[dd] \ar[rr] & & G' \ar[rr] \ar[dd] & & F \ar[dd] \ar[rr] & & 0 \\
\\
0 \ar[rr] & & \IR^{n}  \ar[rr] & & G'' \ar[rr] & & F \ar[rr] & & 0 \\
}
\]
Since $H^{2}(F, \IR^{n}) = 0$, the lower extension is semisplit. Let $L$ be the kernel of the representation $F \rightarrow Gl_{n}(\IR)$ obtained from the splitting. Then we obtain a map of semisplit extensions
\[
\xymatrix{
0 \ar[rr] & & \IR^{n} \ar[dd] \ar[rr] & & G'' \ar[rr] \ar[dd] & & F \ar[dd] \ar[rr] & & 0 \\
\\
0 \ar[rr] & & \IR^{n}  \ar[rr] & & G''' \ar[rr] & & F/L \ar[rr] & & 0 \\
}
\]
and $G'''$ is clearly acting by isometries on $\IR^n$. Since also the map $G'' \rightarrow G'''$ has finite kernel, the composition $G \rightarrow G'''$ has finite kernel. The image of $G$ in $G'''$ is then a crystallographic group of rank $n$ as desired.
\end{proof}

\begin{thm}
The Baum-Connes Conjecture for virtually finitely generated abelian groups is true.
\end{thm}

The proof is quite long and will be by induction on the virtual cohomological dimension of the virtually finitely generated abelian group. It proceeds as follows:

\begin{enumerate}
\item[-] Assume the Baum-Connes conjecture is true for virtually finitely generated abelian groups of $vcd < n$.
\item[-] Then to prove the Baum-Connes conjecture  for virtually finitely generated abelian groups of $vcd = n$, it suffices to prove the Baum-Connes conjecture  for crystallographic groups of rank $n$ by \ref{vfga} and the inheritance results for the Baum-Connes conjecture .
\item[-] Now prove the Baum-Connes conjecture  for a crystallographic group $G$ of rank $n$ by induction on the order of the holonomy group $F = G/\IZ^{n}$. We distinguish 2 cases: If $G$ contains an infinite normal cyclic subgroup, we can reduce to the induction hypothesis by inheritance results for the Baum-Connes conjecture. This also takes care of the induction start $G = \IZ^{n}$. If not, we prove that $G$ is a Farrell-Hsiang group with respect to a family of subgroups for which we already know the Baum-Connes conjecture by the induction hypothesis and use the transitivity principle.
\item[-] A few examples of rank $2$ have to be treated separately to implement the induction start.
\end{enumerate}

Let us begin with the last item. The examples we have to deal with separately are precisely those crystallographic groups of rank $2$ which do have an infinite cyclic normal subgroup, and it turns out there are not so many of those:

\begin{prop}
\label{2dim}
Let $G$ be a crystallographic group of rank $2$ with an infinite cyclic normal subgroup $C$. Then the holonomy group $F$ of $G$ is either trivial, $\IZ/2$ or $\IZ/2 \oplus \IZ/2$ with the $F$-action on $\IZ^{2} \subset G$ being determined by two maps $\sigma_{i}: F \rightarrow \pm 1$ via $f \cdot (a,b) = (\sigma_{1}(f)a, \sigma_{2}(f)b)$. Furthermore, if $F$ is nontrivial $G$ has either precisely two normal, maximal cyclic subgroups or $F = \IZ/2$ and $G = \IZ^{2} \rtimes_{-1} \IZ/2$ is the infinite dihedral group.
\end{prop}

\begin{proof}
We can assume that $C$ is maximal. Let $\rho \colon F \rightarrow GL_{2}(\IZ)$ be the conjugation action and consider its rationalization $\rho_{\IQ} \colon F \rightarrow GL_{2}(\IQ)$. Since $C$ is invariant under the conjugation action of $F$, $\rho_{\IQ}$ contains an invariant subspace of dimension at least one and hence $\IQ^{2}$ splits as $F$-module into two summands $D_{1}, D_{2}$ with $F$ acting on $D_{i}$ via a homomorphism $\sigma_{i}: F \rightarrow \pm 1$. Since $\rho$ is injective and factors through $\pm 1 \times \pm 1$, we must have $F = 0, \IZ/2$ or $\IZ/2 \oplus \IZ/2$ as claimed. \\
Now $G$ has certainly two distinct maximal cyclic subgroups $C_{1}$ and $C_{2}$, namely the unique maximal integral subgroups of $D_{1}$ and $D_{2}$. If $F$ does not act on $\IZ^{2}$ as multiplication by $-1$ (and hence $F =\IZ/2\IZ$), the rationalization of each other normal infinite maximal cyclic subgroup $C$ has to be either $D_{1}$ or $D_{2}$, and hence $C = C_{1}$ or $C = C_{2}$.  Note that in these cases, $\sigma_{1}$ and $\sigma_{2}$ are always different. 
\end{proof}

\begin{rem}
Note that it may happen that $C_{1} \oplus C_{2} \neq \IZ^{2}$; for example for the representation of $\IZ/2$ on $\IZ^{2}$ given by
\begin{align*}
i(1,0) = (1,0) \\
 i(0,1) = (-1,-1)
\end{align*}
with $i$ the generator of $\IZ/2$ has $C_{1}$ generated by $(1,0)$ and $C_{2}$ generated by $(1,2)$. 
\end{rem}

\begin{prop}
Let $G$ be a crystallographic group of rank two which has a normal infinite cyclic subgroup. Then $G$ satisfies the Baum-Connes conjecture. 
\end{prop}

\begin{proof}
This is \cite[2.15]{BFL}. By \ref{2dim}, there are only very few cases we need to consider. Since we already dealt with $\IZ^{2}$ and $\IZ^{2} \rtimes_{-\Id} \IZ/2$, we can assume that $G$ has precisely two maximal infinite cyclic subgroups which are invariant under the action of the holonomy group $F$. \\
Pick a maximal normal infinite cyclic $C \subset \IZ^{2} \subset G$ which is $F$-invariant. Consider the diagram
\[
\xymatrix{
1 \ar[r] & \IZ^{2} \ar[d] \ar[r] & G \ar[d] \ar[r] & F \ar[r] \ar[d] & 1\\
1 \ar[r] & \IZ^{2}/C \ar[r] & G/C \ar[r] & F \ar[r] & 1 \\
}
\] 
with the vertical maps the quotient projections. The group $G/C$ is virtually abelian and hence projects with finite kernel onto a crystallographic group $\Delta_{C}$ of rank one, say via $p: G/C \rightarrow \Delta_{C}$. Let $\IZ \cong A_{C} \subset \Delta_{C}$ be the subgroup of translations and $F_{\Delta_{C}}$ the holonomy group of $\Delta_{C}$. We have that $\Delta_{C}$ is either $\IZ$ or $\IZ \rtimes_{-\id} \IZ/2\IZ$. There is a commutative diagram
\[
\xymatrix{
1 \ar[r] & \IZ^{2}/C \ar[d] \ar[r] & G/C \ar[d]^{p} \ar[r] & F \ar[r] \ar[d] & 1\\
1 \ar[r] & A_{C} \ar[r] & \Delta_{C} \ar[r] & F_{\Delta_{C}} \ar[r] & 1 \\
}
\] 
From the two diagrams, we obtain maps
\begin{align*}
\nu_{C} \colon G &\rightarrow \Delta_{C} \\
\mu_{C} \colon \IZ^{2} &\rightarrow A_{C}
\end{align*}
Let $d_{G}$ and $d_{\Delta_{C}}$ be word metrics. Then there are constants $K_{1}, K_{2}$ such that
\[
d_{\Delta_{C}}(\nu_{C}(g), \nu_{C}(h)) \leq K_{1}d_{G}(g,h)+K_{2}
\]
since $\nu_{C}$ is onto and each two word metrics are quasiisometric. Furthermore, $K_{i}$ can be chosen independent from $C$ since there are only finitely many (in fact, two) maximal infinite $F$-invariant subgroups $C \subset \IZ^{2} \subset G$. \\
The group $\Delta_{C}$ acts on $\IR$ in the usual way. Equip $\IR$ with the simplicial complex structure with $0$-simplices $\{ n/2 \mid n \in \IZ\}$. Then $\Delta_{C}$ acts simplicially on $\IR$. \\
Let $\ev_{C}: \Delta_{C} \rightarrow \IR$ be given by evaluation at one point. By the \v{S}varc-Milnor Lemma \cite[8.19]{BH}, the estimate $d^{1}(x,y) \leq 2d^{\euc}(x,y)$ and the existence of the constants $K_{1}, K_{2}$ above, we conclude that there are constants $D_{1}, D_{2}$ such that
\[
d^{1}(\ev_{C}(\nu_{C}(g)) , \ev_{C}(\nu_{C}(h))) \leq D_{1} d_{G}(g,h)+ D_{2}
\]
for all $g,h \in G$. \\
Assume by induction that the Baum-Connes conjecture is known for all crystallographic groups of rank $2$ whose holonomy group has lower order than $F$. Let $\calf$ be the family obtained by adding these groups to $\mathcal{VCYC}$. We prove that $G$ is a Farrell-Hsiang group with respect to $\calf$, which implies via the transitivity principle that it satisfies the Baum-Connes conjecture. In fact, the induction is not really an induction since the order of $F$ is $2$ or $4$ anyway, so there are only two cases to consider.\\
Fix a natural number $n$. For an odd prime $p$, which we will eventually pick large enough, let $G_{p} = G/p\IZ^{2}$ and consider the extension
\[
1 \rightarrow \IZ^{2}/p\IZ^{2} \rightarrow G_{p} \rightarrow F \rightarrow 1
\]
which exists since the order of $F$ is either $2$ or $4$.  Let $a_{n} \colon  G \rightarrow G_{p}$ be the projection. This will be the map occuring in the definition of a Farrell-Hsiang group. \\
Pick a hyperelementary subgroup $H \subset G_{p}$ and let $\bar H = a_{n}^{-1}(H)$. If the projection $pr: G_{p} \rightarrow F$ does not send $H$ onto $G$, $\bar H \in \calf$ and we can pick $E_{H} = *$. So assume that $pr(H) = F$. \\
We claim that then, $\IZ^{2}/p\IZ^{2} \cap H$ is cyclic. There is a prime $q$, a $q$-group $Q$  and a cyclic group $D$ of order prime to $q$  such that there is a short exact sequence
\[
1 \rightarrow D \rightarrow H \rightarrow Q \rightarrow 1
\]
If $p \neq q$, $\IZ^{2}/p\IZ^{2} \cap H$ is a subgroup of $D$ and hence cyclic. If $p = q$, assume that $\IZ^{2}/p\IZ^{2}  \cap H = \IZ^{2}/p\IZ^{2}$ which is the only chance for $\IZ^{2}/p\IZ^{2}  \cap H$ to be noncyclic. Since $H$ projects onto $F$, this forces $H = G_{p}$. But then the composite 
\[
\IZ^{2}/p\IZ^{2} \rightarrow G_{p} = H \rightarrow Q
\]
 must be an isomorphism since $p$ is odd and $F$ has order $2$ or $4$. It follows that $\IZ^{2}/p\IZ^{2} \rightarrow G_{p}$ splits and that the conjugation action of $F$ on $\IZ^{2}/p\IZ^{2}$ is trivial. Since $p$ is odd, this contradicts the fact from \ref{2dim} that $F$ acts on $\IZ^{2}/p\IZ^{2}$ via multiplication by $-1$ on a cyclic subgroup of $\IZ^{2}/p\IZ^{2}$. Hence $\IZ^{2}/p\IZ^{2} \cap H$ must be cyclic. \\
 Recall that $\IZ^{2}$ has two maximal infinite cyclic subgroups $C_{1}, C_{2}$ on which $F$ acts by multiplication by $\pm 1$. Integrally, it may be $C_{1} \oplus C_{2} \neq \IZ^{2}$, but after reduction$\mod p$, we must have
 \[
 \IZ^{2}/p\IZ^{2} = C_{1}/p \oplus C_{2}/p
 \]
where $F$ acts by multiplication by $\pm 1$ and differently on the two summands. Since $H$ projects onto $F$, $\IZ^{2}/p\IZ^{2} \cap H$ is invariant under the $F$-action and hence contained in one of the two summands $C_{i}/p$. Let $C = C_{i}$ for this $i$. \\
Let $\xi \colon \IZ^{2} \rightarrow \IZ^{2}/C$ be the projection. We claim that
\[
\xi(\bar H \cap \IZ^{2}) \subset p(A/C)
\]
Indeed, $(\bar H \cap \IZ^{2})/p$ is contained in $C/p$, so $\xi(\bar H \cap \IZ^{2}) \subset p(A/C)$.  Furthermore, we have
\[
\mu_C( \bar H \cap \IZ^{2}) \subset pA_{C}
\]
since the map $\phi \colon \IZ^{2}/C \rightarrow A_{C}$ is injective and $\mu_{C} = \phi \circ \xi$. This forces
\[
\mu_{C}(\bar H) \cap A_{C} \subset pA_{C}
\]
since $\mu_C( \bar H \cap \IZ^{2})$ is a finite-index subgroup of $\mu_{C}(\bar H) \cap A_{C}$ of even index since the order of $F$ is $2$ or $4$.
We now need a $p$-expansive map $\phi: \Delta_{C} \rightarrow \Delta_{C}$, together with a $\phi$-invariant affine map
\[
a_{(p,w)} \colon \IR \rightarrow \IR, x \mapsto px+w
\]
such that
\[
\nu_{C}(\bar H) \subset \im(\phi)
\]
If $\Delta_{C} = \IZ$, just take $\phi$ to be multiplication by $p$ and $u = 0$. If $\Delta = \IZ \rtimes_{-\id} \IZ/2$, let $t$ be the generator of $\IZ/2$ and for $u \in \IZ^{n}$ define $\phi_{u} \colon \Delta \rightarrow \Delta$ by $\phi_{u}(t) = ut$ and $\phi_{u}(x) = px$ for $x \in \IZ^{n}$. It is easy to check that this yields a group homomorphism as claimed. 
If $\nu_{C}(\bar H)$ projects to $0$ in $\IZ/2$, we can take $\phi_{0}$. If not, pick $u \in \IZ^{n}$ such that $ut \in \bar H$. We claim that $\phi_{u}$ has the desired properties. If $h \in \nu_{C}(\bar H)$, we either have $h \in \nu_{C}(\bar H) \cap \IZ$, so $h$ is divisible by $p$ and in the image of $\phi_{u}$, or $h = xt$ for $x \in \IZ$. But since $ut \in \im(\phi)$ and $(xt)(ut) = x-u \in \nu_{C}(\bar H) \subset p \IZ$ is also in the image of $\phi$, $xt \in \im(\phi)$ as claimed. Finally, $a_{(p, u/2)}$ is $\phi$-invariant. \\
Let $E_{H}$ be the simplicial complex with underlying space $\IR$ and $0$-simplices $\{ n/2 \mid n \in \IZ\}$. The group $\Delta_{C}$ acts on $E_{H}$. Define
\[
f_{H} \colon G \stackrel{\nu_{C}}{\longrightarrow} \Delta_{C} \stackrel{\ev_{C}}{\longrightarrow} \IR \stackrel{(a_{(p,w)})^{-1}}{\longrightarrow} \IR = E_{H}
\]
Equip $E_{H}$ with the $\bar H$-action where $h \in \bar H$ acts by the unique $g \in \Delta_{C}$ which satisfies $\phi(g) = \nu_{C}(h)$. It is easy to check that $f_{H}$ is $\bar H$-invariant and $E_{H}$ has virtually cyclic isotropy. Finally, we estimate for $g,h \in G$ with $d_{G}(g,h) \leq n$
\begin{align*}
d^{1}(f_{H}(g), f_{H}(h)) &\leq 2d^{\euc}(f_{H}(g), f_{H}(h)) \\
&= \frac{2}{p} d^{\euc}(\ev \circ \nu_{C}(g),\ev \circ \nu_{C}(h)) \\
&\leq \frac{2}{p}(D_{1}d_{G}(g,h)+D_{2}) \\
&\leq \frac{2}{p}(D_{1}n+D_{2})
\end{align*}
Now pick $p$ large enough such that the last expression is at most $\frac{1}{n}$. This finishes the proof.
\end{proof}

Now we can state one more inheritance property we will need.

\begin{prop}
\label{inherit}
Let
\[
1 \rightarrow V \rightarrow G \rightarrow H \rightarrow 1 
\]
be a short exact sequences of groups with $V$ virtually cyclic such that $H$ satisfies the Baum-Connes conjecture. Then also $G$ satisfies the Baum-Connes conjecture.
\end{prop}

\begin{proof}
This is essentially the Baum-Connes version of \cite[2.22]{BFL}. Let $\calf$ be the family of subgroups of $G$ obtained by pulling back the family $\mathcal{VCYC}$ of subgroups of $H$. If all $F \in \calf$ satisfy the Baum-Connes conjecture, the transitivity principle tells us that $G$ satisfies the Baum-Connes conjecture. Each such $F$ sits in a short exact sequence
\[
1 \rightarrow V \rightarrow F \rightarrow K \rightarrow 1 
\]
with also $K$ virtually cyclic. We claim that $F$ satisfies the Baum-Connes conjecture. This is clear if $V$ is finite, so let $C \subset V$ be an infinite cyclic normal subgroup. Let $C' = \cap_{\phi \in \Aut(V)} \phi(C)$. Since $V$ contains only finitely many subgroups of a given index, this is a finite intersection, and $C'$ is characteristic and hence in particular normal in $F$. The group $F$ is virtually finitely generated abelian and hence 
admits a surjection onto a crystallographic group $G$ of rank $2$ with finite kernel by \ref{vfga}. Then the image of $C'$ is an infinite cyclic normal subgroup of $G$ and hence $G$ is a Farrell-Hsiang group with respect to the family $\mathcal{VCYC}$. It follows that $G$ satisfies the Baum-Connes conjecture.
\end{proof}

In the induction step, we will use the following:

\begin{prop}
\label{numbers}
Let $\Delta$ be a crystallographic group and let
\[
0 \rightarrow \IZ^{n} \rightarrow \Delta \rightarrow G \rightarrow 1
\]
be the canonical group extension with $G$ finite of order $2^{k} l$ with $l$ odd. Let $p$ be a prime with $(p,l) = 1$, $(p-1, l) = 1$ and $p = 3 \mod 4$. Consider for some  $r \in \IN$ the extension
\[
0 \rightarrow \IZ^{n}/p^{r}\IZ^{n} \rightarrow \Delta/p^{r} \IZ^n \rightarrow G \rightarrow 1
\]
Let $H \subset \Delta/p^{r}\IZ^n$ be a hyperelementary subgroup such that $H \rightarrow G$ is surjective but not injective. Then $\IZ^{n}$ has an infinite cyclic subgroup which is invariant under the $G$-action.
\end{prop}

\begin{proof}
We find a number $k$ and a prime $q$ such that there is an extension
\[
1 \rightarrow C_{s} \rightarrow H \rightarrow Q \rightarrow 1
\]
with $C_{s}$ cyclic of order $s$ and $Q$ a $q$-group, with $(q,s) = 1$. We know that $H \cap \IZ^n/p^{r}\IZ^{n} \neq 0$. \\
First let us consider the case that $\IZ^{n}/p^{r}\IZ^{n} \cap C_{s} = 0$. The order of $H \cap \IZ^{n}/p^{r}\IZ^{n}$ divides both $q^{i}s$ and $p^{r}$. Also, $H \cap \IZ^{n}/p^{r}\IZ^{n}$ cannot contain elements of order dividing $s$ since $C_{s} \cap \IZ^{n}/p^{r}\IZ^{n} = 0$ and $C_{s}$ is the group of elements in $H$ whose order divides $s$. It follows that $q = p$. So, by our assumption on $p$, we must have $(q, \left\vert G \right\vert) = 1$. Since $H \rightarrow G$ is onto and all elements of order $q^{j}$ have to go to zero since $(q, \left\vert G \right\vert) = 1$, this implies that $C_{s} \rightarrow G$ is onto.
So $G$ is cyclic and we have an extension 
\[
1 \rightarrow \IZ^{n} \rightarrow \Delta \rightarrow \IZ/m\IZ \rightarrow 1
\]
Represent $\Delta$ as a group of isometries of $\IR^{n}$. Let $M \in O(n)$ be a generator of $\IZ/m\IZ$. Regarding $\IZ^{n}$ as sitting inside $\IR^{n}$, pick a basis $v_{1}, \dots v_{n} \in \IZ^{n}$ of $\IR^{n}$. Note that $\IR^{n}$ does not necessarily carry the standard scalar product. There must be an $i$ such that $v = v_{i}+Mv_{i}+ \dots + M^{m-1}v_{i}$ is not the zero vector since $1+x+x^{2}+ \dots + x^{m-1}$ cannot be the minimal polynomial of $M$ since it lacks the Eigenvalue $1$. Then $v$ is an eigenvector of $M$ for the eigenvalue $1$, and hence $v$ is an eigenvector of all orthogonal matrices occuring in $\Delta$. Since we can regard $M$ as a matrix with integral entries, it follows that $v \in \IZ^{n}$. Clearly, $v$ spans a normal infinite cyclic subgroup. \\
Now assume that $\IZ^{n}/p^{r}\IZ^{n} \cap C_{s} \neq 0$. Then no element of $H$ except those whose order divides $s$ can go to $0$ in $G$ since $(s,p^{r}) \neq 1$ forces $q \neq p$. So $H \cap\IZ^{n}/p^{r}\IZ^{n}$ is cyclic of order, say,  $p^{v}$. Now consider the conjugation-representation
\[
\rho: G \rightarrow \Aut(H \cap \IZ^{n}/p^{r}\IZ^{n})
\]
obtained from the short exact sequence
\[
1 \rightarrow H \cap \IZ^{n}/p^{r}\IZ^{n} \rightarrow H \rightarrow G \rightarrow 1
\]
and let $G_{0}$ be the image of $\rho$. We must have
\[
\left\vert G_{0} \right\vert \mid (\left\vert  \Aut(H \cap \IZ^{n}/p^{r}\IZ^{n}) \right\vert, \left\vert G \right\vert) = ((p-1)p^{v-1}, 2^{k}l) = 1 \text{ or } 2
\]
by our assumptions on $p$. Let $F$ be the kernel of $\rho$. Then clearly, $(\IZ^{n}/p^{r}\IZ^{n})^{F} \neq 0$ and so $(\IZ^{n})^{F} \neq 0$ since there is an exact sequence
\[
(\IZ^{n})^{F} \rightarrow (\IZ^{n}/p^{r}\IZ^{n})^{F}   \rightarrow H^{1}(F, p^{r}\IZ^{n}) \stackrel{f}{\rightarrow} H^{1}(F, \IZ^{n})
\]
where the map $f$ is multiplication by $p^{r}$ and hence an isomorphism since $p$ does not divide the order of $F$, so $(\IZ^{n})^{F} \rightarrow (\IZ^{n}/p^{r}\IZ^{n})^{F}$ has to be onto. \\
Let $x \in \IZ^{n}$ be an $F$-fixed vector and $g \in G, \neq F$. Then either $x+\rho(g)(x)$ is nonzero, and then $G$-fixed, or $g$ and hence also all other $g' \in G-F$ act as multiplication by $-1$ on $x$. Then the subgroup of $\IZ^{n}$ generated by $x$ is $G$-invariant.
\end{proof}

Now we are ready for the induction step:

\begin{prop}
Let $G$ be a crystallographic group of rank $n$. Assume that the Baum-Connes conjecture is known for all virtually finitely generated abelian groups of virtual cohomological dimension at most $n-1$ and all crystallographic groups of rank $n$ whose holonomy has smaller rank than the holonomy of $G$. Then $G$ satisfies the Baum-Connes conjecture as well.
\end{prop}

\begin{proof}
\fxnote{this is the proof of part iii) of 11.10; indicate this somehow}
This is \cite[Section 2.3]{BFL}. Let $G$ be a crystallographic group of rank $n$, and let
\[
0 \rightarrow A \rightarrow G \rightarrow F \rightarrow 1
\]
be the associated extension with $A$ finitely generated free abelian. Let $\calf$ be the family of subgroups of $G$ consisting of all subgroups of $G$ of rank at most $n-1$ and all subgroups of $G$ of rank $n$ whose holonomy group has lower order than $F$. By induction hypothesis and the transitivity principle, it suffices to see that $G$ satisfies the Baum-Connes conjecture with respect to the family $\calf$ since the groups in $\calf$ are either virtually finitely generated abelian of rank at most $n-1$ or crystallographic of rank $n$ with lower-order holonomy group.  We proceed to prove that $G$ is a Farrell-Hsiang group with respect to $\calf$.\\
Equip $G$ with some word metric $d_{G}$. Evaluation at some point of $\IR^{n}$ gives a map
\[
\ev \colon G \rightarrow \IR^{n}
\]
and by the \v{S}varc-Milnor Lemma \cite[8.19]{BH}, we find constants $C_{1}, C_{2} > 0$ such that
\[
d(\ev(g), \ev(h)) \leq C_{1} d_{G}(g,h)+ C_{2}.
\]
Fix a natural number $n$. We can equip $\IR^{n}$ with the structure of a simplicial complex such that the $G$-action is simplicial; denote $\IR^{n}$ with this simplicial structure by $E$. All simplicial complexes occuring from now on are $E$, but with varying group actions. \fxnote{reference for this claim?} \\
Since we want to employ \ref{numbers}, write $\left\vert F \right\vert = 2^{k}l$ with $l$ odd and pick a prime $p$ with
\begin{align*}
(p,l) &=  1\\
(p-1,l) &= 1 \\
p &= 3 \mod 4
\end{align*}
There are infinitely many such primes: Pick a number $m$ which is $2$ modulo $l$ and $3$ modulo $4$, which is possible since $(4,l) = 1$, and apply Dirichlet's theorem to conclude that there are infinitely many primes $p$ with
\[
p \mod 4l = m
\]
and each such prime will satisfy $p = 2 \mod l$ and $p = 3 \mod 4$, which forces $(p,l) = (p-1,l) = 1$ since $l$ is at least $3$.
We will eventually use this to choose $p$ very large in order to obtain the necessary contracting properties. Finally, since $(p, 2^{k}l) = 1$, we find a number $r > 0$ such that
\[
p^{r} =1 \mod \left\vert F \right\vert
\]
Since $A$ is normal in $G$ and $p^{r}A$ is a characteristic subgroup of $A$, also $p^{r}A$ is normal in $G$. Let
\begin{align*}
A_{p^{r}} &= A/(p^{r}A) \\
G_{p^{r}} &= G/(p^{r}A)
\end{align*}
The projection
\[
a_{p^{r}}: G \rightarrow G_{p^{r}}
\]
will play the role of the map $\alpha_{n}$ appearing in the definition of a Farrell-Hsiang group. \\
Let $H \subset G_{p^{r}}$ be a hyperelementary subgroup (for the Baum-Connes conjecture, elementary would suffice) and let $\overline{H}$ be its preimage under $a_{p^{r}}$.
If under the projection $pr_{p^{r}}: G_{p^{r}} \rightarrow F$, $H$ is not sent onto all of $F$, $\overline{H}$ is a crystallographic group of the same rank as $G$, but with lower-order holonomy and is hence contained in the family $\calf$. This allows us to set $E_{H} = *$. So we now assume that $pr_{p^{r}}(H) = F$. \\
If $pr_{p^{r}}: H \rightarrow F$ is not an isomorphism, we can employ \ref{numbers} to conclude that $G$ has a normal infinite cyclic subgroup $C$. Then \ref{inherit} shows that it suffices to see the Baum-Connes conjecture for $G/C$ which is a virtually abelian group of virtual cohomological dimension $1$ lower than the rank of $G$, for which we know the Baum-Connes conjecture by induction hypothesis. \\
So we can assume that $pr_{p^{r}}: H \rightarrow F$ is an isomorphism. In particular, the short exact sequence
\[
0 \rightarrow A_{p^{r}} \rightarrow G_{p^{r}} \rightarrow F \rightarrow 1 
\]
splits. Since $p^{r} = 1 \mod \left\vert F \right\vert$, we find by \ref{exp} a $p^{r}$-expansive map $\phi: G \rightarrow G$. The composite $a_{p^{r}} \circ \phi: G \rightarrow G_{p^{r}}$ vanishes on $A$ and hence factors over $\overline{\phi}: F \rightarrow G_{p^{r}}$, i.e. we have
\begin{align*}
a_{p^{r}} \circ \phi &= \overline{\phi} \circ pr \\
pr_{p^{r}} \circ \overline{\phi} &= \id_{F}
\end{align*}
where $pr: G \rightarrow F$ is the projection. Hence $\overline{\phi}$ is another splitting of 
\[
0 \rightarrow A_{p^{r}} \rightarrow G_{p^{r}} \rightarrow F \rightarrow 1 
\]
We conclude from \cite[IV.2.3]{brown} and $H^{1}(F; A_{p^{r}}) = 0$ that the two splittings are conjugated and hence that $H$ and $\im(a_{p^{r}} \circ \phi) = \im(\overline{\phi})$ are conjugated in $G_{p^{r}}$. Hence also their preimages in $G$ are conjugated. \\
If we can handle $ \im(\overline{\phi})$, i.e. find a suitable map $f: G \rightarrow E_{ \im(\overline{\phi})}$, then we can define an $H$-action on $E_{ \im(\overline{\phi})}$ by conjugating the $\im(\overline{\phi})$-action; precomposing $f$ with the conjugation which maps $\im(\overline{\phi})$ to $H$ then yields a map $f_{H}: G \rightarrow E_{H}$ with the desired properties. \\
So we may assume that $H = \im(\overline{\phi})$ and claim that then
\[
\overline{H} = \im{\phi}
\]
Indeed, if $h \in \overline{H}$, we find $g \in G$ with $a_{p^{r}}(h) = a_{p^{r}}(\phi(g))$ since $H = \im(a_{p^{r}} \circ \phi)$ and hence $h^{-1}\phi(g)$ is in the kernel of $a_{p^{r}}$ which is $p^{r}A$, which by definition is in the image of $\phi$.  Conversely, if $g = \phi(s) \in \im(\phi)$, we have $a_{p^{r}}(g) = a_{p^{r}}(\phi(s)) \in H$ and so $g \in \overline{H}$.
The reader should compare this point to the condition $$r(H) \subset p\IZ$$ appearing in the proof for $\IZ^{2}$; it is the analogue in our more complicated situation. \\
Finally, there is by \ref{exp} a $u \in \IR^{n}$ such that the affine map
\[
f = f_{p^{r},u}: \IR^{n} \rightarrow \IR^{n}, x \mapsto p^{r}x+u
\]
is $\phi$-equivariant, i.e. satisfies 
\[
f(gx) = \phi(g)f(x)
\]
for all $g \in G$, $x \in \IR^{n}$. Now consider the map
\[
f_{H}: G \stackrel{\ev}{\rightarrow} \IR^{n} \stackrel{f^{-1}}{\rightarrow} \IR^{n} = E_{H}
\]
This map has the required contracting properties if $p$ is large enough: We estimate for $g,h \in G$ with $d_{G}(g,h) < n$
\begin{align*}
d^{\euc}(f_{H}(g), f_{H}(h)) &= d^{\euc}(f^{-1}(\ev(g)), f^{-1}(\ev(h))) \\
&\leq \frac{1}{p^{r}} d(\ev(g), \ev(h)) \\
&\leq \frac{1}{p^{r}}(C_{1}d_{G}(g,h)+C_{2}) \\
&\leq \frac{1}{p^{r}}(C_{1}n+C_{2})
\end{align*}
The last line can be made arbitrarily small by picking large $p$. Unfortunately, we have to consider the $l^{1}$-metric on $E$ and not the euclidean one, so we are not yet done.  But since the simplicial complex structure on $\IR^{n}$ is locally finite, we find that for given $\epsilon > 0$ there is $\delta > 0$ such that
\[
d^{\euc}(x,y) \leq \delta \Rightarrow d^{l^{1}}(x,y) \leq \epsilon
\]
Picking the $\delta$ for $\epsilon = \frac{1}{n}$ and choosing $p$ large enough such that 
\[
\frac{1}{p^{r}}(C_{1}n+C_{2}) \leq \delta
\]
then shows that $f_{H}$ has the required contracting properties. Finally, we have to define an $\overline{H}$-action on $E_{H}$. We do this by letting $h \in \overline{H}$ act on $E_{H}$ as the unique $g \in G$ with $\phi(g) = h$ which is possible since $\im(\phi) = \overline{H}$. Since $f$ is $\phi$-equivariant, it is easy to check that $f_{H}$ is $\overline{H}$-equivariant. Finally, since the isotropy of $G$ acting on $\IR^{n}$ is finite, also the isotropy of $\overline{H}$ acting on $E$ is finite and hence contained in $\calf$. Hence $G$ is a Farrell-Hsiang group with respect to $\calf$. Since we know the Baum-Connes conjecture for all groups in $\calf$, the Transitivity principle implies that $G$ satisfies the Baum-Connes conjecture.
\end{proof}

\begin{rem}
The reader should note that nothing interesting happens once $F$ is not elementary itself. This can be understood in the context of \ref{trivialFH}: All preimages of elementary subgroups of $F$ satisfy the Baum-Connes conjecture by induction hypothesis, and hence so does $G$.  
\end{rem}

%% file: chapters/outlook.tex
\chapter{Outlook}

\section{The transfer}

An immediate question is whether the transfer from \cite{BLR} or \cite{CAT0} can be carried over to the topological $K$-theory setting. In the construction, three things stand out which do not directly work in topological $K$-theory:

\begin{enumerate}
\item The Waldhausen construction, cf. \cite{Waldhausen}, plays an important role, and it is not immediately clear how to generalize the Waldhausen construction to topological $K$-theory. 
\item The singular chain complex of a space $X$ plays an important role. In the topological setting, we would like that the complex singular chain complex is a complex of pre-Hilbert spaces and bounded operators between them, which it is not. Any potential construction along the lines of \cite{BLR} should run into continuity problems at this point.
\item Finally, the transfer in \cite{BLR} depends on the fact that each morphism is actually controlled. In the topological setup, this is no longer true, and we do not have completion devices at our disposal since the (rather weird) categories considered in \cite{BLR} do not carry norms.
\end{enumerate}

The third point is actually the most serious one, whereas the first one is in fact a non-factor: We seek to prove a surjectivity result, so we may work with algebraic $K$-theory in the first place since there is a natural surjection $K_i^{\alg}(\C) \rightarrow K_i(\C), i = 0,1,$ for any $C^*$-category $\C$. The second one needs some treatment; it could for example be solved by additionally assuming that the space $X$ under consideration is a finite simplicial complex. This, however, would severely restrict the usefulness of the transfer. Altogether, it seems that a genuine new idea is necessary to obtain a transfer in the topological $K$-theory setup under the conditions of \cite{BLR}. \\
In contrast, the transfer we have considered is substantially easier: The necessary section is essentially given by the formal difference of two functors, where ``formal difference''
 only makes sense in $K$-theory. This avoids the three difficulties above altogether. \\

%% file: misc/abstract.tex
\pdfbookmark[1]{Abstract}{Abstract}
\begingroup
\let\clearpage\relax
\let\cleardoublepage\relax
\let\cleardoublepage\relax

\chapter*{Abstract}
The Farrell-Jones conjecture in algebraic $K$-theory proposes a formula for the $K$-groups of group rings. Similarly, the Baum-Connes conjecture proposes a formula for the topological $K$-groups of the reduced group $C^{*}$-algebra. Both conjectures are known for large classes of groups. We present, for the first time, a proof for some cases of the Baum-Connes conjecture which employs the methods used for proving cases of the Farrell-Jones conjecture. The main ingredients are controlled algebra and induction theorems for group representations.

\vfill

\pdfbookmark[1]{Zusammenfassung}{Zusammenfassung}
\chapter*{Zusammenfassung}
Die Farrell-Jones-Vermutung in der algebraischen $K$-Theorie sagt vor\-her, wie sich die $K$-Gruppen von Gruppenringen verhalten. Die Baum-Connes-Vermutung ist eine \"ahnliche Vermutung \"uber die topologi\-sche $K$-Theorie von Gruppen-$C^*$-Algebren. Beide Vermutungen sind f\"ur viele Gruppen bekannt. In dieser Arbeit wird ein Beweis einiger F\"alle der Baum-Connes-Vermutung mit den Beweismethoden der \\ Farrell-Jones-Vermutung gegeben. Die wichtigsten Hilfsmittel sind kontrollierte Algebra und Induktionstheoreme aus der Darstellungstheorie endlicher Gruppen.

\endgroup			

\vfill

%% file: thesislenhardt.bbl
\begin{thebibliography}{BLR08b}

\bibitem[Aus65]{Auslander}
Louis Auslander.
\newblock An account of the theory of crystallographic groups.
\newblock {\em Proc. Amer. Math. Soc.}, 16:1230--1236, 1965.

\bibitem[BEL08]{BEL}
Arthur Bartels, Siegfried Echterhoff, and Wolfgang L{\"u}ck.
\newblock Inheritance of isomorphism conjectures under colimits.
\newblock In {\em {$K$}-theory and noncommutative geometry}, EMS Ser. Congr.
  Rep., pages 41--70. Eur. Math. Soc., Z\"urich, 2008.

\bibitem[BFJR04]{BFJR}
Arthur Bartels, Tom Farrell, Lowell Jones, and Holger Reich.
\newblock On the isomorphism conjecture in algebraic {$K$}-theory.
\newblock {\em Topology}, 43(1):157--213, 2004.

\bibitem[BFL]{BFL}
Arthur Bartels, Tom Farrell, and Wolfgang L\"uck.
\newblock The {F}arrell-{J}ones conjecture for cocompact lattices in virtually
  connected lie groups.
\newblock {P}reprint, 2011, arXiv:1101.0469.

\bibitem[BH99]{BH}
Martin~R. Bridson and Andr{\'e} Haefliger.
\newblock {\em Metric spaces of non-positive curvature}, volume 319 of {\em
  Grundlehren der Mathematischen Wissenschaften [Fundamental Principles of
  Mathematical Sciences]}.
\newblock Springer-Verlag, Berlin, 1999.

\bibitem[BL]{FH}
Arthur Bartels and Wolfgang L\"uck.
\newblock The {F}arrell-{H}siang method revisited.
\newblock {P}reprint, 2011, to appear in \emph{{M}ath. {A}nnalen},
  arXiv:1101.0466.

\bibitem[BL07]{BLind}
Arthur Bartels and Wolfgang L{\"u}ck.
\newblock Induction theorems and isomorphism conjectures for {$K$}- and
  {$L$}-theory.
\newblock {\em Forum Math.}, 19(3):379--406, 2007.

\bibitem[BL12a]{CAT0}
Arthur Bartels and Wolfgang L\"uck.
\newblock The {B}orel conjecture for hyperbolic and {$CAT(0)$}-groups.
\newblock {\em Annals of Mathematics}, 175:631--689, 2012.

\bibitem[BL12b]{catflow}
Arthur Bartels and Wolfgang L\"uck.
\newblock Geodesic flow for {CAT}(0)-groups.
\newblock {\em Geometry and Topology}, 16:1345--1391, 2012.

\bibitem[Bla98]{Blackadar}
Bruce Blackadar.
\newblock {\em {$K$}-theory for operator algebras}, volume~5 of {\em
  Mathematical Sciences Research Institute Publications}.
\newblock Cambridge University Press, Cambridge, second edition, 1998.

\bibitem[BLR08a]{hypercover}
Arthur Bartels, Wolfgang L{\"u}ck, and Holger Reich.
\newblock Equivariant covers for hyperbolic groups.
\newblock {\em Geom. Topol.}, 12(3):1799--1882, 2008.

\bibitem[BLR08b]{BLR}
Arthur Bartels, Wolfgang L{\"u}ck, and Holger Reich.
\newblock The {$K$}-theoretic {F}arrell-{J}ones conjecture for hyperbolic
  groups.
\newblock {\em Invent. Math.}, 172(1):29--70, 2008.

\bibitem[BR05]{BR}
Arthur Bartels and Holger Reich.
\newblock On the {F}arrell-{J}ones conjecture for higher algebraic
  {$K$}-theory.
\newblock {\em J. Amer. Math. Soc.}, 18(3):501--545 (electronic), 2005.

\bibitem[BR07]{coeff}
Arthur Bartels and Holger Reich.
\newblock Coefficients for the {F}arrell-{J}ones conjecture.
\newblock {\em Adv. Math.}, 209(1):337--362, 2007.

\bibitem[Bro94]{brown}
Kenneth~S. Brown.
\newblock {\em Cohomology of groups}, volume~87 of {\em Graduate Texts in
  Mathematics}.
\newblock Springer-Verlag, New York, 1994.
\newblock Corrected reprint of the 1982 original.

\bibitem[CE01]{inheritanceBC}
J{\'e}r{\^o}me Chabert and Siegfried Echterhoff.
\newblock Permanence properties of the {B}aum-{C}onnes conjecture.
\newblock {\em Doc. Math.}, 6:127--183 (electronic), 2001.

\bibitem[CP97]{karoubifil}
Manuel C{\'a}rdenas and Erik~K. Pedersen.
\newblock On the {K}aroubi filtration of a category.
\newblock {\em $K$-Theory}, 12(2):165--191, 1997.

\bibitem[CR06]{curtis-reiner}
Charles~W. Curtis and Irving Reiner.
\newblock {\em Representation theory of finite groups and associative
  algebras}.
\newblock AMS Chelsea Publishing, Providence, RI, 2006.
\newblock Reprint of the 1962 original.

\bibitem[DL98]{DL}
James~F. Davis and Wolfgang L{\"u}ck.
\newblock Spaces over a category and assembly maps in isomorphism conjectures
  in {$K$}- and {$L$}-theory.
\newblock {\em $K$-Theory}, 15(3):201--252, 1998.

\bibitem[Far81]{Farkas}
Daniel~R. Farkas.
\newblock Crystallographic groups and their mathematics.
\newblock {\em Rocky Mountain J. Math.}, 11(4):511--551, 1981.

\bibitem[FH78]{farrellhsiang}
Tom Farrell and Wu-Chung Hsiang.
\newblock The topological-{E}uclidean space form problem.
\newblock {\em Invent. Math.}, 45(2):181--192, 1978.

\bibitem[FJ86]{kd1}
Tom Farrell and Lowell Jones.
\newblock {$K$}-theory and dynamics. {I}.
\newblock {\em Ann. of Math. (2)}, 124(3):531--569, 1986.

\bibitem[FJ87]{kd2}
Tom Farrell and Lowell Jones.
\newblock {$K$}-theory and dynamics. {II}.
\newblock {\em Ann. of Math. (2)}, 126(3):451--493, 1987.

\bibitem[GHT00]{trout}
Erik Guentner, Nigel Higson, and Jody Trout.
\newblock Equivariant {$E$}-theory for {$C^*$}-algebras.
\newblock {\em Mem. Amer. Math. Soc.}, 148(703):viii+86, 2000.

\bibitem[Gra77]{graysonendo}
Daniel~R. Grayson.
\newblock The {$K$}-theory of endomorphisms.
\newblock {\em J. Algebra}, 48(2):439--446, 1977.

\bibitem[Hig90]{kkprimer}
Nigel Higson.
\newblock A primer on {$KK$}-theory.
\newblock In {\em Operator theory: operator algebras and applications, {P}art 1
  ({D}urham, {NH}, 1988)}, volume~51 of {\em Proc. Sympos. Pure Math.}, pages
  239--283. Amer. Math. Soc., Providence, RI, 1990.

\bibitem[HK01]{haagerup}
Nigel Higson and Gennadi Kasparov.
\newblock {$E$}-theory and {$KK$}-theory for groups which act properly and
  isometrically on {H}ilbert space.
\newblock {\em Invent. Math.}, 144(1):23--74, 2001.

\bibitem[HP04]{PH}
Ian Hambleton and Erik~K. Pedersen.
\newblock Identifying assembly maps in {$K$}- and {$L$}-theory.
\newblock {\em Math. Ann.}, 328(1-2):27--57, 2004.

\bibitem[HPR97]{hpr}
Nigel Higson, Erik~Kj{\ae}r Pedersen, and John Roe.
\newblock {$C^\ast$}-algebras and controlled topology.
\newblock {\em $K$-Theory}, 11(3):209--239, 1997.

\bibitem[HR00]{HigsonRoe}
Nigel Higson and John Roe.
\newblock {\em Analytic {$K$}-homology}.
\newblock Oxford Mathematical Monographs. Oxford University Press, Oxford,
  2000.
\newblock Oxford Science Publications.

\bibitem[Kar08]{kar}
Max Karoubi.
\newblock {\em {$K$}-theory}.
\newblock Classics in Mathematics. Springer-Verlag, Berlin, 2008.
\newblock An introduction, Reprint of the 1978 edition, With a new postface by
  the author and a list of errata.

\bibitem[Kas88]{Kasparov}
Gennadi Kasparov.
\newblock Equivariant {$KK$}-theory and the {N}ovikov conjecture.
\newblock {\em Invent. Math.}, 91(1):147--201, 1988.

\bibitem[Laf02]{lafforgue}
Vincent Lafforgue.
\newblock {$K$}-th\'eorie bivariante pour les alg\`ebres de {B}anach et
  conjecture de {B}aum-{C}onnes.
\newblock {\em Invent. Math.}, 149(1):1--95, 2002.

\bibitem[Lan95]{Lance}
E.~Christopher Lance.
\newblock {\em Hilbert {$C^*$}-modules}, volume 210 of {\em London Mathematical
  Society Lecture Note Series}.
\newblock Cambridge University Press, Cambridge, 1995.
\newblock A toolkit for operator algebraists.

\bibitem[LM89]{LM}
H.~Blaine Lawson, Jr. and Marie-Louise Michelsohn.
\newblock {\em Spin geometry}, volume~38 of {\em Princeton Mathematical
  Series}.
\newblock Princeton University Press, Princeton, NJ, 1989.

\bibitem[LR05]{KThandbook}
Wolfgang L{\"u}ck and Holger Reich.
\newblock The {B}aum-{C}onnes and the {F}arrell-{J}ones conjectures in {$K$}-
  and {$L$}-theory.
\newblock In {\em Handbook of {$K$}-theory. {V}ol. 1, 2}, pages 703--842.
  Springer, Berlin, 2005.

\bibitem[Mit02]{cstarcat}
Paul Mitchener.
\newblock ${C}^*$-categories.
\newblock {\em Proc. London Math. Society}, 84:375--404, 2002.

\bibitem[Mur90]{Murphy}
Gerard~J. Murphy.
\newblock {\em {$C^*$}-algebras and operator theory}.
\newblock Academic Press Inc., Boston, MA, 1990.

\bibitem[MY02]{mineyev}
Igor Mineyev and Guoliang Yu.
\newblock The {B}aum-{C}onnes conjecture for hyperbolic groups.
\newblock {\em Invent. Math.}, 149(1):97--122, 2002.

\bibitem[OO01]{oyono}
Herv{\'e} Oyono-Oyono.
\newblock Baum-{C}onnes conjecture and extensions.
\newblock {\em J. Reine Angew. Math.}, 532:133--149, 2001.

\bibitem[Qui73]{quillenktheory}
Daniel Quillen.
\newblock Higher algebraic {$K$}-theory. {I}.
\newblock In {\em Algebraic {$K$}-theory, {I}: {H}igher {$K$}-theories ({P}roc.
  {C}onf., {B}attelle {M}emorial {I}nst., {S}eattle, {W}ash., 1972)}, pages
  85--147. Lecture Notes in Math., Vol. 341. Springer, Berlin, 1973.

\bibitem[Qui12]{quinn}
Frank Quinn.
\newblock Algebraic {$K$}-theory over virtually abelian groups.
\newblock {\em J. Pure Appl. Algebra}, 216(1):170--183, 2012.

\bibitem[RLL00]{RLL}
Mikael R{\o}rdam, Flemming Larsen, and Niels Laustsen.
\newblock {\em An introduction to {$K$}-theory for {$C^*$}-algebras}, volume~49
  of {\em London Mathematical Society Student Texts}.
\newblock Cambridge University Press, Cambridge, 2000.

\bibitem[Ros94]{Ros}
Jonathan Rosenberg.
\newblock {\em Algebraic {$K$}-theory and its applications}, volume 147 of {\em
  Graduate Texts in Mathematics}.
\newblock Springer-Verlag, New York, 1994.

\bibitem[Rud91]{rudin}
Walter Rudin.
\newblock {\em Functional analysis}.
\newblock International Series in Pure and Applied Mathematics. McGraw-Hill
  Inc., New York, second edition, 1991.

\bibitem[Ser77]{Serre}
Jean-Pierre Serre.
\newblock {\em Linear representations of finite groups}.
\newblock Springer-Verlag, New York, 1977.
\newblock Translated from the second French edition by Leonard L. Scott,
  Graduate Texts in Mathematics, Vol. 42.

\bibitem[Str]{cgwh}
Neil Strickland.
\newblock The {C}ategory of {CGWH} {S}paces.
\newblock \url{http://math.mit.edu/~mbehrens/18.906/cgwh.pdf}.

\bibitem[Swa60]{swaninduced}
Richard~G. Swan.
\newblock Induced representations and projective modules.
\newblock {\em Ann. of Math. (2)}, 71:552--578, 1960.

\bibitem[Ull]{ull}
Mark Ullmann.
\newblock Controlled algebra for simplicial rings and the algebraic k-theory
  assembly map.
\newblock Phd thesis, 2011, available online at
  \url{http://docserv.uni-duesseldorf.de/servlets/DocumentServlet?id=17133}.

\bibitem[Val02]{valette}
Alain Valette.
\newblock {\em Introduction to the {B}aum-{C}onnes conjecture}.
\newblock Lectures in Mathematics ETH Z\"urich. Birkh\"auser Verlag, Basel,
  2002.
\newblock From notes taken by Indira Chatterji, With an appendix by Guido
  Mislin.

\bibitem[Wal85]{Waldhausen}
Friedhelm Waldhausen.
\newblock Algebraic {$K$}-theory of spaces.
\newblock In {\em Algebraic and geometric topology ({N}ew {B}runswick,
  {N}.{J}., 1983)}, volume 1126 of {\em Lecture Notes in Math.}, pages
  318--419. Springer, Berlin, 1985.

\bibitem[Weg12]{wegner}
Christian Wegner.
\newblock The {$K$}-theoretic {F}arrell-{J}ones conjecture for {CAT}(0)-groups.
\newblock {\em Proc. Amer. Math. Soc.}, 140(3):779--793, 2012.

\bibitem[WO93]{WO}
Niels~Erik Wegge-Olsen.
\newblock {\em {$K$}-theory and {$C^*$}-algebras}.
\newblock Oxford Science Publications. The Clarendon Press Oxford University
  Press, New York, 1993.
\newblock A friendly approach.

\end{thebibliography}
